\documentclass[english]{amsart}

\usepackage{mathrsfs}
\usepackage[francais,english]{babel}
\usepackage{amssymb,latexsym}

\usepackage{verbatim}
\usepackage{color}
\usepackage[pdftex]{graphicx}
\usepackage[pagebackref,pdftex,hyperindex]{hyperref}

\numberwithin{equation}{section}       
\numberwithin{figure}{section}       
\theoremstyle{plain}
\newtheorem{Thm}{Theorem}[section]
\newtheorem{Prop}[Thm]{Proposition}
\newtheorem{Lemma}[Thm]{Lemma}
\newtheorem{Cor}[Thm]{Corollary}
\newtheorem{Prop-def}[Thm]{Proposition-Definition}

\newtheorem*{Thmempty}{Theorem} 
\newtheorem*{ThmA}{Theorem A} 
\newtheorem*{ThmB}{Theorem B} 
 
\newtheorem*{ThmC}{Theorem C}

\theoremstyle{definition}
\newtheorem{Assum}[Thm]{Assumption}
\newtheorem{Example}[Thm]{Example}
\newtheorem{Exer}[Thm]{Exercise}
\newtheorem{Def}[Thm]{Definition}
\newtheorem{Remark}[Thm]{Remark}
\newtheorem*{Ackn}{Acknowledgments}

\newcommand{\A}{{\mathbf{A}}}
\newcommand{\C}{{\mathbf{C}}}
\newcommand{\D}{{\mathbf{D}}}
\newcommand{\F}{{\mathbf{F}}}
\newcommand{\N}{{\mathbf{N}}}
\renewcommand{\P}{{\mathbf{P}}}
\newcommand{\Q}{{\mathbf{Q}}}
\newcommand{\R}{{\mathbf{R}}}
\newcommand{\Z}{{\mathbf{Z}}}
\renewcommand{\H}{{\mathbf{H}}}

\newcommand{\bX}{{\bar{X}}}

\newcommand{\cB}{{\mathcal{B}}}

\newcommand{\cE}{{\mathcal{E}}}
\newcommand{\cF}{{\mathcal{F}}}

\newcommand{\cH}{{\mathcal{H}}}

\newcommand{\cJ}{{\mathcal{J}}}
\newcommand{\cL}{{\mathcal{L}}}
\newcommand{\cM}{{\mathcal{M}}}

\newcommand{\cO}{{\mathcal{O}}}

\newcommand{\cT}{{\mathcal{T}}}
\newcommand{\cV}{{\mathcal{V}}}

\newcommand{\hcO}{\hat{\mathcal{O}}}
\newcommand{\hcV}{\hat{\mathcal{V}}}
\newcommand{\tcV}{\tilde{\mathcal{V}}}

\newcommand{\hcVm}[1]{\hat{\mathcal{V}}_{\scriptscriptstyle{\subseteq{#1}}}}
\newcommand{\hcVp}[1]{\hat{\mathcal{V}}_{\scriptscriptstyle{\supseteq{#1}}}}
\newcommand{\hcVe}[1]{\hat{\mathcal{V}}_{\scriptscriptstyle{#1}}}
\newcommand{\cWm}[1]{\mathcal{W}_{\scriptscriptstyle{\subseteq{#1}}}}
\newcommand{\cWp}[1]{\mathcal{W}_{\scriptscriptstyle{\supseteq{#1}}}}
\newcommand{\cWe}[1]{\mathcal{W}_{\scriptscriptstyle{#1}}}
\newcommand{\cVe}[1]{\mathcal{V}_{\scriptscriptstyle{#1}}}
\newcommand{\tcVe}[1]{\tilde{\mathcal{V}}_{\scriptscriptstyle{#1}}}

\newcommand{\hKa}{\widehat{K^a}}

\newcommand{\fa}{{\mathfrak{a}}}
\newcommand{\fb}{{\mathfrak{b}}}

\newcommand{\fm}{{\mathfrak{m}}}
\newcommand{\fo}{{\mathfrak{o}}}
\newcommand{\fp}{{\mathfrak{p}}}

\newcommand{\fB}{{\mathfrak{B}}}

\newcommand{\fE}{{\mathfrak{E}}}

\newcommand{\fL}{{\mathfrak{L}}}
\newcommand{\fM}{{\mathfrak{M}}}

\newcommand{\hfa}{{\hat{\mathfrak{a}}}}

\newcommand{\ha}{{\hat{\alpha}}}

\newcommand{\hf}{{\hat{f}}}
\newcommand{\hD}{{\hat{D}}}
\newcommand{\hE}{{\hat{E}}}

\newcommand{\hL}{{\hat{L}}}
\newcommand{\hI}{{\hat{I}}}

\newcommand{\hx}{{\hat{x}}}
\newcommand{\hDelta}{{\hat{\Delta}}}

\newcommand{\hrho}{{\hat{\rho}}}
\newcommand{\hsigma}{{\hat{\sigma}}}

\newcommand{\vE}{{\check{E}}}

\newcommand{\vL}{{\check{L}}}

\newcommand{\ta}{{\tilde{a}}}
\newcommand{\tf}{{\tilde{f}}}

\newcommand{\tC}{{\tilde{C}}}

\newcommand{\tH}{{\tilde{H}}}
\newcommand{\tK}{{\tilde{K}}}

\newcommand{\tDelta}{{\tilde{\Delta}}}

\newcommand{\tphi}{{\tilde{\phi}}}
\newcommand{\tpsi}{{\tilde{\psi}}}

\newcommand{\simto}{\overset\sim\to}
\renewcommand{\=}{:=}
\renewcommand{\a}{\alpha}
\renewcommand{\b}{\beta}
\newcommand{\g}{\gamma}

\newcommand{\e}{\varepsilon}
\newcommand{\f}{\varphi}
\newcommand{\p}{\psi}

\newcommand{\eg}{e.g.\ }
\newcommand{\ie}{i.e.\ }
\newcommand{\cf}{cf.\ }
\newcommand{\loccit}{\textit{loc.\ cit.\ }}

\newcommand{\vv}{{\vec{v}}}
\newcommand{\ww}{{\vec{w}}}

\newcommand{\Aff}{\operatorname{Aff}}

\newcommand{\Aut}{\operatorname{Aut}}
\newcommand{\BerkA}{\operatorname{\A^1_{\mathrm{Berk}}}}
\newcommand{\BerkAone}{\operatorname{\A^1_{\mathrm{Berk}}}}
\newcommand{\BerkAtwo}{\operatorname{\A^2_{\mathrm{Berk}}}}
\newcommand{\BerkAn}{\operatorname{\A^n_{\mathrm{Berk}}}}
\newcommand{\BerkC}{\operatorname{C_{\mathrm{Berk}}}}
\newcommand{\BerkD}{\operatorname{\D_{\mathrm{Berk}}}}
\newcommand{\BerkDtwo}{\operatorname{\D^2_{\mathrm{Berk}}}}
\newcommand{\BerkDn}{\operatorname{\D^n_{\mathrm{Berk}}}}
\newcommand{\BerkP}{\operatorname{\P^1_{\mathrm{Berk}}}}
\newcommand{\BerkPone}{\operatorname{\P^1_{\mathrm{Berk}}}}
\newcommand{\BerkY}{\operatorname{Y_{\mathrm{Berk}}}}
\newcommand{\BDV}{\operatorname{BDV}}

\newcommand{\Div}{\operatorname{Div}}
\newcommand{\emb}{\operatorname{emb}}
\newcommand{\ev}{\operatorname{ev}}

\newcommand{\lcm}{\operatorname{lcm}}

\newcommand{\charac}{\operatorname{char}}

\newcommand{\Gal}{{\operatorname{Gal}}}

\newcommand{\Hom}{\operatorname{Hom}}

\newcommand{\ord}{\operatorname{ord}}

\newcommand{\id}{\operatorname{id}}

\newcommand{\Pic}{\operatorname{Pic}}

\newcommand{\QSH}{\operatorname{QSH}}
\newcommand{\ratrk}{\operatorname{rat.rk}}

\newcommand{\SH}{\operatorname{SH}}

\newcommand{\supp}{\operatorname{supp}}
\newcommand{\trdeg}{\operatorname{tr.deg}}
\newcommand{\triv}{\operatorname{triv}}

\newcommand{\Val}{\operatorname {Val}}

\newcommand{\llbracket}{[\negthinspace[}
\newcommand{\rrbracket}{]\negthinspace]}

 \author{Mattias Jonsson}

\address{Department of Mathematics \\
University of Michigan \\
530 Church Street, 2076 East Hall \\
Ann Arbor, MI 48109-1043 \\
USA}

\email{mattiasj@umich.edu}

\urladdr{www.math.lsa.umich.edu/~mattiasj/}

 \title{Dynamics on Berkovich spaces in low dimensions}

\begin{document}


\begin{abstract}
  These are expanded lecture notes for the summer school on 
  Berkovich spaces that took place at the 
  Institut de Math\'ematiques de Jussieu,
  Paris, during June 28--July 9, 2010. They serve to illustrate
  some techniques and results from the dynamics on 
  low-dimensional Berkovich spaces and to
  exhibit the 
  structure of these spaces.
\end{abstract}
 

%
%
%
%
%

 \maketitle
\setcounter{tocdepth}{1}
\tableofcontents


\section{Introduction}\label{S297}
The goal of these notes is twofold. 
First, I'd like to describe how Berkovich 
spaces enters naturally in certain instances of 
discrete dynamical systems. In particular,
I will try to show how my own work with 
Charles Favre~\cite{eigenval,dyncomp} on
valuative dynamics relates to the dynamics
of rational maps on the Berkovich projective line
as initiated by Juan Rivera-Letelier in his thesis~\cite{Rivera1}
and subsequently studied by him and others.
In order to keep the exposition somewhat focused,
I have chosen three sample problems (Theorems~A,~B and~C below)
for which I will present reasonably complete proofs.

The second objective is to show
some of the simplest Berkovich spaces 
``in action''. While not necessarily representative
of the general situation, they have a structure that 
is very rich, yet can be described in detail.
In particular, they are trees, or cones over trees.

\smallskip
For the purposes of this introduction, the dynamical 
problems that we shall be interested in all arise from 
polynomial mappings 
\begin{equation*}
  f:\A^n\to\A^n,
\end{equation*}
where $\A^n$ denotes affine $n$-space over a 
\emph{valued field}, that is, a field $K$ complete with respect
a norm $|\cdot|$.
Studying the dynamics of $f$ means, in rather vague terms,
studying the asymptotic behavior of the \emph{iterates} of $f$:
\begin{equation*}
  f^m=f\circ f\circ\dots\circ f
\end{equation*}
(the composition is taken $m$ times) as $m\to\infty$. For example,
one may try to identify regular as opposed to chaotic behavior.
One is also interested in invariant objects such as
fixed points, invariant measures, etc.

When $K$ is the field of complex numbers, polynomial mappings can 
exhibit very interesting dynamics both in one and 
higher dimensions. We shall discuss this a little further in~\S\ref{S106}
below. As references we point to~\cite{CG,Milnor} for the one-dimensional
case and~\cite{Sibony} for higher dimensions.

Here we shall instead focus on the case when the norm
on $K$ is \emph{non-Archimedean} in the sense that 
the strong triangle inequality $|a+b|\le\max\{|a|,|b|\}$
holds. Interesting examples of such fields include 
the $p$-adic numbers $\Q_p$, the field of Laurent series 
$\C((t))$, or any field $K$ equipped with the \emph{trivial} norm.

One motivation for investigating the dynamics of polynomial
mappings over non-Archimedean fields is simply to see to what 
extent the known results over the complex (or real) numbers
continue to hold. However, non-Archimedean dynamics 
sometimes plays a role even when the original dynamical 
system is defined over the complex numbers. We shall
see some instances of this phenomenon in these notes;
other examples are provided by the work of 
Kiwi~\cite{Kiwi1}, Baker and DeMarco~\cite{BdM},
and Ghioca, Tucker and Zieve~\cite{GTZ}.

Over the complex numbers, many of the most powerful tools
for studying dynamics are either topological or 
analytical in nature: distortion estimates, 
potential theory, quasiconformal mappings etc. 
These methods do not directly carry over to the 
non-Archimedean setting since $K$ is totally disconnected.

On the other hand, a polynomial mapping $f$ 
automatically induces a selfmap 
\begin{equation*}
  f:\BerkAn\to\BerkAn
\end{equation*}
of the corresponding \emph{Berkovich space} 
$\BerkAn$. 
By definition, $\BerkAn=\BerkAn(K)$ is the set of
multiplicative seminorms on the coordinate ring 
$R\simeq K[z_1,\dots,z_n]$ of $\A^n$ that
extend the given norm on $K$. It carries a 
natural topology in which it it locally compact and 
arcwise connected. It also contains a copy of $\A^n$:
a point $x\in\A^n$ is identified with the seminorm
$\phi\mapsto|\phi(x)|$.
The action of $f$ on $\BerkAn$ is given as follows.
A seminorm $|\cdot|$ is mapped by $f$ to the seminorm
whose value on a polynomial $\phi\in R$ is given by $|f^*\phi|$.

The idea is now to study the dynamics on $\BerkAn$. 
At this level of generality, not very much seems to be
known at the time of writing (although the time may be ripe to
start looking at this). 
Instead, the most interesting results have appeared in situations
when the structure of the space $\BerkAn$ is better understood,
namely in sufficiently low dimensions.

We shall focus on two such situations:
\begin{itemize}
\item[(1)]
  $f:\A^1\to\A^1$ is a polynomial mapping
  of the affine line over a general valued field $K$;
\item[(2)]
  $f:\A^2\to\A^2$ is a polynomial mapping
  of the affine plane over a field $K$
  equipped with the trivial norm.
\end{itemize}  
In both cases we shall mainly treat the case when $K$ is algebraically closed.

In~(1), one makes essential use of the fact that 
the Berkovich affine line $\BerkAone$ is a 
tree.\footnote{For a precise definition of what we mean by ``tree'',
 see~\S\ref{S110}.}
This tree structure was pointed out already by Berkovich 
in his original work~\cite{BerkBook} and 
is described in great detail in the
book~\cite{BRBook} by Baker and Rumely.
It has been exploited by several authors and a very nice 
picture of the global dynamics on this Berkovich space has taken shape.
It is beyond the scope of these notes to give an account of all 
the results that are known. 
Instead, we shall focus on one specific problem:
equidistribution of preimages of points. 
This problem, which will be discussed in further detail
in~\S\ref{S106}, clearly shows the advantage of working 
on the Berkovich space as opposed to the ``classical'' affine line.

As for~(2), the Berkovich affine plane $\BerkAtwo$ is already quite a beast,
but it is possible to get a handle on its structure. We shall be
concerned not with the global dynamics of $f$, but the local dynamics
either at a fixed point $0=f(0)\in\A^2$, or at infinity. There are natural 
subspaces of $\BerkAtwo$ consisting of seminorms that ``live''
at 0 or at infinity, respectively, in a sense that can be made precise. 
These two spaces are cones over a tree and hence reasonably tractable.

While it is of general interest to study the dynamics in~(2)
for a general field $K$, there are surprising
applications to \emph{complex} dynamics when 
using $K=\C$ equipped with the trivial norm. We shall discuss this
in~\S\ref{S109} and~\S\ref{S155} below.
%
%
%
%
\subsection{Polynomial dynamics in one variable}\label{S106}
Our first situation is that of a polynomial mapping 
\begin{equation*}
  f:\A^1\to\A^1
\end{equation*}
of degree $d>1$ over a complete valued field $K$, that we here 
shall furthermore assume to be algebraically closed and,
for simplicity, of characteristic zero.

When $K$ is equal to the (archimedean) field $\C$, there is a 
beautiful theory describing the polynomial dynamics.
The foundation of this theory was built in the 1920's by Fatou and
Julia, who realized that Montel's theorem could be 
used to divide the phase space $\A^1=\A^1(\C)$ into a region where
the dynamics is tame (the Fatou set) and a region where it
is chaotic (the Julia set). In the 1980's and beyond,
the theory was very
significantly advanced, in part because of computer technology
allowing people to visualize Julia sets as fractal objects, 
but more importantly because of the introduction of new tools, 
in particular quasiconformal mappings. For further information on this
we refer the reader to the books~\cite{CG,Milnor}.

In between, however, a remarkable result by Hans Brolin~\cite{Brolin}
appeared in the 1960's. His result seems to have gone
largely unnoticed at the time, but has been of great importance for
more recent developments, especially in higher dimensions. 
Brolin used potential theoretic methods to 
study the asymptotic distribution of preimages of points.
To state his result, let us introduce some terminology. Given a
polynomial mapping $f$ as above, one can consider the 
\emph{filled Julia set} of $f$, consisting of all 
points $x\in\A^1$ whose orbit is bounded. 
This is a compact set. 
Let $\rho_f$ be \emph{harmonic measure} on the filled
Julia set, in the sense of potential theory. 
Now, given a point $x\in\A^1$ we can look at the distribution of 
preimages of $x$ under $f^n$.
There are $d^n$ preimages
of $x$, counted with multiplicity, and we write 
$f^{n*}\delta_x=\sum_{f^ny=x}\delta_y$, where the sum is taken
over these preimages. Thus $d^{-n}f^{n*}\delta_x$ is a probability measure
on $\A^1$. Brolin's theorem now states
\begin{Thmempty}
  For all points $x\in\A^1$, with at most one exception,
  we have 
  \begin{equation*}
    \lim_{n\to\infty}d^{-n}f^{n*}\delta_x\to\rho_f.
  \end{equation*}
  Furthermore, a point $x\in\A^1$ is exceptional iff there exists 
 a global coordinate $z$ on $\A^1$ vanishing at $x$ such that
 $f$ is given by the polynomial $z\mapsto z^d$. In this case, 
 $d^{-n}f^{n*}\delta_x=\delta_x$ for all $n$.
\end{Thmempty}
A version of this theorem for selfmaps of $\P^1$
was later proved independently by Lyubich~\cite{Lyubich} and by 
Freire-Lopez-Ma\~n\'e~\cite{FLM}. There have also been far-reaching
generalizations of Brolin's theorem to higher-dimensional complex 
dynamics. However, we shall stick to the one-dimensional
polynomial case in this introduction. 

It is now natural to ask what happens when we replace
$\C$ by a \emph{non-Archimedean} valued field $K$.
We still assume that $K$ is algebraically closed
and, as above, that it is of characteristic zero.
An important example is $K=\C_p$, the completed
algebraic closure of the $p$-adic numbers $\Q_p$.
However, while most of the early work focused on $\C_p$, 
and certain deep results that are true for this field
do not hold for general $K$, we shall not assume $K=\C_p$
in what follows.

Early on, through work of Silverman, Benedetto, Hsia, Rivera-Letelier
and others~\cite{Benedetto1,Benedetto2,Benedetto3,Hsia,MortonSilverman,Rivera1}
it became clear that there were some significant
differences to the archimedean case.
For example, with the most direct translations of the definitions
from the complex numbers, 
it may well happen that the Julia set of a polynomial 
over a non-Archimedean field $K$ is empty.
This is in clear distinction with the complex case.
Moreover, the topological structure of $K$ is vastly 
different from that of $\C$. Indeed, $K$ is totally disconnected
and usually not even locally compact. 
The lack of compactness is inherited by 
the space of probability measures on $K$: there is a priori
no reason for the sequence of probability measures on $K$
to admit a convergent subsequence. This makes it unlikely
that a na{\"i}ve generalization of Brolin's theorem should hold.

Juan Rivera-Letelier was the first one to realize that Berkovich 
spaces could be effectively used to study the dynamics
of rational functions over non-Archimedean 
fields.
As we have seen above, $\A^1$ embeds naturally into 
$\BerkAone$ and the map $f$ extends to a map
\begin{equation*}
  f:\BerkAone\to\BerkAone.
\end{equation*}

Now $\BerkAone$ has good topological
properties. It is locally compact\footnote{Its one-point compactification is 
  the Berkovich projective line
  $\BerkPone=\BerkAone\cup\{\infty\}$.}
and contractible. This is true for the Berkovich affine space
$\BerkAn$ of any dimension. However, the structure
of the Berkovich affine $\BerkAone$ can be understood in 
much greater detail, and this is quite helpful when analyzing the dynamics.
Specifically, $\BerkAone$ has a structure of a tree
and the induced map $f:\BerkAone\to\BerkAone$ preserves
the tree structure, in a suitable sense.

Introducing the Berkovich space $\BerkAone$ is critical 
for even formulating many of the known results in 
non-Archimedean dynamics. This in particular applies to
the non-Archimedean version of Brolin's theorem:
\begin{ThmA}
  Let $f:\A^1\to\A^1$ be a polynomial map of degree $d>1$
  over an algebraically closed field of characteristic zero.
  Then there exists a probability measure $\rho=\rho_f$ on $\BerkAone$
  such that for all points $x\in\A^1$, with at most one exception,
  we have 
  \begin{equation*}
    \lim_{n\to\infty}d^{-n}f^{n*}\delta_x\to\rho.
  \end{equation*}
  Furthermore, a point $x\in\A^1$ is exceptional iff there exists 
  a global coordinate $z$ on $\A^1$ vanishing at $x$
  such that $f$ is given by the polynomial $z\mapsto z^d$. In this case, 
  $d^{-n}f^{n*}\delta_x=\delta_x$ for all $n$.
\end{ThmA}
In fact, we could have started with any point $x\in\BerkAone$ assuming
we are careful with the definition of $f^{n*}\delta_x$.
Notice that when $x\in\A^1$, the probability measures 
$d^{-n}f^{n*}\delta_x$ are all supported on $\A^1\subseteq\BerkAone$,
but the limit measure may very well give no mass to $\A^1$.
It turns out that if we define the Julia set $J_f$ of $f$ as the 
support of the measure $\rho_f$, then $J_f$ shares many 
properties of the Julia set of complex polynomials.
This explains why we may not see a Julia set when studying the 
dynamics on $\A^1$ itself.

Theorem~A is due to Favre and Rivera-Letelier~\cite{FR2}.
The proof is parallel to Brolin's original proof in that it
uses \emph{potential theory}. Namely, one can define a Laplace
operator $\Delta$ on $\BerkAone$ and to every probability 
measure $\rho$ on $\BerkAone$ associate a subharmonic 
function $\f=\f_\rho$ such that $\Delta\f=\rho-\rho_0$,
where $\rho_0$ is a fixed reference measure (typically a 
Dirac mass at a point of $\BerkAone\setminus\A^1$).
The function $\f$ is unique up to an additive constant.
One can then translate convergence of the measures in 
Theorem~A to the more tractable statement about convergence of potentials.
The Laplace operator itself can be very concretely interpreted
in terms of the tree structure on $\BerkAone$.
All of this will be explained in~\S\S\ref{S110}--\ref{S127}.

The story does not end with Theorem~A. For instance, Favre and
Rivera-Letelier analyze the ergodic properties of $f$ with respect to
the measure $\rho_f$. 
Okuyama has given a quantitative strengthening of 
the equidistribution result in Theorem~A, see~\cite{Oku11b}.
The measure $\rho_f$ also
describes the distribution of periodic points, 
see~\cite[Th\'eor\`eme~B]{FR2} as well as~\cite{Oku11a}.

As already mentioned, there is also a very 
interesting Fatou-Julia theory. We shall discuss this a little
further in~\S\ref{S156} but the discussion will be brief due to limited space.
The reader will find many more details in the book~\cite{BRBook}.
We also recommend  the recent survey by Benedetto~\cite{BenedettoNotes}.
%
%
%
%
\subsection{Local plane polynomial dynamics}\label{S109}
The second and third situations that we will study both deal with 
polynomial mappings
\begin{equation*}
  f:\A^2\to\A^2
\end{equation*}
over a valued field $K$. In fact, they originally arose 
from considerations in \emph{complex} dynamics
and give examples where non-Archimedean methods
can be used to study Archimedean problems.

Thus we start out by assuming that $K=\C$. Polynomial
mappings of $\C^2$ can have quite varied and very interesting
dynamics; see the survey by Sibony~\cite{Sibony} for some of
this. Here we will primarily consider local dynamics, so
we first consider a \emph{fixed point} $0=f(0)\in\A^2$.
For a detailed general discussion of local dynamics in this
setting we refer to Abate's survey~\cite{Abate}.

The behavior of $f$ at the fixed point is largely governed by
the tangent map $df(0)$ and in particular on the eigenvalues
$\lambda_1,\lambda_2$ of the latter. 
For example, if $|\lambda_1|,|\lambda_2|<1$,
then we have an \emph{attracting} fixed point: there exists a
small neighborhood $U\ni 0$ such that $f(\overline{U})\subseteq U$ and
$f^n\to0$ on $U$. Further, when there are no \emph{resonances}
between the eigenvalues $\lambda_1$, $\lambda_2$, the dynamics
can in fact be \emph{linearized}: there exists a local biholomorphism
$\phi:(\A^2,0)\to(\A^2,0)$  such that $f\circ\phi=\phi\circ\Lambda$,
where $\Lambda(z_1,z_2)=(\lambda_1z_1,\lambda_2z_2)$. 
This in particular gives very precise information on the 
rate at which typical orbits converge to the origin: for 
a ``typical'' point $x\approx 0$ we have 
$\|f^n(x)\|\sim\max_{i=1,2}|\lambda_i|^n\|x\|$ as $n\to\infty$.

On the other hand, in the \emph{superattracting} case,
when $\lambda_1=\lambda_2=0$, the action of $f$ on the tangent space
$T_0\C^2$ does not provide much information about the dynamics. 
Let us still try to understand at what rate orbits tend to the fixed point.
To this end, let 
\begin{equation*}
  f=f_c+f_{c+1}+\dots+f_d
\end{equation*}
be the expansion of $f$ in homogeneous 
components: $f_j(\lambda z)=\lambda^jf_j(z)$
and where $f_c\not\equiv0$. Thus $c=c(f)\ge 1$ and the number
$c(f)$ in fact does not depend on the choice of coordinates.
Note that for a typical point $x\approx 0$ we will have 
\begin{equation*}
  \|f(x)\|\sim\|x\|^{c(f)}.
\end{equation*}
Therefore, one expects that the speed at which the orbit of
a typical point $x$ tends to the origin is governed by
the growth of $c(f^n)$ as $n\to\infty$. This can in fact be
made precise, see~\cite{eigenval}, but here we shall only study the 
sequence $(c(f^n))_n$.

Note that this sequence is
supermultiplicative: $c(f^{n+m})\ge c(f^n)c(f^m)$.
This easily implies that the limit
\begin{equation*}
  c_\infty(f):=\lim_{n\to\infty}c(f^n)^{1/n}
\end{equation*}
exists. Clearly $c_\infty(f^n)=c_\infty(f)^n$ for $n\ge 1$.
\begin{Example}
  If $f(z_1,z_2)=(z_2,z_1z_2)$, then $c(f^n)$ is the $(n+2)$th Fibonacci number
  and $c_\infty(f)=\frac12(\sqrt{5}+1)$ is the golden mean.
\end{Example}
Our aim is to give a proof of the following result, originally
proved in~\cite{eigenval}.
\begin{ThmB}
  The number $c_\infty=c_\infty(f)$ is a quadratic integer: there exist
  $a,b\in\Z$ such that $c_\infty^2=ac_\infty+b$. 
  Moreover, there exists a constant 
  $\delta>0$ such that 
  \begin{equation*}
    \delta c_\infty^n\le c(f^n)\le c_\infty^n
  \end{equation*}
  for all $n\ge 1$.
\end{ThmB}
Note that the right-hand inequality $c(f^n)\le c_\infty^n$
is an immediate consequence of supermultiplicativity.
It is the left-hand inequality that is nontrivial. 

To prove Theorem~B we study the induced dynamics 
\begin{equation*}
  f:\BerkAtwo\to\BerkAtwo
\end{equation*}
of $f$ on the Berkovich affine plane 
$\BerkAtwo$.
Now, if we consider $K=\C$ with its standard Archimedean
norm, then it is a consequence of the Gelfand-Mazur theorem
that $\BerkAtwo\simeq\A^2$, so this may not seem like
a particularly fruitful approach. If we instead, however, 
consider $K=\C$ equipped with the \emph{trivial} norm, then the 
associated Berkovich affine plane $\BerkAtwo$ is a 
totally different creature and the induced dynamics is
very interesting.

By definition, the elements of $\BerkAtwo$ are 
multiplicative seminorms on the coordinate ring of $\A^2$, that is,
the polynomial ring $R\simeq K[z_1,z_2]$ in two variables
over $K$.  It turns out to be convenient to instead view these
elements ``additively'' as \emph{semivaluations}
$v:R\to\R\cup\{+\infty\}$ such that 
$v |_{K^*}\equiv0$. The corresponding seminorm 
is $|\cdot|=e^{-v }$.

Since we are interested in the local dynamics of $f$ near 
a (closed) fixed point $0\in\A^2$, we
shall study the dynamics of $f$ on a corresponding 
subspace of $\BerkAtwo$, namely 
the set $\hcV_0$ of semivaluations $v$ such that
$v(\phi)>0$ whenever $\phi$ vanishes at $0$.
In valuative terminology, these are the semivaluations
$v\in\BerkAtwo\setminus\A^2$  
whose \emph{center} on $\A^2$ is the point 0.
It is clear that $f(\hcV_0)\subseteq\hcV_0$.

Note that $\hcV_0$ has the structure of a cone: 
if $v\in\hcV_0$, then $tv\in\hcV_0$ for $0<t\le\infty$.
The apex of this cone is the image of the point $0\in\A^2$
under the embedding $\A^2\hookrightarrow\BerkAtwo$.
The base of the cone can be identified with the subset
$\cV_0\subseteq\hcV_0$ consisting of semivaluations that are
\emph{normalized} by the condition $v (\fm_0)=\min_{p\in\fm_0}v(\phi)=+1$, 
where $\fm_0\subseteq R$ denotes the maximal ideal of $0$.
This space $\cV_0$ is compact and has a structure
of an $\R$-tree. We call it the \emph{valuative tree} at the
point 0. Its structure is investigated in detail in~\cite{valtree} and will be examined
in~\S\ref{S103}.\footnote{In~\cite{valtree,eigenval}, the valuative tree is denoted by $\cV$. We write $\cV_0$ here in order to emphasize the choice of point $0\in\A^2$.} 

Now, $\cV_0$ is in general not invariant by $f$. Instead, $f$
induces a selfmap 
\begin{equation*}
  f_\bullet:\cV_0\to\cV_0
\end{equation*}
and a ``multiplier'' function $c(f,\cdot):\cV_0\to\R_+$ such that 
\begin{equation*}
  f(v)=c(f,v )f_\bullet v 
\end{equation*}
for $v\in\cV_0$. 
The number $c(f)$ above is exactly equal to $c(f,\ord_0)$,
where $\ord_0\in\cV_0$ denotes the order of vanishing at 
$0\in\A^2$. Moreover, we have 
\begin{equation*}
  c(f^n)=c(f^n,\ord_0)=\prod_{i=0}^{n-1}c(f,v _i),
  \quad\text{where $v _i=f^i_\bullet\ord_0$};
\end{equation*}
this equation will allow us to understand the
behavior of the sequence $c(f^n)$ through the 
dynamics of $f_\bullet$ on $\cV_0$.

The proof of Theorem~B given in these notes is simpler
than the one in~\cite{eigenval}. Here is the main idea. 
Suppose that there exists a valuation $v\in\cV_0$
such that $f_\bullet v =v$, so that $f(v)=cv$,
where $c=c(f,v )>0$. Then $c(f^n,v )=c^n$ for $n\ge 1$.
Suppose that $v$ satisfies an Izumi-type bound:
\begin{equation}\label{e163}
  v(\phi)\le C\ord_0(\phi)\quad\text{for all polynomials $\phi$},
\end{equation}
where $C>0$ is a constant independent of $\phi$. This is true
for many, but not all semivaluations $v\in\cV_0$. The 
reverse inequality $v\ge\ord_0$ holds for all 
$v\in\cV_0$ by construction. 
Then we have 
\begin{equation*}
  C^{-1}c^n=C^{-1}c(f^n,v )\le c(f^n)\le c(f^n,v )\le c^n.
\end{equation*}
This shows that $c_\infty(f)=c$ and that the bounds in Theorem~B
hold with $\delta=C^{-1}$. To see that $c_\infty$ is a quadratic
integer, we look at the value group $\Gamma_v$ of $v$.
The equality $f(v) =cv$ implies that 
$c\Gamma_v\subseteq\Gamma_v$.
If we are lucky, then $\Gamma\simeq\Z^d$, where $d\in\{1,2\}$,
which implies that $c_\infty=c$ is an algebraic integer of 
degree one or two. 

The two desired properties of $v$ hold when the eigenvaluation
$v$ is \emph{quasimonomial} valuation.
In general, there may not exist a quasimonomial eigenvaluation, so the argument
is in fact a little more involved. We refer to~\S\ref{S107} for more details.
%
%
%
%
\subsection{Plane polynomial dynamics at infinity}\label{S155}
Again consider a polynomial mapping 
\begin{equation*}
  f:\A^2\to\A^2
\end{equation*}
over the field $K=\C$ of complex numbers.
In the previous subsection, we discussed the dynamics of $f$ at 
a (superattracting) fixed point in $\A^2$. Now we shall
consider the dynamics at infinity and, specifically,
the rate at which orbits tend to infinity.
Fix an embedding 
$\A^2\hookrightarrow\P^2$. 
It is then reasonable to argue that the rate at which ``typical''
orbits tend to infinity is governed by the 
\emph{degree growth sequence} $(\deg f^n)_{n\ge1}$. 
Precise assertions to this end can be found
in~\cite{eigenval,dyncomp}.
Here we shall content ourselves with the study on the 
degree growth sequence.

In contrast to the local case, 
this sequence is \emph{sub}multiplicative:
$\deg f^{n+m}\le\deg f^n\deg f^m$, but again the limit
\begin{equation*}
  d_\infty(f):=\lim_{n\to\infty}(\deg f^n)^{1/n}
\end{equation*}
exists. 
Apart from some inequalities being reversed, the situation is very similar
to the local case, so one may hope for a direct analogue of Theorem~B
above. However, the skew product example 
$f(z_1,z_2)=(z_1^2,z_1z_2^2)$ shows that we may have $\deg f^n\sim nd_\infty^n$.
What does hold true in general is
\begin{ThmC}
  The number $d_\infty=d_\infty(f)$ is a quadratic integer: there exist
  $a,b\in\Z$ such that $d_\infty^2=ad_\infty+b$. 
  Moreover, we are in exactly one of the following two cases:
  \begin{itemize}
  \item[(a)]
    there exists $C>0$ such that 
    $d_\infty^n\le\deg f^n\le Cd_\infty^n$ for all $n$;
  \item[(b)]
    $\deg f^n\sim nd_\infty^n$ as $n\to\infty$.
  \end{itemize}
  Moreover, case~(b) occurs iff $f$, after conjugation 
  by a suitable polynomial automorphism of $\C^2$,
  is a skew product of the form
  \begin{equation*}
    f(z_1,z_2)=(\phi(z_1),\psi(z_1)z_2^{d_\infty}+O_{z_1}(z_2^{d_\infty-1})),
  \end{equation*}
  where $\deg\phi=d_\infty$ and $\deg\p>0$.
\end{ThmC}
As in the local case, we approach this theorem 
by considering the induced dynamics
\begin{equation*}
  f:\BerkAtwo\to\BerkAtwo,
\end{equation*}
where we consider $K=\C$ equipped with the trivial norm.
Since we are interested in the dynamics
of $f$ at infinity, we restrict our attention to the space $\hcVe{\infty}$
consisting of semivaluations $v:R\to\R\cup\{+\infty\}$ whose
center is at infinity, that is, for which $v(\phi)<0$ for some
polynomial $\phi$.
This space has the structure of a pointed\footnote{The apex of the cone does not define an element in $\BerkAtwo$.} cone.
To understand its base, note that our choice of embedding 
$\A^2\hookrightarrow\P^2$ determines the space $\cL$
of affine functions on $\A^2$ (the polynomials of degree 
at most one).
Define 
\begin{equation*}
  \cVe{\infty}:=\{v\in\BerkAtwo\mid\min_{L\in\cL}v (L)=-1\}.
\end{equation*}
We call $\cVe{\infty}$ the~\emph{valuative tree at infinity}.\footnote{In~\cite{eigenval,dyncomp}, the valuative tree at infinity is denoted by $\cV_0$, but the notation $\cVe{\infty}$ seems more natural.}
This subspace at first glance looks very similar to the valuative tree $\cV_0$ 
at a point but there are some important differences. 
Notably, for a semivaluation $v\in\cV_0$ we have $v(\phi)\ge 0$
for all polynomials $\phi$. In contrast, while a semivaluations in $\cVe{\infty}$
must take some negative values, it can take positive values on 
certain polynomials.

Assuming for simplicity that $f$ is \emph{proper}, we obtain
a dynamical system $f:\hcVe{\infty}\to\hcVe{\infty}$,
which we can split into an induced map 
$f_\bullet:\cVe{\infty}\to\cVe{\infty}$ and a multiplier
$d(f,\cdot):\cVe{\infty}\to\R_+$ such that 
$f(v)=d(f,v )f_\bullet v$.

The basic idea in the proof of Theorem~C is again to
look for an eigenvaluation, that is, a 
semivaluation $v\in\cVe{\infty}$ such that $f_\bullet v=v$.
However, even if we can find a ``nice'' (say,
quasimonomial) eigenvaluation, the proof in the local case
does not automatically go through. The reason is that 
Izumi's inequality~\eqref{e163} may fail.

The remedy to this problem is to use
an invariant subtree $\cV'_\infty\subseteq\cVe{\infty}$
where the Izumi bound almost always holds. 
In fact, the valuations $v\in\cV'_\infty$ for
which Izumi's inequality does not hold are
of a very special form, and the case 
when we end up with a fixed point of that type
corresponds exactly to the degree growth $\deg f^n\sim nd_\infty^n$.
In these notes, $\cV'_\infty$ is called the \emph{tight tree at infinity}.
I expect it to have applications beyond the situation here.
%
%
%
%
\subsection{Philosophy and scope}
When writing these notes I was faced with the question of
how much material to present, and at what level of detail to 
present it. Since I decided to have Theorems~A,~B and~C as goals for
the presentation, I felt it was necessary to provide enough background
for the reader to go through the proofs, without too many black boxes.
As it turns out, there is quite a lot of background to cover, so these
notes ended up rather expansive!

All the main results that I present here can be found in the literature,
However, we draw on many different sources that use different notation
and terminology. In order to make the presentation coherent, I have tried
to make it self-contained. 
Many complete proofs are included, others are sketched in reasonable detail.

While the point of these notes is to illustrate the usefulness of 
Berkovich spaces, we only occasionally draw on the general
theory as presented in~\cite{BerkBook,Berkihes}. As a general rule,
Berkovich spaces obtained by analytification of an algebraic variety
are much simpler than the ones constructed by gluing affinoid spaces.
Only at a couple of places in~\S\ref{S101} and~\S\ref{S156}
do we rely on (somewhat) nontrivial facts from the general theory.
On the other hand, these facts, mainly involving the local rings at a
point on the Berkovich space, are very useful. We try to exploit them 
systematically.
It is likely that in order to treat
higher-dimensional questions, one has to avoid simple
topological arguments based on the tree structure and instead use 
algebraic arguments involving the structure sheaf of the space in question.

At the same time, the tree structure of the spaces in question is 
of crucial importance. They can be viewed as the analogue of the conformal structure
on Riemann surfaces. For this reason I have included a self-contained presentation
of potential theory and dynamics on trees, at least to the extent that is
needed for the later applications in these notes.

\smallskip
I have made an attempt to provide a unified point of view of 
dynamics on low-dimensional Berkovich spaces. One can of 
course try to go further and study dynamics on 
higher-dimensional Berkovich spaces
over a field (with either trivial or nontrivial valuation). 
After all, there has been significant progress in higher dimensional 
complex dynamics 
over the last few years. For example, it is reasonable to hope 
for a version of the Briend-Duval equidistribution 
theorem~\cite{BriendDuval}.

\smallskip
Many interesting topics are not touched upon at all in these notes.
For instance, we say very little about the dynamics on, or the structure of the Fatou
set of a rational map and we likewise do not study the ramification locus.
Important contributions to these and other issues have been made by 
Matt Baker,
Robert Benedetto,
Laura DeMarco,
Xander Faber,
Charles Favre,
Liang-Chung Hsia,
Jan Kiwi,
Y\^usuke Okuyama,
Clayton Petsche,
Juan Rivera-Letelier,
Robert Rumely
Lucien Szpiro,
Michael Tepper,
Eugenio Trucco
and others.

For the relevant results we refer to the original 
papers~\cite{BdM,Bak06,Bak09,BH05,BR06,BenedettoThesis,Benedetto1,Benedetto2,Benedetto5,Benedetto6,Benedetto7,Benedetto4,Benedetto8,
Faber09,Faber1,Faber2,Faber3,FKT,FR3,FR1,FR2,Hsia,Kiwi1,Kiwi2,Oku11a,Oku11b,PST,Rivera1,Rivera2,Rivera3,Rivera4,Trucco}.
Alternatively, many of these results can be found in the book~\cite{BRBook} by 
Baker and Rumely or the lecture notes~\cite{BenedettoNotes} by Benedetto.

Finally, we say nothing about arithmetic aspects such as
the equidistribution of points of small 
height~\cite{BRBook,CL,FR1,Yuan,Gubler,Faber09,YZa,YZb}.
For an introduction to arithmetic dynamics, see~\cite{SilvBook} and~\cite{SilvNotes}.
%
%
%
%
\subsection{Comparison to other surveys}\label{S298}
Beyond research articles such as the ones mentioned above,
there are several useful sources that contain a 
systematic treatment of material related to the topics discussed in these notes.

First, there is a significant overlap between these notes and the material in the
\textit{Th\`ese d'Habilitation}~\cite{FavreHab} of Charles Favre.
The latter thesis, which is strongly recommended reading, explains
the usage of tree structures in dynamics and complex analysis.
It treats Theorems~A-C as well as some of my joint work with him on
the singularities of plurisubharmonic functions~\cite{pshsing,valmul}.
However, the presentation here has a different flavor and contains more details.

The book by~\cite{BRBook} by Baker and Rumely treats potential theory
and dynamics on the Berkovich projective line in great detail. The main results 
in~\S\S\ref{S101}--\ref{S127} are contained in this book, but the presentation 
in these notes is at times a little different. We also treat the case when the 
ground field has positive characteristic and discuss the case when it is
not algebraically closed and/or trivially valued. On the other hand,~\cite{BRBook} 
contains a great deal of material not covered here. 
For instance, it contains results on the structure of the Fatou and Julia sets 
of rational maps, and it gives a much more systematic treatment of potential
theory on the Berkovich line.

The lecture notes~\cite{BenedettoNotes} by Benedetto are also recommended reading. 
Just as~\cite{BRBook}, they treat the dynamics on the Fatou and Julia sets in detail. 
It also contains results in ``classical'' non-Archimedean analysis and dynamics, not involving Berkovich spaces.

The Ph.D.\ thesis by Amaury Thuillier~\cite{ThuillierThesis} gives a general treatment of potential theory on Berkovich curves. It is written in a less elementary way than the treatment in, say,~\cite{BRBook} but on the other hand is more amenable to generalizations to higher dimensions. 
Potential theory on curves is also treated in~\cite{Bak08}.

The valuative tree in~\S\ref{S103} is discussed in detail in the monograph~\cite{valtree}. However, the exposition here is self-contained and leads more directly to the dynamical applications that we have in mind.

As already mentioned, we do not discuss arithmetic dynamics in these notes.
For information on this fascinating subject we again refer to the book and lecture notes by 
Silverman~\cite{SilvBook,SilvNotes}.
%
%
%
%
\subsection{Structure}
The material is divided into three parts. 
In the first part,~\S\ref{S110}, we discuss trees
since the spaces on which we do dynamics
are either trees or cones over trees.
The second part,~\S\S\ref{S101}--\ref{S127},
is devoted to the Berkovich affine and 
projective lines and dynamics on them.
Finally, in~\S\S\ref{S105}--\ref{S104} we study 
polynomial dynamics on the Berkovich affine plane
over a trivially valued field.

\smallskip
We now describe the contents of each chapter in more detail.
Each chapter ends with a section called
``Notes and further references'' containing further comments.

\smallskip
In~\S\ref{S110} we gather some general definitions
and facts about trees. Since we shall work on several spaces
with a tree structure, I felt it made sense to collect 
the material in a separate section. See also~\cite{FavreHab}.
First  we define what we mean by a tree, with or without a 
metric. Then we define a Laplace operator on a general metric tree,
viewing the latter as a pro-finite tree.
In our presentation, the Laplace operator is defined on the class
of quasisubharmonic functions and takes values in
the space of signed measures with total mass zero and whose
negative part is a finite atomic measure. 
Finally we study maps between trees. It turns out that simply assuming 
that such a map is finite, open and surjective gives quite strong 
properties. We also prove a fixed point theorem for selfmaps of trees.

\smallskip
The structure of the Berkovich affine and projective lines
is outlined in~\S\ref{S101}. This material is described in 
much more detail in~\cite{BRBook}. One small way in which our
presentation stands out is that we try to avoid coordinates as
far as possible. We also point out some features of the local rings
that turn out to be useful for analyzing the mapping properties
and we make some comments about the case when the ground field is
not algebraically closed and/or trivially valued.

\smallskip
In~\S\ref{S156} we start considering rational maps.
Since we work in arbitrary characteristic, we include
a brief discussion of separable and purely inseparable maps.
Then we describe how polynomial and
rational maps extend to maps on the Berkovich
affine and projective line, respectively.
This is of course only a very special case of the 
analytification functor in the general theory
of Berkovich spaces, but it is useful to see in detail
how to do this. Again our approach differs slightly from
the ones in the literature that I am aware of, in that it
is coordinate free. Having extended a rational map
to the Berkovich projective line, we look at the 
important notion of the local degree at 
a point.\footnote{In~\cite{BRBook}, the local degree is called
\emph{multiplicity}.}
We adopt an algebraic definition of the local degree
and show that it can be interpreted as a local expansion 
factor in the hyperbolic metric. While this important
result is well known, we give an algebraic proof that I believe
is new. We also show that the local degree is the same as the
multiplicity defined by Baker and Rumely, using the Laplacian
(as was already known.)
See~\cite{Faber1,Faber2} for more on the local degree and the 
ramification locus, defined as the subset where the local degree 
is at least two. 
Finally, we discuss the case when the ground field is not algebraically closed
and/or is trivially valued.

\smallskip
We arrive at the dynamics on the Berkovich projective line 
in~\S\ref{S127}. Here we do not really try to survey the known
results. While we do discuss fixed points and the Fatou and Julia
sets, the exposition is very brief and the reader is encouraged
to consult the book~\cite{BRBook} by Baker and Rumely
or the notes~\cite{BenedettoNotes} by Benedetto for much more
information. Instead we focus on Theorem~A in the introduction,
the equidistribution theorem by Favre and Rivera-Letelier.
We give a complete proof which differs in the details from the
one in~\cite{FR2}.
We also give some consequences of the equidistribution theorem.
For example, we prove Rivera-Letelier's dichotomy that the Julia set is
either a single point or else a perfect set.
Finally, we discuss the case when the ground field is not algebraically closed
and/or is trivially valued.

\smallskip
At this point, our attention turns to the Berkovich 
affine plane over a trivially valued field. Here it seems
more natural to change from the multiplicative terminology
of seminorms to the additive notion of semivaluations.
We start in~\S\ref{S105} by introducing the home and the center
of a valuation. This allows us to stratify the Berkovich affine space.
This stratification is very explicit in dimension one, and 
possible (but nontrivial) to visualize in dimension two.  
We also introduce the important notion of a quasimonomial
valuation and discuss the Izumi-Tougeron inequality.

\smallskip
In~\S\ref{S103} we come to the valuative tree at a closed point 0. 
It is the same object as in the monograph~\cite{valtree} but here
it is defined as a subset of the Berkovich affine plane. 
We give a brief, but self-contained description
of its main properties with a presentation that is influenced
by my joint work with Boucksom and Favre~\cite{hiro,siminag,nama} in
hiugher dimensions. As before, our treatment is coordinate-free.
A key result is that the valuative tree at 0 is homeomorphic
to the inverse limit of the dual graphs 
over all birational morphisms above 0.
Each dual graph has a natural metric, so the valuative tree
is a pro-finite metric tree, and hence a metric tree in the sense
of~\S\ref{S110}. In some sense, the cone over the 
valuative tree is an even more natural object. We define a 
Laplace operator on the valuative tree that takes this 
fact into account. The subharmonic functions turn out to
be closely related to ideals in the ring of polynomials that are
primary to the maximal ideal at 0. In general, the geometry of 
blowups of the point~0 can be well understood and we exploit
this systematically.

\smallskip
Theorem~B is proved in~\S\ref{S107}. We give a proof that is
slightly different and shorter than the original one 
in~\cite{eigenval}. In particular, we have a significantly 
simpler argument for the fact that 
the number $c_\infty$ is a quadratic integer. 
The new argument makes more systematic use of the value groups
of valuations.

\smallskip
Next we move from a closed point in $\A^2$ to infinity.
The valuative tree at infinity was first defined in~\cite{eigenval}
and in~\S\ref{S108} we review its main properties. Just as in the
local case, the presentation is supposed to be self-contained and
also more geometric than in~\cite{eigenval}. There is a dictionary 
between the situation at a point and at infinity. For example,
a birational morphism above the closed point $0\in\A^2$
corresponds to a compactification of $\A^2$ and indeed,
the valuative tree at infinity is homeomorphic to the inverse limit
of the dual graphs of all (admissible) compactifications.
Unfortunately, the dictionary is not perfect, and there are many
subtleties when working at infinity. For example,
a polynomial in two variables tautologically 
defines a function on both the valuative tree at a point and at
infinity. At a point, this function is always negative but at
infinity, it takes on both positive and negative values.
Alternatively, the subtelties can be said to stem from the fact that the
geometry of compactifications of $\A^2$ can be much more
complicated than that of blowups of a closed point.

To remedy some shortcomings of the valuative tree at infinity, 
we introduce a subtree, the tight tree at infinity. It is an inverse
limit of dual graphs over a certain class of tight compactifications
of $\A^2$. These have much better properties than general 
compactifications and should have applications to other problems.
In particular, the nef cone of a tight compactification is always
simplicial, whereas the nef cone in general can be quite
complicated.

\smallskip
Finally, in~\S\ref{S104} we come to polynomial dynamics at infinity,
in particular the proof of Theorem~C. We follow the strategy of
the proof of Theorem~B closely, but we make sure to only use
tight compactifications. This causes some additional complications,
but we do provide a self-contained proof, that is
simpler than the one in~\cite{eigenval}.

%
%
%
%
\subsection{Novelties}
While most of the material here is known, certain proofs and ways
of presenting the results are new.

The definitions of a general tree in~\S\ref{S172} and
metric tree in~\S\ref{S131} are new, although equivalent to
the ones in~\cite{valtree}. The class of quasisubharmonic functions
on a general tree also seems new, as are the results in~\S\ref{S201}
on their singularities. The results on tree maps
in~\S\ref{S130} are new in this setting: they can be found
in \eg~\cite{BRBook} for rational maps on the 
Berkovich projective line.

Our description of the Berkovich affine and projective lines
is new, but only in the way that we insist on defining
things in a coordinate free way whenever possible. 
The same applies to the extension of a polynomial or
rational map from $\A^1$ or $\P^1$ to
$\BerkAone$ or $\BerkPone$, respectively.

While Theorem~\ref{T101}, expressing the local degree as a dilatation factor 
in the hyperbolic metric, is due to Rivera-Letelier, the proof
here is directly based on the definition of the local degree and
seems to be new. The remarks in~\S\ref{S254} on the non-algebraic case 
also seem to be new.

The structure of the Berkovich affine plane over a trivially valued
field, described in~\S\ref{S139} was no doubt known to experts but
not described in the literature. In particular, the valuative tree at
a closed point and at infinity were never explicitly identified
as subsets of the Berkovich affine plane.

Our exposition of the valuative tree differs from the treatment
in the book~\cite{valtree} and instead draws on the analysis of
the higher dimensional situation in~\cite{hiro}. 

The proof of Theorem~B in~\S\ref{S107} is new and somewhat simpler than the
one in~\cite{eigenval}. In particular, the fact that $c_\infty$ is
a quadratic integer is proved using value groups, whereas 
in~\cite{eigenval} this was done via rigidification. The same applies
to Theorem~C in~\S\ref{S104}.
%
%
%
%
\begin{Ackn}
I would like to express my 
gratitude to many people, first and foremost to Charles Favre
for a long and fruitful collaboration and 
without whom these notes would not exist.
Likewise, I have benefitted enormously from working with S\'ebastien
Boucksom. 
I thank Matt Baker for many interesting discussions;
the book by Matt Baker and Robert Rumely has also 
served as an extremely useful reference for dynamics 
on the Berkovich projective line.
I am grateful to Michael Temkin and Antoine Ducros for
answering various questions about Berkovich spaces
and to Andreas Blass for help with 
Remark~\ref{R206}
and
Example~\ref{E201}.
Conversations with Dale Cutkosky, William Gignac, Olivier Piltant and Matteo Ruggiero 
have also been very helpful, as have comments by Y\^{u}suke Okuyama.
Vladimir Berkovich of course deserves a special acknowledgment
as neither these notes nor the summer school itself would have been
possible without his work.
Finally I am grateful to the organizers and the sponsors 
of the summer school.
My research has been partially funded by 
grants DMS-0449465 and DMS-1001740  from the NSF.
\end{Ackn}

%
%
%
%
%
%
\newpage
\section{Tree structures}\label{S110}
We shall do dynamics on certain low-dimensional 
Berkovich spaces, or subsets thereof. In all cases, the space/subset
has the structure of a tree. Here we digress to discuss exactly
what we mean by this. We also present a general
version of potential theory on trees.
The definitions that follow are slightly different from,
but equivalent to the ones in~\cite{valtree,BRBook,FavreHab}, 
to which we refer for details.
The idea is that any two points in a 
tree should be joined by a unique interval.
This interval should look like a real interval but may or may not
be equipped with a distance function.
%
%
%
%
\subsection{Trees}\label{S172}
We start by defining a general notion of a tree. 
All our trees will be modeled on the real line 
(as opposed to a general ordered group $\Lambda$).\footnote{Our definition 
of ``tree'' is not the same as the one used in set theory~\cite{Jec03} but we
trust that no confusion will occur. The terminology  ``$\R$-tree'' 
would have been natural, but has already been reserved~\cite{GH}
for slightly different objects.}
In order to avoid technicalities,
we shall also only consider trees that are complete in the sense 
that they contain all their endpoints.
\begin{Def}
  An \emph{interval structure} on a set $I$
  is a partial order $\le$ on $I$ 
  under which $I$ becomes isomorphic 
  (as a partially ordered set) to the real interval $[0,1]$
  or to the trivial real interval $[0,0]=\{0\}$.
\end{Def}
Let $I$ be a set with an interval structure.
A \emph{subinterval} of $I$ is a subset $J\subseteq I$ 
that becomes a subinterval of $[0,1]$ or $[0,0]$ 
under such an isomorphism.
The \emph{opposite} interval structure on $I$ 
is obtained by reversing the partial ordering.
\begin{Def}\label{D201}
  A \emph{tree}\index{tree} is a set $X$
  together with the following data. For each $x,y\in X$, 
  there exists a subset $[x,y]\subseteq X$ containing $x$ and $y$
  and equipped with an interval structure. 
  Furthermore, we have:
  \begin{itemize}
  \item[(T1)]
    $[x,x]=\{x\}$;
  \item[(T2)]
    if $x\ne y$, then $[x,y]$ and $[y,x]$ are equal as subsets of $X$ 
    but equipped with opposite interval structures; they have 
    $x$ and $y$ as minimal elements, respectively;
  \item[(T3)]
    if $z\in[x,y]$ then $[x,z]$ and $[z,y]$ are subintervals of
    $[x,y]$ such that $[x,y]=[x,z]\cup[z,y]$ and $[x,z]\cap[z,y]=\{z\}$; 
  \item[(T4)]
    for any $x,y,z\in X$ there exists a unique element 
    $x\wedge_z y\in[x,y]$ 
    such that $[z,x]\cap[y,x]=[x\wedge_zy,x]$ 
    and $[z,y]\cap[x,y]=[x\wedge_zy,y]$;
  \item[(T5)]
    if $x\in X$ and $(y_\a)_{\a\in A}$ is a net in $X$ such that the segments
    $[x,y_\a]$ increase with $\a$, then there exists $y\in X$ such that 
    $\bigcup_\a[x,y_\a[\,=[x,y[$.
  \end{itemize}  
\end{Def}
In~(T5) we have used the convention $[x,y[\,:=[x,y]\setminus\{y\}$.
Recall that a \emph{net} is a sequence indexed by a directed 
(possibly uncountable) set. The subsets $[x,y]$ above will be
called \emph{intervals} or \emph{segments}.
%
%
\subsubsection{Topology}\label{S119}
A tree as above carries a natural \emph{weak topology}.
Given a point $x\in X$, define two points $y,z\in X\setminus\{x\}$
to be equivalent if $]x,y]\cap\,]x,z]\ne\emptyset$. An equivalence class
is called a \emph{tangent direction} at $x$ and the set of $y\in X$ 
representing a tangent direction $\vv$ is denoted $U(\vv)$.
The weak topology is generated by all such sets $U(\vv)$.
Clearly $X$ is arcwise connected and the connected components 
of $X\setminus\{x\}$ are exactly the sets $U(\vv)$ as $\vv$
ranges over tangent directions at $x$.
A tree is in fact uniquely arc connected in the sense that if 
$x\ne y$ and $\gamma:[0,1]\to X$ is an injective continuous map
with $\gamma(0)=x$, $\gamma(1)=y$, then the image of $\gamma$
equals $[x,y]$.
Since the sets $U(\vv)$ are connected, any point in $X$ admits a
basis of connected open neighborhoods.
We shall see shortly that $X$ is compact in the weak topology.

If $\gamma=[x,y]$ is a nontrivial interval, then the \emph{annulus}
$A(\gamma)=A(x,y)$ is defined by $A(x,y):=U(\vv_x)\cap U(\vv_y)$,
where $\vv_x$ (resp., $\vv_y$) is the tangent direction at 
$x$ containing $y$ (resp., at $y$ containing $x$).

An \emph{end} of $X$ is a point admitting a unique tangent
direction. A \emph{branch point} is a point having at least three
tangent directions.
%
%
\subsubsection{Subtrees}\label{S117}
A \emph{subtree} of a tree $X$ is a subset $Y\subseteq X$
such that the intersection $[x,y]\cap Y$ is either empty 
or a closed subinterval of $[x,y]$ for any $x,y\in X$.
In particular, if $x,y\in Y$, then $[x,y]\subseteq Y$ and 
this interval is then equipped with the 
same interval structure as in $X$. It is easy to see
that conditions~(T1)--(T5) are satisfied so that $Y$
is a tree.
The intersection of any collection of subtrees of $X$
is a subtree (if nonempty). The \emph{convex hull}
of any subset $Z\subseteq X$ is the 
intersection of all subtrees containing $Z$.

A subtree $Y$ is a closed subset of $X$
and the inclusion $Y\hookrightarrow X$ is an embedding.
We can define a \emph{retraction} $r:X\to Y$ as follows:
for $x\in X$ and $y\in Y$ the intersection 
$[x,y]\cap Y$ is an interval of the form $[r(x),y]$;
one checks that $r(x)$ does not depend on the choice of $y$.
The map $r$ is continuous and restricts to the identity on $Y$.
A subtree of $X$
is \emph{finite} if it is the convex hull of a finite set.

Let $(Y_\a)_{\a\in A}$ be an increasing net of finite 
subtrees of $X$, indexed by a directed 
set $A$ (\ie $Y_\a\subseteq Y_\b$ when $\a\le\b$).
Assume that the net is \emph{rich} 
in the sense that for any two distinct points $x_1,x_2\in X$  
there exists $\a\in A$ such that the retraction 
$r_\a:X\to Y_\a$ satisfies $r_\a(x_1)\ne r_\a(x_2)$. 
For example, $A$ could be the set of \emph{all} finite subtrees,
partially ordered by inclusion.
The trees $(Y_\a)$ form an 
inverse system via the retraction maps $r_{\a\b}:Y_\b\to Y_\a$ for $a\le\b$
defined by $r_{\a\b}=r_\a|_{Y_\b}$,
and we can form the inverse limit $\varprojlim Y_\a$,
consisting of points $(y_\a)_{\a\in A}$ in the product space 
$\prod_\a Y_\a$ such that $r_{\a\b}(y_\b)=y_\a$ for all $\a\le\b$. 
This inverse limit is a compact Hausdorff space.
Since $X$ retracts to each $Y_\a$ we get a continuous map
\begin{equation*}
  r:X\to\varprojlim Y_\a,
\end{equation*}
which is injective by the assumption that $A$ is rich.
That $r$ is surjective is a consequence of condition~(T5).
Let us show that the inverse of $r$ is also continuous. This
will show that $r$ is a homeomorphism, so that $X$ is compact.
(Of course, if we knew that $X$ was compact, the continuity
of $r^{-1}$ would be immediate.)

Fix a point $x\in X$ and a tangent direction $\vv$ at $x$.
It suffices to show that $r(U(\vv))$ is open in $\varprojlim Y_\a$.
Pick a sequence $(x_n)_{n\ge 1}$ in $U(\vv)$ such that
$[x_{n+1},x]\subseteq[x_n,x]$ and $\bigcap_n[x_n,x[\,=\emptyset$.
By richness there exists $\a_n\in A$ such that 
$r_{\a_n}(x_n)\ne r_{\a_n}(x)$. Let $\vv_n$ be the tangent direction 
in $X$ at $r_{\a_n}(x)$ represented by $r_{\a_n}(x_n)$.
Then $r(U(\vv_n))$ is open in $\varprojlim Y_\a$; hence so is 
$r(U(\vv))=\bigcup_n r(U(\vv_n))$.
\begin{Remark}\label{R102}
  One may form the inverse limit of any inverse system 
  of finite trees (not necessarily subtrees of a given tree).
  However, such an inverse limit may
  contain a ``compactified long line'' and hence not be a tree!
\end{Remark}
%
%
%
%
\subsection{Metric trees}\label{S131}
Let $I$ be a set with an interval structure. 
A \emph{generalized metric} on $I$ 
is a function $d:I\times I\to[0,+\infty]$ satisfying:
\begin{itemize}
  \item[(GM1)]
   $d(x,y)=d(y,x)$ for all $x,y$, and $d(x,y)=0$ iff $x=y$;
  \item[(GM2)]
   $d(x,y)=d(x,z)+d(z,y)$ whenever $x\le z\le y$
  \item[(GM3)]
    $d(x,y)<\infty$ if neither $x$ nor $y$ is an endpoint of $I$.
  \item[(GM4)]
    if $0<d(x,y)<\infty$, then for every $\e>0$ there exists $z\in I$
    such that $x\le z\le y$ and $0<d(x,z)<\e$.
 \end{itemize}    
 A \emph{metric tree} is a tree $X$ together with
 a choice of generalized metric on each interval $[x,y]$ in $X$
 such that whenever $[z,w]\subseteq[x,y]$, the inclusion
 $[z,w]\hookrightarrow[x,y]$ is an isometry in the obvious sense.

It is an interesting question whether or not 
every tree is \emph{metrizable} in the sense
that it can be equipped with a generalized metric. 
See Remark~\ref{R206} below.
%
%
\subsubsection{Hyperbolic space}
Let $X$ be a metric tree containing more than one point and
let $x_0\in X$ be a point that is not an end.
Define \emph{hyperbolic space} $\H$ to be the set of points $x\in X$
having finite distance from $x_0$. This definition 
does not depend on the choice of $x_0$.
Note that all points in $X\setminus\H$ are ends, but that 
some ends in $X$ may be contained in $\H$.

The generalized metric on $X$ 
restricts to a \emph{bona fide} metric on $\H$.
One can show that $\H$ is complete in this metric and that 
$\H$ is an $\R$-tree in the usual sense~\cite{GH}.
In general, even if $\H=X$, the topology generated by the metric
may be strictly stronger than the weak topology. In fact, the weak
topology on $X$ may not be metrizable. This happens, for example,
when there is a point with uncountable tangent space: such a point
does not admit a countable basis of open neighborhoods.
%
%
\subsubsection{Limit of finite trees}\label{S137}
As noted in Remark~\ref{R102}, the inverse limit
of finite trees may fail to be a tree. However, this 
cannot happen in the setting of metric trees.
A \emph{finite metric tree} is a finite tree equipped
with a generalized metric in which all distances are
finite.
Suppose we are given a directed set $A$, 
a finite metric tree $Y_\a$ for each $\a\in A$ and, for $\a\le\b$:
\begin{itemize}
\item
  an isometric embedding $\iota_{\b\a}:Y_\a\to Y_\b$; this means that 
  each interval in $Y_\a$ maps isometrically onto an interval in $Y_\b$;
\item
  a continuous map $r_{\a\b}:Y_\b\to Y_\a$ 
  such that $r_{\a\b}\circ\iota_{\b\a}=\id_{Y_\a}$
  and such that $r_{\a\b}$ maps each connected component of 
  $Y_\b\setminus Y_\a$ to a single point in $Y_\a$.
\end{itemize}
We claim that the space 
\begin{equation*}
  X:=\varprojlim_\a Y_\a
\end{equation*}
is naturally a metric tree.
Recall that $X$ is the set of points 
$(x_\a)_{\a\in A}$ in the product space $\prod_\a Y_\a$ such that 
$r_{\a\b}(x_\b)=x_\a$ for all $\a\le\b$. It is a compact Hausdorff space.
For each $\a$ we have an injective map $\iota_\a:Y_\a\to X$
mapping $x\in Y_\a$ to $(x_\b)_{\b\in A}$, where $x_\b\in Y_\b$ is
defined as follows: $x_\b=r_{\b\g}\iota_{\g\a}(x)$,
where $\g\in A$ dominates both $\a$ and $\b$.
Abusing notation, we view $Y_\a$ as a subset of $X$.
For distinct points $x,y\in X$ define 
\begin{equation*}
  [x,y]:=\{x\}\cup\bigcup_{\a\in A}[x_\a,y_\a]\cup\{y\}.
\end{equation*}
We claim that $[x,y]$ naturally carries an interval 
structure as well as a generalized metric. 
To see this, pick $\a_0$ such that $x_{\a_0}\ne y_{\a_0}$
and $z=(z_\a)\in\,]x_{\a_0},y_{\a_0}[$.
Then $d_\a(x_\a,z_\a)$ and $d_\a(y_\a,z_\a)$ are finite 
and increasing functions of $\a$, hence converge to 
$\delta_x,\delta_y\in[0,+\infty]$, respectively. 
This gives rise to an isometry of $[x,y]$ onto the interval
$[-\delta_x,\delta_y]\subseteq[-\infty,+\infty]$.
%
%
%
%
\subsection{Rooted and parametrized trees}\label{S111}
Sometimes there is a point in a tree that plays a special role.
This leads to the following notion.
\begin{Def}
  A \emph{rooted tree} is a partially ordered set $(X,\le)$
  satisfying the following properties:
  \begin{itemize}
  \item[(RT1)]
    $X$ has a unique minimal element $x_0$;
  \item[(RT2)]
    for any $x\in X\setminus\{x_0\}$, the set $\{z\in X\mid z\le x\}$ 
    is isomorphic (as a partially ordered set) to the real interval $[0,1]$;
  \item[(RT3)]
    any two points $x,y\in X$ admit an infimum $x\wedge y$ in $X$,
    that is, $z\le x$ and $z\le y$ iff $z\le x\wedge y$;
  \item[(RT4)]
    any totally ordered subset of $X$ has a least upper bound in $X$.
  \end{itemize}
\end{Def}
Sometimes it is natural to reverse the partial ordering so that 
the root is the unique \emph{maximal} element.
\begin{Remark}\label{R201}
  In~\cite{valtree} it was claimed that~(RT3) follows from the other 
  three axioms but this is not true. A counterexample is provided 
  by two copies of the interval $[0,1]$ identified along the half-open
  subinterval $[0,1[\,$. I am grateful to 
  Josnei Novacoski and Franz-Viktor Kuhlmann for pointing this out.
\end{Remark}
Let us compare this notion with the definition of a tree above.
If $(X,\le)$ is a rooted tree, then we can define intervals
$[x,y]\subseteq X$ as follows. First, when $x\le y\in X$,
set $[x,y]:=\{z\in X\mid x\le z\le y\}$ and $[y,x]:=[x,y]$.
For general $x,y\in X$ set $[x,y]:=[x\wedge y,x]\cup[x\wedge y,y]$.
We leave it to the reader to equip 
$[x,y]$ with an interval structure and to verify conditions~(T1)--(T5).
Conversely, given a tree $X$ and a point $x_0\in X$, 
define a partial ordering on $X$ by declaring
$x\le y$ iff  $x\in [x_0,y]$.
One checks that conditions~(RT1)--(RT4) are verified.

A \emph{parametrization} of a rooted tree $(X,\le)$ as above
is a monotone function
$\a:X\to[-\infty,+\infty]$ whose restriction to any 
segment $[x,y]$ with $x<y$ is a homeomorphism onto 
a closed subinterval of $[-\infty,+\infty]$. 
We also require $|\a(x_0)|<\infty$ unless $x_0$ is an endpoint 
of $X$. 
This induces a generalized metric on $X$ by setting
\begin{equation*}
  d(x,y)=|\a(x)-\a(x\wedge y)|+|\a(y)-\a(x\wedge y)|
\end{equation*}
for distinct points $x,y\in X$. The set $\H$ is exactly the locus where $|\a|<\infty$.
Conversely given a generalized metric $d$ on a tree $X$, a point 
$x_0\in\H$ and a real number $\a_0\in\R$, we 
obtain an increasing parametrization 
$\a$ of the tree $X$ rooted in $x_0$ by setting 
$\a(x)=\a_0+d(x,x_0)$.
\begin{Remark}\label{R206}
  A natural question is whether or not every rooted tree admits a parametrization.
  In personal communication to the author, Andreas Blass has outlined
  an example of a rooted tree that cannot be parametrized. His construction
  relies on Suslin trees~\cite{Jec03}, the existence of which cannot be 
  decided from the ZFC axioms. It would be interesting to have a more
  explicit example.
\end{Remark}
%
%
%
%
\subsection{Radon measures on trees}\label{S203}
Let us review the notions of Borel and Radon measures 
on compact topological spaces and, more specifically,
on trees.
\subsubsection{Radon and Borel measures on compact spaces}\label{S204}
A reference for the material in this section is~\cite[\S7.1-2]{Folland}.
Let $X$ be a compact (Hausdorff) space and $\cB$ the associated
Borel $\sigma$-algebra.
A \emph{Borel measure}  on $X$ is a function $\rho:\cB\to[0,+\infty]$
satisfying the usual axioms. A Borel measure $\rho$ is \emph{regular} if
for every Borel set $E\subseteq X$ and every $\e>0$ there exists
a compact set $F$ and an open set $U$ such that 
$F\subseteq E\subseteq U$ and $\rho(U\setminus F)<\e$.

A \emph{Radon measure} on $X$ is a positive linear
functional on the vector space $C^0(X)$ of continuous functions on $X$.
By the Riesz representation theorem, Radon measures can be identified
with regular Borel measures.

If $X$ has the property that every open set of $X$ is $\sigma$-compact, 
that is, a countable union
of compact sets, then every Borel measure on $X$ is Radon.
However, many Berkovich spaces do not have this property.
For example, the Berkovich projective line over any non-Archimedean
field $K$ is a tree, but if the residue field of $K$ is 
uncountable, then the complement of any Type~2 point 
(see~\S\ref{S167}) is an open set that is not $\sigma$-compact.

We write $\cM^+(X)$ for the set of positive Radon measures
on $X$ and endow it with the topology of weak (or vague) convergence.
By the Banach-Alaoglu Theorem, the subspace $\cM^+_1(X)$ of 
Radon probability measure is compact.

A \emph{finite atomic measure} on $X$ is a Radon measure of the form 
$\rho=\sum_{i=1}^Nc_i\delta_{x_i}$, where $c_i>0$.
A \emph{signed} Radon measure is a real-valued linear functional on
$C^0(X;\R)$. The only signed measures that we shall consider will
be of the form $\rho-\rho_0$, where $\rho$ is a Radon measure 
and $\rho_0$ a finite atomic measure.

%
%
\subsubsection{Measures on finite trees}
Let $X$ be a \emph{finite} tree. It is then easy to see
that every connected open set is of the form 
$\bigcap_{i=1}^n U(\vv_i)$, where
$\vv_1,\dots,\vv_n$ are tangent directions in $X$
such that $U(\vv_i)\cap U(\vv_j)\ne\emptyset$ 
but $U(\vv_i)\not\subseteq U(\vv_j)$ for $i\ne j$.
Each such set is a countable union of compact subsets,
so it follows from the above that every Borel measure
is in fact a Radon measure.
%
%
\subsubsection{Radon measures on general trees} 
Now let $X$ be an arbitrary tree in the sense of
Definition~\ref{D201}. It was claimed in~\cite{valtree}
and~\cite{BRBook} that in this case, too, every Borel measure
is Radon, but there is a gap in the proofs. 
\begin{Example}\label{E201}
  Let $Y$ be a set with the following property: there exists
  a probability measure $\mu$ on the maximal $\sigma$-algebra
  (that contains all subsets of $Y$) that gives zero mass to any
  finite set. The existence of such a set, whose cardinality is 
  said to be a \emph{real-valued measurable cardinal} 
  is a well known problem in set theory~\cite{Fremlin}: 
  suffice it to say that its
  existence or nonexistence cannot be decided from the ZFC axioms.
  Now equip $Y$ with the discrete topology and 
  let $X$ be the cone over $Y$, that is
  $X=Y\times[0,1]/\sim$, where $(y,0)\sim(y',0)$ for all $y,y'\in Y$.
  Let $\phi:Y\to X$ be the continuous map defined by $\phi(y)=(y,1)$. 
  Then $\rho:=\phi_*\mu$ is a Borel measure on $X$
  which is not Radon. Indeed, the open set $U:=X\setminus\{0\}$ has 
  measure 1, but any compact subset of $U$ is contained in a 
  \emph{finite} union of intervals $\{y\}\times\,]0,1]$ and thus 
  has measure zero.
\end{Example}
Fortunately, this does not really lead to any problems. The message to
take away is that on a general tree, one should systematically use 
Radon measures, and this is indeed what we shall do here.
%
%
\subsubsection{Coherent systems of measures}
The description of a general tree $X$ as a pro-finite tree 
is well adapted to describe Radon measures on $X$.
Namely, let $(Y_\a)_{\a\in A}$ be a rich net of finite subtrees
of $X$, in the sense of~\S\ref{S117}.
The homeomorphism $X\simto\varprojlim Y_\a$ then 
induces a homeomorphism $\cM^+_1(X)\simto\varprojlim\cM^+_1(Y_\a)$.
Concretely, the right hand side consists of 
collections $(\rho_\a)_{\a\in A}$ of Radon measures on 
each $Y_\a$ satisfying $(r_{\a\b})_*\rho_\b=\rho_\a$ for $\a\le\b$.
Such a collection of measures is called a 
\emph{coherent system of measures} in~\cite{BRBook}.
The homeomorphism above assigns to a Radon probability measure
$\rho$ on $X$
the collection $(\rho_\a)_{\a\in A}$ defined by $\rho_\a:=(r_\a)_*\rho$.
%
%
%
%
\subsection{Potential theory}\label{S116}
Next we outline how to do potential theory
on a metric tree. 
The presentation is adapted to our needs but basically 
follows~\cite{BRBook}, especially~\S1.4 and~\S2.5.
The Laplacian on a tree is a combination of the 
usual real Laplacian with the combinatorially defined 
Laplacian on a simplicial tree.
%
%
\subsubsection{Quasisubharmonic functions on finite metric trees}   
Let $X$ be a \emph{finite} metric tree.
The Laplacian $\Delta$ on $X$ is naturally defined on
the class $\BDV(X)\subseteq C^0(X)$ 
of functions with bounded differential variation, 
see~\cite[\S3.5]{BRBook},
but we shall restrict our attention to the subclass
$\QSH(X)\subseteq\BDV(X)$ of
\emph{quasisubharmonic} functions.

Let $\rho_0=\sum_{i=1}^Nc_i\delta_{x_i}$ 
be a finite atomic measure on $X$.
Define the class
$\SH(X,\rho_0)$ of \emph{$\rho_0$-subharmonic} functions
as the set of continuous functions 
$\f$ that are convex on any segment 
disjoint from the support of $\rho_0$
and such that, for any $x\in X$:
\begin{equation*}
  \rho_0\{x\}+\sum_\vv D_\vv\f\ge0,
\end{equation*}
where the sum is over all tangent directions $\vv$ at $x$.
Here $D_\vv\f$ denotes the directional derivative of 
$\f$ in the direction $\vv$ (outward from $x$): 
this derivative is well defined by the convexity of $\f$.
We leave it to the reader to verify that 
\begin{equation}\label{e210}
  D_\vv\f\le0
  \quad\text{whenever}\quad\rho_0(U(\vv))=0
\end{equation}
for any $\f\in\SH(X,\rho_0)$; this inequality is quite useful.

Define $\QSH(X)$ as the union of $\SH(X,\rho_0)$ 
over all finite atomic measures $\rho_0$. 
Note that if $\rho_0$, $\rho_0'$ are two finite atomic 
measures with $\rho'_0\ge\rho_0$, then 
$\SH(X,\rho_0)\subseteq\SH(X,\rho'_0)$.
We also write $\SH(X,x_0):=\SH(X,\delta_{x_0})$ and refer
to its elements as \emph{$x_0$-sub\-harmonic}.

Let $Y\subseteq X$ be a subtree of $X$ containing 
the support of $\rho_0$.
We have an injection $\iota:Y\hookrightarrow X$ 
and a retraction $r:X\to Y$. 
It follows easily from~\eqref{e210} that 
\begin{equation*}
  \iota^*\SH(X,\rho_0)\subseteq\SH(Y,\rho_0)
  \quad\text{and}\quad
  r^*\SH(Y,\rho_0)\subseteq\SH(X,\rho_0).
\end{equation*}
Moreover, $\f\le r^*\iota^*\f$ 
for any $\f\in\SH(X,\rho_0)$.
%
%
\subsubsection{Laplacian}\label{S223}
For $\f\in\QSH(X)$, 
let $\Delta\f$ be the signed (Borel) measure
on $X$ defined as follows:
if $\vv_1,\dots,\vv_n$ are tangent directions in $X$
such that $U(\vv_i)\cap U(\vv_j)\ne\emptyset$ 
but $U(\vv_i)\not\subseteq U(\vv_j)$ for $i\ne j$,
then 
\begin{equation*}
 \Delta\f(\bigcap_{i=1}^nU(\vv_i))=\sum_{i=1}^nD_{\vv_i}\f.
\end{equation*}
This equation defines $\Delta\f$ uniquely as 
every open set in $X$ is a countable disjoint union of
open sets of the form $\bigcap U(\vv_i)$. 
The mass of $\Delta\f$ at a point $x\in X$ is given by
$\sum_{\vv\in T_x}D_\vv\f$ and the restriction of $\Delta\f$
to any open segment $I\subseteq X$ containing no branch point
is equal to the usual real Laplacian of $\f|_I$.

The Laplace operator is essentially injective. 
Indeed, suppose $\f_1,\f_2\in\QSH(X)$ 
and $\Delta\f_1=\Delta\f_2$. 
We may assume $\f_1,\f_2\in\SH(X,\rho_0)$
for a common positive measure $\rho_0$.
If $\f=\f_1-\f_2$, then 
$\f$ is affine on any closed interval
whose interior is disjoint from the support of $\rho_0$.
Moreover, at any point $x\in X$ we have 
$\sum_{\vv\in T_x}D_\vv\f=0$. These two conditions
easily imply that $\f$ is constant.
(Too see this, first check that $\f$ is locally constant at any end of $X$.)

If $\f\in\SH(X,\rho_0)$, then $\rho_0+\Delta\f$ is a 
positive Borel measure on $X$ of the same mass as $\rho_0$.
In particular, when $\rho_0$ is a probability measure, we 
obtain a map
\begin{equation}\label{e134}
  \SH(X,\rho_0)\ni\f\mapsto\rho_0+\Delta\f\in\cM^+_1(X),
\end{equation}
where $\cM^+_1(X)$ denotes the set of probability measures on $X$.
We claim that this map is surjective. 
To see this, first note that the function $\f_{y,z}$ given by
\begin{equation}\label{e122}
  \f_{y,z}(x)=-d(z,x\wedge_z y),
\end{equation}
with $x\wedge_z y\in X$ as in~(T4),
belongs to $\SH(X,z)$ and satisfies 
$\Delta\f=\delta_y-\delta_z$.
For a general probability measure $\rho$
and finite atomic probability measure $\rho_0$,
the function
\begin{equation}\label{e137}
  \f(x)=\iint\f_{y,z}(x)\,d\rho(y)d\rho_0(z)
\end{equation}
belongs to $\SH(X,\rho_0)$ and 
satisfies $\Delta\f=\rho-\rho_0$.

Let $Y\subseteq X$ be a subtree containing the support of
$\rho_0$ and denote the 
Laplacians on $X$ and $Y$ by $\Delta_X$ and $\Delta_Y$, respectively.
Then, with notation as above, 
\begin{align}
  \Delta_Y(\iota^*\f)=r_*(\Delta_X\f)
  \quad\text{for $\f\in\SH(X,\rho_0)$}\label{e120}\\
  \Delta_X(r^*\f)=\iota_*(\Delta_Y\f)
  \quad\text{for $\f\in\SH(Y,\rho_0)$}\label{e121},
\end{align}
where $\iota:Y\hookrightarrow X$ and $r:X\to Y$
are the inclusion and retraction, respectively. 
%
%
\subsubsection{Equicontinuity}
The spaces $\SH(X,\rho_0)$ have very nice compactness properties
deriving from the fact that if $\rho_0$ is a probability measure then
\begin{equation}\label{e138}
  |D_\vv\f|\le 1
  \quad\text{for all tangent directions $\vv$ and all $\f\in\SH(X,\rho_0)$}.
\end{equation}
Indeed, using the fact that a function in $\QSH(X)$ is determined, up to an
additive constant, by its Laplacian~\eqref{e138}
follows from~\eqref{e122}
when $\rho_0$ and $\rho_0+\Delta\f$ are Dirac masses, 
and from~\eqref{e137} in general.

As a consequence of~\eqref{e138}, the functions in 
$\SH(X,\rho_0)$ are uniformly Lipschitz continuous and in particular
equicontinuous. This shows that pointwise convergence in $\SH(X,\rho_0)$  
implies uniform convergence.

The space $\SH(X,\rho_0)$ is easily seen to be closed in the 
$C^0$-topology, so we obtain several 
compactness assertions from the Arzela-Ascoli theorem.
For example, the set $\SH^0(X,\rho_0)$ of
$\f\in\SH(X,\rho_0)$ for which $\max\f=0$
is compact.

Finally, we have an exact sequence of topological vector spaces
\begin{equation}\label{e139}
  0
  \to\R
  \to\SH(X,\rho_0)
  \to\cM^+_1(X)
  \to0;
\end{equation}
here $\cM^+_1(X)$ is equipped with the weak topology on measures.
Indeed, the construction in~\eqref{e122}-\eqref{e137} gives rise to 
a continuous bijection between $\cM^+_1(X)$
and $\SH(X,\rho_0)/\R\simeq\SH^0(X,\rho_0)$. By compactness,
the inverse is also continuous.
%
%
\subsubsection{Quasisubharmonic functions on general metric trees}\label{S122}
Now let $X$ be a general metric tree and $\rho_0$
a finite atomic measure supported on the associated
hyperbolic space $\H\subseteq X$.

Let $A$ be the set of finite metric subtrees of $X$ that 
contain the support of $\rho_0$. This is a directed set, partially ordered by 
inclusion.
For $\a\in A$, denote the associated metric tree by $Y_\a$.
The net $(Y_\a)_{\a\in A}$ is rich in the sense of~\S\ref{S117},
so the retractions $r_\a:X\to Y_\a$ induce a homeomorphism
$r:X\simto\varprojlim Y_\a$. 

Define $\SH(X,\rho_0)$ to be the set of functions 
$\f:X\to[-\infty,0]$ such that 
$\f|_{Y_\a}\in\SH(Y_\a,\rho_0)$ for all $\a\in A$
and such that $\f=\lim r_\a^*\f$.
Notice that in this case $r_\a^*\f$ in fact decreases to $\f$.
Since $r_\a^*\f$ is continuous for all $\a$, this implies
that $\f$ is upper semicontinuous.

We define the topology on $\SH(X,\rho_0)$ in terms of pointwise
convergence on $\H$. Thus a net $\f_i$ converges
to $\f$ in $\SH(X,\rho_0)$ iff $\f_i|_{Y_\a}$ converges 
to $\f|_{Y_\a}$ for all $\a$. 
Note, however,  that the convergence $\f_i\to\f$ is not required 
to hold on all of $X$.
 
Since, for all $\a$, $\SH(Y_\a,\rho_0)$ is compact in the topology
of pointwise convergence on $Y_\a$, 
it follows that $\SH(X,\rho_0)$ is also compact. 
The space $\SH(X,\rho_0)$ has many nice properties beyond
compactness. 
For example, if $(\f_i)_i$ is a decreasing 
net in $\SH(X,\rho_0)$, and $\f:=\lim\f_i$,
then either $\f_i\equiv-\infty$ on $X$ 
or $\f\in\SH(X,\rho_0)$. 
Further, if $(\f_i)_i$ is a family in $\SH(X,\rho_0)$
with $\sup_i\max_X\f_i<\infty$, 
then the upper semicontinuous regularization of 
$\f:=\sup_i\f_i$ belongs to $\SH(X,\rho_0)$.

As before, we define $\QSH(X)$, the space of \emph{quasisubharmonic functions},
to be the union of $\SH(X,\rho_0)$ over all finite atomic measures $\rho_0$
supported on $\H$.
%
%
\subsubsection{Laplacian}
Let $X$, $\rho_0$ and $A$ be as above.
Recall that a Radon probability measure $\rho$ on $X$
is given by a coherent system $(\rho_\a)_{\a\in A}$ of
(Radon) probability measures on $Y_\a$.

For $\f\in\SH(X,\rho_0)$ we define 
$\rho_0+\Delta\f\in\cM^+_1(X)$ to be the 
unique Radon probability measure such that 
\begin{equation*}
  (r_\a)_*(\rho_0+\Delta\f)=\rho_0+\Delta_{Y_\a}(\f|_{Y_\a})
\end{equation*}
for all $\a\in A$. This makes sense in view of~\eqref{e120}.

The construction in~\eqref{e122}-\eqref{e137} remains valid
and the sequence~\eqref{e139} of topological vector spaces is exact.
For future reference we record that if 
$(\f_i)_i$ is a net in $\SH^0(X,\rho_0)$, then 
$\f_i\to 0$ (pointwise on $\H$)  iff $\Delta\f_i\to0$ 
in $\cM^+_1(X)$.
%
%
\subsubsection{Singularities of quasisubharmonic functions}\label{S201}
Any quasisubharmonic function on a metric tree $X$ is bounded from above
on all of $X$ and Lipschitz continuous on hyperbolic space $\H$, but can
take the value $-\infty$ at infinity.
For example, if $x_0\in\H$ and $y\in X\setminus\H$,
then the function $\f(x)=-d_\H(x_0,x\wedge_{x_0} y)$ is $x_0$-subharmonic
and $\f(y)=-\infty$. 
Note that $\Delta\f=\delta_y-\delta_{x_0}$. The following
result allows us to estimate a quasisubharmonic function from 
below in terms of the mass of its Laplacian at infinity.
It will be used in the proof of the equidistribution result in~\S\ref{S301}.
\begin{Prop}\label{P301}
  Let $\rho_0$ be a finite atomic probability measure on $\H$ and 
  let $\f\in\SH(X,\rho_0)$. Pick $x_0\in\H$ and any number 
  $\lambda>\sup_{y\in X\setminus\H}\Delta\f\{y\}$.
  Then there exists a constant $C=C(x_0,\rho_0,\f,\lambda)>0$ such that 
  \begin{equation*}
    \f(x)\ge\f(x_0)-C-\lambda d_\H(x,x_0)
  \end{equation*}
  for all $x\in\H$.
\end{Prop}
We shall use the following estimates, which are of independent interest.
\begin{Lemma}\label{L212}
  Let $\rho_0$ be a finite atomic probability measure on $\H$ and let $x_0\in\H$.
  Pick $\f\in\SH(X,\rho_0)$ and set $\rho=\rho_0+\Delta\f$. Then 
  \begin{equation*}
    \f(x)-\f(x_0)
    \ge-\int_{x_0}^x\rho\{z\ge y\}d\a(y)
    \ge
    -d_\H(x,x_0)\cdot\rho\{z\ge x\},
  \end{equation*}
  where $\le$ is the partial ordering on $X$ rooted in $x_0$.
\end{Lemma}
\begin{proof}[Proof of Lemma~\ref{L212}]
  It follows from~\eqref{e137} that
  \begin{multline*}
    \f(x)-\f(x_0)
    =-\int_{x_0}^x(\Delta\f)\{z\ge y\}d\a(y)\\
    \ge-\int_{x_0}^x\rho\{z\ge y\}d\a(y)
    \ge-\int_{x_0}^x\rho\{z\ge x\}d\a(y)
    =-d_\H(x,x_0)\cdot\rho\{z\ge x\},
  \end{multline*}
  where we have used that $\rho\ge\Delta\f$ and $x\ge y$.
\end{proof}
\begin{proof}[Proof of Proposition~\ref{P301}]
  Let $\le$ denote the partial ordering rooted in $x_0$ and set 
  \begin{equation*}
    Y_\lambda:=\{y\in X\mid (\rho_0+\Delta\f)\{z\ge y\}\ge\lambda\}.
  \end{equation*}
  Recall that $\rho_0+\Delta\f$ is a probability measure. 
  Thus $Y_\lambda=\emptyset$ if $\lambda>1$.
  If $\lambda\le 1$, then $Y_\lambda$ is a finite subtree of $X$
  containing $x_0$ and having at most $1/\lambda$ ends.
  The assumption that 
  $\lambda>\sup_{y\in X\setminus\H}\Delta\f\{y\}$
  implies that $Y_\lambda$ is in fact contained in $\H$.
  In particular, the number $C:=\sup_{y\in Y_\lambda}d_\H(x_0,y)$
  is finite.
  
  It now follows from Lemma~\ref{L212} that
  \begin{equation*}
    \f(x)-\f(x_0)
    \ge-\int_{x_0}^x(\rho_0+\Delta\f)\{z\ge y\}d\a(y)
    \ge-C-\lambda d_\H(x,x_0),
  \end{equation*}
  completing the proof.
\end{proof}
%
%
\subsubsection{Regularization}\label{S245}
In complex analysis, it is often useful to approximate a 
quasisubharmonic function by a decreasing sequence
of smooth quasisubharmonic functions. 
In higher dimensions, regularization results of this type
play a fundamental role in pluripotential theory, as
developed by Bedford and Taylor~\cite{BT1,BT2}. 
They are also crucial to the approach to non-Archimedean 
pluripotential theory in~\cite{hiro,siminag,nama}.

Let us say that a function $\f\in\SH(X,\rho_0)$ is 
\emph{regular} if it is piecewise affine in the sense that 
$\Delta\f=\rho-\rho_0$, where $\rho$ is a finite atomic 
measure supported on $\H$. 
\begin{Thm}\label{T202}
  For any $\f\in\SH(X,\rho_0)$ there exists a decreasing sequence 
  of regular functions $(\f_n)_{n=1}^\infty$  in $\SH(X,\rho_0)$ such that
  $\f_n$ converges pointwise to $\f$ on $X$.
\end{Thm}
\begin{proof}
  Let $Y_0\subset X$ be a finite tree containing the support of 
  $\rho_0$ and pick a point $x_0\in Y_0$.
  Set $\rho=\rho_0+\Delta\f$. 

  First assume that $\rho$ is supported on a finite subtree 
  contained in $\H$. We may assume $Y_0\subseteq Y$.
  For each $n\ge 1$, write $Y\setminus\{x_0\}$ as a finite disjoint 
  union of half-open segments $\gamma_i=\,]x_i,y_i]$, $i\in I_n$,
  called segments of order $n$, 
  in such a way that each segment of order $n$ has length at most $2^{-n}$
  and is the disjoint union of two segments of order $n+1$.
  Define finite atomic measures $\rho_n$ by
  \begin{equation*}
    \rho_n=\rho\{x_0\}\delta_{x_0}+\sum_{i\in I_n}\rho(\gamma_i)\delta_{y_i} 
  \end{equation*}
  and define $\f_n\in\SH(X,x_0)$ by $\Delta\f_n=\rho_n-\rho_0$,
  $\f_n(x_0)=\f(x_0)$. From~\eqref{e122} and~\eqref{e137} 
  it follows that $\f_n$ decreases to $\f$ pointwise on $X$, as $n\to\infty$.
  Since $\f=r_Y^*\f$ is continuous, the convergence is in fact uniform
  by Dini's Theorem.

  Now consider a general $\f\in\SH(X,\rho_0)$.
  For $n\ge 1$, define $Y'_n\subseteq X$ by 
  \begin{equation*}
    Y'_n:=\{y\in X\mid 
    \rho\{z\ge y\}\ge 2^{-n}\quad\text{and}\quad d_\H(x_0,y)\le 2^n\},
  \end{equation*}
  where $\le$ denotes the partial ordering rooted in $x_0$.
  Then $Y'_n$ is a finite subtree of $X$ and $Y'_n\subseteq Y'_{n+1}$ 
  for $n\ge1$. Let $Y_n$ be the convex hull of
  the union of $Y'_n$ and $Y_0$ and set $\p_n=r_{Y_n}^*\f_n$.
  Since $Y_n\subseteq Y_{n+1}$, we have $\f\le\p_{n+1}\le\p_n$ for all $n$.
  We claim that $\p_n(x)$ converges to $\f(x)$ as $n\to\infty$ for every 
  $x\in X$. Write $x_n:=r_{Y_n}(x)$ so that $\p_n(x)=\f(x_n)$. The points $x_n$
  converge to a point $y\in[x_0,x]$ and $\lim_n\p_n(x)=\f(y)$.
  If $y=x$, then we are done. But if $y\ne x$, then by construction of $Y'_n$,
  the measure $\rho$ puts no mass on the interval $]y,x]$, so 
  it follows from~\eqref{e122} and~\eqref{e137} that $\f(x)=\f(y)$.

  Hence $\p_n$ decreases to $\f$ pointwise on $X$ as $n\to\infty$.
  By the first part of the proof, we can find a regular 
  $\f_n\in\SH(X,\rho_0)$ such that $\p_n\le\f_n\le\p_n+2^{-n}$
  on $X$. Then $\f_n$ decreases to $\f$ pointwise on $X$, as desired.
\end{proof}
\begin{Remark}
  A different kind of regularization is used in~\cite[\S4.6]{FR1}.
  Fix a point $x_0\in\H$ and for each $n\ge1$ let $X_n\subseteq X$ be the
  (a priori not finite) subtree defined by 
  $X_n=\{x\in X\mid d_\H(x_0,x)\le n^{-1}\}$. 
  Let $\f_n\in\SH(X,\rho_0)$ be defined by 
  $\rho_0+\Delta\f_n=(r_n)_*(\rho_0+\Delta\f)$ and $\f_n(x_0)=\f(x_0)$,
  where $r_n:X\to X_n$ is the retraction.
  Then $\f_n$ is bounded and $\f_n$ decreases to $\f$ as $n\to\infty$.
\end{Remark}
%
%
%
%
\subsection{Tree maps}\label{S130}
Let $X$ and $X'$ be trees in the sense of~\S\ref{S131}.
We say that a continuous map $f:X\to X'$ is a \emph{tree map}
if it is open, surjective and finite in the sense that there exists
a number $d$ such that every point in $X'$ has at most $d$
preimages in $X$. The smallest such number $d$ is the
\emph{topological degree} of $f$.
\begin{Prop}\label{P105}
  Let $f:X\to X'$ be a tree map of topological degree $d$.
  \begin{itemize}
  \item[(i)]
    if $U\subseteq X$ is a connected open set, then so is $f(U)$
    and $\partial f(U)\subseteq f(\partial U)$;
  \item[(ii)]
    if $U'\subseteq X'$ is a connected open set
    and $U$ is a connected component of $f^{-1}(U')$, then 
    $f(U)=U'$ and $f(\partial U)=\partial U'$; 
    as a consequence, $f^{-1}(U')$ has at most 
    $d$ connected components;
  \item[(iii)]
    if $U\subseteq X$ is a connected open set and $U'=f(U)$, then
    $U$ is a connected component of $f^{-1}(U')$ iff
    $f(\partial U)\subseteq\partial U'$.
  \end{itemize}
\end{Prop}
The statement is valid for finite surjective open continuous maps 
$f:X\to X'$ between compact Hausdorff spaces,
under the assumption that every point of $X$
admits a basis of connected open neighborhoods.
We omit the elementary proof; see 
Lemma~9.11, Lemma~9.12 and Proposition~9.15
in~\cite{BRBook} for details.
\begin{Cor}\label{C107}
  Consider a point $x\in X$ and set $x':=f(x)\in X'$. 
  Then there exists a connected open neighborhood $V$ of 
  $x$ with the following properties:
  \begin{itemize}
  \item[(i)]
    if $\vv$ is a tangent direction at $x$, then
    there exists a tangent direction $\vv'$ at $x'$  such that 
    $f(V\cap U(\vv))\subseteq U(\vv')$;
    furthermore, either  $f(U(\vv))=U(\vv')$ or $f(U(\vv))=X'$;
  \item[(ii)]
    if $\vv'$ is a tangent direction at $x'$ then there exists
    a tangent direction $\vv$ at $x$ such that 
    $f(V\cap U(\vv))\subseteq U(\vv')$.
  \end{itemize}
\end{Cor}
\begin{Def}
  The \emph{tangent map} of $f$ at $x$  
  is the map that  associates $\vv'$ to $\vv$. 
\end{Def}
The  tangent map is surjective and every tangent direction
has at most $d$ preimages. Since the ends of $X$ are 
characterized by the tangent space being a singleton,
it follows that $f$ maps ends to ends.
\begin{proof}[Proof of Corollary~\ref{C107}]
  Pick $V$ small enough so that  it contains no preimage of $x'$ 
  besides $x$. 
  Note that~(ii) follows from~(i) and the fact that $f(V)$ is 
  an open neighborhood of $x'$.
  
  To prove~(i), note that $V\cap U(\vv)$ is connected for every $\vv$.
  Hence $f(V\cap U(\vv))$ is connected and does not contain $x'$,
  so it must be contained in $U(\vv')$ for some $\vv'$.
  Moreover, the fact that $f$ is open implies
  $\partial f(U(\vv))\subseteq f(\partial U(\vv))=\{x'\}$.
  Thus either $f(U(\vv))=X'$ or $f(U(\vv))$ is a connected open set
  with boundary $\{x'\}$. In the latter case, we must have 
  $f(U(\vv))=U(\vv')$.
\end{proof}
%
%
\subsubsection{Images and preimages of segments}\label{S164}
The following result makes the role of the 
tangent map more precise.
\begin{Cor}\label{C108}
  Let $f:X\to X'$ be a tree map as above. Then: 
  \begin{itemize}
    \item[(i)]
      if $\vv$ is a tangent direction at a point $x\in X$,
      then there exists a point $y\in U(\vv)$ such that 
      $f$ is a homeomorphism of the interval $[x,y]\subseteq X$ onto
      the interval $[f(x),f(y)]\subseteq X'$;
      furthermore, $f$ maps the annulus $A(x,y)$ onto the 
      annulus $A(f(x),f(y))$;
    \item[(ii)]
      if $\vv'$ is a tangent direction at a point $x'\in X'$,
      then there exists $y'\in U(\vv')$ such that if 
      $\gamma':=[x',y']$ then 
      $f^{-1}\gamma'=\bigcup_i\gamma_i$, where the
      $\gamma_i=[x_i,y_i]$ are closed intervals in $X$ with 
      pairwise disjoint interiors and 
      $f$ maps $\gamma_i$ homeomorphically
      onto $\gamma'$ for all $i$;
      furthermore we have $f(A(x_i,y_i))=A(x',y')$ for all $i$ and
      $f^{-1}(A(x',y'))=\bigcup_iA(x_i,y_i)$.
    \end{itemize}
  \end{Cor}
\begin{proof}
  We first prove~(ii).
  Set $U'=U(\vv')$ and let $U$ be a connected component of $f^{-1}(U')$.
  By Proposition~\ref{P105}~(ii), the boundary of $U$ consists of 
  finitely many preimages $x_1,\dots x_m$ of $x'$. 
  (The same preimage of $x'$ can lie on the boundary 
  of several connected components $U$.)
  Since $U$ is connected, there exists, for $1\le i\le m$,
  a unique tangent direction $\vv_i$ at $x_i$
  such that $U\subseteq U(\vv_i)$.
  
  Pick any point $z'\in U'$.
  Also pick points $z_1,\dots,z_m$ in $U$ such that the segments 
  $[x_i,z_i]$ are pairwise disjoint. 
  Then $f(]x_i,z_i])\cap\,]x',z']\ne\emptyset$ for all $i$, so 
  we can find $y'\in\,]x',z']$ and 
  $y_i\in\,]x_i,z_i]\,$ arbitrarily close to $x_i$ such that 
  $f(y_i)=y'$ for all $i$.
  In particular, we may assume that the annulus
  $A_i:=A(x_i,y_i)$ contains no preimage of $z'$.
  By construction it contains no preimage of $x'$ either. 
  Proposition~\ref{P105}~(i) first shows that 
  $\partial f(A_i)\subseteq\{x',y'\}$, so 
  $f(A_i)=A':=A(x',y')$ for all $i$.
  Proposition~\ref{P105}~(iii) then implies that 
  $A_i$ is a connected component of $f^{-1}(A')$.
  Hence $f^{-1}(A')\cap U=\bigcup_iA_i$.

  Write $\gamma_i=[x_i,y_i]$ and $\gamma'=[x',y']$.
  Pick any $\xi\in\,]x_i,y_i[\,$ and set $\xi':=f(\xi)$.
  On the one hand, $f(A(\xi,y_i))\subseteq f(A_i)=A'$.
  On the other hand, $\partial f(A(\xi,y_i))\subseteq\{\xi',y'\}$
  so we must have $f(A(\xi,y_i))=A(\xi',y')$ and 
  $\xi'\in\gamma'$. 
  We conclude that $f(\gamma_i)=\gamma'$ and that 
  $f:\gamma_i\to\gamma'$ is injective, hence a 
  homeomorphism.

  The same argument gives $f(A(x_i,\xi))=A(x',\xi)$.
  Consider any tangent direction $\ww$ at $\xi$ 
  such that $U(\ww)\subseteq A_i$. 
  As above we have $f(U(\ww))\subseteq A'$ and 
  $\partial f(U(\ww))\subseteq\{\xi'\}$, which implies
  $f(U(\ww))=U(\ww')$ for some  tangent direction $\ww$
  at $\xi'$ for which $U(\ww)\subseteq A'$.
  We conclude that $f^{-1}(\gamma')\cap A_i\subseteq\gamma_i$.

  This completes the proof of~(ii), and~(i) is an easy consequence.
\end{proof}
Using compactness, we easily deduce the following result 
from Corollary~\ref{C108}.
See the proof of Theorem~9.35 in~\cite{BRBook}.
\begin{Cor}\label{C109}
  Let $f:X\to X'$ be a tree map as above. Then:
    \begin{itemize}
    \item[(i)]
      any closed interval $\gamma$ in $X$ can be
      written as a finite union of closed intervals $\gamma_i$
      with pairwise disjoint interiors, such that
      $\gamma'_i:=f(\gamma_i)\subseteq X'$ is an interval 
      and $f:\gamma_i\to\gamma'_i$
      is a homeomorphism for all $i$;
      furthermore, $f$ maps the annulus
      $A(\gamma_i)$ onto the annulus $A(\gamma'_i)$;
    \item[(ii)]
      any closed interval $\gamma'$ in $X'$ can be
      written as a union of finitely many intervals $\gamma'_i$
      with pairwise disjoint interiors, such that, for all $i$,
      $f^{-1}(\gamma'_i)$ is a finite union of closed 
      intervals $\gamma_{ij}$ with pairwise disjoint interiors,
      such that $f:\gamma_{ij}\to\gamma'_i$ is a homeomorphism 
      for each $j$;
      furthermore, $f$ maps the annulus
      $A(\gamma_{ij})$ onto the annulus $A(\gamma'_i)$;
      and $A(\gamma_{ij})$ is a connected component of
      $f^{-1}(A(\gamma'_i))$.
    \end{itemize}      
\end{Cor}
%
%
\subsubsection{Fixed point theorem}\label{S165}
It is an elementary fact that any continuous selfmap
of a \emph{finite} tree admits a fixed point.
This can be generalized to arbitrary trees.
Versions of the following fixed point theorem can be found
in~\cite{valtree,Rivera3,BRBook}.
\begin{Prop}\label{P119}
  Any tree map $f:X\to X$ admits a fixed point $x=f(x)\in X$.
  Moreover, we can assume that one of the following two
  conditions hold:
  \begin{itemize}
  \item[(i)]
    $x$ is not an end of $X$;
  \item[(ii)]
    $x$ is an end of $X$ and $x$ is an attracting fixed point:
    there exists an open neighborhood $U\subseteq X$ of $x$
    such that $f(U)\subseteq U$ and $\bigcap_{n\ge 0}f^n(U)=\{x\}$.
  \end{itemize}
\end{Prop}
In the proof we will need the following easy consequence
of Corollary~\ref{C109}~(i).
\begin{Lemma}\label{L124}
  Suppose there are points $x,y\in X$, $x\ne y$,
  with $r(f(x)))=x$ and $r(f(y))=y$,
  where $r$ denotes the retraction of $X$ onto the 
  segment $[x,y]$. Then $f$ has a fixed point on $[x,y]$.
\end{Lemma}
\begin{proof}[Proof of Proposition~\ref{P119}]
  We may suppose that $f$ does not have any fixed point that is 
  not an end of $X$, or else we are in case~(i).
  Pick any non-end $x_0\in X$ and pick a finite 
  subtree $X_0$ that contains $x_0$, all preimages of $x_0$,
  but does not contain any ends of $X$. 
  Let $A$ be the set of finite subtrees of $X$ that contain
  $X_0$ but does not contain any end of $X$. For $\a\in A$,
  let $Y_\a$ be the corresponding subtree. 
  Then $(Y_\a)_{\a\in A}$ is a rich net of subtrees in the 
  sense of~\S\ref{S117}, so $X\simto\varprojlim Y_\a$.

  For each $\a$, define $f_\a:Y_\a\to Y_\a$ by 
  $f_\a=f\circ r_\a$. This is a continuous selfmap of a
  finite tree so the set $F_\a$ of its fixed points is a 
  nonempty compact set.  We will show that 
  $r_\a(F_\b)=F_\a$ when $\b\ge\a$. This will
  imply that there exists $x\in X$ such that 
  $r_\a(f(r_\a(x))=r_\a(x)$ for all $\a$. By assumption,
  $x$ is an end in $X$. Pick a sequence $(x_n)_{n=0}^\infty$
  of points in $X$ such that $x_{n+1}\in\,]x_n,x[\,$ and 
  $x_n\to x$ as $n\to\infty$. 
  Applying what precedes to the subtrees 
  $Y_{\a_n}=X_0\cup[x_0,x_n]$
  we easily conclude that $x$ is an attracting fixed point.

  It remains to show that $r_\a(F_\b)=F_\a$ when $\b\ge\a$. 
  First pick $x_\b\in F_\b$. 
  We will show that $x_\a:=r_\a(x_\b)\in F_\a$.
  This is clear if $x_\b\in Y_\a$
  since $r_\a=r_{\a\b}\circ r_\b$, so suppose
  $x_\b\not\in Y_\a$,
  By assumption, $f(x_\a)\ne x_\a$ and $f(x_\b)\ne x_\b$. 
  Let $\vv$ be the tangent direction at $x_\a$ represented by $x_\b$.
  Then $U(\vv)\cap Y_\a=\emptyset$ so $x_0\not\in f(U(\vv))$
  and hence $f(U(\vv))=U(\vv')$ for some tangent direction $\vv'$
  at $f(x_\a)$. Note that $f(x_\b)\in U(\vv')$. 
  If $f(x_\a)\not\in U(\vv)$, then Lemma~\ref{L124} 
  applied to $x=x_\a$, $y=x_\b$ gives a fixed point for $f$ in
  $[x_\a,x_\b]\subseteq Y_\b$, a contradiction.
  Hence $f(x_\a)\not\in U(\vv)$, so that 
  $r_\a(f(x_\a))=x_\a$, that is, $x_\a\in F_\a$.

  Conversely, pick $x_\a\in F_\a$. By assumption, 
  $f(x_\a)\ne x_\a$. Let $\vv$ be the tangent direction at $x_\a$
  defined by $U(\vv)$.  Then $U(\vv)\cap Y_\a=\emptyset$ 
  so $f(\overline{U(\vv)})\subseteq U(\vv)$. 
  Now $\overline{U(\vv)}\cap Y_\b$ is a finite nonempty 
  subtree of $X$ that is invariant under $f_\b$. Hence $f_\b$
  admits a fixed point $x_\b$ in this subtree. 
  Then $x_\b\in Y_\b$ and $r_\a(x_\b)=x_\a$.
\end{proof}
%
%
%
%
\subsection{Notes and further references}
Our definition of ``tree'' differs from the one in set theory,
see~\cite{Jec03}. It is also not equivalent to the notion
of ``$\R$-tree'' that has been around for 
quite some time (see~\cite{GH}) and found striking applications.
An $\R$-tree is a metric space and usually 
considered with its metric topology. On the other hand,
the notion of the weak topology on an $\R$-tree 
seems to have been rediscovered several times, sometimes
under different names (see~\cite{CLM}).

Our definitions of trees and metric trees are new 
but equivalent\footnote{Except for the missing condition~(RT3), see Remark~\ref{R201}.} 
to the ones given in~\cite{valtree},
where rooted trees are defined first and general
(non-rooted) trees are defined as equivalence classes of 
rooted trees. The presentation here seems
more natural. Following Baker and Rumely~\cite{BRBook} 
we have emphasized viewing a tree as a 
pro-finite tree, that is, an inverse limit of finite trees. 

Potential theory on simplicial graphs is a quite old subject
but the possibility of doing potential theory on general metric trees 
seems to have been 
discovered independently by Favre and myself~\cite{valtree},
Baker and Rumely~\cite{BRBook} and Thuillier~\cite{ThuillierThesis};
see also~\cite{FavreHab}.
Our approach here follows~\cite{BRBook} quite closely in how the
Laplacian is extended from finite to general trees. 
The class of quasisubharmonic functions is modeled on its
complex counterpart, where its compactness properties makes
this class very useful in complex dynamics and geometry.
It is sufficiently large for our purposes
and technically easier to handle than the class of  
functions of bounded differential variations studied in~\cite{BRBook}. 

Note that the interpretation of ``potential theory'' used here is 
quite narrow; for further results and questions we 
refer to~\cite{BRBook,ThuillierThesis}.
It is also worth mentioning that while potential theory on the Berkovich
projective line can be done in a purely tree theoretic way, this
approach has its limitations. In other situations, and especially
in higher dimensions, it seems advantageous to take a more
geometric approach. This point of view is used already
in~\cite{ThuillierThesis}
and is hinted at in our exposition of the valuative tree 
in~\S\ref{S103} and~\S\ref{S108}. 
We should remark that Thuillier in~\cite{ThuillierThesis} does
potential theory on general Berkovich curves. These are not 
always trees in our sense as they can contain loops.

Most of the results on tree maps in~\S\ref{S130} are well known and 
can be found in~\cite{BRBook} in the context of the 
Berkovich projective line. I felt it would be useful
to isolate some properties that are purely topological and
only depend on the map between trees being continuous,
open and finite. In fact, these properties turn out to be
quite plentiful. 

As noted in the text, versions of the fixed point result
in Proposition~\ref{P119} can be found in the work 
of Favre and myself~\cite{eigenval} and 
of Rivera-Letelier~\cite{Rivera3}. 
The proof here is new.
%
%
%
%
%
%
\newpage
\section{The Berkovich affine and projective lines}\label{S101}
Let us briefly describe the Berkovich affine and projective
lines. A comprehensive reference for this material is 
the recent book by Baker and Rumely~\cite{BRBook}.
See also Berkovich's original work~\cite{BerkBook}.
One minor difference to the presentation in~\cite{BRBook} is that 
we emphasize working in a coordinate free way. 
%
%
%
%
\subsection{Non-Archimedean fields}\label{S261}
We start by recalling some facts about non-Archimedean fields.
A comprehensive reference for this material is~\cite{BGR}.
%
%
\subsubsection{Seminorms and semivaluations}   
Let $R$ be a integral domain.
A \emph{multiplicative, non-Archimedean seminorm} on $R$ is 
a function $|\cdot|:R\to\R_+$ satisfying
$|0|=0$, $|1|=1$, $|ab|=|a||b|$ and 
$|a+b|\le\max\{|a|,|b|\}$.
If $|a|>0$ for all nonzero $a$, then $|\cdot|$ is a \emph{norm}.
In any case, the set $\fp\subseteq R$ consisting of
elements of norm zero is a prime ideal
and $|\cdot|$ descends to a norm on the quotient
ring $R/\fp$ and in turn extends to a norm
on the fraction field of the latter. 

Sometimes it is more convenient to work
additively and consider the associated
\emph{semi-valuation}\footnote{Unfortunately,
  the terminology is not uniform across the literature.
  In~\cite{BGR,BerkBook} `valuation'
  is used to denoted multiplicative norms.
  In~\cite{valtree}, `valuation' instead
  of `semi-valuation' is used even when the prime
  ideal $\{v =+\infty\}$ is nontrivial.}
$v:R\to\R\cup\{+\infty\}$
defined by $v =-\log|\cdot|$.
It satisfies the axioms
$v (0)=+\infty$, $v (1)=0$, 
$v (ab)=v (a)+v (b)$
and $v (a+b)\ge\min\{v (a),v (b)\}$.
The prime ideal $\fp$ above is now given by 
$\fp=\{v =+\infty\}$
and $v$ extends uniquely
to a real-valued valuation on the fraction field
of $R/\fp$.

Any seminorm on a field $K$ is a norm.
A \emph{non-Archimedean field} is a 
field $K$ equipped with a non-Archimedean, multiplicative
norm $|\cdot|=|\cdot|_K$ such that 
$K$ is complete in the induced metric.
In general, we allow the norm on $K$ be trivial: 
see Example~\ref{exam:trivialnorm}.
As a topological space, $K$ is totally disconnected.
We write $|K^*|=\{|a|\mid a\in K\setminus\{0\}\}\subseteq\R_+^*$
for the (multiplicative) \emph{value group} of $K$.
%
%
\subsubsection{Discs} 
A \emph{closed disc} in $K$ is a set of the form 
$D(a,r)=\{b\in K\mid |a-b|\le r\}$. 
This disc is \emph{degenerate} if $r=0$,
\emph{rational} if $r\in|K^*|$
and \emph{irrational} otherwise.
Similarly, $D^-(a,r):=\{b\in K\mid |a-b|<r\}$, $r>0$, 
is an \emph{open disc}.

The terminology  is natural but
slightly misleading since nondegenerate discs are both open
and closed in $K$.
Further, if $0<r\not\in|K^*|$, then $D^-(a,r)=D(a,r)$.
Note that any point in a disc in $K$ 
can serve as a center and that when
two discs intersect, one must contain the other.
As a consequence, any two closed discs
admit a unique smallest closed disc containing them both.
%
%
\subsubsection{The residue field}
The \emph{valuation ring} of $K$ is the
ring $\fo_K:=\{|\cdot|\le 1\}$. 
It is a local ring with maximal ideal
$\fm_K:=\{|\cdot|< 1\}$. 
The \emph{residue field} of $K$ is
$\tK:=\fo_K/\fm_K$.
We can identify $\fo_K$ and $\fm_K$
with the closed and open unit discs
in $K$, respectively. 
The \emph{residue characteristic} of $K$
is the characteristic of $\tK$.
Note that if $\tK$ has characteristic zero, then so does $K$.
\begin{Example}\label{exam:trivialnorm}
  We can equip any field $K$ with the 
  \emph{trivial} norm in which 
  $|a|=1$ 
  whenever $a\ne 0$.
  Then $\fo_K=K$, $\fm_K=0$ and $\tK=K$.
\end{Example}
\begin{Example}\label{exam:Qp}
  The field $K=\Q_p$ of $p$-adic numbers is the completion
  of $\Q$ with respect to the $p$-adic norm.
  Its valuation ring $\fo_K$ is the ring of $p$-adic 
  integers $\Z_p$ and the residue field $\tK$
  is the finite field $\F_p$. In particular,
  $\Q_p$ has characteristic zero and residue
 characteristic $p>0$.
\end{Example}
\begin{Example}\label{exam:Cp}
  The algebraic closure of $\Q_p$ is not complete.
  Luckily, the completed algebraic closure $\C_p$
  of $\Q_p$ is both algebraically closed and complete.
  Its residue field is $\overline{\F_p}$, the algebraic
  closure of $\F_p$. Again, $\C_p$
  has characteristic zero and residue
  characteristic $p>0$.
\end{Example}
\begin{Example}\label{exam:laurent}
  Consider the field $\C$ of complex numbers 
  (or any algebraically closed field of characteristic zero)
  \emph{equipped with the trivial norm}.
  Let $K=\C((u))$ be the field of Laurent series
  with coefficients in $\C$. The norm $|\cdot|$ on $K$ 
  is given by $\log|\sum_{n\in\Z}a_nu^n|=-\min\{n\mid a_n\ne 0\}$.
  Then $\fo_K=\C[[u]]$, $\fm_K=u\fo_K$
  and $\tK=\C$.
  We see that $K$ is complete and of residue characteristic zero.
  However, it is not algebraically closed.
\end{Example}
\begin{Example}\label{exam:puiseux}
  Let $K=\C((u))$ be the field of Laurent series.
  By the Newton-Puiseux theorem, the algebraic
  closure $K^a$ of $K$ is the field of 
  \emph{Puiseux series}
  \begin{equation}\label{e101}
    a=\sum_{\beta\in B} a_\beta u^\beta,
  \end{equation}
  where the sum is over a (countable) subset $B\subseteq\Q$ 
  for which there exists $m,N\in\N$ (depending
  on $a$) such that  $m+NB\subseteq\N$.
  This field is not complete; its completion
  $\hKa$ is algebraically closed
  as well as complete.
  It has residue characteristic zero.
\end{Example}
\begin{Example}\label{exam:giant}
  A giant extension of $\C((u))$ is given by the 
  field $K$ consisting of series of the form~\eqref{e101},
  where $B$ ranges over well-ordered subsets of $\R$. 
  In this case, $|K^*|=\R^*$.
\end{Example}
%
%
%
%
\subsection{The Berkovich affine line}
%
%
Write $R\simeq K[z]$ for the ring of polynomials 
in one variable with coefficients in $K$.
The \emph{affine line} $\A^1$ over $K$ is 
the set of maximal ideals in $R$.
Any choice of coordinate $z$ (\ie $R=K[z]$) defines 
an isomorphism 
$\A^1\simto K$.
A (closed or open) disc in $\A^1$ is a disc 
in $K$ under this isomorphism. This makes sense 
since any automorphism $z\mapsto az+b$
of $K$ maps discs to discs. We can also talk about
rational and irrational discs. However, 
the radius of a disc in $\A^1$ is not well defined.
\begin{Def}
  The \emph{Berkovich affine line} $\BerkA=\BerkA(K)$ 
  is the set of multiplicative seminorms $|\cdot|:R\to\R_+$ whose
  restriction to the ground field $K\subseteq R$ 
  is equal to the given norm $|\cdot|_K$.
\end{Def}
Such a seminorm is necessarily non-Archimedean.
Elements of $\BerkA$ are usually denoted 
$x$ and the associated seminorm on $R$ by
$|\cdot|_x$. 
The topology on $\BerkA$ is 
the weakest topology in which all evaluation
maps $x\mapsto|\phi|_x$, $\phi\in R$,
are continuous. There is a natural partial ordering
on $\BerkAone$: $x\le y$ iff $|\phi|_x\le |\phi|_y$ for all
$\phi\in R$.
%
%
%
%
\subsection{Classification of points}\label{S258}
One very nice feature of the Berkovich affine line is that we 
can completely and precisely classify its elements. The situation
is typically much more complicated in higher dimensions.
Following Berkovich~\cite{BerkBook} we shall describe four types
of points in $\BerkAone$, then show that this list is in fact complete.

For simplicity we shall from now on and until~\S\ref{S247}
assume that $K$ is \emph{algebraically closed} and
that the valuation on $K$ is \emph{nontrivial}.
The situation when one or both of these conditions is not
satisfied is discussed briefly in~\S\ref{S247}. 
See also~\S\ref{S211} for a different presentation of the 
trivially valued case.
%
%
\subsubsection{Seminorms from points}\label{S221}
Any closed point $x\in\A^1$ defines a seminorm $|\cdot|_x$ on $R$ through
\begin{equation*}
  |\phi|_x:=|\phi(x)|.
\end{equation*}
This gives rise to an embedding $\A^1\hookrightarrow\BerkAone$.
The images of this map will be called 
\emph{classical points}.\footnote{They are sometimes called \emph{rigid points}
  as they are the points that show up rigid analytic geometry~\cite{BGR}.}
\begin{Remark}
  If we define $\BerkAone$ as above when $K=\C$, then it follows from
  the Gel'fand-Mazur Theorem that all points are classical, that is,
  the map $\A^1\to\BerkAone$ is surjective. 
  The non-Archimedean case is vastly different.
\end{Remark}
%
%
\subsubsection{Seminorms from discs}
Next, let $D\subseteq\A^1$ be a closed disc
and define a seminorm $|\cdot|_D$ on $R$ by
\begin{equation*}
  |\phi|_D\=\max_{x\in D}|\phi(x)|.
\end{equation*}
It follows from Gauss' Lemma that this indeed defines
a multiplicative seminorm on $R$. In fact, the maximum
above is attained for a ``generic'' $x\in D$.
We denote the corresponding element of $\BerkAone$
by $x_D$.
In the degenerate case $D=\{x\}$, $x\in\A^1$, this
reduces to the previous construction: $x_D=x$.
%
%
\subsubsection{Seminorms from nested collections of discs}
It is clear from the construction that 
if $D,D'$ are closed discs in $\A^1$, then 
\begin{equation}\label{e110}
  |\phi|_D\le|\phi|_{D'}
  \ \text{for all $\phi\in R$ iff $D\subseteq D'$}.
\end{equation}
\begin{Def}
 A collection $\cE$ of closed discs in $\A^1$ 
 is \emph{nested} if the following conditions are satisfied:
\begin{itemize}
 \item[(a)]
   if $D,D'\in\cE$ then $D\subseteq D'$ or $D'\subseteq D$;
 \item[(b)]
   if  $D$ and $D'$ are closed discs 
   in $\A^1$ with $D'\in\cE$ and 
   $D'\subseteq D$, then $D\in\cE$;
 \item[(c)]
   if  $(D_n)_{n\ge 1}$ is a decreasing sequence of discs
   in $\cE$ whose intersection is a disc $D$ in $\A^1$,
   then $D\in\cE$.
 \end{itemize}
\end{Def}
In view of~\eqref{e110} we can associate a seminorm
$x_\cE\in\BerkA$ to a nested collection 
$\cE$ of discs by
\begin{equation*}
  x_\cE=\inf_{D\in\cE}x_D;
\end{equation*}
indeed, the limit of an decreasing sequence of seminorms is
a seminorm. 
When the intersection $\bigcap_{D\in\cE}D$ is nonempty,
it is a closed disc $D(\cE)$ 
(possibly of radius 0).
In this case $x_\cE$ is the seminorm associated to 
the disc $D(\cE)$. In general, however, the intersection 
above may be empty (the existence of a nested collection of discs
with nonempty intersection is equivalent to the field $K$
not being \emph{spherically complete}).

The set of nested collections of discs is partially ordered
by inclusion and we have $x_\cE\le x_{\cE'}$ 
iff $\cE'\subseteq\cE$.
%
%
\subsubsection{Classification}\label{S167}
Berkovich proved that all seminorms in $\BerkAone$ 
arise from the construction above.
\begin{Thm}\label{T203}
  For any $x\in\BerkAone$ there exists a unique 
  nested collection
  $\cE$ of discs in $\A^1$ such that $x=x_\cE$.
  Moreover, the map $\cE\to x_\cE$ is an order-preserving
  isomorphism.
\end{Thm}  
\begin{proof}[Sketch of proof]
  The strategy is clear enough: 
  given $x\in\BerkAone$ define $\cE(x)$ as the collection of 
  discs $D$ such that $x_D\ge x$. 
  However, it requires a little work to show that the maps
  $\cE\mapsto x_\cE$ and $x\mapsto\cE(x)$ 
  are order-preserving  and inverse one to another.  
  Here we have to use the assumptions that $K$ is algebraically 
  closed and that the norm on $K$ is nontrivial. The first
  assumption implies that $x$ is uniquely determined
  by its values on \emph{linear} polynomials in $R$. 
  The second assumption is necessary to ensure surjectivity 
  of $\cE\mapsto x_\cE$: if the norm on $K$ is trivial, 
  then there are too few discs in $\A^1$.
  See the proof of~\cite[Theorem~1.2]{BRBook} for details.
\end{proof}
%
%
\subsubsection{Tree structure}\label{S265}
Using the classification theorem above, we can already see that the 
Berkovich affine line is naturally a tree. Namely, let $\fE$ denote the 
set of nested collections of discs in $\A^1$. We also consider the
empty collection as an element of $\fE$. It is then straightforward to 
verify that $\fE$, partially ordered by inclusion, is a rooted tree in the 
sense of~\S\ref{S111}. As a consequence, the set 
$\BerkAone\cup\{\infty\}$ is a rooted metric tree.
Here $\infty$ corresponds to the empty collection of discs in $\A^1$
and can be viewed as the function $|\cdot|_\infty:R\to[0,+\infty]$
given by $|\phi|=\infty$ for any nonconstant polynomial $\phi\in R$
and $|\cdot|_\infty=|\cdot|_K$ on $K$. Then $\BerkAone\cup\{\infty\}$
is a rooted tree with the partial ordering $x\le x'$ iff 
$|\cdot|_x\ge|\cdot|_{x'}$ on $R$. 
See Figure~\ref{F208}.
\begin{figure}[ht]
  \includegraphics[width=0.8\textwidth]{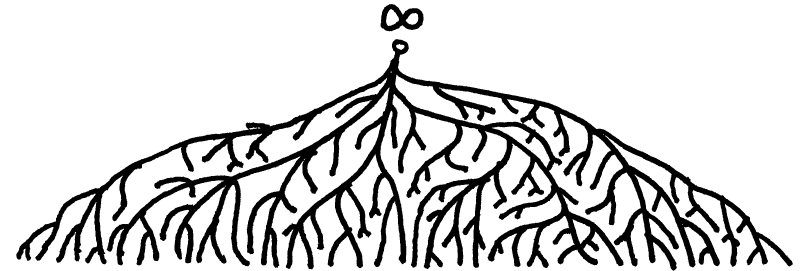}
  \caption{The Berkovich affine line.}\label{F208}
\end{figure}
%
%
\subsubsection{Types of points}\label{S264}
Using the identification with nested collections of discs, 
Berkovich classifies the points in $\BerkAone$ as follows:
\begin{itemize}
\item
  a point of \emph{Type~1} is a classical point, that is, 
  a point in the image
  of the embedding $\A^1\hookrightarrow\BerkAone$;
\item
  a point of \emph{Type~2} is of the form
 $x_D$ where $D$ is a rational disc in $\A^1$;
\item
  a point of \emph{Type~3}
  is of the form $x_D$ 
  where $D$ is an irrational disc in $\A^1$;
\item
  a point of \emph{Type~4} is of the form
  $x_\cE$, where $\cE$ is a
  nested collection of discs with empty intersection.
\end{itemize}
Note that Type~3 points exist iff $|K|\subsetneq\R_+$,
while Type~4 points exist iff $K$ is not spherically complete.
%
%
\subsubsection{Action by automorphisms}\label{S208}
Any automorphism $A\in\Aut(\A^1)$ arises from 
a $K$-algebra automorphism $A^*$ of $R$,
hence extends to an automorphism of $\BerkAone$
by setting 
\begin{equation*}
  |\phi|_{A(x)}:=|A^*\phi|_x
\end{equation*} 
for any polynomial $\phi\in R$. Note that $A$
is order-preserving. 
If $\cE$ is a nested collection of discs in $\A^1$,
then so is $A(\cE)$ and $A(x_\cE)=x_{A(\cE)}$.
It follows that $A$ preserves the type of a point in $\BerkAone$.

Clearly $\Aut(\A^1)$ acts transitively on 
$\A^1$, hence on the Type~1 points in $\BerkAone$.
It also acts transitively on the rational discs in $\A^1$,
hence the Type~2 points. 
In general, it will not act transitively on the set of 
Type~3 or Type~4 points, see~\S\ref{S263}.
%
%
\subsubsection{Coordinates, radii and the Gauss norm}\label{S263}
The description of $\BerkAone$ above was coordinate independent.
Now fix a coordinate $z:\A^1\simto K$. Using $z$,
every disc $D\subseteq\A^1$ becomes a disc in $K$, hence
has a well-defined \emph{radius} $r_z(D)$. 
If $D$ is a closed disc of radius $r=r_z(D)$ centered 
at point in $\A^1$ with coordinate $a\in K$, then 
\begin{equation}\label{e162}
  |z-b|_D=\max\{|a-b|,r\}.
\end{equation}
We can also define the radius
$r_z(\cE):=\inf_{D\in\cE}r_z(D)$ of a nested collection
of discs. The completeness of $K$ implies that if $r_z(\cE)=0$,
then $\bigcap_{D\in\cE}D$ is a point in $\A^1$.

The \emph{Gauss norm} is the norm in $\BerkAone$ defined by 
the unit disc in $K$. 
We emphasize that the Gauss norm depends on a 
choice of coordinate $z$.
In fact, any Type~2 point is the Gauss norm in some coordinate.

The radius $r_z(D)$ of a disc depends on $z$.
However, if we have two closed discs $D\subseteq D'$ in $\A^1$,
then the ratio $r_z(D')/r_z(D)$ does \emph{not} depend on $z$.
Indeed, any other coordinate $w$ is
of the form $w=az+b$, with $a\in K^*$, $b\in K$
and so $r_w(D)=|a|r_z(D)$, $r_w(D')=|a|r_z(D')$.
We think of the quantity $\log\frac{r_z(D')}{r_z(D)}$ as the 
modulus of the annulus $D'\setminus D$. It will
play an important role in what follows.

In the same spirit, the class $[r_z(x)]$ 
of $r_z(x)$ in $\R_+^*/|K^*|$ does not depend on the
choice of coordinate $z$. 
This implies that if $|K|\ne\R_+$, then $\Aut(\A^1)$ does
not act transitively on Type~3 points. Indeed, 
if $|K|\ne\R_+$, then given any Type~3 point $x$ 
we can find another Type~3 point $y\in[\infty,x]$ such that 
$[r_z(x)]\ne[r_z(y)]$. Then $A(x)\ne y$ for any $A\in\Aut(\A^1)$.
The same argument shows that if $K$ admits Type~4 points of
any given radius, then $A$ does not always act transitively on 
Type~4 points. For $K=\C_p$, there does indeed exist Type~4 points 
of any given radius, see~\cite[p.143]{Robert}.
%
%
%
%
\subsection{The Berkovich projective line}\label{S257}
%
%
We can view the projective line $\P^1$ over $K$
as the set of proper valuation rings $A$ of $F/K$, where
$F\simeq K(z)$ is the field of rational functions in
one variable with coefficients in $K$. 
In other words, $A\subsetneq F$ is a subring containing $K$
such that for every nonzero $\phi\in F$ we have $\phi\in A$ or $\phi^{-1}\in A$.
Since $A\ne F$,
there exists $z\in F\setminus A$ such that $F=K(z)$ 
and $z^{-1}\in A$.
The other elements of $\P^1$ are then the
localizations of the ring $R:=K[z]$ at its maximal ideals.
This gives rise to a decomposition
$\P^1=\A^1\cup\{\infty\}$ in which  
$A$ becomes the point $\infty\in\P^1$.

Given such a decomposition  we define
a \emph{closed disc} in $\P^1$ to be a closed disc in $\A^1$,
the singleton $\{\infty\}$, or 
the complement of an open disc in $\A^1$. 
Open discs are defined in the same way. A disc is \emph{rational}
if it comes from a rational disc in $\A^1$.
These notions do not depend on the choice of point $\infty\in\P^1$.
\begin{Def}
  The \emph{Berkovich projective line} 
  $\BerkPone$ over $K$ is the set of functions 
  $|\cdot|:F\to[0,+\infty]$ extending the norm on $K\subseteq F$
  and satisfying $|\phi+\psi|\le\max\{|\phi|,|\psi|\}$ for all $\phi,\psi\in F$,
  and $|\phi\psi|=|\phi||\psi|$ unless 
  $|\phi|=0$, $|\psi|=+\infty$
  or $|\psi|=0$, $|\phi|=+\infty$.
\end{Def}
To understand this, pick a rational function $z\in F$ such that $F=K(z)$. 
Then $R:=K[z]$ is the coordinate ring of $\A^1:=\P^1\setminus\{z=\infty\}$. 
There are two 
cases. Either $|z|=\infty$, in which case $|\phi|=\infty$ for
all nonconstant polynomials $\phi\in R$, 
or $|\cdot|$ is a seminorm on $R$, hence an element of $\BerkAone$.
Conversely, any element $x\in\BerkAone$ defines an element 
of $\BerkPone$ in the sense above. Indeed, every nonzero
$\phi\in F$ is of the form $\phi=\phi_1/\phi_2$ with 
$\phi_1,\phi_2\in R$ having no common
factor. Then we can set $|\phi|_x:=|\phi_1|_x/|\phi_2|_x$; this is
well defined by the assumption on $\phi_1$ and $\phi_2$.
Similarly, the function which is identically 
$\infty$ on all nonconstant polynomials defines a unique element
of $\BerkPone$: each $\phi\in F$ defines a rational function on $\P^1$
and $|\phi|:=|\phi(\infty)|\in[0,+\infty]$. This leads to a decomposition
\begin{equation*}
  \BerkPone=\BerkAone\cup\{\infty\},
\end{equation*}
corresponding to the decomposition
$\P^1=\A^1\cup\{\infty\}$.

We equip $\BerkP$ with the topology of
pointwise convergence. By Tychonoff, 
$\BerkP$ is a compact Hausdorff space
and, as a consequence, $\BerkAone$ is locally compact.
The injection $\A^1\hookrightarrow\BerkAone$
extends to an injection 
$\P^1\hookrightarrow\BerkP$
by associating the function $\infty\in\BerkPone$
to the point $\infty\in\P^1$.

Any automorphism $A\in\Aut(\P^1)$ is given by 
an element $A^*\in\Aut(F/K)$.
hence extends to an automorphism of $\BerkPone$
by setting 
\begin{equation*}
  |\phi|_{A(x)}:=|A^*\phi|_x
\end{equation*} 
for any rational function $\phi\in F$.
As in the case of $\BerkAone$, the type of a point is
preserved.
Further, $\Aut(\P^1)$ acts transitively on the set of 
Type~1 and Type~2 points, but not on the 
Type~3 or Type~4 points in general, see~\S\ref{S263}.
%
%
%
%
\subsection{Tree structure}\label{S256}
We now show that $\BerkPone$ admits natural structures
as a tree and a metric tree. See~\S\ref{S110} for the relevant definitions.

Consider a decomposition $\P^1=\A^1\cup\{\infty\}$ 
and the corresponding decomposition
$\BerkPone=\BerkAone\cup\{\infty\}$. The elements
of $\BerkPone$ define functions on the polynomial ring $R$ with values
in $[0,+\infty]$. This gives rise to a partial ordering on $\BerkPone$:
$x\le x'$ iff and only if $|\phi|_x\ge|\phi|_{x'}$ for all
polynomials $\phi$. 
As already observed in~\S\ref{S265}, 
$\BerkPone$ then becomes a rooted tree in the sense of~\S\ref{S111},
with $\infty$ as its root.
The partial ordering on $\BerkPone$ depends on a choice of 
point $\infty\in\P^1$, but the associated (nonrooted)
tree structure does not.

The ends of $\BerkPone$ are the points of Type~1 and~4, whereas
the branch points are the Type~2 points.
See Figure~\ref{F209}.
\begin{figure}[ht]
  \includegraphics[width=0.6\textwidth]{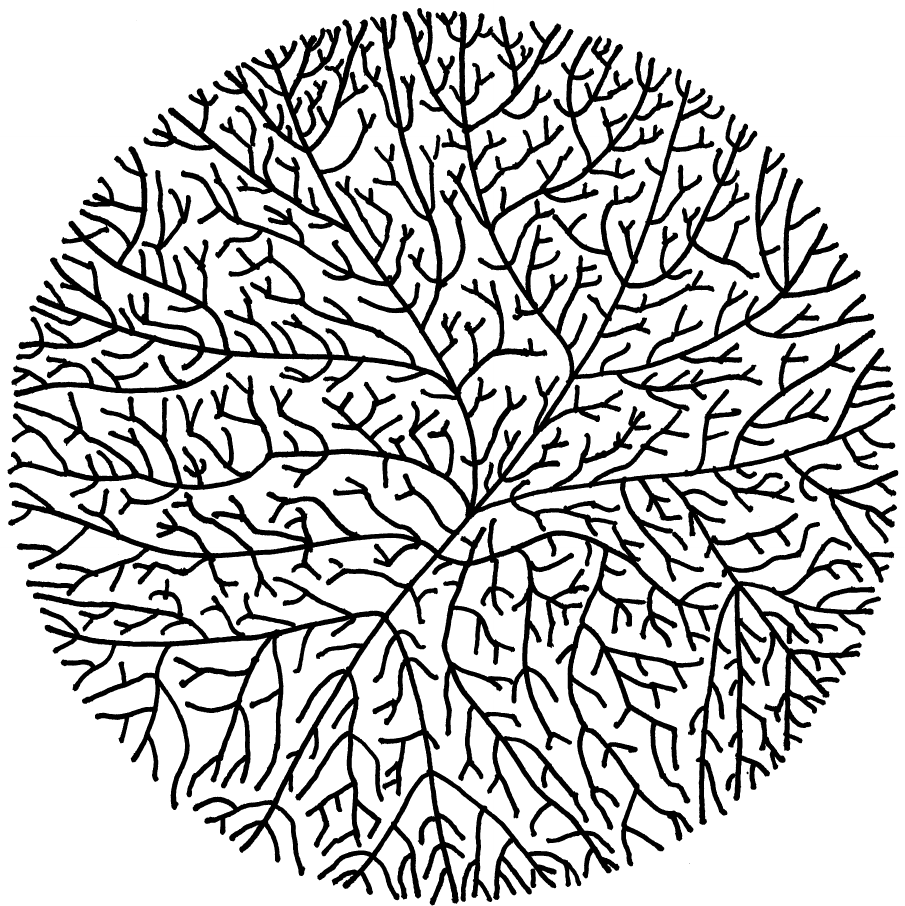}
  \caption{The Berkovich projective line.}\label{F209}
\end{figure}

Given a coordinate $z:\A^1\simto K$ we can parametrize $\BerkPone$
rooted in $\infty$ using radii of discs. 
Instead of doing so literally, we define an decreasing parametrization 
$\a_z:\BerkPone\to[-\infty,+\infty]$ using 
\begin{equation}\label{e111}
  \a_z(x_\cE):=\log r_z(\cE).
\end{equation}
One checks that this is a parametrization in the sense
of~\S\ref{S111}. 
The induced metric tree structure
on $\BerkPone$ does \emph{not} depend on the 
choice of coordinate $z$ and any automorphism of $\P^1$ induces 
an isometry of $\BerkPone$ in this generalized metric.
This is one reason for using the logarithm in~\eqref{e111}.
Another reason has to do with potential theory, 
see~\S\ref{S112}.
Note that $\a_z(\infty)=\infty$ and $\a_z(x)=-\infty$ iff
$x$ is of Type~1.

The associated \emph{hyperbolic space} in the sense of~\S\ref{S110}
is given by 
\begin{equation*}
  \H:=\BerkP\setminus\P^1.
\end{equation*}
The generalized metric on $\BerkPone$ above induces
a complete metric on $\H$ (in the usual sense). 
Any automorphism of $\P^1$ induces an isometry of $\H$.

%
%
%
%
\subsection{Topology and tree structure}\label{S112}
The topology on $\BerkP$ defined above agrees with the weak topology 
associated to the tree structure. To see this, note that $\BerkPone$ is
compact in both topologies. It therefore suffices to show that if 
$\vv$ is a tree tangent direction $\vv$ at a point $x\in\BerkPone$,
then the set $U(\vv)$ is open in the Berkovich topology. We may assume
that $x$ is of Type~2 or~3.
In a suitable coordinate $z$, $x=x_{D(0,r)}$ and $\vv$ is 
represented by the point $x_0$. 
Then $U(\vv)=\{y\in\BerkPone\mid |z|_x<r\}$, which is
open in the Berkovich topology.

A \emph{generalized open Berkovich disc} is a connected component of
$\BerkP\setminus\{x\}$ for some $x\in\BerkP$. When 
$x$ is of Type~2 or~3 we call it an
\emph{open Berkovich disc} and when $x$ of Type~2 a 
\emph{strict open Berkovich disc}.
A  (strict) \emph{simple domain} is a finite intersection of 
(strict) open Berkovich discs.
The collection of all (strict) simple domains 
is a basis for the topology on $\BerkP$.

%
%
%
%
\subsection{Potential theory}\label{S141}
As $\BerkPone$ is a metric tree we can
do potential theory on it, following~\S\ref{S116}. 
See also~\cite{BRBook} for a comprehensive treatment,
and the thesis of Thuillier~\cite{ThuillierThesis}
for potential theory on general Berkovich analytic curves.

We shall not repeat the material in~\S\ref{S116} here,
but given a finite atomic probability measure $\rho_0$ on 
$X$ with support on $\H$, 
we have a space $\SH(\BerkPone,\rho_0)$ of 
$\rho_0$-subharmonic functions, as well as
a homeomorphism 
\begin{equation*}
  \rho_0+\Delta:\SH(\BerkPone,\rho_0)/\R
  \simto\cM^+_1(\BerkPone).
\end{equation*}

Over the complex numbers, the analogue of 
$\SH(\BerkPone,\rho_0)$ is the space
$\SH(\P^1,\omega)$ of $\omega$-subharmonic 
functions on $\P^1$, where $\omega$ is a K\"ahler form.

\begin{Lemma}\label{L109}
  If $\phi\in F\setminus\{0\}$ is a rational function, then 
  the function 
  $\log|\phi|:\H\to\R$ is Lipschitz continuous
  with Lipschitz constant $\deg(\phi)$.
\end{Lemma}
\begin{proof}
  Pick any coordinate $z$ on $\P^1$ and write 
  $\phi=\phi_1/\phi_2$, with $\phi_1,\phi_2$ polynomials. 
  The functions $\log|\phi_1|$ and $\log|\phi_2|$ are 
  decreasing in the partial ordering rooted at $\infty$
  and $\log|\phi|=\log|\phi_1|-\log|\phi_2|$.
  Hence we may assume that $\phi$ is a polynomial.
  Using that $K$ is algebraically closed we 
  further reduce to the case $\phi=z-b$, where $b\in K$. 
  But then the result follows from~\eqref{e162}.
\end{proof}  
\begin{Remark}\label{R205}
  The function $\log|\phi|$ belongs to the space $\BDV(\BerkPone)$
  of functions of bounded differential variation and $\Delta\log|\phi|$ is the 
  divisor of $\phi$, viewed as a signed, finite atomic measure on
  $\P^1\subseteq\BerkPone$; see~\cite[Lemma~9.1]{BRBook}.
  Lemma~\ref{L109} then also follows from a version of~\eqref{e138}
  for functions in $\BDV(\BerkPone)$. 
  These considerations also show that the generalized 
  metric on $\BerkPone$ is the correct one from the point
  of potential theory.
\end{Remark}
%
%
%
%
\subsection{Structure sheaf and numerical invariants}\label{S280}
Above, we have defined the Berkovich projective line as a
topological space, but it also an analytic space in the 
sense of Berkovich  and carries a structure sheaf 
$\cO$. The local rings $\cO_x$ are useful
for defining and studying the local degree of a rational map.
They also allow us to recover Berkovich's classification 
via certain numerical invariants.
%
%
\subsubsection{Structure sheaf}\label{S209}
A holomorphic function on an open set $U\subseteq\BerkPone$ 
is a locally uniform limit of rational functions without poles in $U$. 
To make sense of this, we first need to say where 
the holomorphic functions take their values: 
the value at a point $x\in\BerkPone$ is in a 
non-Archimedean field $\cH(x)$. 

To define $\cH(x)$, assume $x\in\BerkAone$. 
The kernel of the seminorm $|\cdot|_x$
is a prime ideal in $R$ and $|\cdot|_x$
defines a norm on the fraction field of 
$R/\ker(|\cdot|_x)$; the field $\cH(x)$
is its completion. 

When $x$ is of Type~1, $\cH(x)\simeq K$.
If instead $x$ is of Type~3, pick a coordinate 
$z\in R$ such that $r:=|z|_x\not\in|K|$. 
Then $\cH(x)$ is isomorphic to the set of 
series $\sum_{-\infty}^\infty a_jz^j$ with $a_j\in K$
and $|a_j|r^j\to0$ as $j\to\pm\infty$.
For $x$ of Type~2 or~4, I am not aware of a
similar explicit description of $\cH(x)$.

The pole set of a rational function 
$\phi\in F$ can be viewed as a set of 
Type~1 points in $\BerkPone$. 
If $x$ is not a pole of $\phi$, then 
$\phi(x)\in\cH(x)$ is well defined.
The definition of a holomorphic function on 
an open subset $U\subseteq\BerkPone$ now 
makes sense and gives rise to the structure sheaf
$\cO$. 
%
%
\subsubsection{Local rings and residue fields}\label{S279}
The ring $\cO_x$ for $x\in\BerkPone$ 
is the ring of germs of holomorphic functions at $x$.
Denote by $\fm_x$ the maximal ideal of 
$\cO_x$ and by $\kappa(x):=\cO_x/\fm_x$ the 
residue field. Note that the seminorm $|\cdot|_x$
on $\cO_x$ induces a norm on $\kappa(x)$.
The field $\cH(x)$ above is the completion 
of $\kappa(x)$ with respect to the residue norm
and is therefore called the completed residue field.

When $x$ is of Type~1, $\cO_x$ is 
isomorphic to the ring of power series
$\sum_0^\infty a_jz^j$ such that 
$\limsup|a_j|^{1/j}<\infty$,
and $\kappa(x)=\cH(x)=K$.

If $x$ is not of Type~1, then $\fm_x=0$
and $\cO_x=\kappa(x)$ is a field.
This field is usually difficult to describe explicitly.
However, when $x$ is of Type~3 it has a description
analogous to the one of its completion $\cH(x)$
given above. Namely, pick a coordinate 
$z\in R$ such that $r:=|z|_x\not\in|K|$. 
Then $\cO_x$ is isomorphic to the set of 
series $\sum_{-\infty}^\infty a_jz^j$ with $a_j\in K$
for which there exists $r'<r<r''$ such that 
$|a_j|(r'')^j,|a_{-j}|(r')^{-j}\to0$ as $j\to+\infty$.
%
%
\subsubsection{Numerical invariants}\label{S114}
While the local rings $\cO_x$ and the completed residue fields $\cH(x)$
are not always easy to describe explicitly, certain numerical invariants 
of them are more tractable and allow us to recover Berkovich's classification.

First, $x$ is of Type~1 iff the seminorm $|\cdot|_x$ has
nontrivial kernel. Now suppose the kernel is trivial.
Then $\cO_x$ is a field and contains $F\simeq K(z)$
as a subfield. Both these fields are dense in $\cH(x)$
with respect to the norm $|\cdot|_x$.
In this situation we have two basic invariants.

First, the (additive) \emph{value group} is defined by
\begin{equation*}
  \Gamma_x
  :=\log|\cH(x)^*|_x
  =\log|\cO_x^*|_x
  =\log|F^*|_x.
\end{equation*}
This is an additive subgroup of $\R$ containing
$\Gamma_K:=\log|K^*|$.
The \emph{rational rank} $\ratrk x$ 
of $x$ is the dimension of the $\Q$-vector space
$(\Gamma_x/\Gamma_K)\otimes_\Z\Q$.

Second, the three fields 
$\cH(x)$, $\cO_x$ and $F$
have the same residue field
with respect to the norm $|\cdot|_x$.
We denote this field by $\widetilde{\cH(x})$;
it contains the residue field $\tK$ of $K$ as a subfield. 
The \emph{transcendence degree} $\trdeg x$ of $x$
is the transcendence degree of the field extension 
$\widetilde{\cH(x})/\tK$.

One shows as in~\cite[Proposition~2.3]{BRBook} that
\begin{itemize}
\item
  if $x$ is of Type~2, then $\trdeg x=1$ and $\ratrk x=0$;
  more precisely $\Gamma_x=\Gamma_K$ and
  $\widetilde{\cH(x)}\simeq\tK(z)$;
\item
  if $x$ is of Type~3, then $\trdeg x=0$ and $\ratrk x=1$;
  more precisely, $\Gamma_x=\Gamma_K\oplus\Z\a$,
  where $\a\in\Gamma_x\setminus\Gamma_K$,
  and $\widetilde{\cH(x)}\simeq\tK$;
\item
  if $x$ is of Type~4, then $\trdeg x=0$ and $\ratrk x=0$;
  more precisely, $\Gamma_x=\Gamma_K$ and
  $\widetilde{\cH(x)}\simeq\tK$;
\end{itemize}
%
%
\subsubsection{Quasicompleteness of the residue field}\label{S277}
Berkovich proved in~\cite[2.3.3]{Berkihes} that the residue field $\kappa(x)$ is
\emph{quasicomplete} in the sense that the induced norm
$|\cdot|_x$ on $\kappa(x)$ extends uniquely 
to any algebraic extension of $\kappa(x)$. 
This fact is true for any point of a ``good''
Berkovich space. It will be exploited (only) in~\S\ref{S124}.
%
%
\subsubsection{Weak stability of the residue field}\label{S278} 
If $x$ is of Type~2 or~3, then the residue field $\kappa(x)=\cO_x$
is \emph{weakly stable}. By definition~\cite[3.5.2/1]{BGR} this means
that any finite extension $L/\kappa(x)$ is 
\emph{weakly Cartesian}, that is, there exists a linear homeomorphism
$L\simto\kappa(x)^n$, where $n=[L:\kappa(x)]$, 
see~\cite[2.3.2/4]{BGR}. 
Here the norm on $L$ is the unique extension 
of the norm on the quasicomplete field $\kappa(x)$.
The homeomorphism above is not necessarily
an isometry.

The only consequence of weak stability that we shall use is
that if $L/\kappa(x)$ is a finite extension, then 
$[L:\kappa(x)]=[\hL:\cH(x)]$, where $\hL$ denotes the 
completion of $L$, see~\cite[2.3.3/6]{BGR}.
This, in turn, will be used (only) in~\S\ref{S124}. 

Let us sketch a proof that $\kappa(x)=\cO_x$ is weakly stable when
$x$ is of Type~2 or~3. Using the remark at the end
of~\cite[3.5.2]{BGR} it suffices to show that the field extension
$\cH(x)/\cO_x$ is separable. 
This is automatic if the ground field $K$ has 
characteristic zero, so suppose $K$ has characteristic $p>0$.
Pick a coordinate $z\in R$ such that $x$
is associated to a disc centered at $0\in K$. 
It is then not hard to see that $\cO_x^{1/p}=\cO_x[z^{1/p}]$
and it suffices to show that $z^{1/p}\not\in\cH(x)$. 
If $x$ is of Type~3, then this follows from the fact that 
$\frac1p\log r=\log|z^{1/p}|_x\not\in\Gamma_K+\Z\log r=\Gamma_x$.
If instead $x$ is of Type~2, then we may assume that $x$ is the
Gauss point with respect to the coordinate $z$. 
Then $\widetilde{\cH(x)}\simeq\tilde{K}(z)\not\ni z^{1/p}$
and hence $z^{1/p}\not\in\cH(x)$.
%
%
\subsubsection{Stability of the completed residue field}\label{S120}
When $x$ is a Type~2 or Type~3 point,
the completed residue field $\cH(x)$ is \emph{stable} field
in the sense of~\cite[3.6.1/1]{BGR}. This means that 
any finite extension $L/\cH(x)$ admits a basis $e_1,\dots,e_m$
such that $|\sum_ia_ie_i|=\max_i|a_i||e_i|$ for $a_i\in K$. 
Here the norm on $L$ is the unique extension  
of the norm on the complete field $\cH(x)$. 
The stability of $\cH(x)$ is proved in~\cite[6.3.6]{temkinstable}
(the case of a Type~2 point also follows from~\cite[5.3.2/1]{BGR}).

Let $x$ be of Type~2 or~3.
The stability of $\cH(x)$ implies that for any finite extension
$L/\cH(x)$ we have
$[L:\cH(x)]=[\Gamma_L:\Gamma_x]\cdot[\tilde{L}:\widetilde{\cH(x)}]$,
where $\Gamma_L$ and $\tilde{L}$ are the value group and 
residue field of $L$, see~\cite[3.6.2/4]{BGR}.
%
%
\subsubsection{Tangent space and reduction map}\label{S115}
Fix $x\in\BerkPone$. Using the tree structure,
we define as in~\S\ref{S119}
the tangent space $T_x$ of $\BerkPone$ at $x$ 
as well as a tautological ``reduction'' map from 
$\BerkPone\setminus\{x\}$ onto $T_x$.
Let us interpret this procedure algebraically
in the case when $x$ is a Type~2 point.

The tangent space $T_x$ at a Type~2 point
$x$ is the set of valuation rings $A\subsetneq\tH(x)$ 
containing $\tK$. 
Fix a coordinate $z$ such that $x$ becomes the 
Gauss point. Then $\tH(x)\simto\tK(z)$ and
$T_x\simeq\P^1(\tK)$.
Let us define the reduction map $r_x$ of $\BerkPone\setminus\{x\}$
onto $T_x\simeq\P^1(\tK)$. 
Pick a point $y\in\BerkPone\setminus\{x\}$.
If $|z|_y>1$, then we declare $r_x(y)=\infty$.
If $|z|_y\le 1$, then, since $y\ne x$, there exists
$a\in\fo_K$ such that $|z-a|_y<1$. The element $a$
is not uniquely defined, but its class $\ta\in\tK$ is and
we set $r_x(y)=\ta$. One can check that this 
definition does not depend on the choice of coordinate $z$
and gives the same result as the tree-theoretic construction.

The reduction map can be naturally 
understood in the context of formal models,
but we shall not discuss this here.
%
%
%
%
\subsection{Other ground fields}\label{S247} 
Recall that from~\S\ref{S257} onwards, we assumed that the field
$K$ was algebraically closed and nontrivially valued. 
These assumptions were used in the proof of Theorem~\ref{T203}.
Let us briefly discuss what happens when they are removed.

As before, $\BerkAone(K)$ is the set of multiplicative seminorms
on $R\simeq K[z]$ extending the norm on $K$ and 
$\BerkPone(K)\simeq\BerkAone(K)\cup\{\infty\}$. We can equip
$\BerkAone(K)$ and $\BerkPone(K)$ with a partial ordering defined by
$x\le x'$ iff $|\phi(x)|\ge|\phi(x')|$ for all polynomials $\phi\in R$.
%
%
\subsubsection{Non-algebraically closed fields}\label{S273}
First assume that $K$ is nontrivially valued but
not algebraically closed. Our discussion follows~\cite[\S4.2]{BerkBook};
see also~\cite[\S2.2]{Ked1},~\cite[\S5.1]{Ked2} and~\cite[\S6.1]{Ked3}.

Denote by $K^a$ the algebraic closure of $K$ and by $\hKa$ 
its completion. Since $K$ is complete, the norm
on $K$ has a unique extension to $\hKa$.

The Galois group $G:=\Gal(K^a/K)$ acts on the field
$\hKa$ and induces an action on $\BerkAone(\hKa)$,
which in turn extends to 
$\BerkPone(\hKa)=\BerkAone(\hKa)\cup\{\infty\}$
using $g(\infty)=\infty$ for all $g\in G$.
It is a general fact that $\BerkPone(K)$ is isomorphic to 
the quotient $\BerkPone(\hKa)/G$.
The quotient map $\pi:\BerkPone(\hKa)\to\BerkPone(K)$ is 
continuous, open and surjective.

It is easy to see that $g$ maps any segment 
$[x,\infty]$ homeomorphically onto the segment $[g(x),\infty]$.
This implies that $\BerkPone(K)$ is a tree in the sense
of~\S\ref{S172}.
In fact, the rooted tree structure on $\BerkPone(K)$ is
defined by the partial ordering above.

If $g\in G$ and $x\in\BerkPone(\hKa)$, then $x$
and $g(x)$ have the same type. This leads to a classification of 
points in $\BerkPone(K)$ into Types~1-4. 
Note that since $\hKa\ne K^a$ in general, 
there may exist Type~1 points $x\ne\infty$ such that 
$|\phi(x)|>0$ for all polynomials $\phi\in R=K[z]$.

We can equip the Berkovich projective line $\BerkPone(K)$ 
with a generalized metric. In fact, there are two natural ways of doing this. 
Fix a coordinate $z\in R$.
Let $\ha_z:\BerkPone(\hKa)\to[-\infty,+\infty]$ be the 
parametrization defined in~\S\ref{S256}. 
It satisfies $\ha_z\circ g=\ha_z$ for all $g\in G$ and hence induces 
a parametrization $\ha_z:\BerkPone(K)\to[-\infty,+\infty]$.
The associated generalized metric on $\BerkPone(K)$ 
does not depend on the choice of coordinate $z$ and has the
nice feature that the associated hyperbolic space
consists exactly of points of Types~2--4.

However, for potential theoretic considerations, it is better to use
a slightly different metric. 
For this, first define the 
\emph{multiplicity}\footnote{This differs from the ``algebraic degree'' used 
  by Trucco, see~\cite[Definition~5.1]{Trucco}.}
$m(x)\in\Z_+\cup\{\infty\}$ 
of a point $x\in\BerkPone(K)$ as the number of preimages of $x$
in $\BerkPone(\hKa)$. The multiplicity of a Type~2 or Type~3 
point is finite and if $x\le y$, 
then $m(x)$ divides $m(y)$. Note that $m(0)=1$ so all points on the 
interval $[\infty,0]$ have multiplicity 1. We now define a
decreasing parametrization $\a_z:\BerkPone(K)\to[-\infty,+\infty]$ as follows.
Given $x\in\BerkPone(K)$, set $x_0:=x\wedge0$ and 
\begin{equation}\label{e221}
  \a_z(x)=\a_z(x_0)-\int_{x_0}^x\frac1{m(y)}\,d\ha_z(y)
\end{equation}
Again, the associated generalized metric on $\BerkPone(K)$ 
does not depend on the choice of coordinate $z$. The hyperbolic
space $\H$ now contains all points of Types~2--4 but may also
contain some points of Type~1.

One nice feature of the generalized metric induced by $\a_z$
is that if $\rho_0$ is a finite positive measure on $\BerkPone(K)$
supported on points of finite multiplicity and if 
$\f\in\SH(\BerkPone(K),\rho_0)$, then 
$\pi^*\f\in\QSH(\BerkPone(\hKa))$ and 
\begin{equation*}
  \Delta\f=\pi_*\Delta(\pi^*\f).
\end{equation*}
Furthermore, for any rational function $\phi\in F$, the measure
$\Delta\log|\phi|$ on $\BerkPone(K)$ 
can be identified with the divisor of $\phi$, see Remark~\ref{R205}. 
%
%
\subsubsection{Trivially valued fields}\label{S262}
Finally we discuss the case when $K$ is trivially valued,
adapting the methods above.
A different approach is presented in~\S\ref{S211}.

First assume $K$ is algebraically closed. Then a multiplicative
seminorm on $R$ is determined by its values on linear polynomials.
Given a coordinate $z\in R$ it is easy to see that any point 
$x\in\BerkAone$ is of one of the following three types:
\begin{itemize}
\item
  we have $|z-a|_x=1$ for all $a\in K$;
  this point $x$ is the \emph{Gauss point};
\item
  there exists a unique $a\in K$ such that $|z-a|_x<1$;
\item
  there exists $r>1$ such that $|z-a|_x=r$ for all $a\in K$.
\end{itemize}
Thus we can view $\BerkAone$ as the quotient $K\times[0,\infty[\,/\sim$,
where $(a,r)\sim(b,s)$ iff $r=s$ and $|a-b|\le r$.
Note that if $r\ge 1$, then $(a,r)\sim(b,r)$ for all $r$,
whereas if $0\le r<1$, then $(a,r)\simeq(b,r)$ iff $a=b$.

We see that the Berkovich projective line 
$\BerkPone=\BerkAone\cup\{\infty\}$
is a tree naturally rooted at $\infty$ with 
the Gauss point as its only branch point. 
See Figure~\ref{F210}.
The hyperbolic metric is induced by the parametrization 
$\a_z(a,r)=\log r$. In fact, this parametrization does not depend
on the choice of coordinate $z\in R$.

\begin{figure}[ht]
 \includegraphics{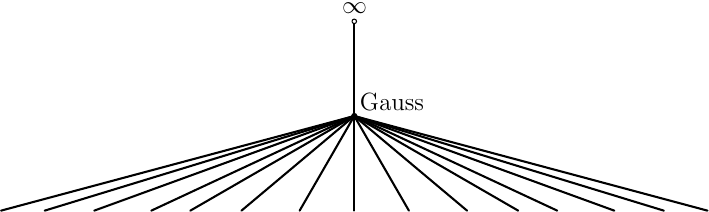}
  \caption{The Berkovich affine line over a trivially valued field.}\label{F210}
\end{figure}

If we instead choose the Gauss point as the root of the 
tree, then we can view the topological space underlying 
$\BerkPone$ as the cone over $\P^1$,
that is, as the quotient $\P^1\times[0,\infty]$, where 
$(a,s)\sim(b,t)$ if $s=t=0$. The Gauss point is the apex
of the cone and its distance to $(a,t)$ is $t$ in the 
hyperbolic metric. See Figure~\ref{F211}.

\begin{figure}[ht]
 \includegraphics{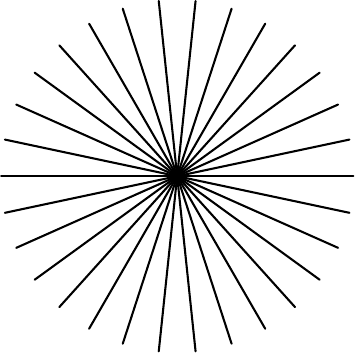}
  \caption{The Berkovich projective line over a trivially valued field.}\label{F211}
\end{figure}

Just as in the nontrivially valued case, the generalized metric on 
$\BerkPone$ is the correct one in the sense that Remark~\ref{R205}
holds also in this case.

\smallskip
Following the terminology of~\S\ref{S167},
a point of the form $(a,t)$ is of Type~1 and Type~2
iff $t=0$ and $t=\infty$, respectively.
All other points are of Type~3; there are no Type~4 points.

\smallskip
We can also describe the structure sheaf $\cO$. 
When $x$ is the Gauss point, the 
local ring $\cO_x$ is the field $F$ of rational functions
and $\cH(x)=\cO_x=F$ is equipped with the trivial norm.
Further, $\Gamma_x=\Gamma_K=0$, so 
$\ratrk x=0$ and $\trdeg x=1$.

Now assume $x\in\BerkPone$ is not the Gauss point and 
pick a coordinate $z\in F$ such that $|z|_x<1$.
If $x$ is of Type~3, that is, $0<|z|_x<1$, then 
$\cO_x=K((z))$ is the field of formal power series 
and $\cH(x)=\cO_x$ is equipped with the norm 
$|\sum_ja_jz^j|_x=r^{\max\{j\mid a_j\ne0\}}$.
Further, $\Gamma_x=\Z\log r$, so 
$\ratrk x=1$, $\trdeg x=0$.

If instead $|z|_x=0$ so that $x$ is of Type~1, then 
we have $\cO_x=K[[z]]$, whereas $\cH(x)\simeq K$ is
equipped with the trivial norm.

\smallskip
Finally, when $K$ is not algebraically closed, we view
$\BerkPone(K)$ as a quotient of $\BerkPone(K^a)$, where
$K^a$ is the algebraic closure of $K$ (note that $K^a$ is 
already complete in this case). 
We can still view the Berkovich projective line as the quotient
$\P^1(K)\times[0,\infty]/\sim$, with $\P^1(K)$ the 
set of closed (but not necessarily $K$-rational) points of
the projective line over $K$ and where $(a,0)\sim(b,0)$
for all $a,b$. The multiplicity (\ie the number of preimages in $\BerkPone(K^a)$
of the Gauss point is 1 and the multiplicity of any point $(a,t)$
is equal to the degree $[K(a):K]$ if $t>0$,
where $K(a)$ is the residue field of $a$. We define a
parametrization of $\BerkPone(K)$ using~\eqref{e221}.
Then the result in Remark~\ref{R205} remains valid.
%
%
%
%
\subsection{Notes and further references}
The construction of the Berkovich affine and projective lines is,
of course, due to Berkovich and everything in this section is,
at least implicitly, contained in his book~\cite{BerkBook}.

For general facts on Berkovich spaces we refer to the 
original works~\cite{BerkBook,Berkihes} or to some of
the recent surveys, \eg the ones by Conrad~\cite{ConradNotes} and
Temkin~\cite{temkinnotes}.
However, the affine and projective lines are very special
cases of Berkovich spaces and in fact we need
very little of the general theory in order to understand them.
I can offer a personal testimony to this fact as I was doing
dynamics on Berkovich spaces before I even knew what a
Berkovich space was!

Having said that, it is sometimes advantageous to use some
general notions, and in particular the structure sheaf,
which will be used to define the local degree of 
a rational map in~\S\ref{S166}. Further, the stability of the 
residue field at Type~2 and Type~3 points is quite useful.
In higher dimensions, simple arguments using the tree structure
are probably less useful than in dimension 1.

The Berkovich affine and projective lines are studied in 
great detail in the book \cite{BRBook} by Baker and Rumely,
to which we refer for more details.
However, our presentation here is slightly different 
and adapted to our purposes. 
In particular, we insist on trying to work in a coordinate free way
whenever possible. For example, the Berkovich unit disc and
its associated Gauss norm play 
an important role in most descriptions of the Berkovich
projective line, but they are only defined once we have chosen a
coordinate; without coordinates all Type~2 points are equivalent.
When studying the dynamics of rational maps, there is
usually no canonical choice of coordinate and hence no
natural Gauss point (the one exception being maps
of simple reduction, see~\S\ref{S160}).

One thing that we do not talk about at all are formal models.
They constitute a powerful geometric tool for studying Berkovich 
spaces, see~\cite{BerkLC1,BerkLC2} but we do not need them here. 
However, the corresponding notion for trivially valued fields is used
systematically in~\S\S\ref{S105}-\ref{S104}.
%
%
%
%
%
%
\newpage
\section{Action by polynomial and rational maps}\label{S156}
We now study how a polynomial or 
a rational map acts on the 
Berkovich affine and projective lines, respectively.
Much of the material in this chapter can be found with details
in the literature. However, as a general rule our presentation is
self-contained, the exception being when we draw more heavily on
the general theory of Berkovich spaces or non-Archimedean geometry.
As before, we strive to work in a coordinate free way 
whenever possible.

Recall that over the complex numbers, the projective line $\P^1$ 
is topologically a sphere. Globally a rational map 
$f:\P^1\to\P^1$ is a branched covering. Locally it is 
of the form $z\mapsto z^m$,  where $m\ge1$ is the local 
degree of $f$ at the point.
In fact, $m=1$ outside the ramification locus of $f$,
which is a finite set.

The non-Archimedean case is superficially very different but 
in fact exhibits many of the same properties when 
correctly translated. The projective line is a tree
and a rational map is a tree map in the sense of~\S\ref{S130}.
Furthermore, there is a natural notion of local degree that we shall 
explore in some detail. The ramification locus can be quite large
and has been studied in detail by Faber~\cite{Faber1,Faber2,Faber3}.
Finally, it is possible to give local normal forms, at least at 
points of Types~1-3.
%
%
%
%
\subsection{Setup}
As before, $K$ is a non-Archimedean field. We assume
that the norm on $K$ is non-trivial
and that $K$ is algebraically closed
but of arbitrary characteristic.
See~\S\ref{S254} for extensions.

Recall the notation $R\simeq K[z]$
for the polynomial ring in one variable with coefficients in $K$,
and $F\simeq K(z)$ for its fraction field.
%
%
%
%
\subsection{Polynomial and rational maps}\label{S259}
We start by recalling some general algebraic facts about 
polynomial and rational maps.
The material in~\S\ref{S291}--\S\ref{S295} is interesting mainly when
the ground field $K$ has positive characteristic. 
General references for that part are~\cite[VII.7]{LangAlgebra} 
and~\cite[IV.2]{Hartshorne}.
%
%
\subsubsection{Polynomial maps}\label{S281}
A nonconstant polynomial map 
$f:\A^1\to\A^1$ 
of the affine line over $K$ is 
given by an injective $K$-algebra homomorphism $f^*:R\to R$.
The \emph{degree} $\deg f$ of $f$ is the length of $R$ as a module
over $f^*R$. 
Given coordinates $z,w\in R$ on $\A^1$,
$f^*w$ is a polynomial in $z$ of degree $\deg f$.
%
%
\subsubsection{Rational maps}\label{S282}
A nonconstant regular map
$f:\P^1\to\P^1$ 
of the projective line over $K$ is 
is defined by an injective homomorphism 
$f^*:F\to F$ of fields over $K$, where $F\simeq K(z)$ is the fraction field 
of $R$. 
The degree of $f$ is the degree of the field extension $F/f^*F$. 
Given coordinates $z,w\in F$ on $\P^1$,
$f^*w$ is a rational function of $z$ of
degree $d:=\deg f$, that is, $f^*w=\phi/\psi$, where
$\phi,\psi\in K[z]$ are polynomials without
common factor and $\max\{\deg\phi,\deg\psi\}=d$.
Thus we refer to $f$ as a rational map, even though it is 
of course regular.

Any polynomial map $f:\A^1\to\A^1$ extends to 
a rational map $f:\P^1\to\P^1$ satisfying $f(\infty)=\infty$.
In fact, polynomial maps can be identified with rational maps
$f:\P^1\to\P^1$ admitting a totally invariant point 
$\infty=f^{-1}(\infty)$.
%
%
\subsubsection{Separable maps}\label{S291}
We say that a rational map $f$ is \emph{separable} 
if the field extension $F/f^*F$ is separable,
see~\cite[VII.4]{LangAlgebra}.
This is always the case if $K$
has characteristic zero.

If $f$ is separable, of degree $d$, then, by the 
Riemann-Hurwitz Theorem~\cite[IV.2/4]{Hartshorne} the 
\emph{ramification divisor} $R_f$ on $\P^1$ is well defined and of degree 
$2d-2$. In particular, all but finitely many points of $\P^1$
have exactly $d$ preimages under $f$, so $f$
has \emph{topological degree} $d$.
%
%
\subsubsection{Purely inseparable maps}\label{S292}
We say that a rational map $f$ is \emph{purely inseparable} if
the field extension $F/f^*F$ is 
purely inseparable. Assuming $\deg f>1$, this can only happen when $K$ has 
characteristic $p>0$ and means that for every $\phi\in F$
there exists $n\ge0$ such that $\phi^{p^n}\in f^*F$, 
see~\cite[VII.7]{LangAlgebra}.
Any purely inseparable map $f:\P^1\to\P^1$ is 
bijective. 
We shall see in~\S\ref{S294} that if $f$ is purely inseparable
of degree $d>1$, then $d=p^n$ for some $n\ge1$ and 
there exists a coordinate $z\in F$ on $\P^1$
such that $f^*z=z^d$.
%
%
\subsubsection{Decomposition}\label{S295}
In general, any algebraic field extension 
can be decomposed into
a separable extension followed by a purely inseparable
extension, see~\cite[VII.7]{LangAlgebra}.
As a consequence, any rational map $f$ can be 
factored as $f=g\circ h$, where $g$ is separable and 
$h$ is purely inseparable.
The topological degree of $f$ is equal to the 
degree of $g$ or, equivalently, the separable 
degree of the field extension $F/f^*F$, see~\cite[VII.4]{LangAlgebra}.
%
%
\subsubsection{Totally ramified points}\label{S293}
We say that a rational map $f:\P^1\to\P^1$
is \emph{totally ramified} at a point $x\in\P^1$ 
if $f^{-1}(f(x))=\{x\}$.
\begin{Prop}\label{P210}
  Let $f:\P^1\to\P^1$ be a rational map of degree $d>1$. 
  \begin{itemize}
  \item[(i)]
    If $f$ is purely inseparable, then $f$ is totally ramified at 
    every point $x\in\P^1$.
  \item[(ii)]
    If $f$ is not purely inseparable, then there are at most 
    two points at which $f$ is totally ramified.
  \end{itemize}
\end{Prop}
\begin{proof}
  If $f$ is purely inseparable, then $f:\P^1\to\P^1$
  is bijective and hence totally ramified at every point.

  Now suppose $f$ is not purely inseparable. Then $f=g\circ h$,
  where $h$ is purely inseparable and $g$ is separable, of degree 
  $\deg g>1$. 
  If $f$ is totally ramified at $x$, then so is $g$, so we may assume
  $f$ is separable. In this case, a direct calculation shows that 
  the ramification divisor has order $d-1$ at $x$. 
  The result follows since the ramification divisor has degree $2(d-1)$.
\end{proof}
%
%
%
%
\subsection{Action on the Berkovich space}\label{S296}
Recall that the affine and projective line $\A^1$ and $\P^1$
embed in the corresponding Berkovich spaces
$\BerkAone$ and $\BerkPone$, respectively.
%
%
\subsubsection{Polynomial maps}
Any nonconstant polynomial map $f:\A^1\to\A^1$ extends to
\begin{equation*}
  f:\BerkAone\to\BerkAone
\end{equation*}
as follows. If $x\in\BerkAone$, then $x'=f(x)$ is 
the multiplicative seminorm $|\cdot|_{x'}$ on $R$ defined by
\begin{equation*}
  |\phi|_{x'}:=|f^*\phi|_x.
\end{equation*} 
It is clear that $f:\BerkAone\to\BerkAone$ is continuous,
as the topology on $\BerkA$ was defined in
terms of pointwise convergence.
Further, $f$ is order-preserving
in the partial ordering on $\BerkAone$
given by $x\le x'$ iff $|\phi|_x\le|\phi|_{x'}$ 
for all polynomials $\phi$. 
%
%
\subsubsection{Rational maps}
Similarly, we can extend any nonconstant rational map $f:\P^1\to\P^1$
to a map
\begin{equation*}
  f:\BerkPone\to\BerkPone.
\end{equation*}
Recall that we defined $\BerkPone$ as the set of 
generalized seminorms $|\cdot|:F\to[0,+\infty]$.
If $x\in\BerkPone$, then the 
value of the seminorm $|\cdot|_{f(x)}$ on 
a rational function $\phi\in F$ is given by 
\begin{equation*}
  |\phi|_{f(x)}:=|f^*\phi|_x.
\end{equation*}

On the Berkovich projective line $\BerkPone$
there is no canonical partial ordering, so in general
it does not make sense to expect $f$ to be order preserving. 
The one exception to this is when there exist points 
$x,x'\in\BerkPone$ such that $f^{-1}(x')=\{x\}$.
In this case one can show that $f:\BerkPone\to\BerkPone$ becomes
order preserving when the source and target spaces are equipped with 
the partial orderings rooted in $x$ and $x'$.
If $x$ and $x'$ are both of Type~2, we can find 
coordinates on the source and target in which $x$ and $x'$ are 
both equal to the Gauss point, in which case one says that $f$
has \emph{good reduction}, see~\S\ref{S160}.
%
%
%
%
\subsection{Preservation of type}\label{S118}
There are many ways of analyzing the mapping properties of 
a rational map $f:\BerkPone\to\BerkPone$. First we show
the type of a point is invariant under $f$. For this, we use
the numerical classification in~\S\ref{S114}.
\begin{Lemma}\label{L211}
  The map $f:\BerkP\to\BerkP$
  sends a point of Type~1-4 to a point
  of the same type.
\end{Lemma}
\begin{proof}
  We follow the proof of~\cite[Proposition~2.15]{BRBook}.
  Fix $x\in\BerkPone$ and write $x'=f(x)$.

  If $|\cdot|_{x'}$ has nontrivial kernel, then clearly
  so does $|\cdot|_x$ and it is not hard to prove
  the converse, using that $K$ is algebraically closed.
  
  Now suppose $|\cdot|_x$ and $|\cdot|_{x'}$ have 
  trivial kernels. In this case, the value group 
  $\Gamma_{x'}$ is a subgroup of $\Gamma_{x}$ of finite index.
  As a consequence, $x$ and $x'$ have the same rational rank.
  Similarly, $\widetilde{\cH(x)}/\widetilde{\cH(x')}$ 
  is a finite field extension,
  so $x$ and $x'$ have the same transcendence degree.
  In view of the numerical classification, $x$ and $x'$
  must have the same type.
\end{proof}
%
%
%
%
\subsection{Topological properties}\label{S260}
Next we explore the basic topological properties 
of a rational map. 
\begin{Prop}\label{P107}
  The map $f:\BerkP\to\BerkP$ 
  is continuous, finite, open  and surjective.
  Any point in $\BerkP$ has at least one
  and at most $d$ preimages, where 
  $d=\deg f$. 
\end{Prop}
We shall see shortly that any point 
has \emph{exactly} $d$ preimages, counted
with multiplicity. 
However, note that for a purely inseparable map,
this multiplicity is equal to $\deg f$ at every point.
\begin{proof}
  All the properties follow quite easily from more general
  results in~\cite{BerkBook,Berkihes}, but we recall the 
  proof from~\cite[p.126]{FR2}.

  Continuity of $f$ is clear from the definition,
  as is the fact that a point of Type~1 has at least one and
  at most $d$ preimages.
  A point in $\H=\BerkP\setminus\P^1$ defines a
  norm on $F$, hence also on the subfield $f^*F$.
  The field extension $F/f^*F$ has degree $d$,
  so by~\cite{ZS} a valuation on $f^*F$ has 
  at least one and at most $d$ extensions to $F$. 
  This means that a point in $\H$
  also has at least one and at most $d$ preimages.

  In particular, $f$ is finite and surjective. By general results 
  about morphisms of Berkovich spaces, this implies
  that $f$ is open, see~\cite[3.2.4]{BerkBook}. 
\end{proof}
Since $\BerkPone$ is a tree, Proposition~\ref{P107} shows that
all the results of~\S\ref{S130} apply
and give rather strong information on the
topological properties of $f$.

One should note, however, that these purely topological 
results seem very hard to replicated for Berkovich spaces of
higher dimensions. The situation over the complex numbers
is similar, where the one-dimensional and higher-dimensional
analyses are quite different.
%
%
%
%
\subsection{Local degree} \label{S166}
It is reasonable to expect that any point in $\BerkPone$ should
have exactly $d=\deg f$ preimages under $f$ counted with 
multiplicity. 
This is indeed true, the only problem being to define
this multiplicity. There are several (equivalent) definitions in the 
literature. Here we shall give the one spelled out by Favre and
Rivera-Letelier~\cite{FR2}, but also used by Thuillier~\cite{ThuillierThesis}.
It is the direct translation of the corresponding notion in 
algebraic geometry.

Fix a point $x\in\BerkP$ and write $x'=f(x)$.
Let $\fm_x$ be the maximal ideal in the 
local ring $\cO_x$ and $\kappa(x):=\cO_x/\fm_x$ the residue
field.
Using $f$ we can $\cO_x$ as an $\cO_{x'}$-module and
$\cO_x/\fm_{x'}\cO_x$ as a $\kappa(x')$-vector space.
\begin{Def}\label{D101}
  The local degree of $f$ at $x$ is 
  $\deg_x f=\dim_{\kappa(x')}(\cO_x/\fm_{x'}\cO_x)$.
\end{Def}
Alternatively, since $f$ is finite, it follows~\cite[3.1.6]{BerkBook} that 
$\cO_x$ is a \emph{finite} $\cO_{x'}$-module.
The local degree $\deg_x f$ is therefore also equal
to the rank of the module $\cO_x$ viewed as 
$\cO_{x'}$-module, see~\cite[Theorem~2.3]{Matsumura}.
From this remark it follows that if 
$f,g:\P^1\to\P^1$ are nonconstant rational maps, then 
\begin{equation*}
  \deg_x(f\circ g)=\deg_xg\cdot\deg_{g(x)}f
\end{equation*}
for any $x\in\BerkPone$.

The definition above of the local degree works also over the complex numbers. 
A difficulty in the non-Archimedean setting is that the local rings $\cO_x$
are not as concrete as in the complex case, where they 
are isomorphic to the ring of convergent power series.

The following result shows that that local degree behaves as
one would expect from the complex case.
See~\cite[Proposition-Definition~2.1]{FR2}.
\begin{Prop}\label{P102}
  For every simple domain $V$ and every 
  connected component $U$ of $f^{-1}(V)$,
  the integer
  \begin{equation}\label{e201}
    \sum_{f(y)=x, y\in U}\deg_y f
  \end{equation}
  is independent of the point $x\in V$.
\end{Prop}
Recall that a simple domain is a finite intersection of 
open Berkovich discs; see~\S\ref{S112}.
The integer in~\eqref{e201} should be interpreted as the degree of the map from 
$U$ to $V$. If we put $U=V=\BerkPone$, then this degree is $d$.

We refer to~\cite[p.126]{FR2} for a proof. The idea is to view 
$f:U\to V$ as a map between Berkovich analytic curves. 
In fact, this is one of the few places in these notes where we draw
more heavily on the general theory of Berkovich spaces.

We would like to give a more concrete interpretation of the local
degree. First, at a Type~1 point, it  can
be read off from a local expansion of $f$:
\begin{Prop}\label{P103}
  Let $x\in\BerkPone$ be a Type~1 point and 
  pick coordinates $z$, $w$ on $\P^1$ such that
  $x=f(x)=0$. 
  Then $\cO_x\simeq K\{z\}$, $\cO_{f(x)}=K\{w\}$
  and we have 
  \begin{equation}\label{e127}
    f^*w=az^k(1+h(z)),
  \end{equation}
  where $a\ne0$, $k=\deg_x(f)$ and $h(0)=0$.
\end{Prop}
\begin{proof}
  The only thing that needs to be checked is that $k=\deg_x(f)$.
  We may assume $a=1$.
  First suppose $\charac K=0$. Then we can find 
  $\phi(z)\in K\{z\}$ 
  such that $1+h(z)=(1+\phi(z))^k$ in $K\{z\}$.
  It is now clear that $\cO_x\sim K\{z\}$ is a
  free module over $f^*\cO_{f(x)}$ of rank $k$,
  with basis given by $(z(1+\phi(z)))^j$, $0\le j\le k-1$,
  so $\deg_x(f)=k$.
  A similar argument can be used the case when $K$ has
  characteristic $p>0$; we refer to~\cite[p.126]{FR2} for 
  the proof.
\end{proof}
We shall later see how the local degree at a Type~2 or Type~3 points 
also appears in a suitable local expansion of $f$.

The following crucial result allows us to interpret the local degree 
quite concretely as a local expansion factor in the hyperbolic metric.
\begin{Thm}\label{T101}
  Let $f:\BerkPone\to\BerkPone$ be as above.
  \begin{itemize}
  \item[(i)]
    If $x$ is a point of Type~1 or~4 and $\gamma=[x,y]$ 
    is a sufficiently small segment, then $f$ maps $\gamma$ 
    homeomorphically onto $f(\gamma)$ and 
    expands the hyperbolic metric on $\gamma$ by a factor 
    $\deg_x(f)$.
  \item[(ii)]
    If $x$ is a point of Type~3 and $\gamma$ 
    is a sufficiently small segment containing $x$
    in its interior,  then $f$ maps $\gamma$ 
    homeomorphically onto $f(\gamma)$ and 
    expands the hyperbolic metric on $\gamma$ by a factor 
    $\deg_x(f)$.
  \item[(iii)]
    If $x$ is a point of Type~2, then for every tangent direction 
    $\vv$ at $x$ there exists an integer $m_\vv(f)$ such 
    that the following holds:
    \begin{itemize}
    \item[(a)]
      for any sufficiently small segment $\gamma=[x,y]$ 
      representing $\vv$, $f$ maps $\gamma$ homeomorphically
      onto $f(\gamma)$ and expands the hyperbolic metric on $\gamma$
      by a factor  $m_\vv(f)$;
    \item[(b)]
      if $\vv$ is any tangent direction at $x$ and 
      $\vv_1,\dots,\vv_m$ are the preimages of $\vv$ 
      under the tangent map, then 
      $\sum_i m_{\vv_i}(f)=\deg_x(f)$.
    \end{itemize}
  \end{itemize}
\end{Thm}
Theorem~\ref{T101} is due to 
Rivera-Letelier~\cite[Proposition~3.1]{Rivera4}
(see also~\cite[Theorem~9.26]{BRBook}). However,
in these references, different (but equivalent) 
definitions of local degree were used.
In~\S\ref{S123} below we will indicate a direct proof of 
Theorem~\ref{T101} using the above definition of 
the local degree.

Since the local degree is bounded by the algebraic degree, we obtain
as an immediate consequence
\begin{Cor}\label{C202}
  If $f:\BerkPone\to\BerkPone$ is as above, then 
  \begin{equation*}
    d_\H(f(x),f(y))\le\deg f\cdot d_\H(x,y)
  \end{equation*}
  for all $x,y\in\H$.
\end{Cor}
Using Theorem~\ref{T101} we can also make Corollary~\ref{C109} more precise:
\begin{Cor}\label{C106}
  Let $\gamma\subseteq\BerkPone$ be a segment
  such that the local degree is constant on the interior
  of $\gamma$. Then $f$ maps $\gamma$ homeomorphically
  onto $\gamma':=f(\gamma)$.
\end{Cor}
\begin{proof}
  By Corollary~\ref{C109} the first assertion is a local statement:
  it suffices to prove that if $x$ belongs to the interior of $\gamma$ 
  then the tangent map of $f$ is injective on the set of tangent directions
  at $x$ defined by $\gamma$. But if this were not the case, the 
  local degree at $x$ would be too high in view of assertion~(iii)~(b)
  in Theorem~\ref{T101}.
\end{proof}
\begin{Remark}
  Using similar arguments, Rivera-Letelier was able to improve 
  Pro\-position~\ref{P105} and describe $f(U)$ for a 
  simple domain $U$. For example, he described  when the image of 
  an open disc is an open disc as opposed to all of $\BerkPone$
  and similarly described the image of an annulus. 
  See Theorems~9.42 and~9.46 in~\cite{BRBook} and also 
  the original papers~\cite{Rivera1,Rivera2}.
\end{Remark}
%
%
%
%
\subsection{Ramification locus}\label{S171}
Recall that, over the complex numbers, a rational map has local
degree 1 except at finitely many points. In the non-Archimedean
setting, the situation is more subtle.
\begin{Def}
  The \emph{ramification locus} $R_f$ of $f$ is the set of $x\in\BerkPone$
  such that $\deg_x(f)>1$.
  We say that $f$ is 
  \emph{tame}\footnote{The terminology ``tame'' follows Trucco~\cite{Trucco}.}
  if $R_f$ is contained in the convex hull of a finite subset of $\P^1$.
\end{Def}
\begin{Lemma}\label{L120}
  If $K$ has residue characteristic zero, then $f$ is tame. 
  More precisely, $R_f$ is a finite union of segments in $\BerkPone$ 
  and is contained in the convex hull of the critical set of $f:\P^1\to\P^1$.
  As a consequence, the local degree is one at all 
  Type~4 points.
\end{Lemma}
We will not prove this lemma here. Instead we refer to
the papers~\cite{Faber1,Faber2} by X.~Faber for a detailed analysis
of the ramification locus, including the case of positive
residue characteristic. 
The main reason why the zero residue characteristic case is
easier stems from the following version of Rolle's Theorem
(see \eg~\cite[Proposition~A.20]{BRBook}):
if $\charac\tK=0$ and $D\subseteq\P^1$ is an open disc
such that $f(D)\ne\P^1$ and $f$ is 
not injective on $D$, then $f$ has a critical point in~$D$.

Se~\S\ref{S210} below for some examples of ramification loci.
%
%
%
%
\subsection{Proof of Theorem~\ref{T101}}\label{S123}
While several proofs of Theorem~\ref{T101} exist in the literature,
I am not aware of any that directly uses Definition~\ref{D101} of the 
local degree. Instead, they use different definitions, which in view of 
Proposition~\ref{P102} are equivalent to the one we use. 
Our proof of Theorem~\ref{T101} uses some
basic non-Archimedean analysis in the spirit of~\cite{BGR}.
%
%
\subsubsection{Type~1 points}
First suppose $x\in\P^1$ is a classical point.
As in the proof of Proposition~\ref{P103},
we find coordinates $z$ and $w$ on
$\P^1$ vanishing at $x$ and $x'$, respectively,
such that $f^*w=az^k(1+h(z))$, where $a\ne0$, 
$k=\deg_x(f)\ge 1$ and $h(0)=0$.
In fact, we may assume $a=1$.
Pick $r_0>0$ so small that $|h(z)|_{D(0,r)}<1$ for $r\le r_0$.
It then follows easily that
$f(x_{D(0,r)})=x_{D(0,r^k)}$ for $0\le r\le r_0$.
Thus $f$ maps the segment 
$[x_0,x_{D(0,r_0)}]$ homeomorphically onto
the segment $[x_0,x_{D(0,r_0^k)}]$
and the hyperbolic metric is expanded by a factor~$k$.
%
%
\subsubsection{Completion}\label{S124}
Suppose $x$ is of Type~2 or~3.
Then the seminorm $|\cdot|_x$ is a norm,
$\cO_x$ is a field having $\cO_{x'}$ as a subfield and 
$\deg_x(f)$ is the degree 
$[\cO_x:\cO_{x'}]$ of the field extension
$\cO_x/\cO_{x'}$.
Recall that $\cH(x)$ is the completion of $\cO_x$.

In general, the degree of a field extension can change
when passing to the completion. However, we have
\begin{Prop}\label{P206}
  For any point $x\in\BerkPone$ of Type 2 or 3 we have
  \begin{equation}\label{e124}
    \deg_x(f)
    =[\cO_x:\cO_{x'}]
    =[\cH(x):\cH(x')]
    =[\Gamma_x:\Gamma_{x'}]
    \cdot[\widetilde{\cH(x)}:\widetilde{\cH(x')}],
  \end{equation}
  where $\Gamma$ and $\widetilde{H}$ denotes the
  value groups and residue fields of the norms 
  under consideration.
\end{Prop}
\begin{proof}
  Recall from~\S\ref{S277} that the field
  $\cO_{x'}$ is quasicomplete in the sense that the norm
  $|\cdot|_{x'}$ on $\cO_{x'}$ extends uniquely 
  to any algebraic extension.
  In particular, the norm $|\cdot|_x$ is the unique
  extension of this norm to $\cO_x$.
  Also recall from~\S\ref{S278} that the field $\cO_{x'}$
  is weakly stable. Thus $\cO_x$ is weakly Cartesian 
  over $\cO_{x'}$, which by~\cite[2.3.3/6]{BGR} 
  implies the second equality in~\eqref{e124}.
  
  Finally recall from~\S\ref{S120} that 
  the field $\cH(x')$ is stable. 
  The third equality in~\eqref{e124} then follows 
  from~\cite[3.6.2/4]{BGR}.
\end{proof}
%
%
\subsubsection{Approximation}\label{S121}
In order to understand the local degree of a rational map,
it is useful to simplify the map in a way similar to~\eqref{e127}.
Suppose $x$ and $x'=f(x)$ are Type~2 or Type~3 points.
In suitable coordinates on the source and target, we can 
write $x=x_{D(0,r)}$ and $x'=x_{D(0,r')}$,
where $0<r,r'\le 1$. If $x$ and $x'$ are Type~2 points,
we can further assume $r=r'=1$.

Write $f^*w=f(z)$ for some rational function $f(z)\in F\simeq K(z)$. 
Suppose we can find a decomposition in $F$ of the form
\begin{equation*}
  f(z)=g(z)(1+h(z)),
  \quad\text{where $|h(z)|_x<1$}.
\end{equation*}
The rational function $g(z)\in F$
induces a rational map $g:\P^1\to\P^1$,
which extends to 
$g:\BerkPone\to\BerkPone$.
\begin{Lemma}\label{L121}
  There exists $\delta>0$ such that 
  $g(y)=f(y)$ and $\deg_y(g)=\deg_y(f)$ 
  for all $y\in\H$ with $d_\H(y,x)\le\delta$.
\end{Lemma}
\begin{proof}
  We may assume that $h(z)\not\equiv0$, 
  or else there is nothing to prove.
  Thus we have $|h(z)|_x>0$.
  Pick $0<\e<1$ such that $|h(z)|_x\le\e^3$,
  set 
  \begin{equation*}
    \delta=(1-\e)\min\left\{\frac{|h(z)|_x}{\deg h(z)},\frac{r'}{2\deg f}\right\}
  \end{equation*}
  and assume $d_\H(y,x)\le\delta$.
  We claim that 
  \begin{equation}\label{e123}
    |f^*\phi-g^*\phi|_y\le\e|f^*\phi|_y
    \quad\text{for all $\phi\in F$}.
  \end{equation}

  Granting~\eqref{e123}, we get
  $|g^*\phi|_y=|f^*\phi|_y$ for all $\phi$ and hence
  $g(y)=f(y)=:y'$.
  Furthermore, $f$ and $g$ give rise to
  isometric embeddings $f^*,g^*:\cH(y')\to\cH(y)$.
  By Proposition~\ref{P206}, the degrees of the 
  two induced field extensions $\cH(y)/\cH(y')$ 
  are equal to $\deg_y f$ and $\deg_y g$, respectively.
  By continuity, the inequality~\eqref{e123} extends to all 
  $\phi\in\cH(y')$. It then 
  follows from~\cite[6.3.3]{temkinstable} 
  that $\deg_y f=\deg_y g$.

  We also remark that~\eqref{e123} implies
  \begin{equation}\label{e125}
    f^*\Gamma_{y'}=g^*\Gamma_{y'}
    \quad\text{and}\quad
    f^*\widetilde{\cH(y')}=g^*\widetilde{\cH(y')}.
  \end{equation}
  Thus $f$ and $g$ give the same embeddings of 
  $\Gamma_{y'}$ and $\widetilde{\cH(y')}$ into 
  $\Gamma_y$ and $\widetilde{\cH(y)}$, respectively.
  When $y$, and hence $y'$ is of Type~2 or~3, the 
  field $\cH(y')$ is stable, and so~\eqref{e124} 
  gives another proof of the equality
  $\deg_y f=\deg_y g$.

  It remains to prove~\eqref{e123}. 
  A simple calculation shows that 
  if~\eqref{e123} holds
  for $\phi,\psi\in F$, then it also holds for
  $\phi\psi$, $1/\phi$ and $a\phi$ for any
  $a\in K$. Since $K$ is algebraically closed, it
  thus suffices to prove~\eqref{e123} 
  for $\phi=w-b$, where $b\in K$.

  Using Lemma~\ref{L109} and the fact that 
  $f(x)=x_{D(0,r')}$, we get
  \begin{multline*}
    |f(z)-b|_y
    \ge|f(z)-b|_x-\delta\deg f
    =|w-b|_{f(x)}-\delta\deg f\ge\\
    \ge r'-\delta\deg f
    \ge\e(r'+\delta\deg f)
    =\e(|f(z)|_x+\delta\deg f)
    \ge\e|f(z)|_y.
  \end{multline*}
  Now Lemma~\ref{L109} 
  and the choice of $\delta$ imply $|h(z)|_y\le\e^2<1$.
  As a consequence, $|g(z)|_y=|f(z)|_y$. We conclude that
  \begin{equation*}
    |f^*(w-b)-g^*(w-b)|_y
    =|h(z)|_y|g(z)|_y
    \le\e^2|f(z)|_y
    \le\e|f(z)-b|_y
    =\e|f^*(w-b)|_y,
  \end{equation*}
  establishing~\eqref{e123} and 
  completing the proof of Lemma~\ref{L121}.
\end{proof}
%
%
\subsubsection{Type~3 points}
Now consider a point $x$ of Type~3. 
In suitable coordinates $z$, $w$ we may assume that 
$x$ and $x'=f(x)$ are associated to 
irrational closed discs $D(0,r)$ and $D(0,r')$, respectively.
In these coordinates, $f$ is locally approximately monomial
at $x$; there exist $\theta\in K^*$ and $k\in\Z\setminus\{0\}$
such that $f^*w=\theta z^k(1+h(z))$, where $h(z)\in K(z)$
satisfies $|h(z)|_x<1$.
Replacing $w$ by $(\theta^{-1}w)^{\pm1}$ we may assume 
$\theta=1$ and $k>0$. In particular, $r'=r^k$.

Let $g:\P^1\to\P^1$ be defined by $g^*w=z^k$.
We claim that $\deg_x(g)=k$. 
Indeed, the field $\cH(x)$ (resp.\ $\cH(x')$)
can be concretely described as the 
set of formal series $\sum_{-\infty}^\infty a_jz^j$ 
(resp.\ $\sum_{-\infty}^\infty b_jw^j$)
with $|a_j|r^j\to 0$ as $|j|\to\infty$
(resp.\ $|b_j|r^{kj}\to 0$ as $|j|\to\infty$).
Then $1,z,\dots,z^{k-1}$ form a basis for $\cH(x)/\cH(x')$.
We can also see that $\deg_x(g)=k$ from~\eqref{e124}
using that 
$\widetilde{\cH(x)}=\widetilde{\cH(x')}=\tK$,
$\Gamma_x=\Gamma_K+\Z\log r$
and
$\Gamma_{x'}=\Gamma_K+k\Z\log r$.

Lemma~\ref{L121} gives $\deg_x(f)=\deg_x(g)$.
Moreover, we must have 
$f(x_{D(0,s)})=x_{D(0,s^k)}$ for $s\approx r$,
so $f$ expands the hyperbolic metric by a factor $k=\deg_x(f)$.
Thus we have established all statements in Theorem~\ref{T101}
for Type~3 points.
%
%
\subsubsection{Type~2 points}
Now suppose $x$ and hence $x'=f(x)$ is of Type~2.
Then $\Gamma_x=\Gamma_{x'}=\Gamma_K$.
We may assume $x$ and $x'$ both equal the 
Gauss point in suitable coordinates $z$ and $w$.
The algebraic tangent spaces $T_x,T_{x'}\simeq\P^1(\tK)$ 
defined in~\S\ref{S115}
have $\widetilde{\cH(x)}\simeq\tK(z)$ 
and $\widetilde{\cH(x')}\simeq\tK(w)$ as function fields.
Now $f$ induces a map 
$f^*:\widetilde{\cH(x')}\to\widetilde{\cH(x)}$
and hence a map $T_x\to T_{x'}$.
By~\eqref{e124}, the latter has degree $\deg_x(f)$.

As opposed to the Type~3 case, we cannot necessarily 
approximate $f$ by a monomial map. However, after 
applying a coordinate change of the form 
$z\mapsto (\theta z)^{\pm1}$,
we can find $g(z)\in F=K(z)$ of the form
\begin{equation}\label{e128}
  g(z)=z^m\frac{\prod_{i=1}^{l-m}(z-a_i)}{\prod_{j=1}^k(z-b_j)},
\end{equation}
with $m\ge 0$, $|a_i|=|b_j|=1$, $a_i\ne b_j$ and $a_ib_j\ne0$
for all $i,j$, such that 
\begin{equation*}
  f^*w=g(z)(1+h(z)),
\end{equation*}
in $F$, where $|h(z)|_x<1=|g(z)|_x$. 

On the one hand, $g(z)$ induces a map 
$g:\P^1(K)\to\P^1(K)$ and hence also a map
$g:\BerkPone\to\BerkPone$.
We clearly have $g(x)=x'$ and Lemma~\ref{L121}
gives $\deg_x(g)=\deg_x(f)$.
On the other hand, $g(z)$ also induces
a map $g:\P^1(\tK)\to\P^1(\tK)$, which
can be identified with the common tangent map
$T_x\to T_{x'}$ of $f$ and $g$.
Both these maps $g$ have degree $\max\{l,k\}$,
so in accordance with~\eqref{e124}, we see that
$\deg_x(f)=[\widetilde{\cH(x)}:\widetilde{\cH(x')}]$.

To prove the remaining statements in Theorem~\ref{T101}~(iii),
define $m_\vv(f)$ as the local degree of the 
algebraic tangent map $T_x\to T_x'$ 
at the tangent direction $\vv$. Statement~(a) in Theorem~\ref{T101}~(iii)
is then clear, so it suffices to show~(b).
We may assume that $\vv$ and
its image $\vv'$ are both represented by $x_0$. 
Then $m(\vv)$ is the integer $m$ in~\eqref{e128}.
We see from~\eqref{e128} and from Lemma~\ref{L121} that
$f(x_{D(0,r)})=x_{D(0,r^m)}$ when $0\ll1-r<1$.
Thus~(b) holds.
%
%
\subsubsection{Type~4 points}
Finally suppose $x$ is a Type~4 point. 
By Corollary~\ref{C108} 
we can find $y\in\BerkPone$ such that $f$ is a homeomorphism
of the segment $\gamma=[x,y]$ onto 
$f(\gamma)$. We first claim that by moving $y$
closer to $x$, $f$ will expand the hyperbolic metric on 
$\gamma$ by a fixed integer constant $m\ge 1$. 

Let $\ww$ be the tangent direction at 
$y$ represented by $x$. By moving $y$
closer to $x$, if necessary, we may assume that $x$ 
is the unique preimage of $x'$ in $U(\ww)$. 

Consider a point $\xi\in\,]x,y[\,$. 
If $\xi$ is of Type~3, then we know that 
$f$ locally expands the hyperbolic metric along $\gamma$
by a factor $m(\xi)$. 
Now suppose $\xi$ is a Type~2 point and let $\vv_+$
and $\vv_-$ be the tangent directions at $\xi$ represented by 
$x$ and $y$, respectively. Then
$f$ locally expands the hyperbolic metric along $\vv_\pm$
by factors $m(\vv_\pm)$. Suppose that $m(\vv_+)<m(\vv_-)$.
Then there must exist a tangent direction $\vv$ at $\xi$ 
different from $\vv_+$ but having the same image as
$\vv_+$ under the tangent map. By Corollary~\ref{C107}
this implies that 
$x'\in f(U(\vv))\subseteq f(U(\ww)\setminus\{x\})$,
a contradiction. Hence $m(\vv_+)\ge m(\vv_-)$. Since
$m(\vv_+)$ is bounded from above by $d=\deg f$,
we may assume that $m(\vv_+)=m(\vv_-)$ at all Type~2
points on $\gamma$. This shows that $f$ expands the 
hyperbolic metric on $\gamma$ by a constant factor $m$.

To see that $m=\deg_x(f)$, note that the above argument shows that
$\deg_\xi(f)=m$ for all $\xi\in\gamma\setminus\{x\}$. Moreover,
if $\ww'$ is the tangent direction at $f(y)$ represented by 
$f(x)$, then the above reasoning shows that 
$U(\ww)$ is a connected component of $f^{-1}(U(\ww'))$
and that $\xi$ is the unique preimage of $f(\xi)$ in $U(\ww)$
for any $\xi\in\gamma$. It then follows from 
Proposition~\ref{P102} that $\deg_x f=m$.
%
%
%
%
\subsection{Laplacian and pullbacks}\label{S134}
Using the local degree we can pull
back Radon measures on $\BerkP$ by $f$. 
This we do by first defining a push-forward operator on continuous functions: 
\begin{equation*}
  f_*H(x)=\sum_{f(y)=x}\deg_y(f)H(y)
\end{equation*}
for any $H\in C^0(\BerkPone)$.
It follows from Proposition~\ref{P102} that $f_*H$ is 
continuous and it is clear that $\|f_*H\|_\infty\le d\|H\|_\infty$,
where $d=\deg f$.
We then define the pull-back of Radon measures by duality:
\begin{equation*}
  \langle f^*\rho,H\rangle = \langle \rho,f_*H\rangle.
\end{equation*}
The pull-back operator is continuous in the 
weak topology of measures. If $\rho$ is a probability measure,
then so is $d^{-1}f^*\rho$. Note that the pull-back of a Dirac mass becomes 
\begin{equation*}
  f^*\delta_x=\sum_{f(y)=x}\deg_y(f)\delta_y.
\end{equation*}

Recall from~\S\ref{S116} that 
given a positive Radon measure $\rho$ on $\BerkPone$ 
and a finite atomic measure $\rho_0$ supported on $\H$
of the same mass as $\rho$, we can write 
$\rho=\rho_0+\Delta\f$ for a
unique function $\f\in\SH^0(\BerkPone,\rho_0)$.
A key property is
\begin{Prop}\label{P104}
  If $\f\in\SH^0(\BerkPone,\rho_0)$, then 
  $f^*\f\in\SH^0(\BerkPone,f^*\rho_0)$ and 
  \begin{equation}\label{e103}
    \Delta(f^*\f)= f^*(\Delta\f).
  \end{equation}
\end{Prop}
This formula, which will be crucial for the proof of the equidistribution 
in the next section, confirms that the generalized metric $d_\H$
on the tree $\BerkP$ is the correct one. See also Remark~\ref{R205}.
\begin{proof}
  By approximating $\f$ by its retractions $\f\circ r_X$,
  where $X$ ranges over finite subtrees of $\H$ containing the support of
  $\rho_0$ we may assume that $\rho:=\rho_0+\Delta\f$ 
  is supported on such a finite subtree $X$. This means that $\f$ is 
  locally constant outside $X$. By further approximation we reduce to the 
  case when $\rho$ is a finite atomic measure supported on 
  Type~2 points of $X$.

  Let $Y=f^{-1}(X)$. 
  Using Corollary~\ref{C108} and Theorem~\ref{T101}
  we can write $X$ (resp.\ $Y$)
  as a finite union $\gamma_i$ (resp.\ $\gamma_{ij}$) 
  of  intervals with mutually disjoint interiors such that
  $f$ maps $\gamma_{ij}$ homeomorphically onto $\gamma_i$
  and the local degree is constant, equal to $d_{ij}$ on the interior of 
  $\gamma_{ij}$. We may also assume that the interior of
  each $\gamma_i$ (resp.\ $\gamma_{ij}$) is disjoint from 
  the support of $\rho$ and $\rho_0$ 
  (resp.\ $f^*\rho$ and $f^*\rho_0$).
  Since $f$ expands the hyperbolic metric on each $\gamma_{ij}$ 
  with a constant factor $d_{ij}$, it follows that 
  $\Delta(f^*\f)=0$ on the interior of $\gamma_{ij}$.

  In particular, $\Delta(f^*\f)$ is a finite atomic measure.
  Let us compute its mass at a point $x$.
  If $\vv$ is a tangent direction at $x$ and $\vv'=Df(\vv)$
  its image under the tangent map, then it follows from
  Theorem~\ref{T101}~(iii) that 
\begin{equation}\label{e222}
  D_\vv(f^*\f)=m_\vv(f)D_{\vv'}(\f)
\end{equation}
and hence
  \begin{multline*}
    \Delta(f^*\f)\{x\}
    =\sum_\vv D_\vv(f^*\f)
    =\sum_\vv m_\vv(f)D_{\vv'}(\f)
    =\sum_{\vv'}D_{\vv'}\f\sum_{Df(\vv)=\vv'} m_\vv(f)\\
    =\deg_x(f)\sum_{\vv'}D_{\vv'}(\f)
    =\deg_x(f)(\Delta\f)\{f(x)\}
    =f^*(\Delta\f)\{x\},
  \end{multline*}
  which completes the proof.
\end{proof}
%
%
%
%
\subsection{Examples}\label{S210}
To illustrate the ideas above, let us study three concrete 
examples of rational maps. Fix a coordinate 
$z\in F$ on $\P^1$. Following standard practice we 
write $f(z)$ for the rational function $f^*z$.
\begin{Example}\label{E207}
  Consider the polynomial map defined by
  \begin{equation*}
    f(z)=a(z^3-3z^2)
  \end{equation*}
  where $a\in K$. Here $K$ has residue characteristic zero.
  The critical points of $f:\P^1\to\P^1$ are $z=0$,
  $z=2$ and $z=\infty$, where the local degree is 
  2, 2 and 3, respectively.
  On $\BerkPone$, the local degree is 3 on the interval 
  $[x_G,\infty]$, where $x_G$ is the Gauss norm.
  The local degree is 2 on the intervals $[0,x_G[\,$ 
  and $[2,x_G[\,$ and it is 1 everywhere else. 
  See Figure~\ref{F101}.
\end{Example}
\begin{figure}[ht]
 \includegraphics{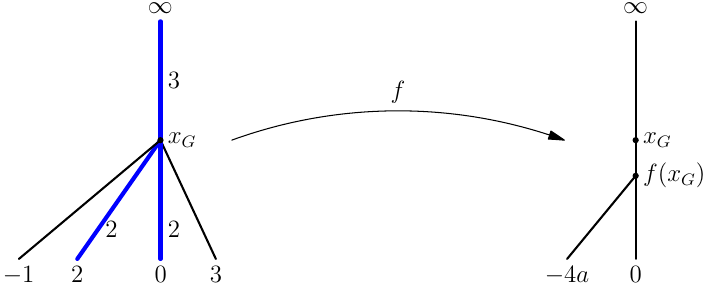}
 \caption{The ramification locus of the map $f(z)=a(z^3-3z^2)$ 
   in Example~\ref{E207} when $|a|<1$.
   Here $x_G$ is the Gauss point. 
   The preimage of the interval $[0,f(x_G)]$ is $[0,x_G]$ 
   (with multiplicity 2) and $[3,x_G]$.
   The preimage of the interval $[-4a,f(x_G)]$ is $[2,x_G]$ 
   (with multiplicity 2) and $[-1,x_G]$.
   The preimage of the interval $[\infty,f(x_G)]$ is $[\infty,x_G]$ 
   (with multiplicity 3).}\label{F101}
\end{figure}
\begin{Example}\label{E208}
  Next consider the polynomial map defined by
  \begin{equation*}
    f(z)=z^p
  \end{equation*}
  for a prime $p$.
  Here the ground field $K$ has characteristic zero.
  If the residue characteristic is different from $p$, 
  then $f$ is tamely ramified and the ramification locus is
  the segment $[0,\infty]$. On the other hand, if the residue
  characteristic is $p$, then $f$ is not tamely ramified.
  A point in $\BerkAone$ corresponding to a disc $D(a,r)$
  belongs to the ramification locus iff $r\ge p^{-1}|a|$.
  The ramification locus is therefore quite large and can be visualized 
  as an ``inverted Christmas tree'', as illustrated in~Figure~\ref{F213}. 
  It is the set of points in $\BerkPone$ having hyperbolic distance at 
  most $\log p$ to the segment $[0,\infty]$. 
  See~\cite[Example~9.30]{BRBook} for more details.
\end{Example}  
\begin{figure}[ht]
  \includegraphics{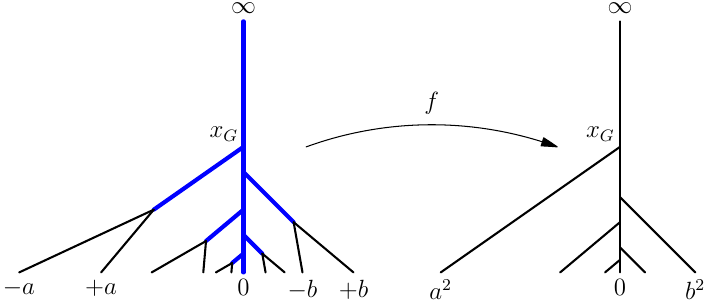}
  \caption{The ramification locus of the map $f(z)=z^2$
    in residual characteristic 2.
    A point in $\BerkAone$ corresponding to a disc $D(a,r)$
    belongs to the ramification locus iff $r\ge 2|a|$.
    The point $x_G$ is the Gauss point.}\label{F213}
\end{figure}
\begin{Example}\label{E212}
  Again consider the polynomial map defined by
  \begin{equation*}
    f(z)=z^p
  \end{equation*}
  for a prime $p$, but now assume 
  that $K$ has characteristic $p>0$.
  Then $f$ is purely inseparable and usually called the 
  \emph{Frobenius map}. 
  We will see in~\S\ref{S294} that every purely inseparable
  map of degree $>1$ is an iterate of the Frobenius map in some
  coordinate $z$.

  The mapping properties of $f$ on the Berkovich projective line 
  are easy to describe. 
  Since $f$ is a bijection, the local degree is equal to $p$ 
  at \emph{all} points of $\BerkPone$. Hence the ramification locus
  is equal to $\BerkPone$.
  The Gauss point $x_G$  in the coordinate $z$ is a fixed point:
  $f(x_G)=x_G$. If $x\in\BerkPone$, then $f$ maps the segment
  $[x_G,x]$ homeomorphically onto the segment $[x_G,f(x)]$
  and expands the hyperbolic metric by a constant factor $p$.
\end{Example}  

For many more interesting examples, see~\cite[\S10.10]{BRBook}.
%
%
%
%
\subsection{Other ground fields}\label{S254}
Above we worked with the assumption that our non-Archimedean field
$K$ was algebraically closed and nontrivially valued. 
Let us briefly discuss what happens when one or both of these 
assumption is dropped.
%
%
\subsubsection{Non-algebraically closed fields}\label{S267}
First suppose $K$ is nontrivially valued but not algebraically closed.
Most of the results above remain
true in this more general setting and can be proved by passing to
the completed algebraic closure $\hKa$ as in~\S\ref{S247}.
Let us outline how to do this.

The definitions and results in~\S\ref{S296} go through unchanged.
Note that $f$ induces a map $\hf:\BerkPone(\hKa)\to\BerkPone(\hKa)$
that is equivariant under the action of the Galois group $G=\Gal(K^a/K)$.
Thus $f\circ\pi=\pi\circ\hf$, where 
$\pi:\BerkPone(\hKa)\to\BerkPone(K)$ is the projection.
The fact that $\hf$ preserves the type of a point (Lemma~\ref{L211}) implies
that $f$ does so as well. Proposition~\ref{P107} remains valid 
and implies that $f$ is a tree map in the sense of~\S\ref{S130}.

We define the local degree of $f$ as in~\S\ref{S166}. 
Proposition~\ref{P102} remains valid. The local degrees
of $f$ and $\hf$ are related as follows. Pick a point 
$\hx\in\BerkPone(\hKa)$ and set $x=\pi(\hx)$,
$\hx':=f(\hx)$ and $x':=\pi(\hx')=f(x)$.
The stabilizer $G_\hx:=\{\sigma\in G\mid \sigma(\hx)=\hx\}$
is a subgroup of $G$ and we have $G_\hx\subseteq G_{\hx'}$.
The index of $G_\hx$ in $G_{\hx'}$ only depends on the projection
$x=\pi(\hx)$ and we set
\begin{equation*}
  \delta_x(f):=[G_{\hx'}:G_\hx];
\end{equation*}
this is an integer bounded by the (topological) degree of $f$. We have
$m(x)=\delta_x(f)m(f(x))$ for any $x\in\BerkPone(K)$, 
where $m(x)$ is the multiplicity of $x$, \ie the number of preimages
of $x$ under $\pi$.
Now
\begin{equation*}
  \deg_x(f)=\delta_x(f)\deg_{\hx}(\hf).
\end{equation*}
Using this relation
(and doing some work), one reduces the assertions in
Theorem~\ref{T101} to the corresponding statements for $f$.
Thus the local degree can still be interpreted as a local 
expansion factor for the hyperbolic metric on $\BerkPone(K)$,
when this metric is defined as in~\S\ref{S247}.
In particular, Corollaries~\ref{C202} and~\ref{C106}
remain valid. 
Finally, the pullback of measures is defined using the local
degree as in~\S\ref{S134} and formulas~\eqref{e103}--\eqref{e222}
continue to hold.
%
%
\subsubsection{Trivially valued fields}\label{S266}
Finally, let us consider the case when $K$ is trivially valued.
First assume $K$ is algebraically closed.
The Berkovich projective line $\BerkPone$ is discussed
in~\S\ref{S262} (see also~\S\ref{S211} below).
In particular, the Berkovich projective line is a cone over the usual
projective line. In other words,
$\BerkPone\simeq\P^1\times[0,\infty]/\sim$,
where $(x,0)\sim(y,0)$ for any $x,y\in\P^1$.
This common point $(x,0)$ is the Gauss point in any 
coordinate. See Figure~\ref{F201}. 
The generalized metric 
on $\BerkPone$ is induced by the parametrization 
$\a:\BerkPone\to[0,+\infty]$ given by $\a(x,t)=t$.

Any rational map $f:\P^1\to\P^1$ of degree $d\ge 1$
induces a selfmap of $\BerkPone$ that fixes the Gauss point.
The local degree is $d$ at the Gauss point. At any point
$(x,t)$ with $t>0$, the local degree is equal to the local 
degree of $f:\P^1\to\P^1$ at $x$. Moreover, 
$f(x,t)=(f(x),t\deg_x(f))$, so $f$ expands the hyperbolic 
metric by a factor equal to the local degree, in accordance
with Theorem~\ref{T101}.

Finally, the case when $K$ is trivially valued but not 
algebraically closed can be treated by passing to the 
algebraic closure $K^a$ (which is of course already 
complete under the trivial norm).
%
%
%
%
\subsection{Notes and further references}
A rational map on the Berkovich projective line is a special case 
of a finite morphism between Berkovich curves, so various
results from~\cite{BerkBook,Berkihes} apply. Nevertheless,
it is instructive to see the mapping properties in more detail,
in particular the interaction with the tree structure.

The fact that the Berkovich projective line can be understood
from many different points of view means that there are several 
ways of defining the action of a rational map. 
In his thesis and early work, Rivera-Letelier viewed the action 
as an extension from $\P^1$ to the hyperbolic space $\H$,
whose points he identified with nested collections of closed discs
as in~\S\ref{S167}.
The definition in~\cite[\S2.3]{BRBook} uses homogeneous coordinates
through a ``Proj'' construction of the Berkovich projective line
whereas~\cite{FR2} simply used the (coordinate-dependent)
decomposition $\BerkPone=\BerkAone\cup\{\infty\}$.
Our definition here seems to be new, but it is of
course not very different from what is already in the literature.
As in~\S\ref{S101}, it is guided by the principle
of trying to work without coordinates whenever possible. 

There are some important techniques that we have not touched upon,
in particular those that take place on the classical (as opposed
to Berkovich) affine and projective lines. For example, we never
employ Newton polygons even though these can be useful
see~\cite[\S{A}.10]{BRBook} or~\cite[\S3.2]{BenedettoNotes}.

The definition of the local degree
is taken from~\cite{FR2} but appears already
in~\cite{ThuillierThesis} and is the natural one in the general
context of finite maps between Berkovich spaces. 
In the early work of Rivera-Letelier, a different definition 
was used, modeled on Theorem~\ref{T101}.
The definition of the local degree (called multiplicity there)
in~\cite{BRBook} uses potential theory and is designed 
to make~\eqref{e103} hold.

As noted by Favre and Rivera-Letelier, Proposition~\ref{P102}
implies that all these different definitions coincide. 
Having said that, I felt it was useful to have a proof of
Theorem~\ref{T101} that is directly based on the algebraic
definition of the local degree. The proof presented here
seems to be new although many of the ingredients are not.

The structure of the ramification locus in the case of 
positive residue characteristic is very interesting.
We refer to~\cite{Faber1,Faber2,Faber3} for details.
%
%
%
%
%
%
\newpage
\section{Dynamics of rational maps in one variable}\label{S127}
Now that we have defined the action of a rational
map on the Berkovich projective line, we would like to
study the dynamical system obtained by iterating 
the map. 
While it took people some time to realize that 
considering the dynamics on $\BerkP$ (as opposed to $\P^1$)
could be useful, it has become abundantly clear
that this is the right thing to do for many questions.

It is beyond the scope of these notes to give an overview
of all the known results in this setting. Instead, in order
to illustrate some of the main ideas, we shall focus on 
an equidistribution theorem due to Favre and Rivera-Letelier~\cite{FR2},
as well as some of its consequences.
For these results we shall, on the other hand, give more or
less self-contained proofs. 

For results not covered 
here---notably on the structure of Fatou and Julia sets---we 
recommend the book~\cite{BRBook} by Baker and Rumely
and the survey~\cite{BenedettoNotes} by Benedetto.
%
%
%
%
\subsection{Setup}
We work over a fixed non-Archimedean field $K$,
of any characteristic.
For simplicity we shall assume that $K$ is algebraically
closed and nontrivially valued. 
The general case is discussed in~\S\ref{S255}.

Fix a rational map $f:\P^1\to\P^1$ of degree $d>1$.
Our approach will be largely coordinate free, but in any case, note
that since we are to study the dynamics of $f$, we must choose
the same coordinates on the source and target. Given a coordinate
$z$, $f^*z$ is a rational function in $z$ of degree $d$.
%
%
%
%
\subsection{Periodic points}\label{S161}
When analyzing a dynamical system, one of the first
things to look at are periodic points. We say that
$x\in\BerkPone$ is a \emph{fixed point} if $f(x)=x$
and a \emph{periodic point} if $f^n(x)=x$ for some $n\ge 1$.
%
%
\subsubsection{Classical periodic points}
First suppose $x=f^n(x)\in\P^1$ is a classical 
periodic point and pick a coordinate $z$ on $\P^1$ vanishing at $x$.
Then
\begin{equation*}
  f^{*n}z=\lambda z+O(z^2)
\end{equation*}
where $\lambda\in K$ is the \emph{multiplier} of 
the periodic point.
We say that $x$ is \emph{attracting} 
if $|\lambda|<1$, \emph{neutral} if $|\lambda|=1$ and 
\emph{repelling} if $|\lambda|>1$. 
The terminology is more or less self-explanatory. For example,
if $x$ is attracting, then there exists a small disc $D\subseteq\P^1$
containing $x$ such that $f^n(D)\subseteq D$ and $f^{nm}(y)\to x$ as 
$m\to\infty$ for every $y\in D$.

The \emph{multiplicity} of a periodic point $x=f^n(x)$ is 
the order of vanishing at $x$ of the rational function 
$f^{n*}z-z$ for any coordinate $z\in F$ vanishing at $x$.
It is easy to see that $f$ has $d+1$ fixed points counted with
multiplicity. Any periodic point of multiplicity at least two must
have multiplier $\lambda=1$.

\begin{Prop}\label{P209}
  Let $f:\P^1\to\P^1$ be a rational map of degree $d>1$.
  \begin{itemize}
  \item[(i)]
    There exist infinitely many distinct classical periodic points.
  \item[(ii)]
    There exists at least one classical nonrepelling fixed point.
  \item[(iii)]
    Any nonrepelling classical fixed point admits a basis of open 
    neighborhoods $U\subseteq\BerkPone$ 
    that are invariant, \ie $f(U)\subseteq U$.
  \end{itemize}
\end{Prop}
Statement~(i) when $K=\C$ goes back at least to Julia. 
A much more precise result was proved by
by I.~N.~Baker~\cite{Bak64}.
Statements~(ii) and~(iii) are due to Benedetto~\cite{BenedettoThesis}
who adapted an argument used by Julia.
\begin{proof}[Sketch of proof]
  To prove~(i) we follow~\cite[pp.102--103]{Beardon}
  and~\cite[Corollary~4.7]{SilvBook}.
  We claim that the following holds 
  for all but at most $d+2$ primes $q$:
  any classical point $x$ with $f(x)=x$ has the same multiplicity 
  as a fixed point of $f$ and as a fixed point of $f^q$.
  This will show that $f^q$ has $d^q-d>1$ fixed points
  (counted with multiplicity) that are not fixed points of $f$.
  In particular, $f$ has infinitely many distinct classical
  periodic points.
  
  To prove the claim, consider a fixed point
  $x\in\P^1$ and pick a coordinate $z\in F$ vanishing at $x$.
  We can write $f^*z=az+bz^{r+1}+O(z^{r+2})$, where $a,b\in K^*$
  and $r>0$. One proves by induction that 
  \begin{equation*}
    f^{n*}z=a^nz+b_nz^{r+1}+O(z^{r+2}),
  \end{equation*}
  where $b_n=a^{n-1}b(1+a^r+\dots+a^{(n-1)r})$.
  If $a\ne 1$, then for all but at most one prime $q$
  we have $a^q\ne 1$ and hence $x$ is a fixed point of
  multiplicity one for both $f$ and $f^q$.
  If instead $a=1$, then $b_q=qb$, so if $q$
  is different from the characteristic of $K$,
  then $x$ is a fixed point of multiplicity 
  $r$ for both $f$ and $f^q$.

  Next we prove~(ii), following~\cite[\S1.3]{BenedettoNotes}.
  Any fixed point of $f$ of multiplicity at least two is 
  nonrepelling, so we may assume that $f$ has exactly
  $d+1$ fixed points $(x_i)_{i=1}^{d+1}$. 
  Let $(\lambda_i)_{i=1}^{d+1}$ be the corresponding multipliers. 
  Hence $\lambda_i\ne 1$ for all $i$. 
  it follows from the Residue Theorem 
  (see~\cite[Theorem~1.6]{BenedettoNotes}) that
  \begin{equation*}
    \sum_{i=1}^{d+1}\frac1{1-\lambda_i}=1.
  \end{equation*}
  If $|\lambda_i|>1$ for all $i$, 
  then the left hand side would have norm $<1$, a
  contradiction. Hence $|\lambda_i|\le 1$ for some $i$
  and then $x_i$ is a nonrepelling fixed point.

  Finally we prove~(iii). Pick a coordinate $z\in F$ vanishing at
  $x$ and write $f^*z=\lambda z+O(z^2)$, with $|\lambda|\le 1$.
  For $0<r\ll1$ we have $f(x_{D(0,r)})=x_{D(0,r')}$, where 
  $r'=|\lambda|r\le r$. Let $U_r:=U(\vv_r)$, where $\vv_r$ is the 
  tangent direction at $x_{D(0,r)}$ determined by $x$. 
  The sets $U_r$ form a basis of open neighborhoods of $x$
  and it follows from Corollary~\ref{C107}~(ii) that
  $f(U_r)\subseteq U_r$ for $r$ small enough.
\end{proof}
%
%
\subsubsection{Nonclassical periodic points}
We say that a fixed point $x=f(x)\in\H$ is \emph{repelling} 
if $\deg_x(f)>1$ and \emph{neutral} otherwise (points in $\H$ cannot be
attracting). This is justified by the interpretation of the 
local degree as an expansion factor in the hyperbolic metric,
see Theorem~\ref{T101}.

The following result is due to 
Rivera-Letelier~\cite[Lemme~5.4]{Rivera2}.
\begin{Prop}\label{P120}
  Any repelling fixed point $x\in\H$ must be of Type 2.
\end{Prop}
\begin{proof}[Sketch of proof]
  We can rule out that $x$ is of Type~3 using value groups.
  Indeed, by~\eqref{e124}  the local degree of $f$ at a Type~3
  point is equal to index of the value group $\Gamma_{f(x)}$ as a
  subgroup of $\Gamma_x$, so if $f(x)=x$, 
  then the local degree is one. 

  I am not aware of an argument of the same style to
  rule out repelling points of Type~4. Instead, Rivera-Letelier
  argues by contradiction. Using Newton polygons he shows
  that any neighborhood of a repelling fixed point of Type~4 
  would contain a classical fixed point. Since there are only
  finitely many classical fixed points, this gives a contradiction.
  See the original paper by Rivera-Letelier
  or~\cite[Lemma~10.80]{BRBook}.
\end{proof}
%
%
\subsubsection{Construction of fixed points}
Beyond Proposition~\ref{P209} there are at least two other 
methods for producing fixed points.

First, one can use Newton polygons to produce 
classical fixed points. This was alluded to in the proof 
of Proposition~\ref{P120} above. 
We shall not describe this further here but instead 
refer the reader to~\cite[\S3.2]{BenedettoNotes} 
and~\cite[{\S}A.10]{BRBook}.

Second, one can use topology. Since $f$ can be viewed as
a tree map, Proposition~\ref{P119} applies and provides
a fixed point in $\BerkPone$. This argument can be refined, 
using that $f$ expands the hyperbolic metric, to produce either
attracting or repelling fixed points. 
See~\cite[\S10.7]{BRBook}.
%
%
%
%
\subsection{Purely inseparable maps}\label{S294}
Suppose $f$ is purely inseparable of degree
$d>1$. In particular, $\charac K=p>0$.
We claim that there exists a coordinate $z\in F$
and $n\ge 1$ such that $f^*z=z^{p^n}$. 
A rational map $f$ such that $f^*z=z^p$ is usually
called the \emph{Frobenius map}, see~\cite[2.4.1--2.4.2]{Hartshorne}.

To prove the claim, we use the fact that
$f$ admits exactly $d+1$ classical fixed points.
Indeed, the multiplier of each fixed point is zero.
Pick a coordinate $z\in F$ such that $z=0$ and $z=\infty$
are fixed points of $f$. 
Since $f$ is purely inseparable there exists 
$n\ge 0$ such that $z^{p^n}\in f^*F$. 
Choose $n$ minimal with this property. Since $\deg f>1$
we must have $n\ge 1$. On the other hand,
the minimality of $n$ shows that 
$z^{p^n}=f^*w$ for some coordinate $w\in F$. 
The fact that $z=0$ and $z=\infty$ are fixed points
imply that $z=aw$ for some $a\in K^*$,
so $f^*z=az^{p^n}$. 
After multiplying $z$ by a 
suitable power of $a$, we get $a=1$, proving the claim.
%
%
%
%
\subsection{The exceptional set}
A classical point $x\in\P^1$ is called \emph{exceptional} 
for $f$ if its total backward orbit $\bigcup_{n\ge 0}f^{-n}(x)$
is finite. The \emph{exceptional set} of $f$ is the set of
exceptional points and denoted $E_f$. Since $f$ is surjective,
it is clear that $E_{f^n}=E_f$ for any $n\ge 1$.
We emphasize that $E_f$ by definition consists of classical
points only.
\begin{Lemma}\label{L213}
  Let $f:\P^1\to\P^1$ be a rational map of degree $d>1$.
  \begin{itemize}
  \item[(i)]
    If $f$ is not purely inseparable, then there are at most two 
    exceptional points. Moreover:
    \begin{itemize}
    \item[(a)]
      if there are two
      exceptional points, then $f(z)=z^{\pm d}$ in a suitable
      coordinate $z$ on $\P^1$
      and $E_f=\{0,\infty\}$;
    \item[(b)]
      if there is exactly 
      one exceptional point,
      then $f$ is a polynomial in a suitable coordinate
      and $E_f=\{\infty\}$.
    \end{itemize}
  \item[(ii)]
    If $f$ is purely inseparable, then the exceptional set is countably 
    infinite and consists of all periodic points of $f$.
  \end{itemize}
\end{Lemma}
Case~(ii) only occurs when $\charac K=p>0$ and $f$
is an iterate of the Frobenius map: $f^*z=z^d$ for $d$ a power of $p$ 
in some coordinate $z\in F$, see~\S\ref{S294}.
\begin{proof}
  For $x\in E_f$ set $F_x:=\bigcup_{n\ge 0}f^{-n}(x)$.
  Then $F_x$ is a finite set with $f^{-1}(F_x)\subseteq F_x\subseteq E_f$.
  Since $f$ is surjective, $f^{-1}(F_x)=F_x=f(F_x)$.
  Hence each point in $F_x$ must be totally ramified 
  in the sense that $f^{-1}(f(x))=\{x\}$.
  
  If $f$ is purely inseparable, then every point in $\P^1$
  is totally ramified, so $F_x$ is finite iff $x$ is periodic. 
  
  If $f$ is not purely inseparable, then it follows from
  Proposition~\ref{P210}~(i) that $E_f$ has at most two elements. 
  The remaining statements are easily verified.
\end{proof}
%
%
%
%
\subsection{Maps of simple reduction}\label{S160}
By definition, the exceptional set consists of classical points
only. The following result by Rivera-Letelier~\cite{Rivera2}
characterizes totally invariant points in hyperbolic space.
\begin{Prop}\label{P207}
  If $x_0\in\H$ is a totally invariant point, $f^{-1}(x_0)=x_0$, then 
  $x_0$ is a Type~2 point.
\end{Prop}
\begin{Def}\label{D202}
  A rational map $f:\P^1\to\P^1$ has \emph{simple reduction}
  if there exists a Type~2 point that is totally invariant for $f$.
\end{Def}
\begin{Remark}
  Suppose $f$ has simple reduction and 
  pick a coordinate $z$ in which the totally invariant Type~2 point
  becomes the Gauss point.
  Then we can write $f^*z=\phi/\psi$,
  where $\phi,\psi\in\fo_K[z]$ and where the rational function
  $\tphi/\tpsi\in\tK(z)$ has
  degree $d=\deg f$. Such a map is usually said to have 
  \emph{good reduction}~\cite{MortonSilverman}.
  Some authors refer to simple reduction 
  as \emph{potentially good reduction}.
  One could argue that dynamically speaking, maps of
  good or simple reduction are not the most interesting ones,
  but they do play an important role.
  For more on this, see~\cite{Benedetto4,Bak09,PST}.
\end{Remark}
\begin{proof}[Proof of Proposition~\ref{P207}]
  A totally invariant point in $\H$ is repelling so the
  result follows from Proposition~\ref{P120}. 
  Nevertheless, we give an alternative proof.

  Define a function $G:\BerkPone\times\BerkPone\to[-\infty,0]$  
  by\footnote{In~\cite{Bak09,BRBook}, the function $-G$ is called the 
    normalized Arakelov-Green's function with respect to the Dirac mass at $x_0$.} 
  \begin{equation*}
    G(x,y)=-d_\H(x_0,x\wedge_{x_0}y).
  \end{equation*}    
  It is characterized by the following properties:
  $G(y,x)=G(x,y)$, $G(x_0,y)=0$ and 
  $\Delta G(\cdot,y)=\delta_y-\delta_{x_0}$.
  
  Pick any point $y\in\BerkPone$. Let $(y_i)_{i=1}^m$ be the preimages
  of $y$ under $f$ and $d_i=\deg_{y_i}(f)$ the corresponding local degrees.
  We claim that
  \begin{equation}\label{e228}
    G(f(x),y)=\sum_{i=1}^m d_iG(x,y_i)
  \end{equation}
  for any $x\in\BerkPone$. 
  To see this, note that 
  since $f^*\delta_{x_0}=d\delta_{x_0}$ it follows from 
  Proposition~\ref{P104} that both sides of~\eqref{e228} are 
  $d\delta_{x_0}$-subharmonic as a function of $x$, with 
  Laplacian $f^*(\delta_y-\delta_{x_0})=\sum_id_i(\delta_{y_i}-\delta_{x_0})$. 
  Now, the Laplacian determines a quasisubharmonic 
  function up to a constant, so
  since both sides of~\eqref{e228} vanish when $x=x_0$
  they must be equal for all $x$, proving the claim.

  Now pick $x$ and $y$ as distinct classical fixed points of $f$.
  Such points exist after replacing $f$ by an iterate,
  see Proposition~\ref{P209}.
  We may assume $y_1=y$. Then~\eqref{e228} gives
  \begin{equation}\label{e229}
    (d_1-1)G(x,y)+\sum_{i\ge2}d_iG(x,y_i)=0
  \end{equation}
  Since $G\le 0$, we must have $G(x,y_i)=0$ for $i\ge 2$ and
  $(d_1-1)G(x,y)=0$.
  
  First assume $x_0$ is of Type~4. Then $x_0$
  is an end in the tree $\BerkPone$,
  so since $x\ne x_0$ and $y_i\ne x_0$ for all $i$, we have
  $x\wedge_{x_0}y_i\ne x_0$ and hence $G(x,y_i)<0$. 
  This contradicts~\eqref{e229}.

  Now assume $x_0$ is of Type~3. Then there are exactly two tangent directions
  at $x_0$ in the tree $\BerkPone$. Replacing $f$ by an iterate, we may assume 
  that these are invariant under the tangent map. We may assume that the 
  classical fixed points $x,y\in\P^1$ above represent the same tangent 
  direction, so that $x\wedge_{x_0}y\ne x_0$. 
  Since $x_0$ is totally invariant, it follows from Corollary~\ref{C107}~(i)
  that all the preimages $y_i$ of $y$ also represent this tangent vector at $x_0$.
  Thus $G(x,y_i)<0$ for all $i$ which again contradicts~\eqref{e229}.
\end{proof}
\begin{Remark}
  The proof in~\cite{Bak09} also uses the function $G$ above 
  and analyzes the lifting of $f$ as a homogeneous polynomial
  map of $K\times K$. 
\end{Remark}
%
%
%
%
\subsection{Fatou and Julia sets}\label{S289}
In the early part of the 20th century, Fatou and Julia
developed a general theory of iteration of rational
maps on the Riemann sphere. Based upon some of those
results, we make the following definition.
\begin{Def}
  The \emph{Julia set} $\cJ=\cJ_f$ is the set of points $x\in\BerkP$
  such that for every open neighborhood $U$ of $x$ we have
  $\bigcup_{n\ge 0}f^n(U)\supseteq\BerkP\setminus E_f$.
  The \emph{Fatou set} is the complement of the Julia set.
\end{Def}
\begin{Remark}\label{R202}
  Over the complex numbers, one usually defines the Fatou
  set as the largest open subset of the Riemann sphere where
  the sequence of iterates is locally equicontinuous.
  One then shows that the Julia set is characterized by 
  the conditions in the definition above.
  Very recently, a non-Archimedean version of this was 
  found by Favre, Kiwi and Trucco, see~\cite[Theorem~5.4]{FKT}.
  Namely, a point $x\in\BerkPone$ belongs to the Fatou set of $f$
  iff the family $\{f^n\}_{n\ge 1}$ is normal in a neighborhood of $x$
  in a suitable sense. We refer to~\cite[\S5]{FKT} for the definition
  of normality, but point out that the situation is more subtle 
  in the non-Archimedean case than over the complex numbers.
\end{Remark}
\begin{Thm}\label{T204}
  Let $f:\P^1\to\P^1$ be any rational map of degree $d>1$.
  \begin{itemize}
  \item[(i)]
    The Fatou set $\cF$ and Julia set $\cJ$ are totally invariant: 
    $\cF=f(\cF)=f^{-1}(\cF)$ and $\cJ=f(\cJ)=f^{-1}(\cJ)$.
  \item[(ii)]
    We have $\cF_f=\cF_{f^n}$ and $\cJ_f=\cJ_{f^n}$ for all $n\ge 1$.
  \item[(iii)]
    The Fatou set is open and dense in $\BerkPone$.
    It contains any nonrepelling classical periodic
    point and in particular any exceptional point.
  \item[(iv)]
    The Julia set is nonempty, compact and has empty interior.
    Further:
    \begin{itemize}
    \item[(a)]
      if $f$ has simple reduction, then $\cJ$ consists of a single Type~2 point;
    \item[(b)]
      if $f$ does not have simple reduction, then $\cJ$ is a 
      perfect set, that is, it has no isolated points.
    \end{itemize}
\end{itemize}
\end{Thm}
\begin{proof}
  It is clear that $\cF$ is open. 
  Since $f:\BerkPone\to\BerkPone$ is an open continuous map, 
  it follows that $\cF$ is totally invariant.
  Hence $\cJ$ is compact and totally invariant.
  The fact that $\cF_{f^n}=\cF_f$, and hence $\cJ_{f^n}=\cJ_f$,
  follow from the total invariance of $E_f=E_{f^n}$.
  
  It follows from Proposition~\ref{P209} that any nonrepelling
  classical periodic point is in the Fatou set. Since such points exist,
  the Fatou set is nonempty. This also implies that the 
  Julia set has nonempty interior. Indeed, if $U$ were an open set 
  contained in the Julia set, then the set $U':=\bigcup_{n\ge1}f^n(U)$ 
  would be contained in the Julia set for all $n\ge1$. 
  Since the Fatou set is open and 
  nonempty, it is not contained in $E_f$, 
  hence $\BerkPone\setminus U'\not\subseteq E_f$,
  so that $U\subseteq\cF$, a contradiction.
  
  The fact that the Julia set is nonempty and that 
  properties~(a) and~(b) hold is nontrivial and will
  be proved in~\S\ref{S284} as a 
  consequence of the equidistribution theorem below.
  See Propositions~\ref{P204} and~\ref{P208}.
\end{proof}

Much more is known about the Fatou and Julia set than what 
is presented here. 
For example, 
as an analogue of the classical result by Fatou and Julia,
Rivera-Letelier proved that $\cJ$ is the closure of the repelling
periodic points of $f$. 

For a polynomial map, the Julia set is also the boundary of the 
filled Julia set, that is, the set of points whose orbits are
bounded in the sense that they are disjoint from a fixed open
neighborhood of infinity. See~\cite[Theorem~10.91]{BRBook}.

Finally, a great deal is known about the dynamics on the Fatou set.
We shall not study this here. Instead we
refer to~\cite{BRBook,BenedettoNotes}.
%
%
%
%
\subsection{Equidistribution theorem}\label{S301}
The following result that describes the distribution of
preimages of points under iteration was proved 
by Favre and Rivera-Letelier~\cite{FR3,FR2}.
The corresponding result over the complex numbers
is due to Brolin~\cite{Brolin} for polynomials and to 
Lyubich~\cite{Lyubich} and Freire-Lopez-Ma\~n\'e~\cite{FLM} 
for rational functions.
\begin{Thm}\label{T201}
  Let $f:\P^1\to\P^1$ be a rational map of degree
  $d>1$.
  Then there exists a unique Radon 
  probability measure $\rho_f$ on $\BerkP$
  with the following property: 
  if $\rho$ is a Radon probability measure on $\BerkPone$,
  then   
  \begin{equation*}
    \frac1{d^n}f^{n*}\rho\to\rho_f
    \quad\text{as $n\to\infty$},
  \end{equation*}
  in the weak sense of measures, iff $\rho(E_f)=0$.
  The measure
  $\rho_f$ puts no mass on any classical point;
  in particular $\rho_f(E_f)=0$. It is totally invariant
  in the sense that $f^*\rho_f=d\rho_f$.
\end{Thm}
Recall that we have assumed that the ground field $K$ is algebraically 
closed and nontrivially valued. 
See~\S\ref{S255} for the general case.

As a consequence of Theorem~\ref{T201}, 
we obtain a more general version of Theorem~A from the introduction, 
namely
\begin{Cor}\label{C201}
  With $f$ as above, we have 
 \begin{equation*}
    \frac1{d^n}\sum_{f^n(y)=x}\deg_y(f^n)\delta_y\to\rho_f
    \quad\text{as $n\to\infty$},
  \end{equation*}
  for any non-exceptional point $x\in\BerkPone\setminus E_f$.
\end{Cor}
Following~\cite{BRBook} we call 
$\rho_f$ the \emph{canonical measure} of $f$.
It is clear that $\rho_f=\rho_{f^n}$ for $n\ge 1$. 
The proof of Theorem~\ref{T201} will be given in~\S\ref{S202}.
\begin{Remark}
 Okuyama~\cite{Oku11b} has proved a 
 quantitative strengthening of Corollary~\ref{C201}.
 The canonical measure is also expected to describe the
 distribution of repelling periodic points. This does not
 seem to be established full generality, but is known in 
 many cases~\cite{Oku11a}.
\end{Remark}
%
%
%
%
\subsection{Consequences of the equidistribution theorem}\label{S284}
In this section we collect some result that follow from Theorem~\ref{T201}.
\begin{Prop}\label{P204}
  The support of the measure $\rho_f$ is 
  exactly the Julia set $\cJ=\cJ_f$. 
  In particular, $\cJ$ is nonempty.
\end{Prop}
\begin{proof}
  First note that the support of $\rho_f$
  is totally invariant. This follows formally from
  the total invariance of $\rho_f$. 
  Further, the support of $\rho_f$ cannot be contained in 
  the exceptional set $E_f$ since $\rho_f(E_f)=0$.

  Consider a point $x\in\BerkP$.
  If $x$ is not in the support of $\rho_f$,
  let $U=\BerkP\setminus\supp\rho_f$.
  Then $f^n(U)=U$ for all $n$. In particular,
  $\bigcup_{n\ge 0}f^n(U)$ is disjoint from 
  $\supp\rho_f$. Since $\supp\rho_f\not\subseteq E_f$,
  $x$ must belong to the Fatou set.

  Conversely, if $x\in\supp\rho_f$
  and $U$ is any open neighborhood of
  $x$, then $\rho_f(U)>0$. 
  For any $y\in\BerkPone\setminus E_f$,
  Corollary~\ref{C201}  implies that $f^{-n}(y)\cap U\ne\emptyset$ for
  $n\gg0$. We conclude that 
  $\bigcup_{n\ge 0}f^n(U)\supseteq\BerkP\setminus E_f$,
  so $x$ belongs to the Julia set.
\end{proof}
We will not study the equilibrium measure $\rho_f$ in detail,
but the following result is not hard to deduce from what we 
already know.
\begin{Prop}\label{P116}
  The following conditions are equivalent.
  \begin{itemize}
  \item[(i)]
    $\rho_f$ puts mass at some point in $\BerkPone$;
  \item[(ii)]
    $\rho_f$ is a Dirac mass at a Type~2 point;
  \item[(iii)]
    $f$ has simple reduction;
  \item[(iv)]
    $f^n$ has simple reduction for all $n\ge 1$;
  \item[(v)]
    $f^n$ has simple reduction for some $n\ge 1$.
  \end{itemize}
\end{Prop}
\begin{proof}
  If $f$ has simple reduction then, by definition, there exists
  a totally invariant Type~2 point $x_0$. We then have 
  $d^{-n}f^{n*}\delta_{x_0}=\delta_{x_0}$ so Corollary~\ref{C201}
  implies $\rho_f=\delta_{x_0}$. 
  Conversely, if $\rho_f=\delta_{x_0}$ for some Type~2 point 
  $x_0$, then $f^*\rho_f=d\rho_f$ implies that 
  $x_0$ is totally invariant, so that $f$ has simple reduction.
  Thus~(ii) and~(iii) are equivalent.
  Since $\rho_f=\rho_{f^n}$, this implies that~(ii)--(v) are
  equivalent.

  Clearly~(ii) implies~(i). We complete the proof by proving 
  that~(i) implies~(v).
  Thus suppose $\rho_f\{x_0\}>0$ for some $x_0\in\BerkPone$.
  Since $\rho_f$ does not put mass on classical points we have 
  $x_0\in\H$. 
  The total invariance of $\rho_f$ implies 
  \begin{equation*}
    0<\rho_f\{x_0\}
    =\frac1d(f^*\rho_f)\{x_0\}
    =\frac1d\deg_{x_0}(f)\rho_f\{f(x_0)\}
    \le\rho_f\{f(x_0)\},
  \end{equation*}
  with equality iff $\deg_{x_0}(f)=d$. 
  Write $x_n=f^n(x_0)$ for $n\ge0$.
  Now the total mass of $\rho_f$ is finite, so
  after replacing $x_0$ by $x_m$ for some 
  $m\ge 0$ we may assume that $x_n=x_0$
  and $\deg_{x_j}(f)=d$ for $0\le j<n$ and some $n\ge 1$.
  This implies that $x_0$ is totally invariant under $f^n$.
  By Proposition~\ref{P207}, $x_0$ is then a Type~2 point
  and $f^n$ has simple reduction.
\end{proof}
With the following result we complete the proof of Theorem~\ref{T204}.
\begin{Prop}\label{P208}
  Let $f:\P^1\to\P^1$ be a rational map of degree $d>1$ and 
  let $\cJ=\cJ_f$ be the Julia set of $f$.
  \begin{itemize}
  \item[(i)]
    If $f$ has simple reduction, then $\cJ$ consists of a single Type~2 point.
 \item[(ii)]
   If $f$ does not have simple reduction, then $\cJ$ is a 
   perfect set.
 \end{itemize}
\end{Prop}
\begin{proof}
  Statement~(i) is a direct consequence of Proposition~\ref{P116}.
  Now suppose $f$ does not have simple reduction. 
  Pick any point $x\in\cJ$ and an open neighborhood 
  $U$ of $x$. It suffices to prove that there exists a point 
  $y\in U$ with $y\ne x$ and $f^n(y)=x$ for some $n\ge1$.
  After replacing $f$ by an iterate we may assume 
  that $x$ is either fixed or not periodic.
  Set $m:=\deg_x(f)$ if $f(x)=x$ and $m:=0$ otherwise.
  Note that $m<d$ as $x$ is not totally invariant.

  Since $x\not\in E_f$, Corollary~\ref{C201} shows that the 
  measure $d^{-n}f^{n*}\delta_x$ converges weakly to $\rho_f$.
  Write $f^{n*}\delta_x=m^n\delta_x+\rho'_n$,
  where 
  \begin{equation*}
    \rho'_n=\sum_{y\ne x, f^n(y)=x}\deg_y(f^n)\delta_y.
  \end{equation*}
  We have $x\in\cJ=\supp\rho_f$ so $\rho_f(U)>0$
  and hence $\liminf_{n\to\infty}(d^{-n}f^{n*}\delta_x)(U)>0$. 
  Since $m<d$ it follows that $\rho'_n(U)>0$ for $n\gg0$.
  Thus there exist points $y\in U$ with $y\ne x$ and $f^n(y)=x$. 
\end{proof}
%
%
%
%
\subsection{Proof of the equidistribution theorem}\label{S202}
To prove the equidistribution theorem we follow 
the approach of Favre and Rivera-Letelier~\cite{FR2}, 
who in turn adapted 
potential-theoretic techniques from complex dynamics 
developed by Forn{\ae}ss-Sibony and others. 
Using the tree Laplacian defined in~\S\ref{S116}
we can study convergence of measures 
in terms of convergence of quasisubharmonic 
functions, a problem for which there are good techniques. 
If anything, the analysis is easier in the nonarchimedean case.
Our proof does differ from the one in~\cite{FR2} in that it avoids
studying the dynamics on the Fatou set.
%
%
\subsubsection{Construction of the canonical measure}\label{S285}
Fix a point $x_0\in\H$.
Since $d^{-1}f^*\delta_{x_0}$ is a probability measure, we 
have 
\begin{equation}\label{e105}
 d^{-1}f^*\delta_{x_0}=\delta_{x_0}+\Delta u
\end{equation} 
for an $x_0$-subharmonic function $u$.
In fact,~\eqref{e122} gives an explicit expression for $u$ and
shows that $u$ is continuous, since $f^{-1}(x_0)\subseteq\H$.

Iterating~\eqref{e105} and using~\eqref{e103} leads to
\begin{equation}
  d^{-n}f^{n*}\delta_{x_0}=\delta_{x_0}+\Delta u_n,
\end{equation}
where $u_n=\sum_{j=0}^{n-1}d^{-j}u\circ f^j$.
It is clear that the sequence $u_n$ converges
uniformly to a continuous $x_0$-subharmonic function $u_\infty$.
We set 
\begin{equation*}
  \rho_f:=\delta_{x_0}+\Delta u_\infty.
\end{equation*}
Since $u_\infty$ is bounded, it follows from~\eqref{e137} 
that $\rho_f$ does not put mass on any classical point.
In particular, $\rho_f(E_f)=0$, since $E_f$ is at most countable.
%
%
\subsubsection{Auxiliary results}\label{S286}
Before starting the proof of equidistribution, let us record
a few results that we need.
\begin{Lemma}\label{L202}
  If $x_0,x\in\H$, then $d_\H(f^n(x),x_0)=O(d^n)$ as $n\to\infty$.
\end{Lemma}
\begin{proof}
  We know that $f$ expands the hyperbolic metric by
  a factor at most $d$, see Corollary~\ref{C202}. 
  Using the triangle inequality and the
  assumption $d\ge2$, this yields
  \begin{equation*}
    d_\H(f^n(x),x)
    \le\sum_{j=0}^{n-1}d_\H(f^{j+1}(x),f^j(x))
    \le\sum_{j=0}^{n-1}d^jd_\H(f(x),x)
    \le d^nd_\H(f(x),x),
  \end{equation*}
  so that 
  \begin{align*}
    d_\H(f^n(x),x_0)
    &\le d_\H(f^n(x),f^n(x_0)) + d_\H(f^n(x_0),x_0)\\
    &\le d^n(d_\H(x,x_0) + d_\H(f(x_0),x_0)),
  \end{align*}
  completing the proof.
\end{proof}
\begin{Lemma}\label{L203}
  Suppose that $f$ is not purely inseparable.
  If $\rho$ is a Radon probability measure on $\BerkPone$
  such that $\rho(E_f)=0$ and we set $\rho_n:=d^{-n}f^{n*}\rho$,
  then $\sup_{y\in\P^1}\rho_n\{y\}\to0$ as $n\to\infty$.
\end{Lemma}
Note that the supremum is taken over classical points only.
Also note that the lemma always applies if the ground field
is of characteristic zero. However, the lemma is false for
purely inseparable maps.
\begin{proof}
  We have $\rho_n\{y\}=d^{-n}\deg_y(f^n)\rho\{f^n(y)\}$,
  so it suffices to show that 
  \begin{equation}\label{e223}
    \sup_{y\in\P^1\setminus E_f}\deg_y(f^n)=o(d^n).
  \end{equation}
  For $y\in\P^1$ and $n\ge0$, 
  write $y_n=f^n(y)$. If $\deg_{y_n}(f)=d$ for 
  $n=0,1,2$, then Proposition~\ref{P210}~(i)
  implies $y\in E_f$. 
  Thus $\deg_y(f^3)\le d^3-1$ 
  and hence $\deg_y(f^n)\le d^2(d^3-1)^{n/3}$
  for $y\in\P^1\setminus E_f$, completing the proof.
\end{proof}
%
%
\subsubsection{Proof of the equidistribution theorem}\label{S287}
Let $\rho$ be a Radon probability measure
on $\BerkPone$ and set $\rho_n=d^{-n}f^{n*}\rho$. 
If $\rho(E_f)>0$, then $\rho_n(E_f)=\rho(E_f)>0$
for all $n$. Any accumulation point of $\{\rho_n\}$ must 
also put mass on $E_f$, so $\rho_n\not\to\rho_f$ as
$n\to\infty$.

Conversely, assume $\rho(E_f)=0$ and let us show that
$\rho_n\to\rho_f$ as $n\to\infty$.
Let $\f\in\SH(\BerkPone,x_0)$ 
be a solution to the equation $\rho=\delta_{x_0}+\Delta \f$.
Applying $d^{-n}f^{n*}$ to both sides of this equation
and using~\eqref{e103}, we get 
\begin{equation*}
  \rho_n
  =d^{-n}f^{n*}\delta_{x_0} + \Delta\f_n
  =\delta_{x_0}+\Delta(u_n+\f_n),
\end{equation*}
where $\f_n=d^{-n}\f\circ f^n$.
Here $\delta_{x_0}+\Delta u_n$ tends to $\rho_f$
by construction. 
We must show that $\delta_{x_0}+\Delta(u_n+\f_n)$ also 
tends to $\rho_f$. By~\S\ref{S122}, this amounts to showing that
$\f_n$ tends to zero pointwise on $\H$. 
Since $\f$ is bounded from above, we always have 
$\limsup_n\f_n\le0$. Hence it remains to show that 
\begin{equation}\label{e230}
  \liminf_{n\to\infty}\f_n(x)\ge0
  \quad\text{for any $x\in\H$}.
\end{equation}

To prove~\eqref{e230} we first consider the case when 
$f$ is not purely inseparable.
Set $\e_m=\sup_{y\in\P^1}\rho_m\{y\}$ for $m\ge 0$.
Then $\e_m\to0$ as $m\to\infty$ by Lemma~\ref{L203}.
Using Lemma~\ref{L202} and Proposition~\ref{P301} we get,
for $m,n\ge 0$
\begin{align*}
  \f_{n+m}(x)
  &=d^{-n}\f_m(f^n(x))\\
  &\ge d^{-n}\f_m(x_0)
  -d^{-n}(C_m+\e_md_\H(f^n(x),x_0))\\
  &\ge -D\e_m-C_md^{-n}
\end{align*}
for some constant $D$ independent of $m$ and $n$
and some constant $C_m$ independent of $n$.
Letting first $n\to\infty$ and then $m\to\infty$
yields $\liminf_n\f_n(x)\ge0$, completing the proof.

\medskip
Now assume $f$ is purely inseparable. In 
particular, $K$ has characteristic $p>0$,
$f$ has degree $d=p^m$ for some $m\ge 1$ and
there exists a coordinate $z\in F$ such that 
$f$ becomes an iterate of the Frobenius map: $f^*z=z^d$.

In this case, we cannot use Lemma~\ref{L203} since~\eqref{e223} 
is evidently false: the local degree is $d$ everywhere on $\BerkPone$.
On the other hand, the dynamics is simple to describe, 
see Example~\ref{E212}.
The Gauss point $x_0$ in the coordinate $z$ is (totally) invariant.
Hence $\rho_f=\delta_{x_0}$. 
The exceptional set $E_f$ is countably infinite and consists
of all classical periodic points. 
Consider the partial ordering on $\BerkPone$ rooted in $x_0$.
Then $f$ is order preserving and  
$d_\H(f^n(x),x_0)=d^nd_\H(x,x_0)$ for any $x\in\BerkPone$. 

As above, write $\rho=\delta_{x_0}+\Delta\f$, 
with $\f\in\SH(\BerkPone,x_0)$.
Pick any point $x\in\H$. 
It suffices to prove that~\eqref{e230} holds, 
where $\f_n=d^{-n}\f(f^n(x))$. 
Using Lemma~\ref{L212} and the fact that 
$d_\H(f^n(x),x_0)=d^nd_\H(x,x_0)$
it suffices to show that 
\begin{equation}\label{e231}
  \lim_{n\to\infty}\rho(Y_n)=0,
  \quad\text{where $Y_n:=\{y\ge f^n(x)\}$}.
\end{equation}
Note that for $m,n\ge1$, either $Y_{m+n}\subseteq Y_n$
or $Y_n$, $Y_{n+m}$ are disjoint. If $\rho(Y_n)\not\to0$,
there must exist a subsequence $(n_j)_j$ such that 
$Y_{n_{j+1}}\subseteq Y_{n_j}$ for all $j$ and 
$\rho(Y_{n_j)}\not\to0$. Since $d_\H(f^n(x),x_0)\to\infty$ we
must have $\bigcap_jY_{n_j}=\{y_0\}$ for a classical point $y_0\in\P^1$.
Thus $\rho\{y_0\}>0$. 
On the other hand, we claim that $y_0$ is periodic, hence
exceptional, contradicting $\rho(E_f)=0$.

To prove the claim, pick $m_1\ge1$ minimal such that 
$Y_{n_1+m_1}=f^{m_1}(Y_{n_1})\subseteq Y_{n_1}$
and set $Z_r=Y_{n_1+rm_1}=f^{rm_1}(Y_{n_1})$ for $r\ge 0$.
Then $Z_r$ forms a decreasing sequence of compact sets
whose intersection consists of a single classical point $y$,
which moreover is periodic: $f^{m_1}(y)=y$.
On the other hand, for $m\ge 1$ we have 
$Y_{n_1+m}\subseteq Y_{n_1}$ iff $m_1$ divides $m$. 
Thus we can write $n_j=n_1+r_jm_1$ with $r_j\to\infty$.
This implies that 
$\{y_0\}=\bigcap_j Y_{n_j}\subseteq\bigcap_r Z_r=\{y\}$
so that $y_0=y$ is periodic. 

The proof of Theorem~\ref{T201} is now complete.
%
%
%
%
\subsection{Other ground fields}\label{S255}
Above we worked with the assumption that our non-Archimedean field
$K$ was algebraically closed and nontrivially 
valued. Let us briefly discuss what happens for other fields, focusing 
on the equidistribution theorem and its consequences.
%
%
\subsubsection{Non-algebraically closed fields}\label{S269}
Suppose $K$ is of arbitrary characteristic and nontrivially valued 
but not algebraically closed. The Berkovich projective line $\BerkPone(K)$
and the action by a rational map were outlined in~\S\ref{S273} and~\S\ref{S267},
respectively.
Let $K^a$ be the algebraic closure of $K$ and $\hKa$ its completion.
Denote by $\pi:\BerkPone(\hKa)\to\BerkPone(K)$ the natural projection.
Write $\hf:\BerkPone(\hKa)\to\BerkPone(\hKa)$ for the induced map.
Define $E_{\hf}$ as the exceptional set for $\hf$ and set $E_f=\pi(E_\hf)$.
Then $f^{-1}(E_f)=E_f$ and $E_f$ has at most two elements, 
except if $K$ has characteristic $p$
and $f$ is purely inseparable, in which case $E_f$ is countable.

We will deduce the equidistribution result in Theorem~\ref{T201} for $f$
from the corresponding theorem for $\hf$. 
Let $\rho_\hf$ be the measure on $\BerkPone(\hKa)$ given by Theorem~\ref{T201}
and set $\rho_f=\pi_*(\rho_\hf)$. Since $E_{\hf}=\pi^{-1}(E_f)$, the measure
$\rho_f$ puts no mass on $E_f$.

Let $\rho$ be a Radon probability measure on $\BerkPone(K)$.
If $\rho(E_f)>0$, then any limit point of $d^{-n}f^{n*}\rho$ puts mass on $E_f$,
hence $d^{-n}f^{n*}\rho\not\to\rho_f$.
Now assume $\rho(E_f)=0$. 
Write $x_0$ and $\hx_0$ for the Gauss point on $\BerkPone(K)$ and 
$\BerkPone(\hKa)$, respectively,
in some coordinate on $K$. 
We have $\rho=\delta_{x_0}+\Delta\f$ for some 
$\f\in\SH(\BerkPone(K),x_0)$. The generalized metric on $\BerkPone(K)$
was defined in such a way that $\pi^*\f\in\SH(\BerkPone(\hKa),x_0)$.
Set $\hrho:=\delta_{\hx_0}+\Delta(\pi^*\f)$.
Then $\hrho$ is a Radon probability measure on $\BerkPone(\hKa)$ such that 
$\pi_*\hrho=\rho$. 
Since $E_\hf$ is countable, $\pi(E_\hf)=E_f$ and $\rho(E_f)=0$
we must have $\hrho(E_f)=0$. Theorem~\ref{T201} therefore
gives $d^{-n}\hf^{n*}\hrho\to\rho_{\hf}$ and hence
$d^{-n}f^{n*}\rho\to\rho_f$ as $n\to\infty$.
%
%
\subsubsection{Trivially valued fields}
Finally let us consider the case when $K$ is equipped with the trivial valuation.
Then the Berkovich projective line is a cone over $\P^1(K)$, 
see~\S\ref{S262}. The equidistribution theorem
can be proved essentially as above, but the proof is in fact much easier.
The measure $\rho_f$ is a Dirac mass at the Gauss point and the exceptional 
set consists of at most two points, except if $f$ is purely inseparable,
The details are left as an exercise to the reader.
%
%
%
%
\subsection{Notes and further references}\label{S290}
The equidistribution theorem is due to Favre and Rivera-Letelier.
Our proof basically follows~\cite{FR2} but avoids studying the 
dynamics on the Fatou set and instead uses the hyperbolic
metric more systematically through Proposition~\ref{P301}
and Lemmas~\ref{L202} and~\ref{L203}.
In any case, both the proof here and the one in~\cite{FR2} 
are modeled on arguments from complex dynamics. 
The remarks in~\S\ref{S255} about general ground fields
seem to be new.

The measure $\rho_f$ is conjectured to describe the distribution of
repelling periodic points, see~\cite[Question~1, p.119]{FR2}.
This is known in certain cases but not in general. 
In characteristic zero, Favre and Rivera-Letelier proved that
the classical periodic points (a priori not repelling) are
distributed according to $\rho_f$, 
see~\cite[Th\'eor\`eme~B]{FR2} as well as~\cite{Oku11a}.

Again motivated by results over the complex
numbers, Favre and Rivera also go beyond
equidistribution and study the ergodic properties of $\rho_f$.

Needless to say, I  have not even scratched the surface when
describing the dynamics of rational maps. I decided to 
focus on the equidistribution theorem since its proof uses 
potential theoretic techniques related to some of the 
analysis in later sections.

One of the many omissions is the Fatou-Julia theory,
in particular the classification of Fatou components,
existence and properties of wandering components etc.
See~\cite[\S10]{BRBook} and~\cite[\S\S6--7]{BenedettoNotes} for this.

Finally, we have said nothing at all about 
arithmetic aspects of dynamical systems.
For this, see \eg the book~\cite{SilvBook} 
and lecture notes~\cite{SilvNotes} by Silverman.
%
%
%
%
%
%
\newpage
\section{The Berkovich affine plane over a trivially valued field}\label{S105}
In the remainder of the paper we will consider polynomial 
dynamics on the Berko\-vich affine plane over a trivially valued field,
at a fixed point and at infinity.
Here we digress and discuss the general structure 
of the Berkovich affine space $\BerkAn$ in the case
of a trivially valued field. While we are primarily interested in the
case $n=2$, many of the notions and results are valid in any dimension.
%
%
%
%
\subsection{Setup}
Let $K$ be any field equipped with the trivial norm. 
(In~\S\S\ref{S129}--\ref{S270} we shall make further restriction on $K$.)
Let $R\simeq K[z_1,\dots,z_n]$ denote the polynomial ring in $n$
variables with coefficients in $K$. Thus $R$ is the coordinate
ring of the affine $n$-space $\A^n$ over $K$.
We shall view $\A^n$ as a scheme equipped with the Zariski topology. 
Points of $\A^n$ are thus prime ideals of $R$ 
and closed points are maximal ideals.
%
%
%
%
\subsection{The Berkovich affine space and analytification}\label{S216}
We start by introducing the basic object that we shall study.
\begin{Def}
  The Berkovich affine space $\BerkAn$
  of dimension $n$ is the set of multiplicative
  seminorms on the polynomial ring $R$ whose restriction to 
  $K$ is the trivial norm.
\end{Def}
This definition is a special case of the \emph{analytification} of a 
variety (or scheme) over $K$. 
Let $Y\subseteq\A^n$ be an irreducible subvariety defined by 
a prime ideal $I_Y\subseteq R$ and having coordinate ring 
$K[Y]=R/I_Y$. Then the analytification $\BerkY$ of 
$Y$ is the set of multiplicative seminorms on $K[Y]$
restricting to the trivial norm 
on $K$.\footnote{The analytification of a general variety or scheme
  over $K$ is defined by gluing the analytifications of open affine subsets,
  see~\cite[\S3.5]{BerkBook}.}
We equip $\BerkY$ with the topology of pointwise convergence.
The map $R\to R/I_Y$ induces a continuous injection 
$\BerkY\hookrightarrow\BerkAn$.

As before, points in $\BerkAn$ will be denoted $x$
and the associated seminorm by $|\cdot|_x$. It is
customary to write $|\phi(x)|:=|\phi|_x$ for a polynomial $\phi\in R$.
Let $\fp_x\subset R$ be the kernel of the seminorm $|\cdot|_x$.
The completed residue field $\cH(x)$ is the 
completion of the ring $R/\fp_x$ with respect to the norm
induced by $|\cdot|_x$. The structure sheaf $\cO$ on $\BerkAn$
can now be defined in the same way as in~\S\ref{S209},
following~\cite[\S1.5.3]{BerkBook}, 
but we will not directly us it. 

Closely related to $\BerkAn$ is the \emph{Berkovich unit polydisc} 
$\BerkDn$.
This is defined\footnote{The unit polydisc is denoted by $E(0,1)$ 
  in~\cite[\S1.5.2]{BerkBook}.}
in~\cite[\S1.5.2]{BerkBook} as the spectrum of the Tate algebra 
over $K$.
Since $K$ is trivially valued, the Tate algebra is the polynomial 
ring $R$ and $\BerkDn$ is the set of multiplicative seminorms on $R$ 
bounded by the trivial norm, that is, 
the set of points $x\in\BerkAn$ such that $|\phi(x)|\le 1$ 
for all polynomials $\phi\in R$.
%
%
%
%
\subsection{Home and center}\label{S219}
To a seminorm $x\in\BerkAn$ we can associate 
two basic geometric objects.
First, the kernel $\fp_x$ of $|\cdot|_x$ defines 
a point in $\A^n$ that we call the \emph{home} of $x$. 
Note that the home of $x$ is equal to $\A^n$ iff $|\cdot|_x$
is a norm on $R$. We obtain a continuous \emph{home map}
\begin{equation*}
  \BerkAn\to\A^n.
\end{equation*}
Recall that  $\A^n$ is viewed as a scheme with the Zariski topology.

Second, we define the \emph{center} of $x$ on $\A^n$ as follows.
If there exists a polynomial $\phi\in R$ such that $|\phi(x)|>1$, then we say that
$x$ has \emph{center at infinity}. Otherwise $x$ belongs to the 
Berkovich unit polydisc $\BerkDn$, in which case we define the 
center of $x$
to be the point of $\A^n$ defined by the prime ideal 
$\{\phi\in R\mid |\phi(x)|<1\}$. 
Thus we obtain a 
\emph{center map}\footnote{The center map is called the 
  \emph{reduction map} in~\cite[\S2.4]{BerkBook}.
  We use the valuative terminology center as in~\cite[\S6]{Vaquie1}
  since it will be convenient to view the elements 
  of $\BerkAn$ as semivaluations rather than seminorms.}
\begin{equation*}
  \BerkDn\to\A^n
\end{equation*}
which has the curious property of being 
\emph{anticontinuous} in the sense that preimages of 
open/closed sets are closed/open. 

The only seminorm in $\BerkAn$ whose center is all of 
$\A^n$ is the trivial norm on $R$.
More generally, if $Y\subseteq\A^n$ is any irreducible 
subvariety, there is a unique seminorm in $\BerkAn$ whose
home and center are both equal to $Y$, namely the image
of the trivial norm on $K[Y]$ under the embedding
$\BerkY\hookrightarrow\BerkAn$,
see also~\eqref{e202} below.
This gives rise to an embedding
\begin{equation*}
  \A^n\hookrightarrow\BerkAn
\end{equation*}
and shows that the home and center maps are both surjective.

The home of a seminorm always contains the center, provided the latter
is not at infinity. By letting the home and center vary over pairs
of points of $\A^n$ we obtain various partitions 
of the Berkovich affine space, see~\S\ref{S126}.

It will occasionally be convenient to identify irreducible
subvarieties of $\A^n$ with their generic points. Therefore, we shall 
sometimes think of the center and home of a seminorm as irreducible
subvarieties (rather than points) of $\A^n$.

There is a natural action of $\R_+^*$ on $\BerkAn$ which to 
a real number $t>0$ and a seminorm $|\cdot|$ associates the 
seminorm $|\cdot|^t$. The fixed points under this action are
precisely the images under the embedding $\A^n\hookrightarrow\BerkAn$ above.
%
%
%
%
\subsection{Semivaluations}\label{S222}
In what follows, it will be convenient to
work additively rather than multiplicatively. Thus we 
identify a seminorm $|\cdot| \in\BerkAn$ with the corresponding 
\emph{semivaluation} 
\begin{equation}\label{e209}
  v=-\log|\cdot|.
\end{equation}
The home of $v$ is now given by the prime ideal $(v=+\infty)$ of $R$.
We say that $v$ is a \emph{valuation} if the home is all of $\A^n$.
If $v(\phi)<0$ for some polynomial $\phi\in R$, then 
$v$ has center at infinity; otherwise $v$ belongs to
the $\BerkDn$ and its center is defined by the prime ideal $\{v>0\}$.
The action of $\R_+^*$ on $\BerkAn$ is now 
given by multiplication: $(t,v)\mapsto tv$.
The image of an irreducible subvariety 
$Y\subseteq\A^n$ under the embedding $\A^n\hookrightarrow\BerkAn$ 
is the semivaluation $\triv_Y$, defined by 
\begin{equation}\label{e202}
  \triv_Y(\phi)
  =\begin{cases}
    +\infty&\text{if $\phi\in I_Y$}\\
    0  &\text{if $\phi\not\in I_Y$},
  \end{cases}
\end{equation}
where $I_Y$ is the ideal of $Y$. 
Note that $\triv_{\A^n}$ is the trivial valuation on $R$.

For $v\in\BerkDn$ we write
\begin{equation*}
  v(\fa):=\min_{p\in\fa}v(\phi)
\end{equation*}
for any ideal $\fa\subseteq R$;
here it suffices to take the minimum over 
any set of generators of $\fa$.
%
%
%
%
\subsection{Stratification}\label{S126}
Let $Y\subseteq\A^n$ be an irreducible subvariety.
To $Y$ we can associate two natural elements of $\BerkAn$:
the semivaluation $\triv_Y$ above and the valuation 
$\ord_Y$\footnote{This is a divisorial valuation given by the order of vanishing
  along the exceptional divisor of the blowup of $Y$, see~\S\ref{S129}.}
defined by
\begin{equation*}
  \ord_Y(\phi)=\max\{k\ge0\mid\phi\in I_Y^k\}.
\end{equation*}
As we explain next, $Y$ also determines several natural subsets of $\BerkAn$.
%
%
\subsubsection{Stratification by home}\label{S224}
Define 
\begin{equation*}
  \cWp{Y},\quad\cWm{Y}\quad\text{and}\quad \cWe{Y}
\end{equation*}
as the set of semivaluations in $\BerkAn$ whose home in $\A^n$
contains $Y$, is contained in $Y$ and is equal to $Y$, respectively.
Note that $\cWm{Y}$ is closed by the continuity of the 
home map. 
We can identify $\cWm{Y}$ with the analytification $\BerkY$ 
of the affine variety $Y$ as defined in~\S\ref{S216}.
In particular, $\triv_Y\in\cWm{Y}$ corresponds to the
trivial valuation on $K[Y]$. 

The set $\cWp{Y}$ is open, since it is the 
complement in $\BerkAn$ of the union of all 
$\cWm{Z}$, where $Z$ ranges over irreducible subvarieties 
of $\A^n$ not containing $Y$.  
The set $\cWe{Y}$, on the other hand, 
is neither open nor closed unless $Y$ is a point or all 
of $\A^n$. It can be identified with the 
set of \emph{valuations} on the coordinate ring $K[Y]$.
%
%
\subsubsection{Valuations centered at infinity}\label{S227}
We define $\hcVe{\infty}$ to be the open subset of $\BerkAn$ consisting of 
semivaluations having center at infinity. Note that $\hcVe{\infty}$ is the 
complement of $\BerkDn$ in $\BerkAn$:
\begin{equation*}
  \BerkAn=\BerkDn\cup\hcVe{\infty}
  \quad\text{and}\quad
  \BerkDn\cap\hcVe{\infty}=\emptyset.
\end{equation*}
The space $\hcVe{\infty}$ is useful for the study of polynomial
mappings of $\A^n$ at infinity
and will be explored in~\S\ref{S108} in the two-dimensional case.
Notice that the action of $\R_+^*$ on $\hcVe{\infty}$
is fixed point free. We denote the quotient by $\cVe{\infty}$:
\begin{equation*}
  \cVe{\infty}:=\hcVe{\infty}/\R_+^*.
\end{equation*}
If we write $R=K[z_1,\dots,z_n]$, 
then we can identify $\cVe{\infty}$ with the set
of semivaluations for which 
$\min_{1\le i\le n}\{v(z_i)\}=-1$.
However, this identification depends on the choice
of coordinates, or at least on the embedding of
$\A^n\hookrightarrow\P^n$.
%
%
\subsubsection{Stratification by center}\label{S225}
We can classify the semivaluations in the Berkovich 
unit polydisc $\BerkDn$ according to their centers.
Given an irreducible subvariety $Y\subseteq\A^n$ we define
\begin{equation*}
  \hcVp{Y},\quad\hcVm{Y}\quad\text{and}\quad \hcVe{Y}
\end{equation*}
as the set of semivaluations in $\BerkDn$ whose center
contains $Y$, is contained in $Y$ and is equal to $Y$,
respectively.
By anticontinuity of the center map,
$\hcVm{Y}$ is open and, consequently, 
$\hcVp{Y}$ closed in $\BerkDn$.
Note that $v\in\hcVm{Y}$ iff $v (I_Y)>0$. 
As before, $\hcVe{Y}$ is neither open nor closed unless
$Y$ is a closed point or all of $\A^n$.

Note that $\cWm{Y}\cap\BerkDn\subseteq\hcVm{Y}$. 
The difference $\hcVm{Y}\setminus\cWm{Y}$ is the open subset 
of $\BerkDn$ consisting of semivaluations $v$ 
satisfying $0<v(I_Y)<\infty$. If we define
\begin{equation}
  \cV_Y:=\{v\in\BerkDn\mid v (I_Y)=1\},
\end{equation}
then $\cV_Y$ is a closed subset of $\BerkDn$
(hence also of $\BerkAn$) 
and the map $v\mapsto v/v(I_Y)$ induces a homeomorphism 
\begin{equation*}
  (\hcVm{Y}\setminus\cWm{Y})/\R_+^*\simto\cV_Y.
\end{equation*}
\begin{Remark}\label{R204}
  In the terminology of Thuillier~\cite{ThuillierStepanov}, 
  $\hcVm{Y}$ is the Berkovich space associated to 
  the completion of $\A^n$ along the closed subscheme $Y$.
  Similarly, the open subset $\hcVm{Y}\setminus\cWm{Y}$ is the 
  generic fiber of this formal subscheme. This terminology differs
  slightly from that of Berkovich~\cite{BerkVan1} who refers to 
  $\hcVm{Y}$ as the generic fiber, see~\cite[p.383]{ThuillierStepanov}.
\end{Remark}
%
%
\subsubsection{Extremal cases}\label{S226}
Let us describe the subsets of $\BerkAn$ introduced above in the case
when the subvariety $Y$ has maximal or minimal dimension.
First, it is clear that 
\begin{equation*}
  \cWm{\A^n}=\BerkAn
  \quad\text{and}\quad
  \hcVm{\A^n}=\BerkDn.
\end{equation*}
Furthermore,
\begin{equation*}
  \hcVp{\A^n}=\hcVe{\A^n}=\cWp{\A^n}=\cWe{\A^n}=\{\triv_{\A^n}\},
\end{equation*}
the trivial valuation on $R$. 
Since $I_{\A^n}=0$, we also have
\begin{equation*}
  \cV_{\A^n}=\emptyset.
\end{equation*}

At the other extreme, for a closed point $\xi\in\A^n$, 
we have
\begin{equation*}
  \cWm{\xi}=\cWe{\xi}=\{\triv_\xi\}.
\end{equation*}
The space $\cV_\xi$ is a singleton when $n=1$ (see~\S\ref{S211})
but has a rich structure when $n>1$.
We shall describe in dimension two in~\S\ref{S103}, in which case
it is a tree in the sense of~\S\ref{S172}.
See~\cite{hiro} for the higher-dimensional case.
%
%
\subsubsection{Passing to the completion}
A semivaluation $v\in\BerkDn$ whose center is equal to an 
irreducible subvariety $Y$ extends uniquely to a semivaluation 
on the local ring $\cO_{\A^n,Y}$
such that $v(\fm_Y)>0$, where $\fm_Y$ is the maximal ideal.
By $\fm_Y$-adic continuity, $v$ further extends uniquely as a 
semivaluation on the completion 
and by Cohen's structure theorem,
the latter is isomorphic to the power series ring 
$\kappa(Y)\llbracket z_1,\dots z_r\rrbracket$,
where $r$ is the codimension of $Y$.
Therefore we can view $\hcVe{Y}$ as the set of semivaluations
$v$ on $\kappa(Y)\llbracket z_1,\dots z_r\rrbracket$ whose restriction to 
$\kappa(Y)$ is trivial and such that $v(\fm_Y)>0$. 
In particular, for a closed point $\xi$,
we can view $\hcVe{\xi}$ (resp., $\cVe{\xi}$)   
as the set of semivaluations
$v$ on $\kappa(\xi)\llbracket z_1,\dots z_n\rrbracket$ whose
restriction to $\kappa(\xi)$ is trivial and such that 
$v(\fm_\xi)>0$ (resp., $v(\fm_\xi)=1$).  
This shows that when $K$ is algebraically closed,
the set $\cVe{\xi}$ above is isomorphic to the space considered   
in~\cite{hiro}. This space was first introduced in dimension $n=2$  
in~\cite{valtree} where it was called the valuative tree. We shall study it 
from a slightly different point of view in~\S\ref{S103}. 
Note that it may happen that a valuation $v\in\hcVe{\xi}$ has home $\xi$
but that the extension of $v$ to $\widehat{\cO}_{\A^n,\xi}$ is a semivaluation
for which the ideal $\{v=\infty\}\subseteq\widehat{\cO}_{\A^n,\xi}$ is
nontrivial.
%
%
%
%
\subsection{The affine line}\label{S211}
Using the definitions above, let us describe the Berkovich affine
line $\BerkAone$ over a trivially valued field $K$.

An irreducible subvariety of $\A^1$ is either $\A^1$ itself or
a closed point. As we noted in~\S\ref{S226} 
\begin{equation*}
  \hcVm{\A^1}=\BerkD,
  \quad \cWm{\A^1}=\BerkAone,
  \quad\hcVp{\A^1}=\hcVe{\A^1}=\cWp{\A^1}=\cWe{\A^1}=\{\triv_{\A^1}\}
\end{equation*}
whereas $\cV_{\A^1}$ is empty.

Now suppose the center of $v\in\BerkAone$ is a closed point 
$\xi\in\A^1$. 
If the home of $v$ is also equal to $\xi$, then $v=\triv_\xi$.
Now suppose the home of $v$ is $\A^1$, so that 
$0<v(I_\xi)<\infty$. After scaling we may assume 
$v(I_\xi)=1$ so that $v\in\cV_\xi$. 
Since $R\simeq K[z]$ is a PID is follows easily that 
$v=\ord_\xi$. This shows that 
\begin{equation*}
  \cWm{\xi}=\cWe{\xi}=\{\triv_\xi\}
  \quad\text{and}\quad 
  \cV_\xi=\{\ord_\xi\},
\end{equation*}
Similarly, if $v\in\BerkAone$ has center at infinity, then,
after scaling, we may assume that $v(z)=-1$, where 
$z\in R$ is a coordinate. It is then clear that 
$v=\ord_\infty$, where $\ord_\infty$ is the valuation on 
$R$ defined by $\ord_\infty(\phi)=-\deg\phi$. 
Thus we have 
\begin{equation*}
  \cVe{\infty}=\{\ord_\infty\}.
\end{equation*}
Note that any polynomial $\phi\in R$ can be viewed as a rational 
function on $\P^1=\A^1\cup\{\infty\}$ and $\ord_\infty(\phi)\le0$ 
is the order of vanishing of $\phi$ at $\infty$.

We leave it as an exercise to the reader to compare the 
terminology above with the one in~\S\ref{S262}.
See Figure~\ref{F102} for a picture of the Berkovich affine
line over a trivially valued field.

\begin{figure}[ht]
 \includegraphics{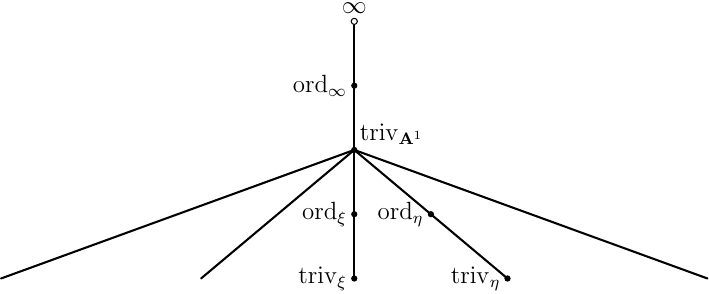}
  \caption{The Berkovich affine line over a trivially valued field.
    The trivial valuation $\triv_{\A^1}$ is the only point with center $\A^1$.
    The point $\triv_\xi$ for $\xi\in\A^1$ has home $\xi$. 
    All the points on the open segment $]\triv_{\A^1},\triv_\xi[$ 
    have home $\A^1$ and center $\xi$ and are proportional
    to the valuation $\ord_\xi$. 
    The point $\infty$ does not belong to $\BerkAone$.
    The points on the open segment $]\triv_{\A^1},\infty[$ 
    have home $\A^1$, center at infinity
    and are proportional to the valuation $\ord_\infty$.
  }\label{F102}
\end{figure}

%
%
%
%
\subsection{The affine plane}\label{S139}
In dimension $n=2$, the Berkovich affine space is 
significantly more complicated than in dimension one, 
but can still---with some effort---be visualized.

An irreducible subvariety of $\A^2$ is either 
all of $\A^2$, a curve, or a closed point.
As we have seen,
\begin{equation*}
  \hcVm{\A^2}=\BerkDtwo,
  \quad \cWm{\A^2}=\BerkAtwo,
  \quad\hcVp{\A^2}=\hcVe{\A^2}=\cWp{\A^2}=\cWe{\A^2}=\{\triv_{\A^2}\}
\end{equation*}
whereas $\cV_{\A^2}$ is empty.

\smallskip
Now let $\xi$ be a closed point. 
As before, $\cWm{\xi}=\cWe{\xi}=\{\triv_\xi\}$, where $\triv_\xi$
is the image of $\xi$ under the embedding 
$\A^2\hookrightarrow\BerkAtwo$. 
The set $\hcVm{\xi}=\hcVe{\xi}$ is open and 
$\hcVe{\xi}\setminus\{\triv_\xi\}=\hcVe{\xi}\setminus\cWe{\xi}$
is naturally a punctured cone with base $\cV_\xi$.
The latter  will be called \emph{the valuative tree}
(at the point $\xi$) and is studied in detail in~\S\ref{S103}.
Suffice it here to say that it is a tree in the sense of~\S\ref{S172}.
The whole space $\hcVe{\xi}$ is a cone over the valuative tree with 
its apex at $\triv_\xi$. The boundary of $\hcVe{\xi}$ 
consists of all semivaluations whose center strictly contains $\xi$,
so it is the union of  $\triv_{\A^2}$ and $\hcVe{C}$, 
where $C$ ranges over curves containing $C$. As we shall
see, the boundary therefore has the structure of a tree 
naturally rooted in $\triv_{\A^2}$. See Figure~\ref{F103}.
If $\xi$ and $\eta$ are two different closed points, 
then the open sets $\hcVe{\xi}$ and $\hcVe{\eta}$ are disjoint.

\smallskip
Next consider a curve $C\subseteq\A^2$.
By definition, the set $\cWm{C}$ consists all semivaluations whose
home is contained in $C$. This means that $\cWm{C}$ is the
image of the analytification $\BerkC$ of $C$ under the 
embedding $\BerkC\hookrightarrow\BerkAtwo$. 
As such, it looks quite similar to 
the Berkovich affine line $\BerkAone$, see~\cite[\S1.4.2]{BerkBook}.
More precisely, the semivaluation $\triv_C$ is the unique
semivaluation in $\cWm{C}$ having center $C$. 
All other semivaluations in $\cWm{C}$ have center at a closed
point $\xi\in C$. The only such semivaluation having home $\xi$
is $\triv_\xi$; the other semivaluations in $\cWm{C}\cap\hcVe{\xi}$ 
have home $C$ and center $\xi$. We can normalize them by $v(I_\xi)=1$. 
If $\xi$ is a nonsingular point on $C$, then there is a unique
normalized semivaluation $v_{C,\xi}\in\BerkAtwo$ having home $C$ and center
$\xi$. When $\xi$ is a singular point on $C$, the set of such
semivaluations is instead in bijection with the set of 
local branches\footnote{A local branch is a preimage of a point of $C$
  under the normalization map.}
of $C$ at $\xi$. 
We see that $\cWm{C}$ looks like $\BerkAone$ except that there may be
several intervals joining $\triv_C$ and $\triv_\xi$: one for
each local branch of $C$ at $\xi$. See Figure~\ref{F201}.

Now look at the closed set $\hcVp{C}$
of semivaluations whose center contains $C$. 
It consists of all semivaluations $t\ord_C$ for 
$0\le t\le\infty$. Here $t=\infty$ and $t=0$ correspond to 
$\triv_C$ and $\triv_{\A^2}$, respectively.
As a consequence, for any closed point $\xi$, 
$\partial\hcVe{\xi}$ has the structure of a tree, much 
like the Berkovich affine line $\BerkAone$. 

The set $\hcVm{C}$ is open and its boundary consists of semivaluations
whose center strictly contains $C$. In other words, the boundary is
the singleton $\{\triv_{\A^2}\}$. 
For two curves $C,D$, the intersection 
$\hcVm{C}\cap\hcVm{D}$ is the union of sets 
$\hcVe{\xi}$ over all closed points $\xi\in C\cap D$. 

The set $\cV_C\simeq(\hcVm{C}\setminus\cWm{C})/\R_+^*$ 
looks quite similar to the valuative tree at a closed point.
To see this, note that the valuation $\ord_C$ is the only semivaluation
in $\cV_C$ whose center is equal to $C$. All other semivaluations in 
$\cV$ have center at a closed point $\xi\in C$. 
For each semivaluation $v\in\cV_\xi$ whose home is not equal to $C$,
there exists a unique $t=t(\xi,C)>0$ such that $tv\in\cV_C$;
indeed, $t=v(I_C)$. Therefore, $\cV_C$ can be obtained by taking
the disjoint union of the trees $\cV_\xi$ over all $\xi\in C$ and identifying the
semivaluations having home $C$ with the point $\ord_C$.
If $C$ is nonsingular, then $\cV_C$ will be a tree naturally rooted
in $\ord_C$.

We claim that if $C$ is a line, then $\cV_C$ can be identified with the 
Berkovich unit disc over the field of Laurent series in one variable
with coefficients in $K$. To see this, pick affine coordinates 
$(z_1,z_2)$ such that $C=\{z_1=0\}$. Then $\cV_C$ is the set of 
semivaluations $v:K[z_1,z_2]\to\R_+\cup\{\infty\}$ such that 
$v(z_1)=1$. Let $L=K((z_1))$ be the field of Laurent series,
equipped with the valuation $v_L$ that is trivial on $K$ and takes
value 1 on $z_1$. Then the Berkovich unit disc 
$\BerkD$ over $L$ is the set of semivaluations 
$L[z_2]\to\R_+\cup\{\infty\}$ extending $v_L$. 
Every element of $\BerkD$ defines an element of 
$\cV_C$ by restriction. 
Conversely, pick $v\in\cV_C$. If $v=\ord_C$, then $v$ extends
uniquely to an element of $\BerkD$, namely the Gauss point.
If $v\ne\ord_C$, then the center of $v$ is a closed point 
$\xi\in C$ and $v$ extends uniquely to the fraction field
of the completion $\hcO_\xi$. This fraction field contains $L[z_2]$.

\smallskip
The open subset $\hcVe{\infty}=\BerkAtwo\setminus\BerkDtwo$ 
of semivaluations centered at infinity 
is a punctured cone over a base $\cVe{\infty}$. The latter space is called 
\emph{the valuative tree at infinity} and will be studied in detail
in~\S\ref{S108}. Superficially, its structure is quite similar to the 
valuative tree at a closed point $\xi$. In particular it is a tree in the 
sense of~\S\ref{S172}. 
The boundary of $\hcVe{\infty}$ is the union of $\hcVp{C}$ 
over \emph{all} affine curves $C$, that is, the set of semivaluations 
in $\BerkDtwo$ whose center is \emph{not} a closed point. Thus the boundary 
has a structure of a tree rooted in $\triv_{\A^2}$. See Figure~\ref{F212}.
We emphasize that there is no point $\triv_\infty$ in $\hcVe{\infty}$.

\smallskip
To summarize the discussion, $\BerkAtwo$ contains a closed
subset $\Sigma$ with empty interior consisting of semivaluations 
having center of dimension one or two. This set is a naturally
a tree, which can be viewed as the cone over the collection of all
irreducible affine curves.
The complement of $\Sigma$ is an
open dense subset whose connected components are $\hcVe{\infty}$,
and $\hcVe{\xi}$, where $\xi$ ranges over closed points of $\A^2$.
The set $\hcVe{\infty}$ is a punctured cone over a tree $\cVe{\infty}$
and its boundary is all of $\Sigma$. 
For a closed point $\xi$, $\hcVe{\xi}$ is a cone over a tree
$\cVe{\xi}$ and its boundary is a subtree of $\Sigma$, namely
the cone over the collection of all irreducible affine curves containing $\xi$.
\begin{figure}[ht]
  \includegraphics[width=\textwidth]{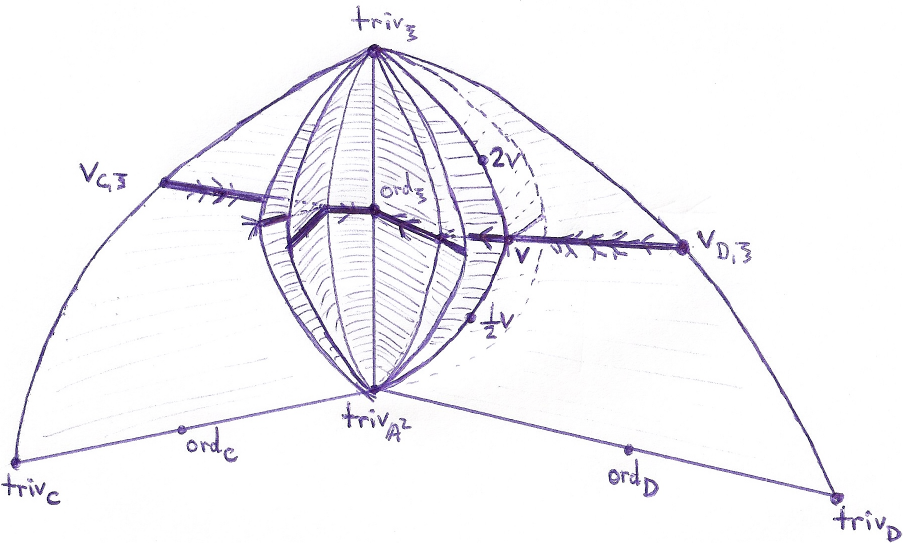}
  \caption{The Berkovich affine plane over a trivially
    valued field. The picture shows the closure of the set $\hcVe{\xi}$
    of semivaluations having center at a closed point $\xi\in\A^2$.
    Here $C$, $D$ are irreducible curves containing $\xi$. 
    The semivaluation $\triv_\xi\in\hcVe{\xi}$ has home $\xi$.
    All semivaluations in $\hcVe{\xi}\setminus\{\triv_\xi\}$ are 
    proportional to a semivaluation $v$ in the valuative tree
    $\cV_\xi$ at $\xi$. We have $tv\to\triv_\xi$ as $t\to\infty$.
    As $t\to0+$, $tv$ converges to the semivaluation $\triv_Y$,
    where $Y$ is the home of $v$. 
    The semivaluations $v_{C,\xi}$ and $v_{D,\xi}$ belong to $\cV_\xi$
    and have home $C$ and $D$, respectively.
    The boundary of $\hcVe{\xi}$ is a tree consisting of all segments 
    $[\triv_{\A^2},\triv_C]$ for all irreducible affine 
    curves $C$ containing both $\xi$.
    Note that the segment $[\triv_C,\triv_\xi]$ in the closure of 
   $\hcVe{\xi}$ is also a segment in the analytification 
   $\BerkC\subseteq\BerkAtwo$ of $C$,
   see Figure~\ref{F201}.
}\label{F103}
\end{figure}
\begin{figure}[ht]
  \includegraphics[width=\textwidth]{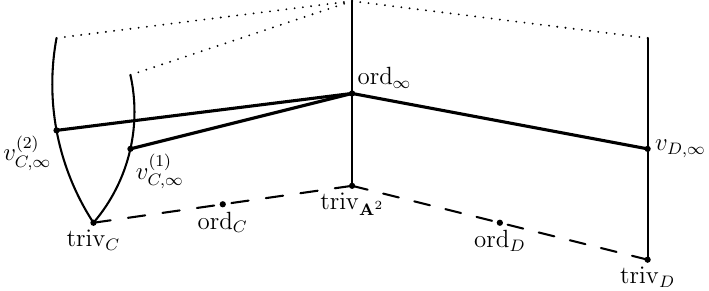}
  \caption{The Berkovich affine plane over a trivially
    valued field. The picture shows (part of) the closure of the set $\hcVe{\infty}$
    of semivaluations having center at infinity.
    Here $C$ and $D$ are affine curves having two and one places
    at infinity, respectively.
    The set $\hcVe{\infty}$ is a cone whose base is $\cVe{\infty}$, 
    the valuative tree at infinity. Fixing an embedding $\A^2\hookrightarrow\P^2$
    allows us to identify $\cVe{\infty}$ with a subset of
    $\hcVe{\infty}$
    and the valuation $\ord_\infty$ is the order of vanishing along
    the line at infinity in $\P^2$.
    The semivaluations $v_{D,\infty}$ and $v_{C,\infty}^{(i)}$,
    $i=1,2$ have home $D$ and $C$, respectively; 
    the segments $[\ord_\infty,v_{D,\infty}]$ and
    $[\ord_\infty,v_{C,\infty}^{(i)}]$, $i=1,2$ belong to $\cVe{\infty}$.
   The segments $[\triv_{\A^2},\triv_C]$ and $[\triv_{\A^2},\triv_D]$
    at the bottom of the picture belong to the 
    boundary of $\hcVe{\infty}$: the full boundary is a tree 
    consisting of all such segments and whose only branch point is
    $\triv_{\A^2}$.
    The dotted segments in the top of the picture do not belong to
    the Berkovich affine plane.
}\label{F212}
\end{figure}
\begin{figure}[ht]
 \includegraphics{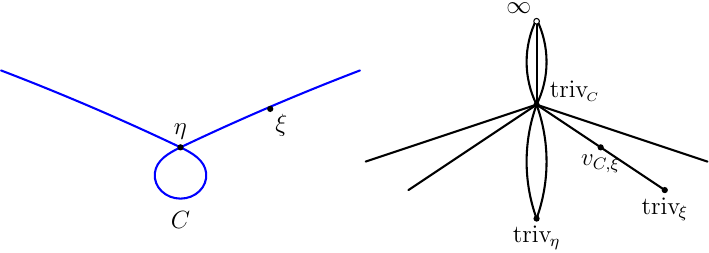}
 \caption{The analytification $\BerkC$ of an affine curve $C$ over a trivially
    valued field. The semivaluation $\triv_C$ is the only semivaluation
    in $\BerkC$ having center $C$ and home $C$. To each closed point $\xi\in C$
    is associated a unique semivaluation $\triv_\xi\in\BerkC$ with center and home 
    $\xi$. The set of elements of $\BerkC$ with home $C$ and center at a given closed 
    point $\xi$ is a disjoint union of open intervals, one for each 
    local branch of $C$ at $\xi$. Similarly, the set of elements of $\BerkC$ 
    with home $C$ and center at infinity is a disjoint union of open intervals, 
    one for each branch of $C$ at infinity.
    The left side of the picture shows a nodal cubic curve $C$ and the
    right side shows its analytification $\BerkC$.
    Note that for a smooth point $\xi$ on $C$,
    the segment $[\triv_C,\triv_\xi]$ in $\BerkC$ also lies
    in the closure of the cone $\hcVe{\xi}$,
    see Figure~\ref{F103}.
   }\label{F201}
\end{figure}
%
%
%
%
\subsection{Valuations}\label{S220}
A semivaluation $v$ on $R\simeq K[z_1,\dots,z_n]$ 
is a \emph{valuation} if the corresponding seminorm 
is a norm, that is, if $v(\phi)<\infty$ for all nonzero polynomials 
$\phi\in R$. A valuation $v$ extends to the fraction field 
$F\simeq K(z_1,\dots,z_n)$ of $R$ by setting 
$v(\phi_1/\phi_2)=v(\phi_1)-v(\phi_2)$.

Let $X$ be a variety over $K$ whose function field is equal to $F$.
The \emph{center} of a valuation $v$ on $X$, if it exists,
is the unique (not necessarily closed) point $\xi\in X$ defined 
by the properties that $v\ge 0$ on the local ring $\cO_{X,\xi}$ 
and $\{v>0\}\cap\cO_{X,\xi}=\fm_{X,\xi}$.
By the valuative criterion of properness,
the center always exists and is unique 
when $X$ is proper over $K$.

Following~\cite{graded} we write $\Val_X$ for the set of valuations 
of $F$ that admit a center on $X$. As usual, this set is endowed with the 
topology of pointwise convergence. 
Note that $\Val_X$ is a subset of $\BerkAn$ that can in fact be shown to
be dense. One nice feature of $\Val_X$ is that any proper
birational morphism $X'\to X$ induces an isomorphism $\Val_{X'}\simto\Val_X$. 
(In the same situation, the analytification $X'_\mathrm{Berk}$ 
maps onto $X_\mathrm{Berk}$, but this map is not injective.)

We can view the Berkovich unit polydisc $\BerkDn$ as the disjoint union 
of $\Val_Y$, where $Y$ ranges over irreducible subvarieties of $X$.
%
%
%
%
\subsection{Numerical invariants}\label{S133}
To a valuation $v\in\BerkAn$ we can associate several invariants.
First, the \emph{value group} of $v$ is defined by 
$\Gamma_v:=\{v(\phi)\mid\phi\in F\setminus\{0\}\}$.
The \emph{rational rank} $\ratrk v$ of $v$ is
the dimension of the $\Q$-vector space 
$\Gamma_v\otimes_\Z\Q$. 

Second, the valuation ring 
$R_v=\{\phi\in F\mid v(\phi)\ge 0\}$ of $v$ is a local
ring with maximal ideal $\fm_v=\{v(\phi)>0\}$.
The \emph{residue field} $\kappa(v)=R_v/\fm_v$ contains $K$ as
a subfield and 
the \emph{transcendence degree} of $v$
is the transcendence degree of the field extension 
$\kappa(v)/K$. 

In our setting, the fundamental \emph{Abhyankar inequality} 
states that
\begin{equation}\label{e129}
  \ratrk v +\trdeg v \le n.
\end{equation}
The valuations for which equality holds are of particular 
importance. At least in characteristic zero, 
they admit a nice geometric description that we discuss next.
%
%
%
%
\subsection{Quasimonomial and divisorial valuations}\label{S129}
Let $X$ be a smooth variety over $K$ with function field $F$.
We shall assume in this section that the field $K$ has characteristic zero or
that $X$ has dimension at most two. This allows us to freely
use resolutions of  singularities.

Let $\xi\in X$ be a point (not necessarily closed) with residue
field $\kappa(\xi)$. 
Let $(\zeta_1,\dots,\zeta_r)$ be a system of
algebraic coordinates at $\xi$ (\ie a regular
system of parameters of $\cO_{X,\xi}$).
We say that a valuation $v\in\Val_X$ is \emph{monomial}
in coordinates $(\zeta_1,\dots,\zeta_r)$ with weights 
$t_1,\dots,t_r\ge 0$ if the following holds:
if we write $\phi\in\widehat{\cO}_{X,\xi}$ as
$\phi=\sum_{\b\in\Z_{\ge0}^m}c_\b\zeta^\b$
with each $c_\b\in\widehat{\cO}_{X,\xi}$ either zero or a unit, then
\begin{equation*}
  v(\phi)=\min\{\langle t,\b\rangle\mid c_\b\ne0\},
\end{equation*}
where $\langle t,\b\rangle=t_1\b_1+\dots+t_r\b_r$.
After replacing $\xi$ by the (generic point of the)
intersection of all divisors $\{\zeta_i=0\}$ we may in fact
assume that $t_i>0$ for all $i$.

We say that a valuation $v\in\Val_X$ is \emph{quasimonomial}
(on $X$) if it is monomial in some birational model of $X$. 
More precisely, we require that there exists a proper birational
morphism $\pi:X'\to X$, with $X'$ smooth,
such that $v$ is monomial in some algebraic coordinates 
at some point $\xi\in X'$.
As explained in~\cite{graded}, in this case we can
assume that the divisors $\{\zeta_i=0\}$ are irreducible components
of a reduced, effective simple normal crossings divisor $D$ on $X'$
that contains the exceptional locus of $\pi$.
(In the two-dimensional situation that we shall be primarily
interested in, arranging this is quite elementary.)

It is a fact that a valuation $v\in\Val_X$ is quasimonomial
iff equality holds in Abhyankar's inequality~\eqref{e129}.
For this reason, quasimonomial valuations are sometimes
called Abhyankar valuations. See~\cite[Proposition~2.8]{ELS}.

Furthermore, we can arrange the situation so that 
the weights $t_i$ are all strictly positive and linearly independent over
$\Q$: see~\cite[Proposition~3.7]{graded}.
In this case the residue field of $v$ is isomorphic
to the residue field of $\xi$, and hence 
$\trdeg v=\dim(\overline\xi)=n-r$. Furthermore, the value group of
$v$ is equal to 
\begin{equation}\label{e130}
  \Gamma_v=\sum_{i=1}^r\Z t_i,
\end{equation}
so $\ratrk v=r$.

A very important special case of quasimonomial 
valuations are given by \emph{divisorial valuations}.
Numerically, they are characterized by $\ratrk=1$, $\trdeg=n-1$.
Geometrically, they are described as follows: there exists
a birational morphism $X'\to X$, a prime divisor 
$D\subseteq X'$ and a constant $t>0$ such that 
$t^{-1}v(\phi)$ is the order of vanishing along $D$ for all
$\phi\in F$.
%
%
%
%
\subsection{The Izumi-Tougeron inequality}\label{S270}
Keep the same assumptions on $K$ and $X$ as in~\S\ref{S129}.
Consider a valuation $v\in\Val_X$ and let $\xi$ be its 
center on $X$. Thus $\xi$ is a (not necessarily closed) 
point of $X$. By definition, $v$ is nonnegative on 
the local ring $\cO_{X,\xi}$ and strictly positive on the 
maximal ideal $\fm_{X,\xi}$. 
Let $\ord_\xi$ be the order of vanishing at $\xi$.
It follows from the valuation axioms that 
\begin{equation}\label{e132}
  v\ge c\ord_\xi,
\end{equation}
on $\cO_{X,\xi}$, where $c=v(\fm_{X,\xi})>0$. 

It will be of great importance to us that if 
$v\in\Val_X$ is quasimonomial
then the reverse inequality holds 
in~\eqref{e132}. Namely, 
there exists a constant
$C=C(v)>0$ such that 
\begin{equation}\label{e131}
  c\ord_\xi\le v\le C\ord_\xi 
\end{equation}
on $\cO_{X,\xi}$.
This inequality is often referred to as Izumi's 
inequality (see~\cite{Izumi,Rees,HuSw,ELS}) 
but in the smooth case we are considering it goes back at least to 
Tougeron~\cite[p.178]{Tougeron}.
More precisely, Tougeron proved this 
inequality for divisorial valuations, but that easily 
implies the general case.

As in~\S\ref{S124}, a valuation $v\in\Val_X$ having center 
$\xi$ on $X$ extends uniquely to a semivaluation on 
$\widehat{\cO}_{X,\xi}$. The Izumi-Tougeron inequality~\eqref{e131} 
implies that if $v$ is quasimonomial, then this extension 
is in fact a valuation. 
In general, however, the extension may not be a valuation,
so the Izumi-Tougeron inequality certainly does not hold
for \emph{all} valuations in $\Val_X$ having center $\xi$ on $X$.
For a concrete example, let $X=\A^2$, let $\xi$ be 
the origin in coordinates $(z,w)$ and let 
$v(\phi)$ be defined as the order of vanishing 
at $u=0$ of $\phi(u,\sum_{i=1}^\infty\frac{u^i}{i!})$.
Then $v(\phi)<\infty$ for all nonzero polynomials $\phi$,
whereas $v(w-\sum_{i=1}^\infty\frac{u^i}{i!})=0$.
%
%
%
%
\subsection{Notes and further references}\label{S228}
It is a interesting feature of Berkovich's theory that
one can work with trivially valued fields: this is definitely
not possible in rigid geometry (see e.g.~\cite{ConradNotes}
for a general discussion of rigid geometry and various other
aspects of non-Archimedean geometry).

In fact, Berkovich spaces over trivially valued fields 
have by now seen several interesting and 
unexpected applications. In these notes we focus on 
dynamics, but one can also study use Berkovich spaces
to study the singularities of plurisubharmonic
functions~\cite{pshsing,hiro} and various asymptotic 
singularities in algebraic geometry,
such as multiplier ideals~\cite{valmul,graded}.
In other directions, Thuillier~\cite{ThuillierStepanov} 
exploited Berkovich spaces to give a new proof of a theorem
by Stepanov in birational geometry, and Berkovich~\cite{BerkHodge} 
has used them in the context of mixed Hodge structures.

The Berkovich affine space of course also comes with a structure sheaf $\cO$.
We shall not need use it in what follows but it is surely a useful tool for 
a more systematic study of polynomial mappings on the $\BerkAn$.

The spaces $\hcVe{\xi}$, $\cVe{\xi}$ and $\cVe{\infty}$
were introduced (in the case of $K$ algebraically closed
of characteristic zero) and studied in~\cite{valtree,eigenval,hiro}
but not explicitly identified as subset of the Berkovich affine
plane. The structure of the Berkovich affine space does
not seem to have been written down in detail before,
but see~\cite{YZb}.

The terminology ``home'' is not standard.
Berkovich~\cite[\S1.2.5]{BerkBook} uses this construction
but does not give it a name. The name ``center''
comes from valuation theory, see~\cite[\S6]{Vaquie1}
whereas non-Archimedean geometry tends to use
the term ``reduction''.
Our distinction between (additive) valuations and
(multiplicative) norms is not always made in the literature.
Furthermore, in~\cite{valtree,hiro}, the term `valuation' instead
of `semi-valuation' is used even when the prime
ideal $\{v =+\infty\}$ is nontrivial.

The space $\Val_X$ was introduced in~\cite{graded} for the 
study of asymptotic invariants of graded sequences of ideals.
In~\loccit it is proved that $\Val_X$ is an inverse limit
of cone complexes, in the same spirit as~\S\ref{S248} below.
%
%
%
%
%
%
\newpage
\section{The valuative tree at a point}\label{S103}
Having given an overview of the Berkovich affine plane
over a trivially valued field, we now study the set 
of semivaluations centered at a closed point. As indicated 
in~\S\ref{S139}, this is a cone over a space that we call
the valuative tree.

The valuative tree is treated in detail in
the monograph~\cite{valtree}. However, 
the self-contained presentation here has a different focus. 
In particular, we emphasize
aspects that generalize to higher dimension.
See~\cite{hiro} for some of these generalizations.
%
%
%
%
\subsection{Setup}
Let $K$ be field equipped with the trivial norm.
For now we assume that $K$ is algebraically closed
but of arbitrary characteristic.
(See~\S\ref{S230} for a more general case).
In applications to complex dynamics we would
of course pick $K=\C$, but
we emphasize that the norm is then \emph{not}
the Archimedean one.
As in~\S\ref{S105}
we work additively rather than multiplicatively
and consider $K$ equipped with the
trivial valuation, whose value on nonzero elements is zero
and whose value on $0$ is $+\infty$.

Let $R$ and $F$ be the coordinate ring and
function field of $\A^2$.
Fix a closed point $0\in\A^2$
and write $\fm_0\subseteq R$ for the corresponding maximal ideal.
If $(z_1,z_2)$ are global coordinates on $\A^2$ 
vanishing at $0$, then $R=K[z_1,z_2]$, $F=K(z_1,z_2)$ and $\fm_0=(z_1,z_2)$.
We say that an ideal $\fa\subseteq R$ is $\fm_0$-\emph{primary} 
or simply \emph{primary} if it contains some power of $\fm_0$.

Recall that the Berkovich affine plane $\BerkAtwo$ is the set of 
semivaluations on $R$ that restrict to the trivial valuation
on $K$. Similarly, the Berkovich unit bidisc $\BerkDtwo$ is the 
set of semivaluations $v\in\BerkAtwo$ that are nonnegative on $R$.
If $\fa\subseteq R$ is an ideal and $v\in\BerkDtwo$, then we write 
$v(\fa)=\min\{v(\phi)\mid\phi\in\fa\}$.
In particular, $v(\fm_0)=\min\{v(z_1),v(z_2)\}$.
%
%
%
%
\subsection{The valuative tree}\label{S229}
Let us recall some definitions from~\S\ref{S225} and~\S\ref{S139}.
Let $\hcV_0\subseteq\BerkDtwo$ 
be the subset of semivaluations whose center on $\A^2$
is equal to the closed point $0\in\A^2$. 
In other words, 
$\hcV_0$ is the set of semivaluations $v:R\to[0,+\infty]$
such that $v|_{K^*}\equiv0$ and $v(\fm_0)>0$.

There are now two cases. Either $v(\fm_0)=+\infty$,
in which case $v=\triv_0\in\BerkAtwo$ 
is the trivial valuation associated
to the point $0\in\A^2$, or $0< v(\fm_0)<\infty$.
Define $\hcV_0^*$ as the set of semivaluations of the 
latter type. This set is naturally a pointed cone and
admits the following set as a ``section''.
\begin{Def}
  The \emph{valuative tree} $\cV_0$ at the point $0\in\A^2$
  is the set of semivaluations
  $v:R\to[0,+\infty]$ satisfying $v(\fm_0)=1$.
\end{Def}
To repeat, we have 
\begin{equation*}
  \hcV_0=\{\triv_0\}\cup\hcV_0^*
  \quad\text{and}\quad
  \hcV_0^*=\R_+^*\cV_0.
\end{equation*}

We equip $\cV_0$ and $\hcV_0$ with the subspace topology from
$\BerkAtwo$, that is, the weakest topology for
which all evaluation maps $v\mapsto v(\phi)$
are continuous, where $\phi$ ranges over polynomials in $R$.
It follows easily from Tychonoff's theorem that 
$\cV_0$ is a compact Hausdorff space.

Equivalently, we could demand that $v\mapsto v(\fa)$
be continuous for any primary ideal $\fa\subseteq R$.
For many purposes it is indeed quite natural to evaluate semivaluations
in $\hcV_0^*$ on primary ideals rather than polynomials.
For example, we have $v(\fa+\fb)=\min\{v(\fa),v(\fb)\}$ for any 
primary ideals $\fa$, $\fb$, whereas we only
have $v(\phi+\psi)\ge\min\{v(\phi),v(\psi)\}$ for polynomials $\phi,\psi$.

An important element of $\cV_0$ is the valuation $\ord_0$
defined by 
\begin{equation*}
  \ord_0(\phi)=\max\{k\ge 0\mid \phi\in\fm_0^k\}.
\end{equation*}
Note that $v(\phi)\ge\ord_0(\phi)$ for all $v\in\cV_0$ and all
$\phi\in R$.

Any semivaluation $v\in\BerkAtwo$ extends as a 
function $v:F\to[-\infty,+\infty]$, where $F$ is the 
fraction field of $R$, by setting $v(\phi_1/\phi_2)=v(\phi_1)-v(\phi_2)$; 
this is well defined since $\{v=+\infty\}\subseteq R$ is a prime ideal.

Our goal for now is to justify the name
``valuative tree'' by showing that $\cV_0$ can be equipped with a 
natural tree structure, rooted at $\ord_0$. 
This structure can be obtained 
from many different points of view, as 
explained in~\cite{valtree}. Here we focus on
a geometric approach that is
partially generalizable to higher dimensions (see~\cite{hiro}).
%
%
%
%
\subsection{Blowups and log resolutions}\label{S154} 
We will consider birational morphisms
\begin{equation*}
  \pi:X_\pi\to\A^2,
\end{equation*}
with $X_\pi$ smooth, 
that are isomorphisms above $\A^2\setminus\{0\}$.
Such a morphism is necessarily a finite 
composition of point blowups; somewhat sloppily 
we will refer to it simply as a \emph{blowup}.
The set $\fB_0$ of blowups is a partially ordered
set: we say $\pi\le\pi'$ if the induced 
birational map $X_{\pi'}\to X_\pi$ is a morphism 
(and hence itself a composition of point blowups).
In fact, $\fB_0$ is a directed system: any two 
blowups can be dominated by a third. 
%
%
\subsubsection{Exceptional primes}\label{S234}
An irreducible component $E\subseteq\pi^{-1}(0)$ is called
an \emph{exceptional prime} (divisor) of $\pi$. There are as many
exceptional primes as the number of points blown up. 
We often identify an exceptional prime of $\pi$ with its  
strict transform to any blowup $\pi'\in\fB_0$ dominating $\pi$.
In this way we can identify an exceptional prime $E$
(of some blowup $\pi$) with the corresponding divisorial valuation $\ord_E$. 

If $\pi_0$ is the simple blowup of the origin, then there is a unique
exceptional prime $E_0$ of $\pi_0$ whose associated divisorial valuation is
$\ord_{E_0}=\ord_0$. Since any blowup $\pi\in\fB_0$ 
factors through $\pi_0$, $E_0$ is an exceptional prime of 
any $\pi$.
%
%
\subsubsection{Free and satellite points}\label{S274}
The following terminology is convenient and commonly used in the 
literature. Consider a closed point $\xi\in\pi^{-1}(0)$ for some
blowup $\pi\in\fB_0$. We say that $\xi$ is a \emph{free} point 
if it belongs to a unique exceptional prime; otherwise it is the
intersection point of two distinct exceptional primes and is called
a \emph{satellite} point.
%
%
\subsubsection{Exceptional divisors}\label{S212}
A divisor on $X_\pi$ is \emph{exceptional} if
its support is contained in $\pi^{-1}(0)$.
We write $\Div(\pi)$ for the abelian group of 
exceptional divisors on $X_\pi$.
If $E_i$, $i\in I$, are the exceptional primes of $\pi$, then 
$\Div(\pi)\simeq\bigoplus_{i\in I}\Z E_i$.

If $\pi,\pi'$ are blowups and $\pi'=\pi\circ\mu\ge\pi$,
then there are natural maps 
\begin{equation*}
  \mu^*:\Div(\pi)\to\Div(\pi')
  \quad\text{and}\quad
  \mu_*:\Div(\pi')\to\Div(\pi)
\end{equation*}
satisfying the projection formula $\mu_*\mu^*=\id$.
In many circumstances it is natural to
identify an exceptional divisor $Z\in\Div(\pi)$
with its pullback $\mu^*Z\in\Div(\pi')$.
%
%
\subsubsection{Intersection form}\label{S243}
We denote by $(Z\cdot W)$ the intersection number 
between exceptional divisors $Z,W\in\Div(\pi)$.
If $\pi'=\pi\circ\mu$, 
then $(\mu^*Z\cdot W')=(Z\cdot\mu_*W')$ 
and hence $(\mu^*Z\cdot\mu^*W)=(Z\cdot W)$ for
$Z,W\in\Div(\pi)$, $Z'\in\Div(\pi')$.
\begin{Prop}\label{P203}
  The intersection form on $\Div(\pi)$ is negative definite and 
  unimodular.
\end{Prop}
\begin{proof}
  We argue by induction on the number of blowups in $\pi$.
  If $\pi=\pi_0$ is the simple blowup of $0\in\A^2$, then 
  $\Div(\pi)=\Z E_0$ and $(E_0\cdot E_0)=-1$. 
  For the inductive step, suppose $\pi'=\pi\circ\mu$, where
  $\mu$ is the simple blowup of a closed point on $\pi^{-1}(0)$,
  resulting in an exceptional prime $E$. 
  Then we have an orthogonal decomposition 
  $\Div(\pi')=\mu^*\Div(\pi)\oplus\Z E$. 
  The result follows since $(E\cdot E)=-1$.

  Alternatively, we may view $\A^2$ as embedded in $\P^2$
  and $X_\pi$ accordingly embedded in a smooth compact
  surface $\bX_\pi$. The proposition can then be obtained
  as a consequence of the Hodge Index Theorem~\cite[p.364]{Hartshorne}
  and Poincar{\'e} Duality applied to the smooth rational 
  surface $\bX_\pi$.
\end{proof}
%
%
\subsubsection{Positivity}\label{S159}
It follows from Proposition~\ref{P203} that for any $i\in I$ 
there exists a unique divisor $\vE_i\in\Div(\pi)$ such that 
$(\vE_i\cdot E_i)=1$ and $(\vE_i\cdot E_j)=0$ for 
$j\ne i$. 

An exceptional divisor $Z\in\Div(\pi)$ is 
\emph{relatively nef}\footnote{The acronym ``nef'' is due to M.~Reid
 who meant it to stand for ``numerically eventually free'' although
 many authors refer to it as ``numerically effective''.} 
if $(Z\cdot E_i)\ge 0$ for all exceptional primes $E_i$. 
We see that the set of relatively nef divisors is a free semigroup 
generated by the $\vE_i$, $i\in I$. 
Similarly, the set of \emph{effective} divisors is  
a free semigroup generated by the $E_i$, $i\in I$.

Using the negativity of the intersection form and some 
elementary linear algebra, one shows that the divisors $\vE_i$ have strictly 
negative coefficients in the basis $(E_j)_{j\in I}$. Hence any
relatively nef divisor is antieffective.\footnote{A higher-dimensional 
  version of this result is known as the ``Negativity Lemma''
  in birational geometry: see~\cite[Lemma~3.39]{KollarMori}
  and also~\cite[Proposition~2.11]{BdFF}.}

We encourage the reader to explicitly construct the divisors $\vE_i$ 
using the procedure in the proof of Proposition~\ref{P203}. 
Doing this, one sees directly that $\vE_i$ is antieffective.
See also~\S\ref{S252}.
%
%
\subsubsection{Invariants of exceptional primes}\label{S249}
To any exceptional prime $E$ (or the associated divisorial valuation 
$\ord_E\in\hcV^*_0)$ we can associate two basic numerical
invariants $\a_E$ and $A_E$. We shall not directly use them
in this paper, but they seem quite fundamental and their
cousins at infinity (see~\S\ref{S151}) will be of great importance.

To define $\a_E$, pick a blowup $\pi\in\fB_0$ for which 
$E$ is an exceptional prime.
Above we defined the divisor $\vE=\vE_\pi\in\Div(\pi)$ by duality:
$(\vE_\pi\cdot E)=1$ and $(\vE_\pi\cdot F)=0$ for all 
exceptional primes $F\ne E$ of $\pi$.
Note that if $\pi'\in\fB_0$ dominates $\pi$,
then the divisor $\vE_{\pi'}\in\Div(\pi')$
is the pullback of $\vE_\pi$ under the morphism 
$X_{\pi'}\to X_\pi$. In particular, the self-intersection number 
\begin{equation*}
  \a_E:=\a(\ord_E):=(\vE\cdot\vE)
\end{equation*}
is an integer independent of the choice of $\pi$. Since $\vE$ is
antieffective, $\a_E\le -1$.

The second invariant is the 
\emph{log discrepancy} $A_E$.\footnote{The log discrepancy is called 
  \emph{thinness} in~\cite{valtree,pshsing,valmul,eigenval}.}
This is an important invariant in higher dimensional
birational geometry, see~\cite{Kollar}.
Here we shall use a definition adapted to our purposes.
Let $\omega$ be a nonvanishing regular 2-form on $\A^2$. 
If $\pi\in\fB_0$ is a blowup,  then $\pi^*\omega$ is a 
regular 2-form on $X_\pi$. For any exceptional 
prime $E$ of $\pi$ with associated divisorial valuation 
$\ord_E\in\hcV^*_0$, we define
\begin{equation}\label{e149}
  A_E:=A(\ord_E):=1+\ord_E(\pi^*\omega).
\end{equation}
Note that $\ord_E(\pi^*\omega)$ is simply the order of 
vanishing along $E$ of the Jacobian determinant of $\pi$.
The log discrepancy $A_E$ is a positive integer 
whose value does not depend on the choice of $\pi$ or $\omega$. 
A direct calculation shows that $A(\ord_0)=2$.
%
%
\subsubsection{Ideals and log resolutions} 
A \emph{log resolution} of a primary ideal $\fa\subseteq R$
is a blowup $\pi\in\fB_0$ such that the ideal sheaf
$\fa\cdot\cO_{X_\pi}$ on $X_\pi$ is locally principal:
\begin{equation}\label{e159}
  \fa\cdot\cO_{X_\pi}=\cO_{X_\pi}(Z)
\end{equation}
for some exceptional divisor $Z=Z_\pi(\fa)\in\Div(\pi)$.
This means that the pullback of the ideal $\fa$ to 
$X_\pi$ is locally generated by a single monomial
in coordinates defining the exceptional primes.
It is an important basic fact that 
any primary ideal $\fa\subseteq R$
admits a log resolution.

If $\pi$ is a log resolution of $\fa$ and 
$\pi'=\pi\circ\mu\ge\pi$,
then $\pi'$ is also a log resolution of $\fa$ and 
$Z_{\pi'}(\fa)=\mu^*Z_\pi(\fa)$.
\begin{Example}\label{E209}
  The ideal 
  $\fa=(z_2^2-z_1^3,z_1^2z_2)$ 
  admits a log resolution that is a
  composition of four point blowups.
  Each time we blow up the base locus 
  of the strict transform of $\fa$. 
  The first blowup is at the origin.
  In the terminology of~\S\ref{S274}, the second and fourth
  blowups occur at free points whereas the third blowup is at a
  satellite point.
  See Figure~\ref{F205}.
\end{Example}
\begin{figure}[ht]
\includegraphics{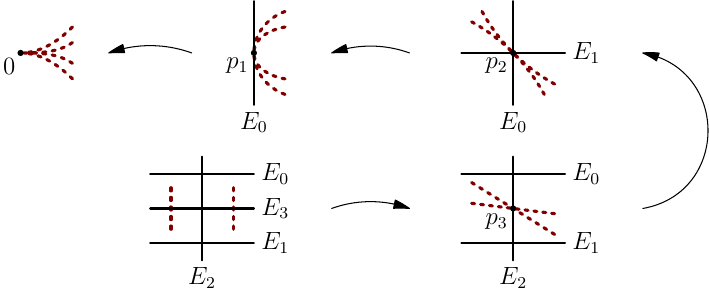}
\caption{A log resolution of the primary ideal
  $\fa=(z_2^2-z_1^3,z_1^2z_2)$.
  The dotted curves show the strict transforms of curves of
  the form $C_a=\{z_2^2-z_1^3=az_1^2z_2\}$ for two different values
  of $a\in K^*$.
  The first blowup is the blowup of the 
  origin; then we successively blow up
  the intersection of the exceptional divisor with the 
  strict transform of the curves $C_a$.
  In the terminology of~\S\ref{S274}, the second and fourth
  blowups occur at free points whereas the third blowup is at a
  satellite point.
}\label{F205}
\end{figure}
%
%
\subsubsection{Ideals and positivity}\label{S215}
The line bundle $\cO_{X_\pi}(Z)$ on $X_\pi$ in~\eqref{e159} is
\emph{relatively base point free}, that is, it admits
a nonvanishing section at any point of $\pi^{-1}(0)$.
Conversely, if $Z\in\Div(\pi)$ is an exceptional divisor 
such that $\cO_{X_\pi}(Z)$ is relatively base point free, then
$Z=Z_\pi(\fa)$ for $\fa=\pi_*\cO_{X_\pi}(Z)$.

If a line bundle $\cO_{X_\pi}(Z)$ is relatively base point free, then 
its restriction to any exceptional prime $E$ 
is also base point free, implying 
$(Z\cdot E)=\deg(\cO_{X_\pi}(Z)|_E)\ge0$, 
so that $Z$ is relatively nef.
It is an important fact that the converse implication also holds:
\begin{Prop}\label{P202}
  If $Z\in\Div(\pi)$ is relatively nef, then the line bundle
  $\cO_{X_\pi}(Z)$ is relatively base point free.
\end{Prop}

Since $0\in\A^2$ is a trivial example of 
a rational singularity, Proposition~\ref{P202} is merely 
a special case of a result by Lipman, see~\cite[Proposition~12.1~(ii)]{Lipman}.
The proof in \loccit uses sheaf cohomology
as well as the Zariski-Grothendieck theorem on formal functions,
techniques that will not be exploited elsewhere in the paper. 
Here we outline a more elementary proof, 
taking advantage of $0\in\A^2$ being a smooth point
and working over an algebraically closed ground field.
\begin{proof}[Sketch of proof of Proposition~\ref{P202}]
  By the structure of the semigroup of relatively nef divisors, 
  we may assume $Z=\vE$ for an exceptional prime $E$ of $\pi$. 
  Pick two distinct free points 
  $\xi_1$, $\xi_2$ on $E$ and formal curves $\tC_i$
  at $\xi_i$, $i=1,2$, intersecting $E$ transversely.
  Then $C_i:=\pi(\tC_i)$, $i=1,2$ are formal curves at $0\in\A^2$
  satisfying $\pi^*C_i=\tC_i+G_i$, where $G_i\in\Div(\pi)$
  is an exceptional divisor. 
  Now $(\pi^*C_i\cdot F)=0$ for every exceptional prime $F$ of 
  $\pi$, so $(G_i\cdot F)=-(\tC_i\cdot F)=-\delta_{EF}=(-\vE\cdot F)$.
  Since the intersection pairing on $\Div(\pi)$ is nondegenerate, 
  this implies $G_i=-\vE$, that is,  $\pi^*C_i=\tC_i-\vE$ for $i=1,2$. 

  Pick $\phi_i\in\widehat{\cO}_{\A^2,0}$ defining $C_i$.
  Then the ideal $\hfa$ generated by $\phi_1$ and $\phi_2$
  is primary so the ideal $\fa:=\hfa\cap\cO_{\A^2,0}$ is also primary and
  satisfies $\fa\cdot\widehat{\cO}_{\A^2,0}=\hfa$. 
  Since $\ord_F(\fa)=\ord_F(\phi_i)=-\ord_F(\vE)$, $i=1,2$, 
  for any exceptional prime $F$ and the
  (formal) curves $\tC_i$ are disjoint, it follows that 
  $\fa\cdot\cO_{X_\pi}=\cO_{X_\pi}(\vE)$ as desired.
\end{proof}
%
%
%
%
\subsection{Dual graphs and fans}\label{S217}
To a blowup $\pi\in\fB_0$ we can associate two basic 
combinatorial objects, equipped with additional structure.
%
%
\subsubsection{Dual graph}
First we have the classical notion of the \emph{dual graph}
$\Delta(\pi)$. This is an abstract simplicial complex of dimension
one.  Its vertices correspond to exceptional primes of $\pi$ and 
its edges to proper intersections between exceptional primes.
In the literature one often labels each vertex with the self-intersection
number of the corresponding exceptional prime.
We shall not do so here since this number is not an invariant of
the corresponding divisorial valuation but
depends also on the blowup $\pi$. 
From the point of view of these notes, it is more 
natural to use invariants such as the ones in~\S\ref{S249}.

The dual graph $\Delta(\pi)$ is connected and simply connected.
This can be seen using the decomposition of $\pi$ as a
composition of point blowups, see~\S\ref{S157}.
Alternatively, the connectedness of $\Delta(\pi)$ follows
from Zariski's Main Theorem~\cite[p.280]{Hartshorne}
and the simple connectedness can be deduced from
sheaf cohomology considerations, see~\cite[Corollary~7]{Artin}.

See Figure~\ref{F207} for an example of a dual graph.
\begin{figure}[ht]
\includegraphics{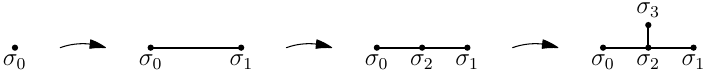}
\caption{The dual graphs of the blowups leading up to 
  the log resolution of the primary ideal 
  $\fa=(z_2^2-z_1^3,z_1^2z_2)$ described in Example~\ref{E209}
  and depicted in Figure~\ref{F205}.
  Here $\sigma_i$ is the vertex corresponding to $E_i$.
}\label{F207}
\end{figure}
%
%
\subsubsection{Dual fan}\label{S231}
While the dual graph $\Delta(\pi)$ is a natural object, the 
\emph{dual fan} $\hDelta(\pi)$ is arguably more canonical. 
To describe it, we use basic notation and terminology from 
toric varieties, see~\cite{KKMS,Fulton,Oda}.\footnote{We shall not, however, 
  actually consider the toric variety defined by the fan $\hDelta(\pi)$.}
Set 
\begin{equation*}
  N(\pi):=\Hom(\Div(\pi),\Z).
\end{equation*}
If we label the exceptional primes $E_i$, $i\in I$, then 
we can write $N(\pi)=\bigoplus_{i\in I}\Z e_i\simeq\Z^I$
with $e_i$ satisfying $\langle e_i,E_j\rangle=\delta_{ij}$.
Note that if we identify $N(\pi)$ with $\Div(\pi)$ using the 
unimodularity of the intersection product (Proposition~\ref{P203}),
then $e_i$ corresponds to the divisor $\vE_i$ in~\S\ref{S159}.

Set $N_\R(\pi):=N(\pi)\otimes_\Z\R\simeq\R^I$.
The one-dimensional cones in $\hDelta(\pi)$ are then 
of the form $\hsigma_i:=\R_+e_i$, $i\in I$, and the two-dimensional
cones are of the form $\hsigma_{ij}:=\R_+e_i+\R_+e_j$, 
where $i,j\in I$ are such that $E_i$ and $E_j$ intersect properly.
Somewhat abusively, we will write $\hDelta(\pi)$ 
both for the fan and for its support (which is a subset of $N_\R(\pi)$).

Note that the dual fan $\hDelta(\pi)$ is naturally a cone over
the dual graph $\Delta(\pi)$. In~\S\ref{S136} we shall see how to 
embed the dual graph inside the dual fan.

A point $t\in\hDelta(\pi)$ is \emph{irrational} if
$t=t_1e_1+t_2e_2$ with $t_i>0$ and $t_1/t_2\not\in\Q$;
otherwise $t$ is \emph{rational}. Note that the rational
points are always dense in $\hDelta(\pi)$. The irrational points
are also dense except if $\pi=\pi_0$, the simple blowup
of $0\in\A^2$.
%
%
\subsubsection{Free and satellite blowups}\label{S157} 
Using the factorization of birational surface maps into simple
point blowups, we can understand the structure of the dual
graph and fan of a blowup $\pi\in\fB_0$.

First, when $\pi=\pi_0$ is a single blowup of the origin, there is
a unique exceptional prime $E_0$, so $\hDelta(\pi_0)$ 
consists of a single, one-dimensional cone 
$\hsigma_0=\R_+e_0$ and $\Delta(\pi)=\{\sigma_0\}$ is 
a singleton.

Now suppose $\pi'$ is obtained from $\pi$ by blowing 
up a closed point $\xi\in\pi^{-1}(0)$.
Let $E_i$, $i\in I$ be the exceptional primes of $\pi$.
Write $I=\{1,2,\dots,n-1\}$, where $n\ge 2$.
If $E_n\subseteq X_{\pi'}$ is the preimage of $\xi$, then the exceptional
primes of $\pi'$ are $E_i$, $i\in I'$, where $I'=\{1,2,\dots,n\}$.
Recall that we are identifying an exceptional prime of $\pi$
with its strict transform in $X_{\pi'}$.

To see what happens in detail,
first suppose $\xi$ is a free point, belonging 
to a unique exceptional prime of $\pi$, 
say $E_1$. In this case, 
the dual graph $\Delta(\pi')$ is obtained from $\Delta(\pi)$ 
by connecting a new vertex $\sigma_n$ to $\sigma_1$.
See Figure~\ref{F106}.

If instead $\xi$ is a satellite point,
belonging to two distinct exceptional primes of $\pi$,
say $E_1$ and $E_2$, then 
we obtain $\Delta(\pi')$ from $\Delta(\pi)$ 
by subdividing the edge $\sigma_{12}$
into two edges $\sigma_{1n}$ and $\sigma_{2n}$.
Again see Figure~\ref{F106}.
\begin{figure}[ht]
  \begin{center}
    \includegraphics{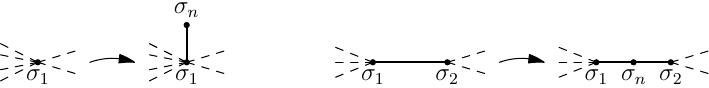}
 \end{center}
  \caption{Behavior of the dual graph under a single blowup.
    The left part of the picture illustrates the blowup of a free
   point on $E_1$, creating a new vertex $\sigma_n$ connected 
   to the vertex $\sigma_1$.
   The right part of the picture illustrates the blowup of the satellite
   point $E_1\cap E_2$, creating a new vertex $\sigma_n$ and
   subdividing the segment $\sigma_{12}$ into two segments 
   $\sigma_{1n}$ and $\sigma_{2n}$.
}\label{F106}
\end{figure}
%
%
\subsubsection{Integral affine structure}\label{S232}
We define the \emph{integral affine structure} on $\hDelta(\pi)$
to be the lattice 
\begin{equation*}
  \Aff(\pi):=\Hom(N(\pi),\Z)\simeq\Z^I
\end{equation*}
and refer to its elements as integral affine functions. 
By definition,  $\Aff(\pi)$ can be identified with the group $\Div(\pi)$
of exceptional divisors on $X_\pi$.
%
%
\subsubsection{Projections and embeddings}\label{S233}
Consider blowups $\pi,\pi'\in\fB_0$ with $\pi\le\pi'$,
say $\pi'=\pi\circ\mu$, with $\mu:X_{\pi'}\to X_\pi$
a birational morphism.
Then $\mu$ gives rise to an injective homomorphism 
$\mu^*:\Div(\pi)\to\Div(\pi')$ and we let
\begin{equation*}
  r_{\pi\pi'}:N(\pi')\to N(\pi)
\end{equation*}
denote its transpose. 
It is clear that $r_{\pi\pi'}\circ r_{\pi'\pi''}=r_{\pi\pi''}$ when $\pi\le\pi'\le\pi''$.
\begin{Lemma}\label{L210}
  Suppose $\pi,\pi'\in\fB_0$ and $\pi\le\pi'$.  Then:
  \begin{itemize}
  \item[(i)]
    $r_{\pi\pi'}(\hDelta(\pi'))=\hDelta(\pi)$;
  \item[(ii)]
    any irrational point in $\hDelta(\pi)$ has a unique preimage 
    in $\hDelta(\pi')$;
  \item[(iii)]
    if $\hsigma'$ is a $2$-dimensional cone in $\hDelta(\pi)$ then
    either $r_{\pi\pi'}(\hsigma')$ is a one-dimensional cone in $\hDelta(\pi)$,
    or $r_{\pi\pi'}(\hsigma')$ is a $2$-dimensio\-nal cone contained in 
    a $2$-dimensional cone $\hsigma$ of $\hDelta(\pi)$.
    In the latter case, the restriction of $r_{\pi\pi'}$ to $\hsigma'$
    is unimodular in the sense that 
    $r_{\pi\pi'}^*\Aff(\pi)|_{\hsigma'}=\Aff(\pi')|_{\hsigma'}$.
  \end{itemize}
\end{Lemma}
We use the following notation.
If $e_i$ is a basis element of $N(\pi)$ associated to an exceptional prime $E_i$,
then $e'_i$ denotes the basis element of $N(\pi')$ associated to 
the strict transform of $E_i$.
\begin{proof}
  It suffices to treat the case when $\pi'=\pi\circ\mu$, where $\mu$ is 
  a single blowup of a closed point $\xi\in\pi^{-1}(0)$. 
  As in~\S\ref{S157} we let
  $E_i$, $i\in I$ be the exceptional primes of $\pi$.
  Write $I=\{1,2,\dots,n-1\}$, where $n\ge 2$.
  If $E_n\subseteq X_{\pi'}$ is the preimage of $\xi$, then the exceptional
  primes of $\pi'$ are $E_i$, $i\in I'$, where $I'=\{1,2,\dots,n\}$.
  
  First suppose $\xi\in E_1$ is a free point.
  Then $r_{\pi\pi'}(e'_i)=e_i$ for $1\le i<n$ and $r_{\pi\pi'}(e'_n)=e_1$.
  Conditions~(i)-(iii) are immediately verified:
  $r_{\pi\pi'}$ maps the cone $\hsigma'_{1n}$ onto $\hsigma_1$
  and  maps all other cones $\hsigma'_{ij}$ onto the corresponding
  cones $\hsigma_{ij}$, preserving the integral affine structure.

  Now suppose $\xi\in E_1\cap E_2$ is a satellite point.
  The linear map $r_{\pi\pi'}$ is then determined by 
  $r_{\pi\pi'}(e'_i)=e_i$ for $1\le i<n$ and  $r_{\pi\pi'}(e'_n)=e_1+e_2$. 
  We see that the cones $\hsigma_{1n}$ and $\hsigma_{2n}$ in $\hDelta(\pi')$ 
  map onto the subcones 
  $\R_+e_1+\R_+(e_1+e_2)$ and $\R_+e_2+\R_+(e_1+e_2)$, respectively,
  of the cone $\hsigma_{12}$ in $\hDelta(\pi)$.
  Any other cone $\hsigma'_{ij}$ of $\hDelta(\pi')$ is mapped onto the 
  corresponding cone $\hsigma_{ij}$ of $\hDelta(\pi)$,
  preserving the integral affine structure.
  Conditions~(i)--(iii) follow.
\end{proof}
Using Lemma~\ref{L210} we can show that $r_{\pi\pi'}$ admits a 
natural one-side inverse.
\begin{Lemma}\label{L204}
  Let $\pi,\pi'\in\fB_0$ be as above. Then 
  there exists a unique continuous, homogeneous map 
  $\iota_{\pi'\pi}:\hDelta(\pi)\to\hDelta(\pi')$ such that:
  \begin{itemize}
  \item[(i)] 
    $r_{\pi\pi'}\circ\iota_{\pi'\pi}=\id$ on $\hDelta(\pi)$;
  \item[(ii)]
    $\iota_{\pi'\pi}(e_i)=e'_i$ for all $i$.
  \end{itemize}  
  Further, a two-dimensional cone $\hsigma'$ in 
  $\hDelta(\pi')$ is contained in the image of $\iota_{\pi'\pi}$ 
  iff $r_{\pi\pi'}(\hsigma')$ is two-dimensional.
\end{Lemma}
It follows easily from the uniqueness statement that 
$\iota_{\pi''\pi}=\iota_{\pi''\pi'}\circ\iota_{\pi'\pi}$
when $\pi\le\pi'\le\pi''$.
We emphasize that $\iota_{\pi'\pi}$ is only \emph{piecewise} linear
and not the restriction to $\hDelta(\pi)$ of a linear
map $N_\R(\pi)\to N_\R(\pi')$.
\begin{proof}[Proof of Lemma~\ref{L204}]
  Uniqueness is clear: when $\pi=\pi_0$ is the simple blowup
  of $0\in\A^2$, $\iota_{\pi'\pi}$ is determined by~(ii) and 
  when $\pi\ne\pi_0$, the irrational points are dense 
  in $\hDelta(\pi)$ and uniqueness is a consequence of
  Lemma~\ref{L210}~(ii).

  As for existence, it suffices to treat the case when
  $\pi'=\pi\circ\mu$, where $\mu$ is 
  a simple blowup of a closed point $\xi\in\pi^{-1}(0)$.
  
  When $\xi\in E_1$ is a free point, $\iota_{\pi'\pi}$ maps
  $e_i$ to $e'_i$ for $1\le i<n$ and 
  maps any cone $\hsigma_{ij}$ in $\hDelta(\pi)$
  onto the corresponding cone $\hsigma'_{ij}$ 
  in $\hDelta(\pi')$ linearly
  via $\iota_{\pi'\pi}(t_ie_i+t_je_j)=(t_ie'_i+t_je'_j)$.

  If instead $\xi\in E_1\cap E_2$ is a satellite point,
  then $\iota_{\pi'\pi}(e_i)=e'_i$ for $1\le i<n$.
  Further, $\iota_{\pi'\pi}$ is piecewise linear on the cone 
  $\hsigma_{12}$:
  \begin{equation}\label{e216}
    \iota_{\pi'\pi}(t_1e_1+t_2e_2)=
    \begin{cases}
      (t_1-t_2)e'_1+t_2e'_n &\text{if $t_1\ge t_2$}\\
      (t_2-t_1)e'_2+t_1e'_n &\text{if $t_1\le t_2$}\\
    \end{cases}
  \end{equation}
  and maps any other two-dimensional cone 
  $\hsigma_{ij}$ onto $\hsigma'_{ij}$ linearly
  via $\iota_{\pi'\pi}(t_ie_i+t_je_j)=(t_ie'_i+t_je'_j)$.
\end{proof}
%
%
\subsubsection{Embedding the dual graph in the dual fan}\label{S136}
We have noted that $\hDelta(\pi)$ can be viewed as a cone over $\Delta(\pi)$.
Now we embed $\Delta(\pi)$ in $\hDelta(\pi)\subseteq N_\R$,
in a way that remembers the maximal ideal $\fm_0$.
For $i\in I$ define an integer $b_i\ge 1$ by
\begin{equation*}
  b_i:=\ord_{E_i}(\fm_0),
\end{equation*}
where $\ord_{E_i}$ is the divisorial valuation given by 
order of vanishing along $E_i$.
There exists a unique function $\f_0\in\Aff(\pi)$
such that $\f_0(e_i)=b_i$. It is the integral affine
function corresponding to the exceptional divisor 
$-Z_0\in\Div(\pi)$, where $Z_0=-\sum_{i\in I}b_iE_i$.
Note that $\pi$ is a log resolution of the maximal ideal $\fm_0$
and that $\fm_0\cdot\cO_{X_\pi}=\cO_{X_\pi}(Z_0)$.

We now define $\Delta(\pi)$ as the subset of $\hDelta(\pi)$
given by $\f_0=1$. In other words, the vertices of $\Delta(\pi)$
are of the form 
\begin{equation*}
  \sigma_i:=\hsigma_i\cap\Delta(\pi)= b_i^{-1}e_i
\end{equation*}
 and the edges of the form 
\begin{equation*}
  \sigma_{ij}:=\hsigma_{ij}\cap\Delta(\pi)=\{t_ie_i+t_je_j\mid t_i,t_j\ge0, b_it_i+b_jt_j=1\}.
\end{equation*}
If $\pi,\pi'\in\fB_0$ and 
$\pi'\ge\pi$, then $r_{\pi\pi'}(\Delta(\pi'))=\Delta(\pi)$
and $\iota_{\pi'\pi}(\Delta(\pi))\subseteq\Delta(\pi')$.
%
%
\subsubsection{Auxiliary calculations}\label{S252}
For further reference let us record a  few calculations involving the numerical 
invariants $A$, $\a$ and $b$ above.

If $\pi_0\in\fB_0$ is the simple blowup of the origin, then 
\begin{equation*}
  A_{E_0}=2,
  \quad
  b_{E_0}=1,
  \quad
  \vE_0=-E_0
  \quad\text{and}\quad
  \a_{E_0}=-1.
\end{equation*}

Now suppose $\pi'=\pi\circ\mu$, where $\mu$ is the simple blowup
of a closed point $\xi$ and let us check how the numerical invariants behave.
We use the notation of~\S\ref{S157}.
In the case of a free blowup we have
\begin{equation}\label{e212}
  A_{E_n}=A_{E_1}+1,
  \quad
  b_{E_n}=b_{E_1}
  \quad\text{and}\quad
  \vE_n=\vE_1-E_n,
\end{equation}
where, in the right hand side, we identify 
the divisor $\vE_1\in\Div(\pi)$ with its
pullback in $\Div(\pi')$. 
Since $(E_n\cdot E_n)=-1$ we derive as a consequence,
\begin{equation}\label{e213}
 \a_{E_n}:=(\vE_n\cdot\vE_n)=(\vE_1\cdot\vE_1)-1=\a_{E_1}-1.
\end{equation}

In the case of a satellite blowup,
\begin{equation}\label{e214}
  A_{E_n}=A_{E_1}+A_{E_2},
  \quad
  b_{E_n}=b_{E_1}+b_{E_2} 
  \quad\text{and}\quad
  \vE_n=\vE_1+\vE_2-E_n.
\end{equation}
Using $(E_n\cdot E_n)=-1$ this implies
\begin{equation}\label{e215}
  \a_{E_n}:=\a_{E_1}+\a_{E_2}+2(\vE_1\cdot\vE_2)-1.
\end{equation}

We also claim that if $E_i$, $E_j$ are exceptional primes
that intersect properly in some $X_\pi$, then 
\begin{equation}\label{e217}
  ((b_i\vE_j-b_j\vE_i)\cdot(b_i\vE_j-b_j\vE_i))=-b_ib_j.
\end{equation}
Note that both sides of~\eqref{e217} are independent 
of the blowup $\pi\in\fB_0$ but we have to assume that $E_i$
and $E_j$ intersect properly in \emph{some} blowup.

To prove~\eqref{e217}, we proceed inductively. 
It suffices to consider the case when $E_i$ is obtained 
by blowing up a closed point $\xi\in E_j$.
When $\xi$ is free, we have $b_i=b_j$,
$\vE_i=\vE_j-E_i$ and~\eqref{e217} reduces to the 
fact that $(E_i\cdot E_i)=-1$.
When instead $\xi\in E_j\cap E_k$ is a satellite point,
we have $((b_i\vE_k-b_k\vE_i)\cdot(b_i\vE_k-b_k\vE_i))=-b_ib_k$
by induction. Furthermore, $b_i=b_j+b_k$,
$\vE_i=\vE_j+\vE_k-E_i$; we obtain~\eqref{e217} from these
equations and from simple algebra.

In the dual graph depicted in Figure~\ref{F207} we have
$b_0=b_1=1$, 
$b_2=b_3=2$,
$\a_0=-1$,
$\a_1=-2$,
$\a_2=-6$,
$\a_3=-7$,
$A_0=2$,
$A_1=3$,
$A_2=5$
and
$A_3=6$.
%
%
\subsubsection{Extension of the numerical invariants}\label{S251}
We extend the numerical invariants $A$ and $\a$ in~\S\ref{S133} 
to functions on the dual fan
\begin{equation*}
  A_\pi:\hDelta(\pi)\to\R_+
  \quad\text{and}\quad
  \a_\pi:\hDelta(\pi)\to\R_-
\end{equation*}
as follows.
First we set $A_\pi(e_i)=A_{E_i}$ and extend $A_\pi$
uniquely as an (integral) linear function on $\hDelta(\pi)$.
Thus we set $A_\pi(t_ie_i)=t_iA_\pi(e_i)$ and 
\begin{equation}\label{e218}
  A_\pi(t_ie_i+t_je_j)
  =t_iA_\pi(e_i)+t_jA_\pi(e_j).
\end{equation}
In particular, $A_\pi$ is integral affine on each simplex in
the dual graph $\Delta(\pi)$. 

Second, we set $\a_\pi(e_i)=\a_{E_i}=(\vE_i\cdot\vE_i)$ and extend $\a_\pi$
as a homogeneous function of order \emph{two} on $\hDelta(\pi)$
which is affine on each simplex in the dual graph $\Delta(\pi)$.
In other words, we set $\a_\pi(t_ie_i)=t_i^2\a_\pi(e_i)$ for any $i\in I$
and 
\begin{align}\label{e219}
  \a_\pi(t_ie_i+t_je_j)\notag
  &=(b_it_i+b_jt_j)^2
  \left(
    \frac{b_it_i}{b_it_i+b_jt_j}\a_\pi(\sigma_i)
    +\frac{b_jt_j}{b_it_i+b_jt_j}\a_\pi(\sigma_j)
  \right)\\
  &=(b_it_i+b_jt_j)
  \left(
    \frac{t_i}{b_i}\a_\pi(e_i)+
    \frac{t_j}{b_j}\a_\pi(e_j)
  \right)
\end{align}
whenever $E_i$ and $E_j$ intersect properly.

Let us check that 
\begin{equation*}
  A_{\pi'}\circ\iota_{\pi'\pi}=A_\pi
  \quad\text{and}\quad
  \a_{\pi'}\circ\iota_{\pi'\pi}=\a_\pi
\end{equation*}
on $\hDelta(\pi)$ whenever $\pi'\ge\pi$.
It suffices to do this when $\pi'=\pi\circ\mu$ and
$\mu$ is the blowup of $X_{\pi}$ at a closed point $\xi$.
Further, the only case that requires verification is when 
$\xi\in E_1\cap E_2$ is a satellite point, in which case 
it suffices to prove 
$A_\pi(e_1+e_2)=A_{\pi'}(e'_n)$ and
$\a_\pi(e_1+e_2)=\a_{\pi'}(e'_n)$. 
The first of these formulas follows from~\eqref{e214} and~\eqref{e218}
whereas the second results from~\eqref{e215},~\eqref{e217} and~\eqref{e219}.
The details are left to the reader.

In the dual graph depicted in Figure~\ref{F207} we have
$A_\pi(\sigma_0)=2$,
$A_\pi(\sigma_1)=3$,
$A_\pi(\sigma_2)=5/2$,
$A_\pi(\sigma_3)=3$,
$\a_\pi(\sigma_0)=-1$, 
$\a_\pi(\sigma_1)=-2$, 
$\a_\pi(\sigma_2)=-3/2$, 
and
$\a_\pi(\sigma_3)=-7/4$.
%
%
\subsubsection{Multiplicity of edges in the dual graph}\label{S250}
We define the \emph{multiplicity} $m(\sigma)$ of an edge $\sigma$ in
a dual graph $\Delta(\pi)$ as follows. Let 
$\sigma=\sigma_{ij}$ have endpoints $v_i=b_i^{-1}e_i$ and $v_j=b_j^{-1}e_j$.
We set
\begin{equation}\label{e224}
  m(\sigma_{ij}):=\gcd(b_i,b_j).
\end{equation}
Let us see what happens when $\pi'$ is obtained from 
$\pi$ by blowing up a closed point $\xi\in\pi^{-1}(0)$. We use
the notation above. See also Figure~\ref{F106}.

If $\xi\in E_1$ is a free point, then we have seen in~\eqref{e212} that 
$b_n=b_1$ and hence 
\begin{equation}\label{e160}
  m(\sigma_{1n})=b_1.
\end{equation} 
If instead $\xi\in E_1\cap E_2$ is a satellite point, then~\eqref{e214}
gives $b_n=b_1+b_2$ and hence
\begin{equation}\label{e161}
  m(\sigma_{1n})=m(\sigma_{2n})=m(\sigma_{12}).
\end{equation}
This shows that the multiplicity does not change when subdividing
a segment.

In the dual graph depicted in Figure~\ref{F207} we have
$m_{02}=m_{12}=1$ and $m_{23}=2$.
%
%
\subsubsection{Metric on the dual graph}\label{S135}
Having embedded $\Delta(\pi)$ inside $\hDelta(\pi)$, the 
integral affine structure $\Aff(\pi)$ gives rise to an 
abelian group of functions on $\Delta(\pi)$ by restriction. 
Following~\cite[p.95]{KKMS}, this further induces a volume form 
on each simplex in $\Delta(\pi)$.
In our case, this simply means a metric on each edge $\sigma_{ij}$.
The length of $\sigma_{ij}$ is the largest positive number $l_{ij}$ such that 
$\f(\sigma_i)-\f(\sigma_j)$ is an integer multiple of $l_{ij}$ 
for all $\f\in\Aff(\pi)$. From this description it follows
that $l_{ij}=\lcm(b_i,b_j)^{-1}$.

However, it turns out that the ``correct'' metric for doing potential
theory is the one for which 
\begin{equation}\label{e225}
  d_\pi(\sigma_i,\sigma_j)
  =\frac{1}{b_ib_j}
  =\frac{1}{m_{ij}}\cdot\frac{1}{\lcm(b_i,b_j)},
\end{equation}
where $m_{ij}=\gcd(b_i,b_j)$ is the multiplicity of the edge $\sigma_{ij}$
as in~\S\ref{S250}.

We have seen that the dual graph is 
connected and simply connected.
It follows that $\Delta(\pi)$ is a metric tree. 
The above results imply that if $\pi,\pi'\in\fB_0$ and 
$\pi'\ge\pi$, then $\iota_{\pi'\pi}:\Delta(\pi)\hookrightarrow\Delta(\pi')$ 
is an isometric embedding.

Let us see more concretely what happens when $\pi'$ is obtained from 
$\pi$ by blowing up a closed point $\xi\in\pi^{-1}(0)$. We use
the notation above.

If $\xi\in E_1$ is a free point, then $b_n=b_1$
and the dual graph $\Delta(\pi')$
is obtained from $\Delta(\pi)$ 
by connecting a new vertex $\sigma_n$ to $\sigma_1$
using an edge of length $b_1^{-2}$. See Figure~\ref{F106}.

If instead $\xi\in E_1\cap E_2$ is a satellite point, then
$b_n=b_1+b_2$ and we obtain $\Delta(\pi')$ from $\Delta(\pi)$ 
by subdividing the edge $\sigma_{12}$, which is of length $\frac1{b_1b_2}$
into two edges $\sigma_{1n}$ and $\sigma_{2n}$, of lengths 
$\frac1{b_1(b_1+b_2)}$ and $\frac1{b_2(b_1+b_2)}$, respectively.
Note that these lengths add up to $\frac1{b_1b_2}$.
Again see Figure~\ref{F106}.

In the dual graph depicted in Figure~\ref{F207} we have
$d(\sigma_0,\sigma_2)=d(\sigma_1,\sigma_2)=1/2$
and $d(\sigma_2,\sigma_3)=1/4$.
%
%
\subsubsection{Rooted tree structure}\label{S253}
The dual graph $\Delta(\pi)$ is a tree in the sense of~\S\ref{S172}. 
We turn it into a rooted tree by declaring the root to be the 
vertex $\sigma_0$ corresponding to the strict transform of $E_0$,
the exceptional prime of $\pi_0$, the simple blowup of $0$.

When restricted to the dual graph, the functions $\a_\pi$ and $A_\pi$ on the dual fan
$\hDelta(\pi)$ described in~\S\ref{S251} define 
parametrizations 
\begin{equation}\label{e226}
  \a_\pi:\Delta(\pi)\to\,]\!-\infty,-1]
  \quad\text{and}\quad
  A_\pi:\Delta(\pi)\to[2,\infty[
\end{equation}
satisfying $A_{\pi'}\circ\iota_{\pi'\pi}=A_\pi$ and 
$\a_{\pi'}\circ\iota_{\pi'\pi}=\a_\pi$ whenever $\pi'\ge\pi$.
 
We claim that $\a_\pi$ induces the metric on the dual graph given 
by~\eqref{e225}. For this, it suffices to show that 
$|\a_\pi(\sigma_i)-\a_\pi(\sigma_j)|=\frac1{b_ib_j}$ when $E_i$, $E_j$
are exceptional primes intersecting properly. In fact, it suffices to 
verify this when $E_i$ is obtained by blowing up a free point on $E_j$.
But then $b_i=b_j$ and it follows from~\eqref{e213} that 
\begin{equation*}
  \a_\pi(\sigma_i)-\a_\pi(\sigma_j)
  =b_i^{-2}(\a_{E_i}-\a_{E_j})
  =-b_i^{-2}
  =-d(\sigma_i,\sigma_j).
\end{equation*}

In a similar way we see that the parametrization $A_\pi$ of
$\Delta(\pi)$ induces by the log discrepancy gives rise to
the metric induced by the integral affine structure 
as in~\S\ref{S135}. In other words,
if $E_i$, $E_j$ are exceptional primes of $X_\pi$ 
intersecting properly, then 
\begin{equation}\label{e227}
  A(\sigma_j)-A(\sigma_i)=-m_{ij}(\a(\sigma_j)-\a(\sigma_i)),
\end{equation}
where $m_{ij}=\gcd(b_i,b_j)$ is the multiplicity of the edge $\sigma_{ij}$.
%
%
%
%
\subsection{Valuations and dual graphs}\label{S143}
Now we shall show how to embed the dual graph into the valuative tree.
%
%
\subsubsection{Center}\label{S239}
It follows from the valuative criterion of properness that 
any semivaluation $v\in\hcV_0^*$ admits a \emph{center}
on $X_\pi$, for any blowup $\pi\in\fB_0$.
The center is the unique (not necessarily closed) point 
$\xi=c_\pi(v)\in X_\pi$ such that $v\ge 0$ 
on the local ring $\cO_{X_\pi,\xi}$ 
and such that $\{ v>0\}\cap\cO_{X_\pi,\xi}$ 
equals the maximal ideal $\fm_{X_\pi,\xi}$.
If $\pi'\ge\pi$, then the map $X_{\pi'}\to X_\pi$ sends
$c_{\pi'}(v)$ to $c_\pi(v)$. 
%
%
\subsubsection{Evaluation}\label{S238}
Consider a semivaluation $v\in\hcV_0^*$ and a blowup $\pi\in\fB_0$.
We can evaluate $v$ on exceptional divisors $Z\in\Div(\pi)$.
Concretely, if $Z=\sum_{i\in I} r_iE_i$, 
$\xi=c_\pi(v)$ is the center of $v$ on $X_\pi$ and 
$E_j$, $j\in J$ are the exceptional primes containing $\xi$, 
then $v(Z)=\sum_{j\in J}r_jv(\zeta_j)$, where
$\zeta_j\in\cO_{X_\pi,\xi}$ and $E_j=\{\zeta_j=0\}$. 

This gives rise to an \emph{evaluation map} 
\begin{equation}\label{e203}
 \ev_\pi:\hcV_0^*\to N_\R(\pi)
\end{equation}
that is continuous, more or less by definition.
The image of $\ev_\pi$ is contained in
the dual fan $\hDelta(\pi)$. Furthermore, the embedding of 
the dual graph $\Delta(\pi)$ in the dual fan $\hDelta(\pi)$ was
exactly designed so that $\ev_\pi(\cV_0)\subseteq\Delta(\pi)$.
In fact, we will see shortly that these inclusions are equalities.

It follows immediately from the definitions that 
\begin{equation}
  r_{\pi\pi'}\circ\ev_{\pi'}=\ev_\pi\label{e205}
\end{equation}
when $\pi'\ge\pi$.

Notice that if the center of $v\in\hcV_0^*$ on $X_\pi$ is the generic point
of $\bigcap_{i\in J}E_i$, then $\ev_\pi(v)$ lies in the relative
interior of the cone $\sum_{i\in J}\R_+e_i$.
%
%
\subsubsection{Embedding and quasimonomial valuations}
Next we construct a one-sided inverse to the evaluation
map in~\eqref{e203}.
\begin{Lemma}\label{L205}
  Let $\pi\in\fB_0$ be a blowup. 
  Then there exists a unique continuous map
  $\emb_\pi:\hDelta^*(\pi)\to\hcV^*_0$
  such that:
  \begin{itemize}
  \item[(i)]
    $\ev_\pi\circ\emb_\pi=\id$ on $\hDelta^*(\pi)$;
  \item[(ii)]
    for $t\in\hDelta^*(\pi)$, the center of $\emb_\pi(t)$ is the generic
    point of the intersection of all exceptional primes 
    $E_i$ of $\pi$ such that $\langle t,E_i\rangle>0$.
  \end{itemize}
  Furthermore, condition~(ii) is superfluous except in the case when 
  $\pi=\pi_0$ is a simple blowup of $0\in\A^2$ in which case the dual graph
  $\Delta(\pi)$ is a singleton.
\end{Lemma}
As a consequence of~(i), $\emb_\pi:\hDelta^*(\pi)\to\hcV^*_0$ 
is injective and $\ev_\pi:\hcV^*_0\to\hDelta^*(\pi)$ surjective.
\begin{Cor}\label{C203}
  If $\pi,\pi'\in\fB_0$ and $\pi'\ge\pi$, 
  then $\emb_{\pi'}\circ\iota_{\pi'\pi}=\emb_\pi$.
\end{Cor}
As in~\S\ref{S129} we say that a valuation $v\in\hcV_0^*$ 
is \emph{quasimonomial} if it lies in the image of $\emb_\pi$ 
for some blowup $\pi\in\fB_0$. By Corollary~\ref{C203},
$v$ then lies in the image of $\emb_{\pi'}$ for all $\pi'\ge\pi$.
\begin{proof}[Proof of Corollary~\ref{C203}]
  We may assume $\pi'\ne\pi$ so that $\pi'$ is not the simple
  blowup of $0\in\A^2$. 
  The map $\emb'_\pi:=\emb_{\pi'}\circ\iota_{\pi'\pi}:\Delta(\pi)\to\cV_0$
  is continuous and satisfies
  \begin{equation*}
   \ev_\pi\circ\emb'_\pi
   =r_{\pi\pi'}\circ\ev_{\pi'}\circ\emb_{\pi'}\circ\iota_{\pi'\pi}
   =r_{\pi\pi'}\circ\iota_{\pi'\pi}
   =\id.
 \end{equation*}
 By Lemma~\ref{L205} this implies $\emb'_\pi=\emb_\pi$.
\end{proof}
\begin{proof}[Proof of Lemma~\ref{L205}]
  We first prove existence.
  Consider a point $t=\sum_{i\in I}t_ie_i\in\hDelta^*(\pi)$
  and let $J\subseteq I$ be the set of indices $i$ such that $t_i>0$.
  Let $\xi$ be the generic point of $\bigcap_{i\in J}E_i$
  and write $E_i=(\zeta_i=0)$ in local algebraic coordinates 
  $\zeta_i$, $i\in J$  at $\xi$. Then we let $\emb_\pi(t)$ 
  be the monomial valuation 
  with weights $t_i$ on $\zeta_i$ as in~\S\ref{S129}.
  More concretely, after relabeling 
  we may assume that either $J=\{1\}$ is a singleton and 
  $\emb_\pi(t)=t_1\ord_{E_1}$ is a divisorial valuation, 
  or $J=\{1,2\}$ in which case $v_t$ is defined on 
  $R\subseteq\widehat{\cO}_{X_\pi,\xi}\simeq K\llbracket \zeta_1,\zeta_2\rrbracket$
 by
  \begin{equation}\label{e208}
    \emb_\pi(t)(\sum_{\b_1,\b_2\ge 0}
    c_{\beta_1\beta_2}\zeta_1^{\b_1}\zeta_2^{\b_2})
    =\min\{t_1\b_1+t_2\b_2\mid c_\b\ne 0\}.
  \end{equation}
  It is clear that $\emb_\pi$ is continuous and 
  that $\ev_\pi\circ\emb_\pi=\id$.

 The uniqueness statement is clear when $\pi=\pi_0$
  since the only valuation whose center on $X_\pi$
  is the generic point of the exceptional divisor $E_0$
  is proportional to $\ord_{E_0}=\ord_0$.
  
  Now suppose $\pi\ne\pi_0$ and that 
  $\emb'_\pi:\hDelta^*(\pi)\to\hcV^*_0$
  is another continuous map satisfying 
  $\ev_\pi\circ\iota_\pi=\id$. 
  It suffices to show that $\emb'_\pi(t)=\emb_\pi(t)$ for 
  any \emph{irrational} $t\in\hDelta^*(\pi)$.
  But if $t$ is irrational, the value of $\emb'_\pi(t)$ 
  on a monomial $\zeta_1^{\b^1}\zeta_2^{\b_2}$ is $t_1\b_1+t_2\b_2$.
  In particular, the values on distinct monomials are
  distinct, so it follows that the value of $\emb'_\pi(t)$
  on a formal power series is given as in~\eqref{e208}.
  Hence $\emb'_\pi(t)=\emb_\pi(t)$,
  which completes the proof.
 
  In particular the divisorial valuation in $\cV_0$
  associated to the exceptional prime $E_i$ is given by 
  \begin{equation*}
    v_i:=b_i^{-1}\ord_{E_i}
    \quad\text{where}\quad
    b_i:=\ord_{E_i}(\fm_0)\in\N
  \end{equation*}
 \end{proof} 

The embedding $\emb_\pi:\hDelta^*(\pi)\hookrightarrow\hcV^*_0\subseteq\BerkAtwo$ 
extends to the full cone fan $\hDelta(\pi)$ 
and maps the apex $0\in\hDelta(\pi)$ 
to the trivial valuation $\triv_{\A^2}$ on $R$.
The boundary of $\emb_\pi:\hDelta^*(\pi)$ inside $\BerkAtwo$ consists
of $\triv_{\A^2}$ and the semivaluation $\triv_0$.
Thus $\emb_\pi(\hDelta(\pi))$ looks like a ``double cone''.
See Figure~\ref{F104}. 
\begin{figure}[ht]
\includegraphics[width=0.45\textwidth]{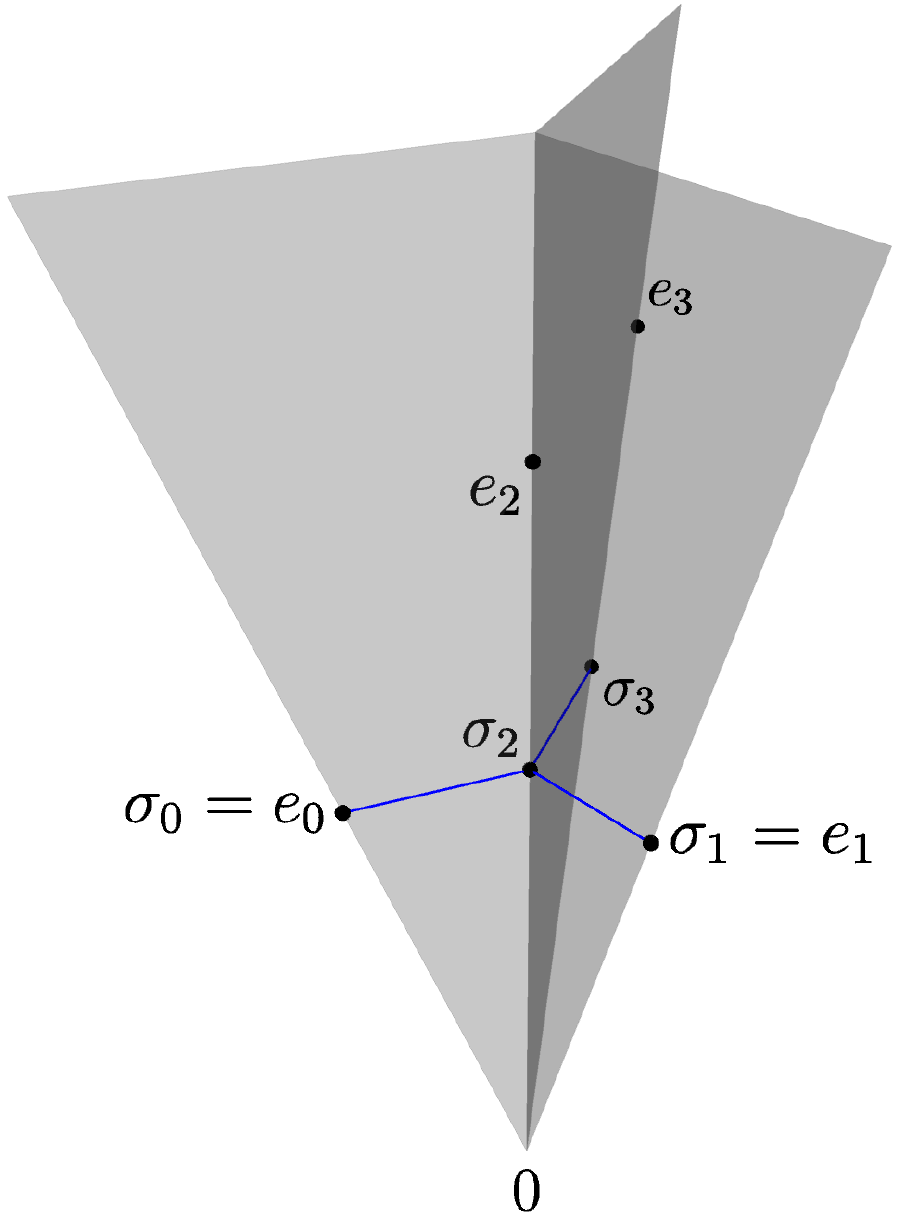}
\hfill 
\includegraphics[width=0.45\textwidth]{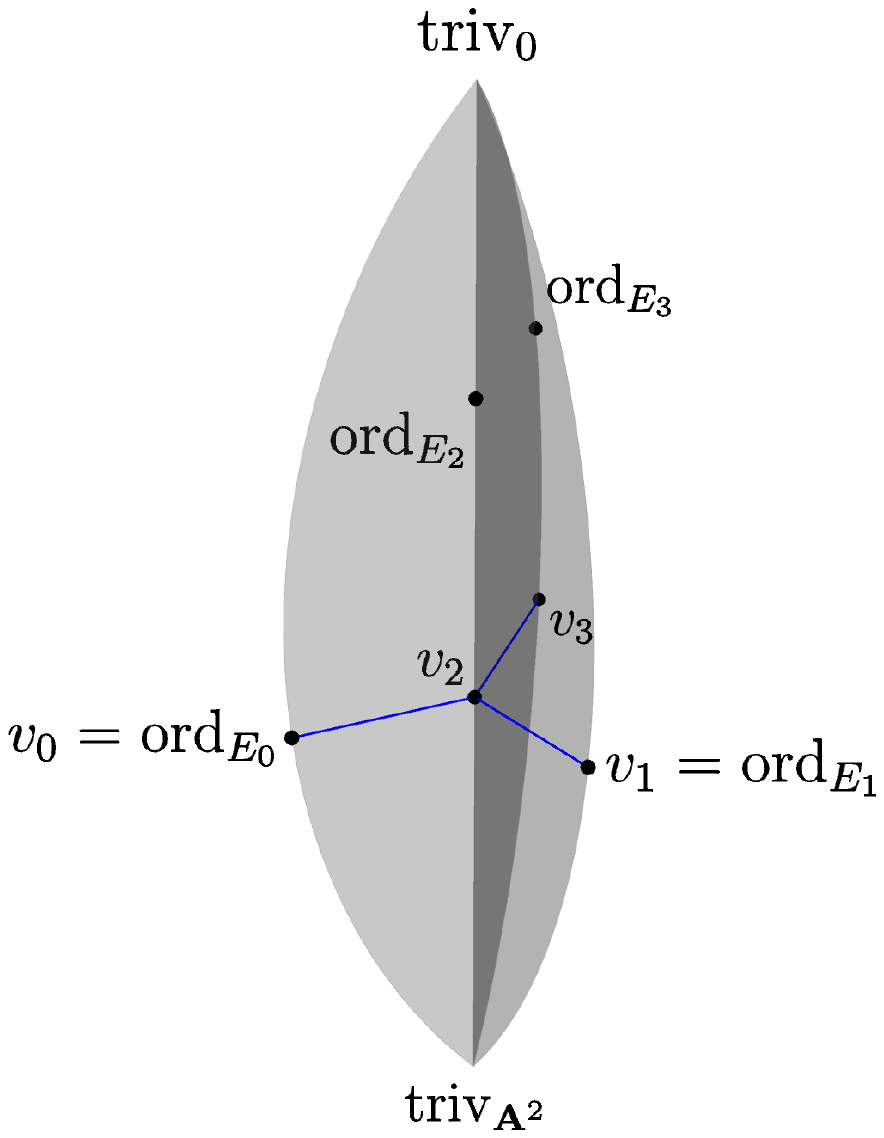}
\caption{The dual fan of a blowup. The picture on the left illustrates
  the dual fan $\hDelta(\pi)$, where $\pi$ is the log resolution
  illustrated in Figure~\ref{F205}. The picture on the left
  illustrates the closure of the embedding of the dual fan inside
  the Berkovich affine plane. The line segments illustrate the dual
  graph $\Delta(\pi)$ and its embedding inside the valuative tree $\cV_0$.
}\label{F104}
\end{figure}
%
%
\subsubsection{Structure theorem}\label{S248}
Because of~\eqref{e205}, the evaluation maps $\ev_\pi$
induce a continuous map
\begin{equation}\label{e106}
  \ev:\cV_0\to\varprojlim_\pi\Delta(\pi),
\end{equation} 
where the right hand side is equipped with the inverse limit topology.
Similarly, the embeddings $\emb_\pi$ define an embedding 
\begin{equation}\label{e206}
  \emb:\varinjlim_\pi\Delta(\pi)\to\cV_0,
\end{equation}  
where the direct limit is defined using the maps  
$\iota_{\pi'\pi}$ and is equipped with the direct limit
topology. 
The direct limit is naturally a dense subset of the inverse limit
and under this identification we have $\ev\circ\emb=\id$.
\begin{Thm}\label{T103}
  The map $\ev:\cV_0\to\varprojlim\Delta(\pi)$ is a homeomorphism.
\end{Thm}
By homogeneity, we also obtain a homeomorphism
$\ev:\hcV_0^*\to\varprojlim\hDelta^*(\pi)$.
\begin{proof}
  Since $r$ is continuous and both sides of~\eqref{e106} are compact,
  it suffices to show that $r$ is bijective.
  The image of $r$ contains the dense subset $\varinjlim\Delta(\pi)$
  so surjectivity is clear. 

  To prove injectivity, pick $v,w\in\cV_0$ with 
  $v\ne w$. Then there exists a primary ideal
  $\fa\subseteq R$ such that $v(\fa)\ne w(\fa)$.
  Let $\pi\in\fB_0$ be a log resolution of $\fa$
  and write $\fa\cdot\cO_{X_\pi}=\cO_{X_\pi}(Z)$,
  where $Z\in\Div(\pi)$. Then
  \begin{equation*}
    \langle\ev_\pi(v),Z\rangle
   =-v(\fa)
    \ne-w(\fa)
    =\ev_\pi(Z)
    =\langle\ev_\pi(w),Z\rangle,
  \end{equation*}
  so that $\ev_\pi(v)\ne\ev_\pi(w)$ and hence $\ev(v)\ne\ev(w)$.
\end{proof}
%
%
\subsubsection{Integral affine structure}
We set
\begin{equation*}
  \Aff(\hcV_0^*)=\varinjlim_\pi\ev_\pi^*\Aff(\pi).
\end{equation*}
Thus a function $\f:\hcV_0^*\to\R$ is integral affine iff it is of the form
$\f=\f_\pi\circ\ev_\pi$, with $\f_\pi\in\Aff(\pi)$.
In other words, $\f$ is defined by an exceptional divisor
in some blowup. 
%
%
%
%
\subsection{Tree structure on $\cV_0$}\label{S244}
Next we use Theorem~\ref{T103} to equip $\cV_0$
with a tree structure.
%
%
\subsubsection{Metric tree structure}\label{S138}
The metric on a dual graph $\Delta(\pi)$ defined in~\S\ref{S135}
turns this space into a finite metric tree in the sense of~\S\ref{S131}.
Further, if $\pi'\ge\pi$, then the embedding 
$\iota_{\pi'\pi}:\Delta(\pi)\hookrightarrow\Delta(\pi')$
is an isometry.
It then follows from the discussion in~\S\ref{S137} that 
$\cV_0\simeq\varprojlim\Delta(\pi)$ is a metric tree.
\begin{Lemma}
  The ends of $\cV_0$ are exactly the valuations that are not
  quasimonomial.
\end{Lemma}
\begin{proof}
  The assertion in the lemma amounts to the ends of 
  the tree $\varprojlim\Delta(\pi)$ being exactly the
  points that do not belong to any single dual graph.
  It is clear that all points of the latter type are ends.
  On the other hand, 
  if $t\in\Delta(\pi)$ for some blowup $\pi$, then there exists a blowup $\pi'\in\fB_0$
  dominating $\pi$ such that $\iota_{\pi'\pi}(t)$ is not an end of 
  $\Delta(\pi')$. When $t$ is already not an endpoint of $\Delta(\pi)$, 
  this is clear. Otherwise $t=b_i^{-1}e_i$, in which case $\pi'$ can be 
  chosen as the blowup of a free point on the associated exceptional 
  prime $E_i$.
\end{proof}
The hyperbolic space $\H\subseteq\cV_0$ induced by the generalized
metric on $\cV_0$ contains all
quasimonomial valuations but also some non-quasimonomial ones,
see~\S\ref{S235}.
%
%
\subsubsection{Rooted tree structure}\label{S240}
We choose the valuation $\ord_0$ as the root of the tree
$\cV_0$ and write $\le$ for the corresponding partial ordering.

The two parametrizations $\a_\pi$ and $A_\pi$ on the dual graph 
$\Delta(\pi)$ in~\S\ref{S253} give rise to parametrizations\footnote{The increasing parametrization $-\a$ is denoted by $\a$ and called \emph{skewness} in~\cite{valtree}. The increasing parametrization $A$ is called \emph{thinness} in \loccit.}
\begin{equation}\label{e143}
  \a:\cV_0\to[-\infty,-1]
  \quad\text{and}\quad
  A:\cV_0\to[2,\infty].
\end{equation}
The parametrization $\a$ gives rise to the generalized metric
on $\cV_0$ and we have 
\begin{equation}\label{e207}
  \a(v)=-(1+d(v,\ord_0)).
\end{equation} 
The choice of parametrization will be justified in~\S\ref{S144}.
Note that hyperbolic space $\H\subseteq\cV_0$ is given by 
$\H=\{\a>-\infty\}$.

There is also a unique, lower semicontinuous 
\emph{multiplicity} function 
\begin{equation*}
  m:\cV_0\to\N\cup\{\infty\}
\end{equation*}
on $\cV_0$ induced by the multiplicity on dual graphs. 
It has the property that $m(w)$ divides $m(v)$ if $w\le v$.
The two parametrizations $\a$ and $A$ are related through the 
multiplicity by
\begin{equation*}
  A(v)=2+\int_{\ord_0}^vm(w)\,d\a(w);
\end{equation*}
this follows from~\eqref{e227}.

There is also a generalized metric induced by $A$, but we shall not use it.
%
%
\subsubsection{Retraction}
It will be convenient to regard the dual graph and fan as subsets 
of the valuation spaces $\cV_0$ and $\hcV_0$, respectively.
To this end, we introduce
\begin{equation*}
  |\Delta(\pi)|:=\emb_\pi(\Delta(\pi))
  \quad\text{and}\quad
  |\hDelta^*(\pi)|:=\emb_\pi(\hDelta^*(\pi)).
\end{equation*}
Note that if $\pi'\ge\pi$, then $|\hDelta^*(\pi)|\subseteq|\hDelta^*(\pi')|$.

The evaluation maps now give rise to \emph{retractions}
\begin{equation*}
  r_\pi:=\emb_\pi\circ\ev_\pi
\end{equation*}
of $\hcV_0^*$ and $\cV_0$ onto $|\Delta^*_0|$ and $|\Delta(\pi)|$,
respectively.
It is not hard to see that $r_\pi'\circ r_\pi=r_\pi$ when $\pi'\ge\pi$.

Let us describe the retraction in more detail.
Let $\xi=c_\pi(v)$ be the center of $v$ on $X_\pi$
and let $E_i$, $i\in J$ be the exceptional primes containing $\xi$.
Write $E_i=(\zeta_i=0)$ in local algebraic coordinates $\zeta_i$ at $\xi$ and set 
$t_i=v(\zeta_i)>0$. 
Then $w:=r_\pi(v)\in|\hDelta^*(\pi)|$ is the monomial valuation
such that $w(\zeta_i)=t_i$, $i\in J$.

It follows from Theorem~\ref{T103} that 
\begin{equation*}
  r_\pi\to\id 
  \quad\text{as $\pi\to\infty$}.
\end{equation*}
In fact, we have the following more precise result.
\begin{Lemma}\label{L125}
  If $v\in\hcV^*_0$ and $\pi\in\fB_0$ is a blowup, then
  \begin{equation*}
    (r_\pi v)(\fa)\le v(\fa)
  \end{equation*}
  for every ideal $\fa\subseteq R$,
  with equality if the strict transform of $\fa$ to $X_\pi$ does
  not vanish at the center of $v$ on $X_\pi$.
  In particular, equality holds if 
  $\fa$ is primary and  $\pi$ is a log resolution of $\fa$.
\end{Lemma}
\begin{proof}
  Pick $v\in\hcV^*_0$ and set $w=r_\pi(v)$. 
  Let $\xi$ be the center of $v$ on $X_\pi$ and 
  $E_i=(\zeta_i=0)$, $i\in J$, the exceptional primes of $\pi$ containing $\xi$.
  By construction, $w$ is the smallest valuation on 
  $\widehat{\cO}_{X_\pi,\xi}$ taking the same values as 
  $v$ on the $\zeta_i$. 
  Thus $w\le v$ on $\widehat{\cO}_{X_\pi,\xi}\supseteq R$,
  which implies $w(\fa)\le v(\fa)$ for all ideals $\fa\subseteq R$. 

  Moreover, if the strict transform of $\fa$ to $X_\pi$ does
  not vanish at $\xi$, then $\fa\cdot\widehat{\cO}_{X_\pi,\xi}$ is 
  generated by a single monomial in the $\zeta_i$, and then it is clear
  that $v(\fa)=w(\fa)$.
\end{proof} 
%
%
%
%
\subsection{Classification of valuations}\label{S169}
Similarly to points in the Berkovich
affine line, we can classify semivaluations in the valuative 
tree into four classes. The classification is 
discussed in detail in~\cite{valtree} but already appears
in a slightly different form in the work of Spivakovsky~\cite{spiv}.
One can show that the set of semivaluations of each of the four types 
below is dense in $\cV_0$, see~\cite[Proposition~5.3]{valtree}.

Recall that any semivaluation $v\in\hcV_0^*$ extends
to the fraction field $F$ of $R$. In particular, it extends to 
the local ring $\cO_0:=\cO_{\A^2,0}$.  Since $v(\fm_0)>0$,
$v$ also defines a semivaluation on the completion $\widehat{\cO}_0$.
%
%
\subsubsection{Curve semivaluations}\label{S237}
The subset 
$\fp:=\{v=\infty\}\subsetneq\widehat{\cO}_0$ is a prime ideal
and $v$ defines a valuation on the quotient ring
$\widehat{\cO}_0/\fp$. If $\fp\ne0$, then $\widehat{\cO}_0/\fp$ is principal
and we say that $v$ is a \emph{curve semivaluation}
as $v(\phi)$ is proportional to the order of vanishing at 0
of the restriction of $\phi$ to the formal curve defined by $\fp$.
A curve semivaluation $v\in\cV_0$ is always an endpoint in
the valuative tree. One can check that they satisfy $\a(v)=-\infty$
and $A(v)=\infty$.
%
%
\subsubsection{Numerical invariants}
Now suppose $v$ defines a \emph{valuation} on $\widehat{\cO}_0$,
that is, $\fp=(0)$.
As in~\S\ref{S133} we associate to $v$ two basic numerical invariants:
the rational rank and the transcendence degree.
It does not make a difference whether we compute these in 
$R$, $\cO_0$ or $\widehat{\cO}_0$. The Abhyankar inequality 
says that 
\begin{equation*}
  \trdeg v+\ratrk v\le 2
\end{equation*}
and equality holds iff $v$ is a quasimonomial valuation.
%
%
\subsubsection{Divisorial valuations}\label{S163}
A valuation $v\in\hcV^*_0$ is 
\emph{divisorial} if it has the numerical invariants 
$\trdeg v=\ratrk v=1$.
In this situation there exists a blowup $\pi\in\fB_0$ 
such that the center of $v$ 
on $X_\pi$ is the generic point of an exceptional prime $E_i$ of $\pi$.
In other words, $v$ belongs to the one-dimensional 
cone $\hsigma_i$ of the dual fan $|\hDelta^*(\pi)|$ and 
$v=t\ord_{E_i}$ for some $t>0$. 
We then set $b(v):=b_i=\ord_{E_i}(\fm_0)$. 

More generally, suppose $v\in\hcV^*_0$ is divisorial and
$\pi\in\fB_0$ is a blowup such that the center of $v$
on $X_\pi$ is a closed point $\xi$. 
Then there exists a blowup $\pi'\in\fB_0$  
dominating $\pi$ in which the (closure of the)
center of $v$ is an exceptional prime of $\pi'$. 
Moreover, by a result of Zariski
(\cf~\cite[Theorem~3.17]{Kollar}),
the birational morphism $X_{\pi'}\to X_\pi$ is an 
isomorphism above $X_\pi\setminus\{\xi\}$ and 
can be constructed by successively blowing
up the center of $v$.

We will need the following result in~\S\ref{S168}.
\begin{Lemma}\label{L123}
  Let $\pi\in\fB_0$ be a blowup and $v\in\hcV_0^*$ a semivaluation.
  Set $w:=r_\pi(v)$.
  \begin{itemize}
  \item[(i)]
    if $v\not\in|\hDelta^*(\pi)|$, then $w$ is necessarily divisorial;
  \item[(ii)]
    if $v\not\in|\hDelta^*(\pi)|$ and $v$ is divisorial, then $b(w)$ divides $b(v)$;
  \item[(iii)]
    if $v\in|\hDelta^*(\pi)|$, then $v$ is divisorial iff it is
    a rational point in the given integral affine structure; 
    in this case, there exists a 
    blowup $\pi'\ge\pi$  such that $|\hDelta^*(\pi')|=|\hDelta^*(\pi)|$
    as subsets of $\hcV_0^*$ and such that $v$ belongs to a one-dimensional
    cone of $|\hDelta^*(\pi')|$;
  \item[(iv)]
    if $v\in|\hDelta^*(\pi)|$ is divisorial and lies in the interior of
    a two-dimensional cone, say $\hsigma_{12}$ of $|\hDelta(\pi)|$,
    then $b(v)\ge b_1+b_2$. 
  \end{itemize}  
\end{Lemma}
\begin{proof}[Sketch of proof]
  For~(i), let $\xi$ be the  common center of 
  $v$ and $w$ on $X_\pi$. If there is a unique exceptional prime 
  $E_1$ containing $\xi$, then it is clear that $w$ is proportional to 
  $\ord_{E_1}$ and hence divisorial. Now suppose $\xi$ is the intersection
  point between two distinct exceptional primes $E_1$ and $E_2$.
  Pick coordinates $\zeta_1$, $\zeta_2$ at $\xi$ such that 
  $E_i=(\zeta_i=0)$ for $i=1,2$.
  If $v(\zeta_1)$ and $v(\zeta_2)$ are rationally independent, then $v$
  gives different values to all monomials
  $\zeta_1^{\b_1}\zeta_2^{\b_2}$, so we must have $v=w$,
  contradicting $v\not\in|\hDelta^*(\pi)|$. 
  Hence $w(\zeta_1)=v(\zeta_1)$ and $w(\zeta_2)=v(\zeta_2)$ are 
  rationally dependent, so $\ratrk w=1$.
  Since $w$ is quasimonomial, it must be divisorial.

  For~(iii), we may assume that the center of $v$ on $X_\pi$ 
  is the intersection point between two distinct exceptional 
  primes $E_1=(\zeta_1=0)$ and $E_2=(\zeta_2=0)$ as above.
  Then $v$ is monomial in coordinates $(\zeta_1,\zeta_2)$ and it is clear
  that $\ratrk v=1$ if $v(\zeta_1)/v(\zeta_2)\in\Q$ and 
  $\ratrk v=2$ otherwise. This proves the first statement.
  Now suppose $v$ is divisorial.
  We can construct $\pi'$ in~(iii) by successively blowing
  up the center of $v$ using the result of Zariski referred to 
  above. Since $v$ is monomial, the center is always a satellite point
  and blowing it up does not change the dual fan, viewed as a 
  subset of $\hcV^*_0$.

  When proving~(ii) we may by~(iii) assume that $w$ belongs to a 
  one-dimensional cone $\hsigma_1$ of $|\hDelta(\pi)|$. 
  Then $b(w)=b_1$. We now successively blow up the 
  center of $v$. This leads to a sequence of divisorial valuations
  $w_0=w,w_1,\dots,w_m=v$. Since the first blowup is at a free point,
  we have $b(w_1)=b_1$ in view of~\eqref{e160}.
  Using~\eqref{e160} and~\eqref{e161} one now shows by induction
  that $b_1$ divides $b(w_j)$ for $j\le m$, concluding the proof of~(ii).

  Finally, in~(iv) we obtain $v$ after a finitely many satellite blowups,
  so the result follows from~\eqref{e161}.
\end{proof}
%
%
\subsubsection{Irrational valuations}\label{S236}
A valuation $v\in\hcV_0^*$ is \emph{irrational} if $\trdeg v=0$, $\ratrk v=2$.
In this case $v$ is not divisorial but still quasimonomial;
it belongs to a dual fan $|\hDelta^*(\pi)|$ for some 
blowup $\pi\in\fB_0$ and for any such $\pi$, $v$ belongs to the interior
of a two-dimensional cone.
%
%
\subsubsection{Infinitely singular valuations}\label{S235}
A valuation $v\in\hcV_0^*$ is \emph{infinitely singular} if it has the 
numerical invariants $\ratrk v=1$, $\trdeg v=0$.
Every infinitely singular valuation in 
the valuative tree $\cV_0$ is an end.
However, some of these ends still belong to hyperbolic space $\H$, 
\begin{Example}\label{E210}
  Consider a sequence $(v_j)_{j=0}^\infty$ defined as follows.
  First, $v_0=\ord_0=\ord_{E_0}$. Then $v_j=b_j^{-1}\ord_{E_j}$ is
  defined inductively as follows: for $j$ odd, $E_j$ is obtained by 
  blowing up a free point on $E_{j-1}$ and for $j$ even, 
  $E_j$ is obtained by blowing up the satellite point 
  $E_{j-1}\cap E_{j-2}$. The sequence $(v_{2j})_{j=0}^\infty$ is increasing
  and converges to an infinitely singular valuation $v$,
  see Figure~\ref{F206}.
  We have $b_{2n}=b_{2n+1}=2^{-n}$, $A(v_{2n})=3-2^{-n}$ and 
  $\a(v_{2n})=-\frac13(5-2^{1-2n})$.
  Thus $\a(v)=-5/3$ and $A(v)=3$. 
  In particular, $v\in\H$.
\end{Example}
For more information on infinitely singular valuations, 
see~\cite[Appendix~A]{valtree}. 
We shall not describe them further here, but they do play a role in dynamics.
\begin{figure}[ht]
 \includegraphics{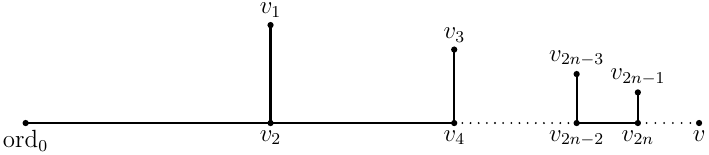}
 \caption{An infinitely singular valuation.
   The divisorial valuation $v_j$ is obtained by performing a sequence
   of $j+1$ blowups, every other free, and every other a satellite blowup.
   The picture is not to scale: we have 
   $d(v_{2n},v_{2n+2})=d(v_{2n+1},v_{2n+2})=2^{-(2n+1)}$ for $n\ge0$.
   Further, $\a(v)=-5/3$, $A(v)=-3$ and $d(\ord_0,v)=2/3$. 
   In particular, $v$ belongs to  hyperbolic space $\H$.
 }\label{F206}
\end{figure}
 %
%
%
%
\subsection{Potential theory}\label{S113}
In~\S\ref{S116} we outlined the first elements of
a potential theory on a general metric tree and
in~\S\ref{S134} we applied this to the 
Berkovich projective line.
 
However, the general theory applied literally to the
valuative tree $\cV_0$  does not quite lead
to a satisfactory notion. The reason is that 
one should really view a function on $\cV_0$ as
the restriction of a homogeneous function on the 
cone $\hcV_0$. In analogy with the 
situation over the complex numbers, one expects that 
for any ideal $\fa\subseteq R$, the function $\log|\fa|$
defined by\footnote{The notation reflects the fact that $|\cdot|:=e^{-v}$ 
  is a seminorm on $R$, see~\eqref{e209}.}
\begin{equation*}
  \log|\fa|(v):=-v(\fa)
\end{equation*}
should be plurisubharmonic on $\hcV_0$.
Indeed, $\log|\fa|$ is a maximum of finitely 
many functions of the form $\log|\phi|$,
where $\phi\in R$ is a polynomial.
As a special case, the function 
$\log|\fm_0|$ should be plurisubharmonic on $\hcV_0$.
This function has a pole (with value $-\infty$) 
at the point $\triv_0$ and so should 
definitely not be pluriharmonic on 
$\hcV_0$. However, it is constantly equal to $-1$ on $\cV_0$,
and so would be harmonic there with the usual definition of the Laplacian.
%
%
\subsubsection{Subharmonic functions and Laplacian on $\cV_0$}\label{S144}
An \textit{ad hoc} solution to the problem above is to
extend the valuative tree $\cV_0$ to a slightly larger tree
$\tcV_0$ by connecting the root $\ord_0$ to 
a ``ground'' point $G\in\tcV_0$ using an interval of length one.
See Figure~\ref{F105}.

\begin{figure}[ht]
  \includegraphics{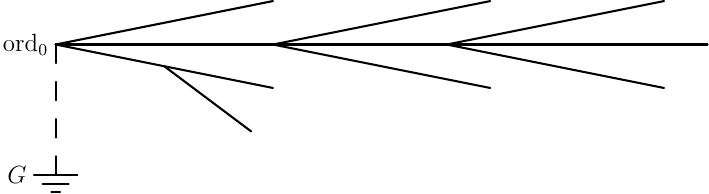}
 \caption{Connecting the valuative tree $\cV_0$ to ``ground'' 
    gives rise to the auxiliary tree $\tcV_0$.}\label{F105}
\end{figure}

Denote the Laplacian on $\tcV_0$ by $\tDelta$.
We define the class $\SH(\cV_0)$ 
of \emph{subharmonic functions}\footnote{If $\f\in\SH(\cV_0)$, then $-\f$
is a \emph{positive tree potential} in the sense of~\cite{valtree}.}
on $\cV_0$ as the set of restrictions to $\cV_0$ of functions
$\f\in\QSH(\tcV_0)$ with
\begin{equation*}
  \f(G)=0
  \quad\text{and}\quad
  \tDelta\f=\rho-a\delta_G,
\end{equation*}
where $\rho$ is a positive measure supported
on $\cV_0$ and $a=\rho(\cV_0)\ge0$. 
In particular, $\f$ is affine
of slope $\f(\ord_0)$ on the segment 
$[G,\ord_0[\,=\tcV_0\setminus\cV_0$.
We then define 
\begin{equation*}
  \Delta\f:=\rho=(\tDelta\f)|_{\cV_0}.
\end{equation*}
For example, if $\f\equiv -1$ on $\cV_0$, 
then $\tDelta\f=\delta_{\ord_0}-\delta_G$
and $\Delta\f=\delta_{\ord_0}$.

From this definition and the analysis in~\S\ref{S116} 
one deduces:
\begin{Prop}\label{P108}
  Let $\f\in\SH(\cV_0)$ and write $\rho=\Delta\f$.
  Then:
  \begin{itemize}
  \item[(i)]
    $\f$ is decreasing in the partial 
    ordering of $\cV_0$ rooted in $\ord_0$;
  \item[(ii)]
    $\f(\ord_0)=-\rho(\cV_0)$;
  \item[(iii)]
    $|D_\vv\f|\le\rho(\cV_0)$ for all
    tangent directions $\vv$ in $\cV_0$.
  \end{itemize}
\end{Prop}
As a consequence we have the estimate
\begin{equation}\label{e144}
  -\a(v)\f(\ord_0)\le\f(v)\le\f(\ord_0)\le 0
\end{equation}
for all $v\in\cV_0$, where $\a:\cV_0\to[-\infty,-1]$
is the parametrization given by~\eqref{e207}.
The exact sequence in~\eqref{e139} shows that 
\begin{equation}\label{e142}
  \Delta:\SH(\cV_0)\to\cM^+(\cV_0),
\end{equation}
is a homeomorphism whose inverse is given by 
\begin{equation}\label{e164}
  \f(v)=\int_{\cV_0}\a(w\wedge_{\ord_0}v)d\rho(w).
\end{equation}
In particular, for any $C>0$, the set 
$\{\f\in\SH(\cV_0)\mid\f(\ord_0)\ge-C\}$
is compact. 
Further, if $(\f_i)_i$ is a decreasing 
net in $\SH(\cV_0)$, and $\f:=\lim\f_i$,
then either $\f_i\equiv-\infty$ on $\cV_0$ 
or $\f\in\SH(\cV_0)$. 
Moreover, if $(\f_i)_i$ is a family in $\SH(\cV_0)$
with $\sup_i\f(\ord_\infty)<\infty$, 
then the upper semicontinuous regularization of 
$\f:=\sup_i\f_i$ belongs to $\SH(\cV_0)$.
%
%
\subsubsection{Subharmonic functions from ideals}\label{S145}
The definitions above may seem arbitrary, but the next result 
justifies them. It shows that the Laplacian is intimately connected to
intersection numbers and shows that the generalized metric on $\cV_0$
is the correct one.
\begin{Prop}\label{P110}
  If $\fa\subseteq R$ is a primary ideal, then 
  the function $\log|\fa|$ on $\cV_0$ is subharmonic.
  Moreover, if $\pi\in\fB_0$ is a log resolution of $\fa$,
  with exceptional primes $E_i$, $i\in I$, 
  and if we write $\fa\cdot\cO_{X_\pi}=\cO_{X_\pi}(Z)$, then
  \begin{equation*}
    \Delta\log|\fa|=\sum_{i\in I}b_i(Z\cdot E_i)\delta_{v_i},
  \end{equation*}
  where $b_i=\ord_{E_i}(\fm_0)$ and 
  $v_i=b_i^{-1}\ord_{E_i}\in\cV_0$.
\end{Prop}
\begin{proof}
  Write $\f=\log|\fa|$. 
  It follows from Lemma~\ref{L125} that 
  $\f=\f\circ r_\pi$, so 
  $\Delta\f$
  is supported on the dual graph $|\Delta(\pi)|\subseteq\cV_0$.
  Moreover, the proof of the same lemma shows that $\f$
  is affine on the interior of each 1-dimensional 
  simplex  so $\Delta\f$ is zero there. Hence it suffices to compute
  the mass of $\Delta\f$ at each $v_i$.

  Note that $\pi$ dominates $\pi_0$, the simple blowup of 0.
  Let $E_0$ be the strict transform of the exceptional divisor of 
  $\pi_0$.  
  Write $\fm_0\cdot\cO_{X_\pi}=\cO_{X_\pi}(Z_0)$,
  where  $Z_0=-\sum_i b_iE_i$. 
  Since $\pi_0$ already is a log resolution of $\fm_0$
  we have 
  \begin{equation}\label{e145}
    (Z_0\cdot E_0)=1
    \quad\text{and}\quad
    (Z_0\cdot E_j)=0,\ j\ne 0.
  \end{equation}
  Fix $i\in I$ and let $E_j$, $j\in J$  be the exceptional primes 
  that intersect $E_i$ properly. 
  First assume $i\ne 0$. 
  Using~\eqref{e145} and $(E_i\cdot E_j)=1$ for $j\in J$ we get 
  \begin{multline*}
    \Delta\f\{v_i\}
    =\sum_{j\in J}\frac{\f(v_j)-\f(v_i)}{d(v_i,v_j)}
    =\sum_{j\in J}b_ib_j(\f(v_j)-\f(v_i))=\\
    =\sum_{j\in J}(b_i\ord_{E_j}(Z)-b_j\ord_{E_i}(Z))(E_i\cdot E_j)=\\
    =b_i(Z\cdot E_i)-\ord_{E_i}(Z)(Z_0\cdot E_i)
    =b_i(Z\cdot E_i).
  \end{multline*}  
  If instead $i=0$, then, by the definition 
  of the Laplacian on $\cV_0\subseteq\tcV_0$, we get
  \begin{multline*}
    \Delta\f\{v_0\}
    =\sum_{j\in J}\frac{\f(v_j)-\f(v_0)}{d(v_0,v_j)}
    +\f(\ord_0)
    =\sum_{j\in J}b_j(\f(v_j)-\f(v_i))
    +\f(\ord_0)=\\
    =\sum_{j\in J}(\ord_{E_j}(Z)-b_j\f(\ord_0))(E_j\cdot E_0)
    +\f(\ord_0)=\\
    =(Z\cdot E_0)-\f(\ord_0)(Z_0\cdot E_0)+\f(\ord_0)
    =(Z\cdot E_0),
  \end{multline*}  
  which completes the proof. (Note that $b_0=1$.)
\end{proof}
\begin{Cor}\label{C110}
  If $v=v_E=b_E^{-1}\ord_E\in\cV_0$ is a divisorial valuation, 
  then there exists a primary ideal $\fa\subseteq R$ 
  such that $\Delta\log|\fa|=b_E\delta_{v_E}$.
\end{Cor}
\begin{proof}
  Let $\pi\in\fB_0$ be a blowup such that
  $E$ is among the exceptional primes $E_i$, $i\in I$.
  As in~\S\ref{S159} above, define 
  $\vE\in\Div(\pi)$ by $(\vE\cdot F)=\delta_{EF}$.
  Thus $\vE$ is relatively nef, so by Proposition~\ref{P202}
  there exists a primary ideal $\fa\subseteq R$ such that 
  $\fa\cdot\cO_{X_\pi}=\cO_{X_\pi}(\vE)$. 
  The result now follows from Proposition~\ref{P110}.
\end{proof}
\begin{Remark}
  One can show that the function $\log|\fa|$ determines
  a primary ideal $\fa$ up to integral closure.
  (This fact is true in any dimension.)
  Furthermore, the product of two integrally closed ideals is
  integrally closed. Corollary~\ref{C110} therefore shows that 
  the assignment $\fa\mapsto\Delta\log|\fa|$ is a semigroup isomorphism
  between integrally closed primary ideals of $R$ and finite
  atomic measures on $\cV_0$ whose mass at a divisorial valuation 
  $v_E$ is an integer divisible by $b_E$.
\end{Remark}
\begin{Cor}\label{C111}
  If $\phi\in R\setminus\{0\}$ is a nonzero polynomial, then the function 
  $\log|\phi|$ on $\cV_0$ is subharmonic.
  More generally, the function $\log|\fa|$ is subharmonic
  for any nonzero ideal $\fa\subseteq R$.
\end{Cor}
\begin{proof}
  For $n\ge 1$, the ideal $\fa_n:=\fa+\fm_0^n$ 
  is primary. Set $\f_n=\log|\fa_n|$.
  Then $\f_n$ decreases pointwise on $\cV_0$ to
  $\f:=\log|\fa|$. Since the $\f_n$ are subharmonic,  so is~$\f$.
\end{proof}
\begin{Exer}\label{E211}
  If $\phi\in\fm_0$ is a nonzero irreducible polynomial, show that 
  \begin{equation*}
    \Delta\log|\phi|=\sum_{j=1}^nm_j\delta_{v_j}
  \end{equation*}
  where $v_j$, $1\le j\le n$ are the curve valuations associated to the 
  local branches $C_j$ of $\{\phi=0\}$ at 0 and 
  where $m_j$ is the multiplicity of $C_j$ at $0$, that is,
  $m_j=\ord_0(\phi_j)$, where $\phi_j\in\hcO_0$ is a local equation of
  $C_j$.
  \textit{Hint} Let $\pi\in\fB_0$ be an embedded resolution of singularities of 
  the curve $C=\{\phi=0\}$. 
\end{Exer}
This exercise confirms that the generalized metric on $\cV_0$ 
is the correct one.

While we shall not use it, we have the following regularization result.
\begin{Thm}\label{T104}
  Any subharmonic function on $\cV_0$ is a decreasing limit of 
  a sequence $(\f_n)_{n\ge 1}$, where 
  $\f_n=c_n\log|\fa_n|$, with $c_n$ a positive rational number and 
  $\fa_n\subseteq R$ a primary ideal.
\end{Thm}
\begin{proof}
  By Theorem~\ref{T202} (applied to the tree $\tcV_0$) any given 
  function $\f\in\SH(\cV_0)$ is the limit of a 
  decreasing sequence $(\f_n)_n$ of functions in $\SH(\cV_0)$
  such that $\Delta\f_n$ is a finite atomic measure supported on 
  quasimonomial valuations.
  Let $\pi_n\in\fB_0$ be a blowup such that $\Delta\f_n$ is 
  supported on the dual graph $|\Delta(\pi_n)|$. Since the divisorial
  valuations are dense in $|\Delta(\pi_n)|$, we may pick 
  $\p_n\in\SH(\cV_0)$ such that 
  $\Delta\p_n$ is a finite atomic
 measure supported on divisorial valuations in $|\Delta(\pi_n)|$,
 with rational weights, such that $|\p_n-\f_n|\le2^{-n}$ on $\cV_0$.
  The sequence $(\p_n+3\cdot2^{-n})_{n\ge1}$ is then decreasing and 
  $\Delta(\p_n+3\cdot2^{-n})=\Delta\p_n+3\cdot2^{-n}\delta_{\ord_0}$
  is a finite atomic measure supported on divisorial valuations in $|\Delta(\pi_n)|$,
  with rational weights. The result now follows from Corollary~\ref{C110}.
\end{proof}
Regularization results such as Theorem~\ref{T104} play an important
role in higher dimensions, but the above proof, which uses tree
arguments together with Lipman's result in Proposition~\ref{P202},
does not generalize. Instead, one can construct the ideals $\fa_n$
as \emph{valuative multiplier ideals}. This is done in~\cite{valmul}
in dimension two, and in~\cite{hiro} in higher dimensions.
%
%
%
%
\subsection{Intrinsic description of the tree structure on $\cV_0$}\label{S207}
As explained in~\S\ref{S244}, the valuative tree inherits
a partial ordering and a (generalized) metric from the dual graphs. 
We now describe these two structures intrinsically, 
using the definition of elements in $\cV_0$ as functions on $R$.
The potential theory in~\S\ref{S113} is quite useful for this purpose.
%
%
\subsubsection{Partial ordering}
The following result
gives an intrinsic description of the partial ordering on $\cV_0$.
\begin{Prop}\label{P109}
  If $w, v\in\cV_0$, then the following are equivalent:
  \begin{itemize}
  \item[(i)]
    $v\le w$ in the partial ordering induced by 
    $\cV_0\simeq\varprojlim\Delta(\pi)$;
  \item[(ii)]
    $v(\phi)\le w(\phi)$ for all polynomials $\phi\in R$;
  \item[(iii)]
    $v(\fa)\le w(\fa)$ for all primary ideals $\fa\subseteq R$.
  \end{itemize}
\end{Prop}
\begin{proof}
  The implication~(i)$\implies$(ii) is a consequence of 
  Proposition~\ref{P108} and the fact that $\log|\phi|$ is subharmonic.
  That~(ii) implies~(iii) is obvious.
  It remains to prove that~(iii) implies~(i). 
  Suppose that $v\not\le w$ in the sense of~(i). 
  After replacing $v$ and $w$ by $r_\pi(v)$ and $r_\pi(w)$,
  respectively, for a sufficiently large $\pi$, we may assume that 
  $v,w\in|\Delta(\pi)|$. 
  Set $v':=v\wedge w$. Then $v'<v$, $v'\le w$
  and $\,]v',v]\cap[v',w]=\emptyset$.
  Replacing $v$ by a divisorial valuation in $]v',v]$
  we may assume that $v$ is divisorial.
  By Corollary~\ref{C110} we can find an ideal $\fa\subseteq R$
  such that $\Delta\log|\fa|$ is supported at $v$.
  Then $w(\fa)=v'(\fa)<v(\fa)$, so~(iii) does not hold.
\end{proof}
%
%
\subsubsection{Integral affine structure}
Next we give an intrinsic description of the 
integral affine structure.
\begin{Prop}\label{P113}
  If $\pi\in\fB_0$ is a blowup, then 
  a function $\f:\hcV_0\to\R$ belongs to $\Aff(\pi)$ iff
  it is of the form 
  $\f=\log|\fa|-\log|\fb|$, where
  $\fa$ and $\fb$ are primary ideals of $R$
  for which $\pi$ is a common log resolution.
\end{Prop}
\begin{proof}[Sketch of proof]
  After unwinding definitions
  this boils down to the fact that 
  any exceptional divisor can be
  written as the difference of two relatively nef divisors.
  Indeed, by Proposition~\ref{P202}, if $Z$ is
  relatively nef, then there exists a primary
  ideal $\fa\subseteq R$ such that  $\fa\cdot\cO_{X_\pi}=\cO_{X_\pi}(Z)$.
\end{proof} 
\begin{Cor}\label{C112}
  A function $\f:\hcV^*_0\to\R$ is integral affine iff
  it is of the form 
  $\f=\log|\fa|-\log|\fb|$, where
  $\fa$ and $\fb$ are primary ideals in $R$.
\end{Cor}
%
%
\subsubsection{Metric}
Recall the parametrization $\a$ of $\cV_0\simeq\varprojlim\Delta(\pi)$
given by~\eqref{e207}.
\begin{Prop}\label{P106}
  For any $v\in\cV_0$ we have 
  \begin{equation*}
    \a(v)
    =-\sup\left\{\frac{v(\phi)}{\ord_0(\phi)}\ \bigg|\ 
      p\in\fm_0\right\}
    =-\sup\left\{\frac{v(\fa)}{\ord_0(\fa)}\ \bigg|\ 
      \fa\subseteq R\ \fm_0\text{-primary}\right\}
  \end{equation*}
  and the suprema are attained when $v$ is quasimonomial.
\end{Prop}
In fact, one can show that supremum in the second equality is 
attained \emph{only} if $v$ is quasimonomial. Further, the supremum in the first 
equality is never attained when $v$ is infinitely singular, but is
attained if $v$ is a curve semivaluation (in which case $\a(v)=-\infty$), 
and we allow $\phi\in\fm_0\cdot\widehat{\cO}_0$.
\begin{proof}
  Since the functions $\log|\fa|$ and $\log|\phi|$ are 
  subharmonic,~\eqref{e144} shows that 
  $v(\fa)\le-\a(v)\ord_0(\fa)$ and 
  $v(\phi)\le-\a(v)\ord_0(\phi)$
  for all $\fa$ and all $\phi$.

  Let us prove that equality can be achieved when 
  $v$ is quasimonomial. Pick a blowup $\pi\in\fB_0$
  such that $v\in|\Delta(\pi)|$ and pick $w\in|\Delta(\pi)|$
  divisorial with $w\ge v$. By Corollary~\ref{C110}
  there exists a primary ideal $\fa$ such that 
  $\Delta\log|\fa|$ is supported at $w$. This implies
  that the function $\log|\fa|$ is affine with slope 
  $-\ord_0(\fa)$ on the segment $[\ord_0,w]$. 
  In particular, $v(\fa)=-\a(v)\ord_0(\fa)$.
  By picking $\phi$ as a general element in $\fa$ we also
  get $v(\phi)=-\a(v)\ord_0(\phi)$. 

  The case of a general $v\in\cV_0$ follows from what precedes,
  given that $r_\pi v(\phi)$, $r_\pi v(\fa)$ and $\a(r_\pi(v))$ converge 
  to $v(\phi)$, $v(\fa)$ and $\a(v)$, respectively, as $\pi\to\infty$.
\end{proof}
Notice that Proposition~\ref{P106} gives a very precise 
version of the Izumi-Tougeron inequality~\eqref{e131}.
Indeed, $\a(v)>-\infty$ for all quasimonomial valuations $v\in\cV_0$.
%
%
\subsubsection{Multiplicity}\label{S275}
The multiplicity function 
$m:\cV_0\to\N\cup\{\infty\}$ can also
be characterized intrinsically.
For this, one first notes that if $v=v_C$ is a curve semivaluation,
defined by a formal curve $C$, then $m(v)=\ord_0(C)$.
More generally, one can show that 
\begin{equation*}
  m(v)=\min\{m(C)\mid v\le v_C\}.
\end{equation*}
In particular, $m(v)=\infty$ iff $v$ cannot be dominated by a curve 
semivaluation, which in turn is the case iff $v$ is infinitely singular.
%
%
\subsubsection{Topology}\label{S162}
Theorem~\ref{T103} shows that the topology on $\cV_0$ induced from 
$\BerkAtwo$ coincides with the tree topology 
on $\cV_0\simeq\varprojlim\Delta(\pi)$.
It is also possible to give a more geometric description.

For this, consider a blowup $\pi\in\fB_0$ and a \emph{closed}
point $\xi\in\pi^{-1}(0)$. Define $U(\xi)\subseteq\cV_0$ as the set of
semivaluations having center $\xi$ on $X_\pi$. 
This means precisely that $v(\fm_\xi)>0$, where 
$\fm_\xi$ is the maximal ideal of the local ring $\cO_{X_\pi,\xi}$.
Thus $U(\xi)$ is open in $\cV_0$.
One can in fact show that these sets $U(\xi)$ generate the topology
on $\cV_0$.

If $\xi$ is a \emph{free} point, belonging to a unique 
exceptional prime $E$ of $X_\pi$, then we have 
$U(\xi)=U(\vv)$ for a tangent direction $\vv$ at $v_E$ in 
$\cV_0$, namely, the tangent direction for which 
$\ord_\xi\in U(\vv)$. 
As a consequence, the open set $U(\xi)$ is connected
and its boundary is a single point: $\partial U(\xi)=\{v_E\}$.
%
%
%
%
\subsection{Relationship to the Berkovich unit disc}\label{S132}
Let us briefly sketch how to relate
the valuative tree with the Berkovich unit disc.
Fix global coordinates $(z_1,z_2)$ on $\A^2$ vanishing at 0
and let $L=K((z_1))$ be the field of Laurent series in $z_1$. 
There is a unique extension of the trivial valuation on $K$ to a
valuation $v_L$ on $L$ for which $v_L(z_1)=1$. 
The Berkovich open unit disc over $L$ is the set 
of semivaluations $v:L[z_2]\to\R_+$, extending $v_L$,
for which $v(z_2)>0$.
If $v$ is such a semivaluation, then $v/\min\{1, v(z_2)\}$ 
is an element in the valuative tree $\cV_0$. Conversely, 
if $v\in\cV_0$ is not equal to the curve semivaluation 
$v_C$ associated to the curve $(z_1=0)$, then 
$v/ v(z_1)$ defines an element in the Berkovich open
unit disc over $L$.

Even though $L$ is not algebraically closed, the classification 
of the points in the Berkovich affine line into Type~1-4 points
still carries over, see~\S\ref{S273}. Curve valuations become Type~1 points, 
divisorial valuations become Type~2 points and irrational valuations
become Type~3 points. An infinitely singular valuation $v\in\cV_0$
is of Type~4 or Type~1, depending on whether the log discrepancy $A(v)$
is finite or infinite.
The parametrization and partial orderings on $\cV_0$
and the Berkovich unit disc are related, but different.
See~\cite[\S3.9, \S4.5]{valtree} for more details.

Note that the identification of the valuative tree with the 
Berkovich unit disc depends on a choice of coordinates.
In the study of polynomial dynamics in~\S\ref{S107},
it would usually not be natural to fix coordinates. 
The one exception to this is when studying the dynamics
of a skew product 
\begin{equation*}
  f(z_1,z_2)=(\phi(z_1),\psi(z_1,z_2)),
\end{equation*}
with $\phi(0)=0$, in a neighborhood of the invariant line $z_1=0$.
However, it will be more efficient to study general polynomial mappings 
in two variables using the Berkovich affine plane over the
trivially valued field $K$.

As noted in~\S\ref{S139}, the Berkovich unit disc over the 
field $K((z_1))$ of Laurent series is in fact more naturally 
identified with the space $\cVe{C}$, where $C=\{z_1=0\}$.
%
%
%
%
\subsection{Other ground fields}\label{S230}
Let us briefly comment on the case when the field $K$
is not algebraically closed.

Let $K^a$ denote the 
algebraic closure and $G=\Gal(K^a/K)$ the Galois group.
Using general theory we have an identification 
$\BerkAtwo(K)\simeq\BerkAtwo(K^a)/G$.

First suppose that the closed point $0\in\A^2(K)$ is
$K$-rational, that is, $\cO_0/\fm_0\simeq K$. Then $0$ 
has a unique preimage $0\in\A^2(K^a)$.
Let $\cV_0(K^a)\subseteq\BerkAtwo(K^a)$ denote the 
valuative tree at $0\in\A^2(K^a)$. 
Every $g\in G$ induces an automorphism of $\BerkAtwo(K^a)$
that leaves $\cV_0(K^a)$ invariant.
In fact, one checks that $g$ preserves the partial ordering 
as well as the parametrizations $\a$ and $A$ and the 
multiplicity $m$. 
Therefore, the quotient $\cV_0(K)\simeq\cV_0(K^a)$
also is naturally a tree. 
As in~\S\ref{S273} we define a parametrization $\a$ of $\cV_0(K)$
using the corresponding parametrization of $\cV_0(K^a)$ and
the degree of the map $\cV_0(K^a)\to\cV_0(K)$. 
This parametrization gives rise to the correct generalized
metric in the sense that the analogue of Exercise~\ref{E211} holds.

When the closed point $0$ is not $K$-rational, it has 
finitely many preimages $0_j\in\A^2(K^a)$. At each $0_j$
we have a valuative tree $\cV_{0_j}\subseteq\BerkAtwo(K^a)$
and $\cV_0$, which is now the quotient of the disjoint union of
the $\cV_{0_j}$ by $G$, still has a natural metric tree structure.

In fact, even when $K$ is not algebraically closed, we can 
analyze the valuative tree using blowups and dual graphs
much as we have done above. One thing to watch out for,
however, is that the intersection form on $\Div(\pi)$ is
no longer unimodular. Further, when $E_i$, $E_j$ are 
exceptional primes intersecting properly, it is
no longer true that $(E_i\cdot E_j)=1$. In order to
get the correct metric on the valuative tree,
so that Proposition~\ref{P110} holds for instance,
we must take into account the degree over $K$
of the residue field whenever we blow up a closed point $\xi$.
The resulting metric is the same as the one obtained
above using the Galois action.
%
%
%
%
\subsection{Notes and further references}\label{S241}
The valuative tree was introduced and studied extensively in the
monograph~\cite{valtree} by Favre and myself. One of our original
motivations was in fact to study superattracting fixed points,
but it turned out that while valuations on surfaces had been
classified by Spivakovsky, the structure of this valuation 
space had not been explored.

It was not remarked in~\cite{valtree} that the valuative tree 
can be viewed as a subset of the Berkovich affine plane
over a trivially valued field. The connection that was made was
with the Berkovich unit disc over the field of Laurent series.

In~\cite{valtree}, several approaches to the valuative tree are
pursued. The first approach is algebraic, using 
\emph{key polynomials} as developed 
by MacLane~\cite{MacLane}. While beautiful, this method is coordinate
dependent and involves some quite delicate combinatorics.
In addition, even though there is a notion
of key polynomials in higher dimensions~\cite{Vaquie2},
these seem hard to use for our purposes.

The geometric approach, using blowups and dual graphs is
also considered in~\cite{valtree} but perhaps not
emphasized as much as here. As already mentioned, this
approach can be partially generalized to higher dimensions,
see~\cite{hiro}, where it is still true that the
valuation space $\cV_0$ is an inverse limit of dual graphs. 
The analogue of the Laplace operator on $\cV_0$ is then a
nonlinear Monge-Amp\`ere operator, but this operator is
defined geometrically, using intersection theory,
rather than through the simplicial structure of the space.
In higher dimensions, the relation between the different
positivity notions on exceptional divisors is much more
subtle than in dimension two.
Specifically, Proposition~\ref{P202} is no longer true.

Granja~\cite{Granja} has generalized the construction 
of the valuative tree to a general two-dimensional regular local ring.

The valuative tree gives an efficient way to encode singularities in
two dimensions. For example, it can be used to study the singularities
of planar plurisubharmonic functions, see~\cite{pshsing,valmul}.
It is also related to many other constructions in singularity theory.
We shall not discuss this further here, but refer to the 
paper~\cite{Popescu-Pampu} by Popescu-Pampu for further references.
In this paper, the author, defines an interesting object, the 
\emph{kite} (cerf-volant), which also encodes the combinatorics 
of the exceptional primes of a blowup.

In order to keep notes reasonably coherent, and in order to reflect
changing trends, I have taken the freedom to change some of the
notation and terminology from~\cite{valtree}. 
Notably, in~\cite{valtree}, the valuative tree is simply denoted $\cV$
and its elements are called valuations. Here we wanted
to be more precise, so we call them semivaluations.
What is called subharmonic functions here correspond to 
positive tree potentials in~\cite{valtree}. 
The valuation $\ord_0$ is called $\nu_\fm$ in~\cite{valtree}.
%
%
%
%
%
%
\newpage
\section{Local plane polynomial dynamics}\label{S107}
Next we will see how the valuative tree can be used to study
superattracting fixed points for polynomial maps of $\A^2$. 
%
%
%
%
\subsection{Setup}
Let $K$ be an algebraically closed field, equipped with the trivial valuation.
(See~\S\ref{S271} for the case of other ground fields.)
Further, $R$ and $F$ are the coordinate ring and
function field of the affine plane $\A^2$ over $K$.
Recall that the Berkovich affine plane $\BerkAtwo$ 
is the set of semivaluations on $R$ that restrict to the trivial valuation
on $K$.
%
%
%
%
\subsection{Definitions and results}
We briefly recall the setup from~\S\ref{S109} of the introduction.
Let $K$ be an algebraically closed field of characteristic zero.
Consider a polynomial mapping 
$f:\A^2\to\A^2$ over $K$.
We assume that $f$ is \emph{dominant}, since otherwise
the image of $f$ is contained in a curve.
Consider a (closed) fixed point $0=f(0)\in\A^2$ and define
\begin{equation*}
  c(f):=\ord_0(f^*\fm_0),
\end{equation*}
where $\fm_0$ denotes the maximal ideal at $0$.
We say that $f$ is \emph{superattracting} if $c(f^n)>1$ 
for some $n\ge 1$.
\begin{Exer}
  Show that if $f$ is superattracting, then in fact
  $c(f^2)>1$. On the other hand, find an example 
  of a superattracting $f$ for which $c(f)=1$.
\end{Exer}
\begin{Exer}
  Show that if $f$ is superattracting and $K=\C$, 
  then there exists a neighborhood $0\in U\subseteq\A^2$ 
  (in the usual Euclidean topology)
  such that $f(U)\subseteq U$, and $f^n(z)\to 0$ as $n\to\infty$ 
  for any $z\in U$.
\end{Exer}

As mentioned in the introduction, the sequence $(c(f^n))_{n\ge1}$
is supermultiplicative, so the limit 
\begin{equation*}
  c_\infty(f):=\lim_{n\to\infty}c(f^n)^{1/n}=\sup_{n\to\infty}c(f^n)^{1/n}
\end{equation*}
exists.
\begin{Exer}
  Verify these statements!
  Also show that $f$ is superattracting iff
  $c_\infty(f)>1$ iff $df_0$ is nilpotent.
\end{Exer}
\begin{Exer}
  In coordinates $(z_2,z_2)$ on $\A^2$, let $f_c$ be the homogeneous
  part of $f$ of degree $c=c(f)$. 
  Show that if $f_c^2\not\equiv0$, then in fact $f_c^n\ne0$ for all $n\ge 1$,
  so that $c(f^n)=c^n$ and $c_\infty=c=c(f)$ is an integer.
\end{Exer}
\begin{Example}
  If $f(z_1,z_2)=(z_2,z_1z_2)$, then $c(f^n)$ is the $(n+2)$th Fibonacci number
  and $c_\infty=\frac12(\sqrt{5}+1)$ is the golden mean.
\end{Example}
For the convenience of the reader, we recall the result that we are
aiming for:
\begin{ThmB}
  The number $c_\infty=c_\infty(f)$ is a quadratic integer: there exists
  $a,b\in\Z$ such that $c_\infty^2=ac_\infty+b$. 
  Moreover, there exists a constant 
  $\delta>0$ such that 
  \begin{equation*}
    \delta c_\infty^n\le c(f^n)\le c_\infty^n
  \end{equation*}
  for all $n\ge 1$.
\end{ThmB}
Here it is the left-hand inequality that is nontrivial. 
%
%
%
%
\subsection{Induced map on the Berkovich affine plane}
As outlined in~\S\ref{S109}, we approach Theorem~B 
by studying the induced map
\begin{equation*}
  f:\BerkAtwo\to\BerkAtwo
\end{equation*}
on the Berkovich affine plane $\BerkAtwo$. Recall the subspaces 
\begin{equation*}
  \cV_0\subseteq\hcV_0^*\subseteq\hcV_0\subseteq\BerkAtwo
\end{equation*}
introduced in~\S\ref{S103}: $\hcV_0$ is the set of semivaluations
whose center on $\A^2$ is the point $0$. It has the structure of
a cone over the valuative tree $\cV_0$, with apex at $\triv_0$. 
It is clear that 
\begin{equation*}
  f(\hcV_0)\subseteq\hcV_0
  \quad\text{and}\quad
  f(\triv_0)=\triv_0.
\end{equation*}
In general, $f$ does not map the pointed cone
$\hcV_0^*$ into itself. Indeed, suppose there exists
an algebraic curve $C=\{\phi=0\}\subseteq\A^2$
passing through $0$ and contracted
to $0$ by $f$. Then any semivaluation
$v\in\hcV_0^*$ such that $v(\phi)=\infty$
satisfies $f(v)=\triv_0$.
To rule out this behavior, we introduce
\begin{Assum}
  From now on, and until~\S\ref{S128} we assume
  that the germ $f$ is \emph{finite}.
\end{Assum}
This assumption means that the ideal $f^*\fm_0\subseteq\cO_0$ 
is primary, that is, $\fm_0^s\subseteq f^*\fm_0$ 
for some $s\ge 1$,
so it exactly rules out the existence of contracted curves.
Certain modifications are required to handle the
more general case when $f$ is merely dominant.
See~\S\ref{S128} for some of this.

The finiteness assumption implies that 
$f^{-1}\{\triv_0\}=\{\triv_0\}$. Thus we obtain a
well-defined map
\begin{equation*}
  f:\hcV_0^*\to\hcV_0^*,
\end{equation*}
which is clearly continuous and homogeneous.

While $f$ preserves $\hcV_0^*$, it does not preserve the 
``section'' $\cV_0\subseteq\hcV_0^*$ given by the condition
$v(\fm_0)=1$.
Indeed, if $v(\fm_0)=1$, there is no reason why $f(v)(\fm_0)=1$.
Rather, we define 
\begin{equation*}
  c(f, v):= v(f^*\fm_0)
  \quad\text{and}\quad
  f_\bullet v:=\frac{f(v)}{c(f, v)}. 
\end{equation*}
The assumption that $f$ is finite at 0 is equivalent to the existence
of a constant $C>0$ such that $1\le c(f, v)\le C$ for all $v\in\cV_0$.
Indeed, we can pick $C$ as any integer $s$ 
such that $f^*\fm_0\supseteq\fm_0^s$.
Also note that 
\begin{equation*}
  c(f)=c(f,\ord_0).
\end{equation*}
The normalization factors $c(f, v)$ naturally define a
dynamical \emph{cocycle}. Namely, we can look at 
$c(f^n, v)$ for every $n\ge 0$ and $v\in\cV_0$
and we then have
\begin{equation*}
  c(f^n, v)=\prod_{i=0}^{n-1}c(f, v_i),
\end{equation*}
where $v_i=f^i_\bullet v$ for $0\le i<n$.

Apply this equality to $v=\ord_0$.
By definition, we have $v_i=f^i_\bullet\ord_0\ge\ord_0$
for all $i$. This gives $c(f, v_i)\ge c(f,\ord_0)=c(f)$,
and hence $c(f^n)\ge c(f)^n$, as we already knew.
More importantly, 
we shall use the \emph{multiplicative cocycle} $c(f^n, v)$ in order to 
study the \emph{supermultiplicative sequence} $(c(f^n))_{n\ge0}$.

%
%
%
%
\subsection{Fixed points on dual graphs}\label{S168}
Consider a blowup $\pi\in\fB_0$.
We have seen that the dual graph of $\pi$ embeds as
a subspace $|\Delta(\pi)|\subseteq\cV_0$ of the valuative tree, 
and that there is a retraction $r_\pi:\cV_0\to|\Delta(\pi)|$.
We shall study the selfmap
\begin{equation*}
  r_\pi f_\bullet:|\Delta(\pi)|\to|\Delta(\pi)|.
\end{equation*}
Notice that this map is continuous since $r_\pi$ and $f_\bullet$ are.
Despite appearances, 
it does not really define an induced dynamical system on $|\Delta(\pi)|$,
as, in general, we may have 
$(r_\pi f_\bullet)^2\ne r_\pi f_\bullet^2$.
However, the fixed points of $r_\pi f_\bullet$ will
play an important role.

It is easy to see that a continuous selfmap of a finite simplicial tree
always has a fixed point. (See also~Proposition~\ref{P119}.)
Hence we can find $v_0\in|\Delta(\pi)|$
such that $r_\pi f_\bullet v_0= v_0$. There are then three possibilities:
\begin{itemize}
\item[(1)]
  $v_0$ is divisorial and $f_\bullet v_0= v_0$;
\item[(2)]
  $v_0$ is divisorial and $f_\bullet v_0\ne v_0$;
\item[(3)]
  $v_0$ is irrational and $f_\bullet v_0= v_0$.
\end{itemize}
Indeed, if $v\in\cV_0\setminus|\Delta(\pi)|$ is any valuation, 
then $r_\pi(v)$ is divisorial, see Lemma~\ref{L123}.
The same lemma also allows us to assume,
in cases~(1) and~(2), that the center of $v_0$ on $X_\pi$
is an exceptional prime $E\subseteq X_\pi$.

In case~(2), this means that the center of $f_\bullet v_0$ on $X_\pi$
is a \emph{free} point $\xi\in E$, that is, a point that does not belong to 
any other exceptional prime of $\pi$.
%
%
%
%
\subsection{Proof of Theorem~B}\label{S148}
Using the fixed point $v_0$ that we just constructed, and still assuming 
$f$ finite, we can now prove Theorem~B. 

The proof that $c_\infty$ is a quadratic 
integer relies on a calculation using value groups. 
Recall that the value group of a valuation $v$ is 
defined as $\Gamma_v= v(F)$, where $F$ is the fraction field of $R$.
\begin{Lemma}
  In the notation above, we have 
  $c(f, v_0)\Gamma_{v_0}\subseteq\Gamma_{v_0}$.
  As a consequence, $c(f, v_0)$ is a quadratic integer.
\end{Lemma}
We shall see that under suitable assumptions on the blowup
$\pi$ we have $c(f, v_0)=c_\infty(f)$. This will show that 
$c_\infty(f)$ is a quadratic integer.
\begin{proof}
  In general, $\Gamma_{f(v)}\subseteq\Gamma_v$
  and $\Gamma_{r_\pi(v)}\subseteq\Gamma_v$ for 
  $v\in\hcV_0^*$. 
  If we write $c_0=c(f, v_0)$, then this leads to 
  \begin{equation*}
    c_0\Gamma_{v_0}
    =c_0\Gamma_{r_\pi f_\bullet v_0}
    \subseteq c_0\Gamma_{f_\bullet v_0}
    =\Gamma_{c_0f_\bullet v_0}
    =\Gamma_{f(v_0)}
    \subseteq\Gamma_{v_0},
  \end{equation*}
  which proves the first part of the lemma.  
  
  Now $v_0$ is quasimonomial, so the structure of 
  its value group is given by~\eqref{e125}.
  When $v_0$ is divisorial, 
  $\Gamma_{v_0}\simeq\Z$ and the inclusion
  $c_0\Gamma_{v_0}\subseteq\Gamma_{v_0}$
  immediately implies that $c_0$ is an integer.
  If instead $v_0$ is irrational, 
  $\Gamma_{v_0}\simeq\Z\oplus\Z$ 
  and $c_0$ is a quadratic integer.
  Indeed, if we write 
  $\Gamma_{v_0}=t_1\Z\oplus t_2\Z$,
  then there exist integers $a_{ij}$ such that 
  $c_0t_i=\sum_{j=1}^2a_{ij}t_j$ for $i=1,2$.
  But then $c_0$ is an eigenvalue of the matrix $(a_{ij})$,
  hence a quadratic integer.
\end{proof}
It remains to be seen that $c(f, v_0)=c_\infty(f)$ and that the 
estimates in Theorem~B hold.
We first consider cases~(1) and~(3) above, so that $f_\bullet v_0= v_0$.
It follows from~\eqref{e144} that the valuations $v_0$ and $\ord_0$ 
are comparable. More precisely, 
$\ord_0\le v_0\le-\alpha_0\ord_0$, where $\alpha_0=\alpha(v_0)$.
The condition $f_\bullet v_0= v_0$ means that $f(v_0)=c v_0$,
where $c=c(f, v_0)$. This leads to
\begin{equation*}
  c(f^n)
  =\ord_0(f^{n*}\fm_0)
  \le v_0(f^{n*}\fm_0)
  =(f^n_* v_0)(\fm_0)
  =c^n v_0(\fm_0)
  =c^n
\end{equation*}
and, similarly, $c^n\le-\alpha_0c(f^n)$.
In view of the definition of $c_\infty$, this implies that 
$c_\infty=c$, so that 
\begin{equation*}
  f(v_0)=c_\infty v_0
  \quad\text{and}\quad
  -\alpha_0^{-1}c_\infty^n\le c(f^n)\le c_\infty^n,
\end{equation*}
proving Theorem~B in this case.

Case~(2) is more delicate and is in some sense the typical case.
Indeed, note that we have not made any restriction on 
the modification $\pi$. For instance, $\pi$ could be a simple blowup
of the origin. In this case $|\Delta(\pi)|=\{\ord_0\}$ is a singleton,
so $v_0=\ord_0$ but there is no reason why $f_\bullet\ord_0=\ord_0$.
To avoid this problem, we make
\begin{Assum}
  The map $\pi:X_\pi\to\A^2$ defines a log resolution of the 
  ideal $f^*\fm$. In other words, the ideal sheaf 
  $f^*\fm\cdot\cO_{X_\pi}$ is locally principal.
\end{Assum}
Such a $\pi$ exists by resolution of singularities. Indeed our 
current assumption
that $f$ be a \emph{finite} germ implies that $f^*\fm$ is an 
$\fm$-primary ideal.

For us, the main consequence of $\pi$ being a log resolution of $f^*\fm$
is that 
\begin{equation*}
  c(v)= v(f^*\fm_0)=(r_\pi v)(f^*\fm_0)=c(r_\pi v)
\end{equation*}
for all $v\in\cV_0$, see Lemma~\ref{L125}.

As noted above, we may assume that the center of $v_0$
on $X_\pi$ is an exceptional prime $E$. Similarly, 
the center of $f_\bullet v_0$ on $X_\pi$ is a free point $\xi\in E$.
Let $U(\xi)$ be the set of all valuations $v\in\cV_0$ whose center
on $X_\pi$ is the point $\xi$. 
By~\S\ref{S162}, this is a connected open set and its closure is
given by $\overline{U(\xi)}=U(\xi)\cup\{ v_0\}$.
We have $r_\pi U(\xi)=\{ v_0\}$, so $c(f, v)=c(f, v_0)$
for all $v\in U(\xi)$ by Lemma~\ref{L125}.

We claim that $f_\bullet(\overline{U(\xi)})\subseteq U(\xi)$. 
To see this, we could use~\S\ref{S130}
but let us give a direct argument.
Note that $v\ge v_0$, and hence 
$f(v)\ge f(v_0)$ for all $v\in\overline{U(\xi)}$. 
Since $c(f, v)=c(f, v_0)$, this implies
$f_\bullet v\ge f_\bullet v_0> v_0$ 
for all $v\in\overline{U(\xi)}$. In particular, 
$f_\bullet v\ne v_0$ for all $v\in\overline{U(\xi)}$,
so that 
\begin{equation*}
  \overline{U(\xi)}\cap f_\bullet^{-1}U(\xi)
  =\overline{U(\xi)}\cap f_\bullet^{-1}\overline{U(\xi)}.
\end{equation*}
It follows that $\overline{U(\xi)}\cap f_\bullet^{-1}U(\xi)$ is a subset of 
$\overline{U(\xi)}$ that is both open and closed. It is also nonempty,
as it contains $v_0$. By connectedness of $\overline{U(\xi)}$, 
we conclude that $f_\bullet(\overline{U(\xi)})\subseteq U(\xi)$.

The proof of Theorem~B can now be concluded in the same way
as in cases~(1) and~(3). Set $v_n:=f^n_\bullet v_0$
for $n\ge 0$. Then we have $v_n\in\overline{U(\xi)}$ and hence 
$c(f, v_n)=c(f, v_0)=:c$ for all $n\ge 0$. This implies
$c(f^n, v_0)=\prod_{i=0}^{n-1}c(f, v_i)=c^n$ for all $n\ge 1$.
As before, this implies that $c=c_\infty$ and 
$-\alpha_0^{-1}c_\infty^n\le c(f^n)\le c_\infty^n$,
where $\alpha_0=\alpha(v_0)<\infty$.
%
%
%
%
\subsection{The case of a non-finite germ}\label{S128}
Let us briefly discuss the situation when 
$f:\A^2\to\A^2$ is dominant but not finite at a fixed point $0=f(0)$. 
In other words, the ideal $f^*\fm_0\subseteq\fm_0$ is not
primary. In this case, the subset 
$I_f\subseteq\cV_0$ given by $c(f,\cdot)=+\infty$ 
is nonempty but finite. Each element of $I_f$ is a curve valuation
associated to an irreducible germ of a curve $C$ at $0$ such that 
$f(C)=0$.  In particular, $I_f$ does not contain any 
quasimonomial valuations.
Write $\hI_f=\R_+^*I_f$,
$\hD_f:=\hcV_0^*\setminus\hI_f=\{c(f,\cdot)<+\infty\}$
and $D_f:=\cV_0\setminus I_f=\hD_f\cap\cV_0$. 
For $v\in\hI_f$ we have $f(v)=\triv_0$. We can view
$f:\hcV_0^*\dashrightarrow\hcV_0^*$
as a partially defined map having domain of definition $\hD_f$.
On $D_f$ we define $f_\bullet$ as before,
namely $f_\bullet v=f(v)/c(f, v)$. 
One can show that $f_\bullet$ extends 
continuously through $I_f$ to a map 
$f_\bullet:\cV_0\to\cV_0$. More precisely, any $v\in I_f$
is associated to an analytically irreducible branch of 
an algebraic curve $D\subseteq\A^2$ for which $f(D)=0$.
The valuation $f(\ord_D)$ is divisorial and has 0 as 
its center on $\A^2$, hence $f(\ord_D)=r v_E$,
where $r\in\N$ and $v_E\in\cV_0$ is divisorial.
The continuous extension of $f_\bullet$ across 
$v$ is then given by $f_\bullet v= v_E$. 
In particular, $f_\bullet I_f\cap I_f=\emptyset$.

Now we can find a log resolution $\pi:X_\pi\to\A^2$ 
of the ideal $f^*\fm_0$.
By this we mean that the ideal sheaf $f^*\fm_0\cdot\cO_{X_\pi}$ 
on $X_\pi$ is
locally principal and given by a normal crossings divisor
in a neighborhood of $\pi^{-1}(0)$. We can embed the dual
graph of this divisor as a finite subtree 
$|\Delta|\subseteq\cV_0$. Note that $|\Delta|$ 
contains all elements of $I_f$.
There is a continuous retraction map 
$r:\cV_0\to|\Delta|$.  Thus we get a continuous selfmap 
$rf_\bullet:|\Delta|\to|\Delta|$, which admits a
fixed point $v\in|\Delta|$. 
Note that $v\not\in I_f$ since $f_\bullet I_f\cap I_f=\emptyset$
and $r^{-1}I_f=I_f$. Therefore $v$ is quasimonomial.
The proof now goes through exactly as in the finite case. 
%
%
%
%
\subsection{Further properties}
Let us outline some further results from~\cite{eigenval} that
one can obtain by continuing the analysis. 

First, one can construct an \emph{eigenvaluation}, by which
we mean a semivaluation $v\in\cV_0$ such that 
$f(v)=c_\infty v$. Indeed, suppose $f$ is finite for simplicity
and look at the three cases~(1)--(3)
in~\S\ref{S168}. In cases~(1) and~(3) the valuation $v_0$
is an eigenvaluation. In case~(2) one can show that the sequence
$(f^n_\bullet v_0)_{n=0}^\infty$ increases to an eigenvaluation.

Second, we can obtain local normal forms for the dynamics.
For example, in Case~(2) in~\S\ref{S168} we showed that 
$f_\bullet$ mapped the open set $U(\xi)$ into itself, where 
$U(\xi)$ is the set of semivaluations whose center of $X_\pi$ is
equal to $\xi$, the center of $f_\bullet v_0$ on $X_\pi$.
This is equivalent to the the lift 
$f:X_\pi\dashrightarrow X_\pi$ being regular at $\xi$ and $f(\xi)=\xi$.
By choosing $X_\pi$ judiciously one can even guarantee that
$f:(X_\pi,\xi)\to(X_\pi,\xi)$ is a \emph{rigid} germ, a dynamical
version of simple normal crossings singularities. 
Such a rigidification result was proved in~\cite{eigenval} for 
superattracting germs and later extended by 
Matteo Ruggiero~\cite{Ruggiero} to more general germs.

When $f$ is finite,
$f_\bullet:\cV_0\to\cV_0$ is a tree map in the sense of~\S\ref{S130},
so the results in that section apply, but in our approach here we 
did not need them. In contrast, the approach in~\cite{eigenval}
consists of first using the tree analysis in~\S\ref{S130} to 
construct an eigenvaluation. 

Using numerical invariants one can show that $f$ preserves the
type of a valuation in the sense of~\S\ref{S169}. There is also
a rough analogue of the ramification locus for selfmaps
of the Berkovich projective line as in~\S\ref{S171}.
At least in the case of a finite map,
the ramification locus is a finite subtree given by the convex 
hull of the preimages of the root $\ord_0$.

While this is not pursued in~\cite{eigenval}, the induced dynamics
on the valuative tree is somewhat similar to the dynamics of a selfmap
of the unit disc over $\C$. Indeed, recall from~\S\ref{S132} that we 
can embed the valuative tree inside the Berkovich unit disc over
the field of Laurent series (although this does not seem very useful
from a dynamical point of view). 
In particular, the dynamics is (essentially) globally attracting. This
is in sharp contrast with selfmaps
of the Berkovich projective line that are nonrepelling on hyperbolic 
space $\H$.

For simplicity we only studied the dynamics of polynomial maps,
but the analysis goes through also for formal fixed point germs.
In particular, it applies to fixed point germs defined by rational
maps of a projective surface and to holomorphic (perhaps
transcendental) fixed point germs. In the latter case, one
can really interpret $c_\infty(f)$ as a speed at which typical 
orbits tend to 0, see~\cite[Theorem~B]{eigenval}.
%
%
%
%
\subsection{Other ground fields}\label{S271}
Let us briefly comment on the case when the field $K$
is not algebraically closed. Specifically, let us argue why
Theorem~B continues to hold in this case.

Let $K^a$ be the algebraic closure of $K$ and $G=\Gal(K^a/K)$
the Galois group. 
Then $\A^2(K)\simeq\A^2(K^a)/G$ and 
any polynomial mapping $f:\A^2(K)\to\A^2(K)$ induces a equivariant 
polynomial mapping $f:\A^2(K^a)\to\A^2(K^a)$.

If the point $0\in\A^2(K)$ is $K$-rational, then it has
a unique preimage in $0\in K^a$ and the value of 
$\ord_0(\phi)$, for $\phi\in R$, is the same when 
calculated over $K$ or over $K^a$. The same therefore holds
for $c(f^n)$, so since Theorem~B holds over $K^a$, it also holds over $K$.

In general, $0\in\A^2$ has finitely many preimages $0_j\in\A^2(K^a)$
but if $\phi\in R$ is a polynomial with coefficients in $K$, then 
$\ord_0(\phi)=\ord_{0_j}(\f)$ for all $j$. Again we can deduce
Theorem~B over $K$ from its counterpart over $K^a$, although some care
needs to be taken to prove that $c_\infty$ is a quadratic integer in
this case.

Alternatively, we can consider the action of $f$ directly on
$\BerkAtwo(K)$.
As noted in~\S\ref{S230}, the subset of semivaluations centered at $0$
is still the cone over a tree and we can consider the induced
dynamics. The argument for proving that $c_\infty$ is a quadratic
integer, using value groups, carries over to this setting.
%
%
%
%
\subsection{Notes and further references}
In~\cite{eigenval} and~\cite{dyncomp} we used the notation
$f_*v$ instead of $f(v)$ as the action of $f$ on the valuative tree is
given as a pushforward. However, one usually does not 
denote induced maps on Berkovich spaces as pushforwards,
so I decided to deviate from \loccit in order to 
keep the notation uniform across these notes.

In analogy with the degree growth of polynomial maps
(see~\ref{S170}) I would expect the sequence $(c(f^n))_{n=0}^\infty$ 
to satisfy an integral linear recursion relation, but this has
not yet been established.\footnote{The existence
of such a relation has in fact recently 
been established by W. Gignac and M. Ruggiero in
\texttt{arXiv:1209.3450}}.

My own path to Berkovich spaces came through joint work with 
Charles Favre. Theorem~B, in a version for holomorphic selfmaps
of $\P^2$, has ramifications for problem of equidistribution
to the Green current. See~\cite{brolin} and 
also~\cite{DSAENS,Parra} for higher dimensions.
%
%
%
%
%
%
\newpage
\section{The valuative tree at infinity}\label{S108}
In order to study the dynamics at infinity of 
polynomial maps of $\A^2$ we will use the subspace
of the Berkovich affine plane $\BerkAtwo$ consisting of 
semivaluations centered at infinity. 
As in the case of semivaluations centered at a point, 
this is a cone over a tree that we call 
the \emph{valuative tree at infinity}.\footnote{The notation in these notes
differs from~\cite{eigenval,dyncomp} where 
the valuative tree at infinity is denoted by $\cV_0$.
In \loccit the valuation $\ord_\infty$ defined in~\eqref{e146} 
is denoted by $-\deg$.}
Its structure is superficially similar to that of the valuative tree
at a point, which we will refer to as the \emph{local case},
but, as we will see, there are some significant differences.
%
%
%
%
\subsection{Setup}
Let $K$ be an algebraically closed field 
of characteristic zero, equipped with the trivial valuation.
(See~\S\ref{S246} for the case of other ground fields.)
Further, $R$ and $F$ are the coordinate ring and
function field of the affine plane $\A^2$ over $K$.
Recall that the Berkovich affine plane $\BerkAtwo$ 
is the set of semivaluations on $R$ that restrict to the trivial valuation
on $K$.

A \emph{linear system} $|\fM|$ of curves on $\A^2$
is the projective space associated to a nonzero, 
finite-dimensional vector space
$\fM\subseteq R$. The system is \emph{free} if its base 
locus is empty, that is, for every point $\xi\in\A^2$ there exists
a polynomial $\phi\in\fM$ with $\phi(\xi)\ne 0$. 
For any linear system $|\fM|$ and any $v\in\BerkAtwo$
we write $v(|\fM|)=\min\{v(\phi)\mid\phi\in\fM\}$.
%
%
%
%
\subsection{Valuations centered at infinity}
We let $\hcVe{\infty}\subseteq\BerkAtwo$ denote
the set of semivaluations $v$ having center at infinity, that is,
such that $v(\phi)<0$ for some polynomial $\phi\in R$.
Note that $\hcVe{\infty}$ is naturally a pointed cone:
in contrast to $\hcV_0$ there is no element `$\triv_\infty$'.

The valuative tree at infinity is the base of this cone
and we want to realize it as a ``section'.
In the local case, the valuative tree at a closed point $0\in\A^2$ was defined
using the maximal ideal $\fm_0$. In order to do something
similar at infinity, we fix an embedding $\A^2\hookrightarrow\P^2$. 
This allows us to define the \emph{degree} of a polynomial
in $R$ and in particular defines the free linear system $|\fL|$ of \emph{lines},
associated to the subspace $\fL\subseteq R$ of \emph{affine functions} 
on $\A^2$, that is, polynomials of degree at most one. 
Note that $v\in\BerkAtwo$ has center at infinity iff $v(|\fL|)<0$.

We say that two polynomials $z_1,z_2$ are affine coordinates
on $\A^2$ if $\deg z_i=1$ and $R=K[z_1,z_2]$.
In this case, $F=K(z_1,z_2)$ and $v(|\fL|)=\min\{v(z_1),v(z_2)\}$.
\begin{Def}
  The \emph{valuative tree at infinity} $\cVe{\infty}$ is the set of 
  semivaluations $v\in\BerkAtwo$ such that $v(|\fL|)=-1$.
\end{Def}
The role of $\ord_0\in\cV_0$ is played by the valuation 
$\ord_\infty\in\cVe{\infty}$, 
defined by
 \begin{equation}\label{e146}
  \ord_\infty(\phi)=-\deg(\phi).
\end{equation}
In particular, $v(\phi)\ge\ord_\infty(\phi)$ for every $\phi\in R$
and every $v\in\cVe{\infty}$.
We emphasize that both $\cVe{\infty}$ and $\ord_\infty$ depend
on a choice of embedding $\A^2\hookrightarrow\P^2$.

We equip $\cVe{\infty}$ and $\hcVe{\infty}$ with the subspace topology from
$\BerkAtwo$.
It follows from Tychonoff's theorem that 
$\cVe{\infty}$ is a compact Hausdorff space.
The space $\hcVe{\infty}$ is open in $\BerkAtwo$ 
and its boundary consists of the trivial valuation $\triv_{\A^2}$ and 
the set of semivaluations centered at a curve in $\A^2$. 

As in the local case, we can classify the elements of $\hcVe{\infty}$ into 
curve semivaluations, divisorial valuations, irrational valuations 
and infinitely singular valuations. We do this by considering 
$v$ as a semivaluation on the ring $\widehat{\cO}_{\P^2,\xi}$, where 
$\xi$ is the center of $\xi$ on $\P^2$.
%
%
%
%
\subsection{Admissible compactifications}\label{S213}
The role of a blowup of $\A^2$ above a closed point is played here
by a \emph{compactification} of $\A^2$, by which we mean a 
projective surface containing $\A^2$ as Zariski open subset.
To make the analogy even stronger, recall that we have fixed an 
embedding $\A^2\hookrightarrow\P^2$. We will use
\begin{Def}
  An \emph{admissible compactification} of $\A^2$ is a 
  smooth projective surface $X$ containing $\A^2$ as a
  Zariski open subset, such that the induced birational map
  $X\dashrightarrow\P^2$ induced by the identity on $\A^2$,
  is regular.
\end{Def}
By the structure theorem of birational surface maps,
this means that the morphism $X\to\P^2$ is a finite composition 
of point blowups above infinity. The set of admissible 
compactifications is naturally partially ordered and in fact a 
directed set: any two admissible compactifications are
dominated by a third.

Many of the notions below will in fact not depend on the 
choice of embedding $\A^2\hookrightarrow\P^2$ but would
be slightly more complicated to state without it.
\begin{Remark}\label{R203}
  Some common compactifications of $\A^2$, for instance $\P^1\times\P^1$,
  are not admissible in our sense. However,
  the set of admissible compactifications is cofinal 
  among compactifications of $\A^2$:
  If $Y$ is an irreducible, normal projective surface 
  containing $\A^2$ as a Zariski open subset, 
  then there exists an admissible 
  compactification $X$ of $\A^2$ such that the birational map
  $X\dashrightarrow Y$ induced by the identity on $\A^2$
  is regular. Indeed, $X$ is obtained
  by resolving the indeterminacy points of the similarly defined
  birational map $\P^2\dashrightarrow Y$. 
  See~\cite{Morrow,Kishimoto} for a classification of smooth 
  compactifications of $\A^2$.
\end{Remark}
%
%
\subsubsection{Primes and divisors at infinity}  
Let $X$ be an admissible compactification of $\A^2$.
A \emph{prime at infinity} of $X$  is an irreducible component 
of $X\setminus\A^2$.
We often identify a prime of $X$ at infinity with its  
strict transform in any compactification $X'$ dominating $X$. 
In this way we can identify a prime at infinity $E$ (of some 
admissible compactification) with the 
corresponding divisorial valuation $\ord_E$. 

Any admissible compactification contains a special prime $L_\infty$,
the strict transform of $\P^2\setminus\A^2$.
The corresponding divisorial valuation is $\ord_{L_\infty}=\ord_\infty$.

We say that a point in $X\setminus\A^2$ is a 
\emph{free point} if it belongs to a unique prime at infinity;
otherwise it is a \emph{satellite point}.

A \emph{divisor at infinity} on $X$ is a divisor
supported on $X\setminus\A^2$. We write 
$\Div_\infty(X)$ for the abelian group of divisors
at infinity. 
If $E_i$, $i\in I$ are the primes of $X$ at infinity,
then $\Div_\infty(X)\simeq\bigoplus_i\Z E_i$.
%
%
\subsubsection{Intersection form and linear equivalence}
We have the following basic facts.
\begin{Prop}\label{P205}
  Let $X$ be an admissible compactification of $\A^2$. Then
  \begin{itemize}
    \item[(i)]
      Every divisor on $X$ is linearly 
      equivalent to a unique divisor at infinity, so 
      $\Div_\infty(X)\simeq\Pic(X)$.
    \item[(ii)]
      The intersection form on $\Div_\infty(X)$ is nondegenerate
      and unimodular. It has signature $(1,\rho(X)-1)$. 
    \end{itemize}
  \end{Prop}
\begin{proof}
  We argue by induction on the number of blowups needed
  to obtain $X$ from $\P^2$. 
  If $X=\P^2$, then the statement is clear: 
  $\Div_\infty(X)=\Pic(X)=\Z L_\infty$ and $(L_\infty\cdot L_\infty)=1$.
  For the inductive step, suppose $\pi'=\pi\circ\mu$, where
  $\mu$ is the simple blowup of a closed point on $X\setminus\A^2$,
  resulting in an exceptional prime $E$. 
  Then we have an orthogonal decomposition 
  $\Div_\infty(X')=\mu^*\Div_\infty(X)\oplus\Z E$,
  $\Pic(X')=\mu^*\Pic(X)\oplus\Z E$
  and $(E\cdot E)=-1$.

  Statement~(ii) about the intersection form is also
  a consequence of the Hodge Index Theorem
  and Poincar{\'e} Duality. 
\end{proof}
Concretely, the isomorphism $\Pic(X)\simeq\Div_\infty(X)$ can
be understood as follows. Any irreducible curve $C$ in $X$
that is not contained in $X\setminus\A^2$
is the closure in $X$ of an affine curve $\{\phi=0\}$ for
some polynomial $\phi\in R$. Then $C$ is linearly equivalent to 
the element in $\Div_\infty(X)$ defined as the divisor of poles of
$\phi$, where the latter is viewed as a rational function on $X$.

Let $E_i$, $i\in I$ be the primes of $X$ at infinity.
It follows from Proposition~\ref{P205} 
that for each $i\in I$ there exists a divisor 
$\vE_i\in\Div_\infty(X)$ such that 
$(\vE_i\cdot E_i)=1$ and $(\vE_i\cdot E_j)=0$ for all $j\ne i$.
%
%
\subsubsection{Invariants of primes at infinity}\label{S151}
Analogously to the local case (see~\ref{S249}) we associate two
basic numerical invariants $\a_E$ and $A_E$ to any prime $E$ at infinity 
(or, equivalently, to the associated divisorial valuation 
$\ord_E\in\hcVe{\infty}$.

To define $\a_E$, pick an admissible compactification $X$ of $\A^2$
in which $E$ is a prime at infinity. Above we defined the 
divisor $\vE=\vE_X\in\Div_\infty(X)$ by duality:
$(\vE_X\cdot E)=1$ and $(\vE_X\cdot F)=0$ for all primes $F\ne E$
of $X$ at infinity. Note that if $X'$ is an admissible
compactification dominating $X$, then the divisor $\vE_{X'}$ on 
$X'$ is the pullback of $\vE_X$ under the morphism 
$X'\to X$. In particular, the self-intersection number 
\begin{equation*}
  \a_E:=\a(\ord_E):=(\vE\cdot\vE)
\end{equation*}
is an integer independent of the choice of $X$.

The second invariant is the \emph{log discrepancy} $A_E$.
Let $\omega$ be a nonvanishing regular 2-form on $\A^2$. 
If $X$ is an admissible compactification of $\A^2$, then $\omega$
extends as a rational form on $X$. For any prime $E$ of $X$ at infinity, 
with associated divisorial valuation $\ord_E\in\hcVe{\infty}$, we define
\begin{equation}\label{e220}
  A_E:=A(\ord_E):=1+\ord_E(\omega).
\end{equation}
This is an integer whose value does not depend
on the choice of $X$ or $\omega$.
Note that $A_{L_\infty}=-2$ since 
$\omega$ has a pole of order 3 along $L_\infty$. In general, $A_E$
can be positive or negative. 

We shall later need the analogues of~\eqref{e212} and~\eqref{e213}.
Thus let $X$ be an admissible compactification of $\A^2$ and 
$X'$ the blowup of $X$ at a free point $\xi\in X\setminus\A^2$.
Let $E'$ be the ``new'' prime of $X'$, that is, the inverse image of $\xi$ in $X'$.
Then 
\begin{equation}\label{e152}
  A_{E'}=A_E+1,\
  b_{E'}=b_E 
  \quad\text{and}\quad
  \vE'=\vE-E',
\end{equation}
where, in the right hand side, we identify 
the divisor $\vE\in\Div_\infty(X)$ with its
pullback to $X'$. 
As a consequence,
\begin{equation}\label{e155}
 \a_{E'}:=(\vE'\cdot\vE')=(\vE\cdot\vE)-1=\a_E-1.
\end{equation}

Generalizing both~\S\ref{S249} and~\S\ref{S151}, the 
invariants $\a_E$ and $A_E$ can in fact be defined for any 
divisorial valuation $\ord_E$ in the Berkovich affine plane.
%
%
\subsubsection{Positivity}\label{S214}
Recall that in the local case, the notion of 
relative positivity was very well behaved and easy
to understand, see~\S\ref{S159}.
Here the situation is much more subtle,
and this will account for several difficulties.

As usual, we say that a divisor $Z\in\Div(X)$ is \emph{effective}
if it is a positive linear combination of prime divisors on $X$.
We also say that $Z\in\Div(X)$ is \emph{nef} if 
$(Z\cdot W)\ge 0$ for all effective divisors $W$. 
These notions make sense also for $\Q$-divisors. 
It is a general fact that if $Z\in\Div(X)$ is nef, then $(Z\cdot Z)\ge0$.

Clearly, the semigroup of effective divisors in $\Div_\infty(X)$ is
freely generated by the primes $E_i$, $i\in I$ at infinity.
A divisor $Z\in\Div_\infty(X)$ is \emph{nef at infinity} if 
$(Z\cdot W)\ge0$ for every effective divisor $W\in\Div_\infty(X)$.
This simply means that $(Z\cdot E_i)\ge 0$ for all $i\in I$. 
It follows easily that the subset of $\Div_\infty(X)$ consisting 
of divisors that are nef at infinity is a free semigroup generated
by the $\vE_i$, $i\in I$.

We see that a divisor $Z\in\Div_\infty(X)$ is nef  iff
it is nef at infinity and, in addition, $(Z\cdot C)\ge 0$
whenever $C$ is the closure in $X$ of an irreducible 
curve in $\A^2$. 
In general, a divisor that is nef at infinity may not be nef. 
\begin{Example}\label{E202}
  Consider the surface $X$ obtained by first blowing up
  any closed point at infinity, creating the prime $E_1$,
  then blowing up a free point on $E_1$, creating the prime $E_2$.
  Then the divisor $Z:=\vE_2=L_\infty-E_2$ is nef at infinity but 
  $Z$ is not nef since $(Z\cdot Z)=-1<0$.
\end{Example}
However, a divisor $Z\in\Div_\infty(X)$ 
that is nef at infinity and effective
is always nef: as above it suffices to show that $(Z\cdot C)\ge0$ 
whenever $C$ is the closure in $X$ of a curve in $\A^2$. 
But $(E_i\cdot C)\ge0$ for all $i\in I$, so since $Z$ has
nonnegative coefficients in the basis $E_i$, $i\in I$,
we must have $(Z\cdot C)\ge0$.

On the other hand, it is possible for a divisor to be nef 
but not effective. 
The following example was communicated by Adrien Dubouloz~\cite{Dubouloz}.
\begin{Example}\label{E206}
  Pick two distinct points $\xi_1$, $\xi_2$ on the line at infinity 
  $L_\infty$ in $\P^2$ and let $C$ be a conic passing through $\xi_1$
  and $\xi_2$. 
  Blow up $\xi_1$ and let $D$ be the exceptional divisor.
  Now blow up $\xi_2$, creating $E_1$, blow up $C\cap E_1$, 
  creating $E_2$ and finally blow up $C\cap E_2$ creating $F$. 
  We claim that the non-effective divisor $Z=2D+5L_\infty+3E_1+E_2-F$
  on the resulting surface $X$ is nef. 
  
  To see this, we successively contract the primes 
  $L_\infty$, $E_1$ and $E_2$. A direct computation shows that 
  each of these is a $(-1)$-curve at the time we contract
  it, so by Castelnuovo's criterion we obtain a birational morphism
  $\mu:X\to Y$, with $Y$ a smooth rational surface. Now 
  $Y$ is isomorphic to $\P^1\times\P^1$. Indeed,
  one checks that $(F\cdot F)=(C\cdot C)=0$
  and $(F\cdot C)=1$ on $Y$ and it is easy to see in coordinates
  that each of $F$ and $C$ is part of a fibration on $Y$.
  Now $Z$ is the pullback of the divisor $W=2D-F$ on $Y$,
  Further, $\Pic(Y)\simeq\Z C\oplus\Z F$ and 
  $(W\cdot C)=1>0$ and $(W\cdot F)=2>0$, so $W$
  is ample on $Y$ and hence $Z=\mu^*W$ is nef on $X$.
\end{Example}
Finally, in contrast to the local case (see Proposition~\ref{P202})
it can happen that a divisor $Z\in\Div_\infty(X)$ is nef
but that the line bundle $\cO_X(Z)$ has base points,
that is, it is not generated by its global sections.
\begin{Example}\label{E203}
  Consider the surface $X$ obtained from blowing $\P^2$ nine times,
  as follows. First blow up at three distinct points on 
  $L_\infty$, creating primes $E_{1j}$, $j=1,2,3$. 
  On each $E_{1j}$ blow up a free point, creating a new prime
  $E_{2j}$. Finally blow up a free point on each $E_{2j}$,
  creating a new prime $E_{3j}$. 
  Set $Z=3L_\infty+\sum_{j=1}^3(2E_{2j}+E_{1j})$.
  Then $Z=\sum_{j=1}^3\vE_{3j}$, so $Z$ is nef at infinity.
  Since $Z$ is also effective, it must be nef.

  However, we claim that if the points at which we
  blow up are generically chosen, then the line bundle
  $\cO_X(Z)$ is not generated by its global sections.
  To see this, consider a global section of $\cO_X(Z)$
  that does not vanish identically along $L_\infty$.
  Such a section is given by a polynomial 
  $\phi\in R$ of degree 3 satisfying $\ord_{E_{ij}}(\phi)=3-i$, $1\le i,j\le 3$. 
  This gives nine conditions on $\phi$. Note that if $\phi$ is such a
  section, then so is $\phi-c$ for any constant $c$, so we may assume
  that  $\phi$ has zero constant coefficient. Thus $\phi$ is given by 
  eight coefficients. For a generic choice of points blown up, no such 
  polynomial $\phi$ will exist. This argument is of course not rigorous, but can
  be made so by an explicit computation in coordinates 
  that we invite the reader to carry out.
\end{Example}
%
%
%
%
\subsection{Valuations and dual fans and graphs}\label{S218}
Analogously to~\S\ref{S143} we can realize
$\hcVe{\infty}$ and $\cVe{\infty}$ as inverse limits
of dual fans and graphs, respectively. 

To an admissible compactification $X$ of $\A^2$
we associate a dual fan $\hDelta(X)$ with integral affine structure 
$\Aff(X)\simeq\Div_\infty(X)$. 
This is done exactly as in the local case,
replacing exceptional primes with primes at infinity.
Inside the dual fan we embed the dual graph $\Delta(X)$
using the integral affine function associated to the divisor
$\pi^*L_\infty=\sum_ib_iE_i\in\Div_\infty(X)$.
The dual graph is a tree.

The numerical invariants $A_E$ and $\a_E$
uniquely to homogeneous functions $A$ and $\a$ on
the dual fan $\hDelta(X)$ of degree one and two, respectively
and such that these functions are affine on the dual graph.
Then $A$ and $\a$ give parametrizations of the dual graph rooted
in the vertex corresponding to $L_\infty$. We equip the dual graph
with the metric associated to the parametrization $\a$:
the length of a simplex $\sigma_{ij}$ is equal to $1/(b_ib_j)$.
We could also (but will not) use $A$ to define a metric on 
the dual graph. This metric is the same as the one induced by the
integral affine structure: the length of the simplex $\sigma_{ij}$
is $m_{ij}/(b_ib_j)$, where $m_{ij}=\gcd\{b_i,b_j\}$ is the 
multiplicity of the segment.

Using monomial valuations we embed the dual fan as a 
subset $|\hDelta(X)|$ of the Berkovich affine plane. 
The image $|\hDelta^*(X)|$ of the punctured dual fan lies
in $\hcVe{\infty}$.
The preimage of $\cVe{\infty}\subseteq\hcVe{\infty}$ under the 
embedding $|\hDelta^*(X)|\subseteq\hcVe{\infty}$
is exactly $|\Delta(X)|$. In particular, a vertex $\sigma_E$ of 
the dual graph is identified with the corresponding 
normalized valuation $v_E\in\cVe{\infty}$, defined by 
\begin{equation}\label{e154}
  v_E=b_E^{-1}\ord_E
  \quad\text{where}\
  b_E:=-\ord_E(|\fL|).
\end{equation}
Note that $v_{L_\infty}=\ord_{L_\infty}=\ord_\infty$.

We have a retraction $r_X:\hcVe{\infty}\to|\hDelta^*(X)|$
that maps $\cVe{\infty}$ onto $|\Delta(X)|$. 
The induced maps
\begin{equation}\label{e135}
  r:\cVe{\infty}\to\varprojlim_X|\Delta(X)|
  \quad\text{and}\quad
  r:\hcVe{\infty}\to\varprojlim_X|\hDelta^*(X)|
\end{equation} 
are homeomorphisms.
The analogue of Lemma~\ref{L123} remains true
and we have the following analogue of 
Lemma~\ref{L125}.
\begin{Lemma}\label{L126}
  If $v\in\hcVe{\infty}$ and $X$ is an admissible compactification 
  of $\A^2$, then 
  \begin{equation*}
    (r_X v)(\phi)\le v(\phi)
  \end{equation*}
  for every polynomial $\phi\in R$, 
  with equality if the closure in $X$ of the curve $(\phi=0)\subseteq\A^2$
  does not pass through the center of $v$ on $X$. 
\end{Lemma}

The second homeomorphism in~\eqref{e135}  equips $\hcVe{\infty}$ with an
integral affine structure: a function $\f$ on $\hcVe{\infty}$ is integral 
affine if it is of the form $\f=\f_X\circ r_X$,
where $\f_X\in\Aff(X)$.

The first homeomorphism in~\eqref{e135} induces a metric tree structure on 
$\cVe{\infty}$ as well as two
parametrizations\footnote{In~\cite{valtree} the parametrization $A$ is
  called \emph{thinness} whereas $-\a$ is called \emph{skewness}.} 
\begin{equation}\label{e140}
  \a:\cVe{\infty}\to[-\infty,1]
  \quad\text{and}\quad
  A:\cVe{\infty}\to[2,\infty]
\end{equation}
of $\cVe{\infty}$, viewed as a tree rooted in $\ord_\infty$.
We extend $A$ and $\a$ as homogeneous functions on $\hcVe{\infty}$
of degrees one and two, respectively.
%
%
%
%
\subsection{Potential theory}
Since $\cVe{\infty}$ is a metric tree, we can do potential theory on it, 
but just as in the case of the valuative tree at a closed point, we
need to tweak the general approach in~\S\ref{S116}. The reason is again 
that one should view a function on $\cVe{\infty}$ as the restriction
of a homogeneous function on $\hcVe{\infty}$. 

A first guideline is that functions of the form $\log|\fM|$, 
defined by\footnote{As in~\S\ref{S113} the notation reflects the fact that $|\cdot|:=e^{-v}$ 
  is a seminorm on $R$.}
\begin{equation}\label{e211}
  \log|\fM|(v)=-v(|\fM|)
\end{equation}
should be subharmonic on $\cVe{\infty}$, for any 
linear system $|\fM|$ on $\A^2$.
In particular, the function $\log|\fL|\equiv1$ should
be subharmonic (but not harmonic).
A second guideline is that the Laplacian should be closely
related to the intersection product on divisors at infinity.
%
%
\subsubsection{Subharmonic functions and Laplacian on $\cVe{\infty}$}
As in~\S\ref{S144} we extend the valuative tree $\cVe{\infty}$ 
to a slightly larger tree $\tcVe{\infty}$ by connecting the root $\ord_\infty$ to 
a point $G$ using an interval of length one. Let $\tDelta$
denote the Laplacian on $\tcVe{\infty}$.

We define the class $\SH(\cVe{\infty})$ 
of \emph{subharmonic functions} on 
$\cVe{\infty}$ as the set of restrictions to $\cVe{\infty}$ of functions
$\f\in\QSH(\tcVe{\infty})$ such that
\begin{equation*}
  \f(G)=2\f(\ord_\infty)
  \quad\text{and}\quad
  \tDelta\f=\rho-a\delta_G,
\end{equation*}
where $\rho$ is a positive measure supported
on $\cVe{\infty}$ and $a=\rho(\cVe{\infty})\ge0$.
In particular, $\f$ is affine
of slope $-\f(\ord_\infty)$ on the segment 
$[G,\ord_\infty[\,=\tcVe{\infty}\setminus\cVe{\infty}$.
We then define $\Delta\f:=\rho=(\tDelta\f)|_{\cVe{\infty}}$.
For example, if $\f\equiv 1$ on $\cVe{\infty}$, 
then $\f(G)=2$,
$\tDelta\f=\delta_{\ord_\infty}-\delta_G$
and $\Delta\f=\delta_{\ord_\infty}$.

From this definition and the analysis in~\S\ref{S116} 
one deduces:
\begin{Prop}\label{P111}
  Let $\f\in\SH(\cVe{\infty})$ and write $\rho=\Delta\f$.
  Then:
  \begin{itemize}
  \item[(i)]
    $\f$ is decreasing in the partial 
    ordering of $\cVe{\infty}$ rooted in $\ord_\infty$;
  \item[(ii)]
    $\f(\ord_\infty)=\rho(\cVe{\infty})$;
  \item[(iii)]
    $|D_\vv\f|\le\rho(\cVe{\infty})$ for all
    tangent directions $\vv$ in $\cVe{\infty}$.
  \end{itemize}
\end{Prop}
As a consequence we have the estimate
\begin{equation}\label{e151}
  \a(v)\f(\ord_\infty)\le\f(v)\le\f(\ord_\infty)
\end{equation}
for all $v\in\cVe{\infty}$. 
Here $\a:\cVe{\infty}\to[-\infty,+1]$
is the parametrization in~\eqref{e140}.
It is important to remark that a subharmonic function 
can take both positive and negative values.
In particular,~\eqref{e151} is not so useful when $\a(v)<0$.

The exact sequence in~\eqref{e139} shows that 
\begin{equation}\label{e158}
  \Delta:\SH(\cVe{\infty})\to\cM^+(\cVe{\infty}),
\end{equation}
is a homeomorphism whose inverse is given by 
\begin{equation}\label{e147}
  \f(v)=\int_{\cVe{\infty}}\a(w\wedge_{\ord_\infty}v)d\rho(w).
\end{equation}

The compactness properties in~\S\ref{S116} carry over to 
the space $\SH(\cVe{\infty})$. 
In particular, for any $C>0$, the set 
$\{\f\in\SH(\cVe{\infty})\mid\f(\ord_\infty)\le C\}$
is compact. 
Further, if $(\f_i)_i$ is a decreasing 
net in $\SH(\cVe{\infty})$, and $\f:=\lim\f_i$,
then $\f\in\SH(\cVe{\infty})$.
Moreover, if $(\f_i)_i$ is a family in $\SH(\cVe{\infty})$
with $\sup_i\f(\ord_\infty)<\infty$, 
then the upper semicontinuous regularization of 
$\f:=\sup_i\f_i$ belongs to $\SH(\cVe{\infty})$.

While the function $-1$ on $\cVe{\infty}$ is not subharmonic,
it is true that $\max\{\f,r\}$ is subharmonic
whenever $\f\in\SH(\cVe{\infty})$ and $r\in\R$.
%
%
\subsubsection{Laplacian of integral affine functions}
Any integral affine function $\f$ on $\hcVe{\infty}$ 
is associated to a divisor
at infinity $Z\in\Div_\infty(X)$ for some admissible 
compactification $X$ of $\A^2$: the value of $\f$ at 
a divisorial valuation $\ord_{E_i}$ is the coefficient $\ord_{E_i}(Z)$ 
of $E_i$ in $Z$. Using the same computations as in the proof
of Proposition~\ref{P110} we show that
\begin{equation*}
  \Delta\f=\sum_{i\in I}b_i(Z\cdot E_i)\delta_{v_i},
\end{equation*}
where $b_i=-\ord_{E_i}(|\fL|)\ge1$ and $v_i=b_i^{-1}\ord_{E_i}$.
In particular, $\f$ is subharmonic iff $Z$ is nef at infinity.

Recall that we have defined divisors $\vE_i\in\Div_\infty(X)$ such that 
$(\vE_i\cdot E_i)=1$ and $(\vE_i\cdot E_j)=0$ for all $j\ne i$.
The integral affine function 
$\f_i$ on $\cVe{\infty}$ associated to $\vE_i$
is subharmonic and satisfies 
$\Delta\f_i=b_i\delta_{v_i}$.
In view of~\eqref{e147}, this shows that 
$\min_{\cVe{\infty}}\f_i=\f_i(v_i)=b_i\a(v_i)$.
This implies
\begin{equation}\label{e148}
  \a_{E_i}
  =(\vE_i\cdot\vE_i)
  =\ord_{E_i}(\vE_i)
  =b_i^2\a(v_i)
  =\a(\ord_{E_i}).
\end{equation}
\begin{Prop}\label{P201}
  Let $E$ be a divisor at infinity on some admissible 
  compactification $X$ of $\A^2$. Let $\vE\in\Div_\infty(X)$
  be the associated element of the dual basis and 
  $v_E=b_E^{-1}\ord_E\in\cVe{\infty}$ the associated 
  normalized divisorial valuation. 
  Then $\vE$ is nef at infinity and the following 
 statements are equivalent:
  \begin{itemize}
  \item[(i)]
    $\vE$ is nef;
  \item[(ii)]
    $(\vE\cdot\vE)\ge0$;
  \item[(iii)]
    $\a(v_E)\ge0$.
 \end{itemize}
\end{Prop}
\begin{proof}
  That $\vE$ is nef at infinity is clear from the definition
  and has already been observed.
  That~(ii) is equivalent to~(iii) is an immediate consequence
  of~\eqref{e148}.
  If $\vE$ is nef, then $(\vE\cdot\vE)\ge0$, showing 
  that~(i) implies~(ii). On the other hand, if 
  $\a(v_E)\ge0$, then we have seen above that the minimum on
  $\cVe{\infty}$ of the integral affine function $\f$ associated to
  $\vE$ is attained at $v_E$ and is nonnegative. Thus $\vE$ is
  effective. Being nef at infinity and effective, $\vE$ must 
  be nef, proving that~(ii) implies~(i).
\end{proof}
%
%
\subsubsection{Subharmonic functions from linear systems}\label{S206}
Let $|\fM|$ be a nonempty linear system of affine curves.
We claim that the function $\log|\fM|$, defined by~\eqref{e211}
is subharmonic on $\cVe{\infty}$. To see this, note that 
$\log|\fM|=\max\log|\phi|$, where $\phi$ ranges over polynomials
defining the curves in $|\fM|$. The claim therefore follows from
\begin{Exer}
  If $\phi\in R$ is an irreducible polynomial, show that 
  $\log|\phi|$ is subharmonic on $\cVe{\infty}$ and that 
  \begin{equation*}
    \Delta\log|\phi|=\sum_{j=1}^nm_j\delta_{v_j}
  \end{equation*}
  where $v_j$, $1\le j\le n$ are the curve valuations associated to the 
  all the local branches $C_j$ of $\{\phi=0\}$ at infinity
  and where $m_j=(C_j\cdot L_\infty)$ is the local intersection number of 
  $C_j$ with the line at infinity in $\P^2$.
\end{Exer}
\begin{Example}\label{E204}
  Fix affine coordinates $(z_1,z_2)$ on $\A^2$ and let
  $\fM\subseteq R$ be the vector space 
  spanned by $z_1$ 
  Then  $\log|\fM|(v)=\max\{-v(z_1),0\}$ and 
  $\Delta\log|\fM|$ is a Dirac mass at the monomial valuation
  with $v(z_1)=0$, $v(z_2)=-1$.
\end{Example}
\begin{Example}\label{E205}
  Fix affine coordinates $(z_1,z_2)$ on $\A^2$ and let
  $\fM\subseteq R$ be the vector space 
  spanned by $z_1z_2$ and the constant function 1.
  Then  $\log|\fM|(v)=\max\{-(v(z_1)+v(z_2),0\}$ and 
  $\Delta\log|\fM|=\delta_{v_{-1,1}}+\delta_{v_{1,-1}}$, where 
  $v_{t_1,t_2}$ is the monomial valuation with weights 
  $v_{t_1,t_2}(z_i)=t_i$, $i=1,2$.
\end{Example}
\begin{Prop}\label{P118}
  Let $|\fM|$ be a linear system of affine curves on $\A^2$.
  Then the following conditions are equivalent:
  \begin{itemize}
  \item[(i)]
    the base locus of $|\fM|$ on $\A^2$ contains no curves;
  \item[(ii)]
    the function $\log|\fM|$ is bounded on $\cVe{\infty}$;
  \item[(iii)]
    the measure $\Delta\log|\fM|$ on $\cVe{\infty}$ is supported
    at divisorial valuations.
  \end{itemize}
\end{Prop}
Linear systems $|\fM|$ satisfying these equivalent conditions are
natural analogs of primary ideals $\fa\subseteq R$ in the
local setting. 
\begin{proof}[Sketch of proof]
  That~(iii) implies~(ii) follows from~\eqref{e147}.
  If the base locus of $|\fM|$ contains an affine curve
  $C$, let $v\in\cVe{\infty}$ be a curve valuation 
  associated to one of the branches at infinity of $C$.
  Then $\log|\fM|(v)=-v(\f)=-\infty$ so~(ii) implies~(i).

  Finally, let us prove that~(i) implies~(iii). 
  Suppose the base locus on $|\fM|$ on $\A^2$
  contains no curves. Then we can pick an admissible
  compactification of $\A^2$ such that the strict transform of $|\fM|$
  to $X$ has no base points at infinity. In this case one shows that 
  $\Delta\log|\fM|$ is an atomic measure supported on the divisorial
  valuations associated to some of the primes of $X$ at infinity.
\end{proof}
In general, it seems very hard to characterize
the measures on $\cVe{\infty}$ appearing in~(iii).
Notice that if $\Delta\log|\fM|$ is a Dirac mass at 
a divisorial valuation $v$ then $\a(v)\ge 0$, 
as follows from~\eqref{e147}. There are also sufficient conditions:
using the techniques in the proof of Theorem~\ref{T105} one can
show that if $\rho$ is an atomic measure with rational coefficients
supported on divisorial valuations in the tight tree $\cV'_\infty$
(see~\S\ref{S276}) 
then there exists a linear system $|\fM|$ such that $\log|\fM|\ge 0$
and $\Delta\log|\fM|=n\rho$ for some integer $n\ge 1$.
%
%
\subsection{Intrinsic description of tree structure on $\cVe{\infty}$}\label{S205}
We can try to describe the tree structure on 
$\cVe{\infty}\simeq\varprojlim|\Delta(X)|$ intrinsically,
viewing the elements of $\cVe{\infty}$ purely as 
semivaluations on the ring $R$.
This is more complicated than in the case of the 
valuative tree at a closed point (see~\S\ref{S207}). 
However, the partial ordering can be characterized essentially as expected:
\begin{Prop}\label{P112}
  If $w, v\in\cVe{\infty}$, then the following are equivalent:
  \begin{itemize}
  \item[(i)]
    $v\le w$ in the partial ordering induced by 
    $\cVe{\infty}\simeq\varprojlim|\Delta(X)|$;
  \item[(ii)]
    $v(\phi)\le w(\phi)$ for all polynomials $\phi\in R$;
  \item[(iii)]
    $v(|\fM|)\le w(|\fM|)$ for all free linear systems $|\fM|$ on $\A^2$.
 \end{itemize}
\end{Prop}
\begin{proof}
  The implication (i)$\implies$(ii) follows from the 
  subharmonicity of $\log|\phi|$ 
  together with Proposition~\ref{P111}~(i). 
  The implication (ii)$\implies$(iii) is obvious. It remains to prove
  (iii)$\implies$(i).

  Suppose $v\not\le w$ in the partial ordering on
  $\cVe{\infty}\simeq\varprojlim|\Delta(X)|$. 
  We need to find a free linear system $|\fM|$  on $\A^2$
  such that $v(|\fM|)>w(|\fM|)$.
  First assume that $v$ and $w$ are quasimonomial and 
  pick an admissible compactification $X$ of $\A^2$ such that 
  $v,w\in|\Delta(X)|$. Let $E_i$, $i\in I$, be the primes of 
  $X$ at infinity. One of these primes is $L_\infty$ and there exists
  another prime (not necessarily unique) $E_i$ such that 
  $v_i\ge v$. Fix integers $r$, $s$ with $1\ll r\ll s$ and
  define the divisor $Z\in\Div_\infty(X)$ by 
  \begin{equation*}
    Z:=\sum_{j\in I}\vE_j+r\vE_i+s\vL_\infty.
  \end{equation*}
  We claim that $Z$ is an \emph{ample} divisor on $X$.
  To prove this, it suffices, by the Nakai-Moishezon criterion,
  to show  that $(Z\cdot Z)>0$, $(Z\cdot E_j)>0$ for all $j\in I$
  and $(Z\cdot C)>0$ whenever $C$ is the closure in $X$ of a
  curve $\{\phi=0\}\subseteq\A^2$. 

  First, by the definition of $\hE_j$ it follows that $(Z\cdot E_j)\ge 1$
  for all $j$. Second, we have $(\vL_\infty\cdot C)=\deg\phi$ 
  and $(\vE_j\cdot C)=-\ord_{E_i}(\phi)\ge\a(v_i)\deg\phi$ for all
  $j\in I$ in view of~\eqref{e151}, so that 
  $(Z\cdot C)>0$ for $1\le r\ll s$.
  Third, since $(\vL_\infty\cdot\vL_\infty)=1$, a similar argument shows 
  that $(Z\cdot Z)>0$ for $1\le r\ll s$.

  Since $Z$ is ample, there exists an integer $n\ge1$ such that 
  the line bundle $\cO_X(nZ)$ is base point free. 
  In particular, the corresponding 
  linear system $|\fM|:=|\cO_X(nZ)|$ is free on $\A^2$.
  Now, the integral affine function on $|\Delta(X)|$
  induced by $\vL_\infty$ is the constant function $+1$.
  Moreover, the integral affine function on $|\Delta(X)|$
  induced by $\vE_i$ is the function 
  $\f_i=b_i\a(\cdot\wedge_{\ord_\infty} v_i)$.
  Since $v_i\ge v$ and $v\not\le w$, 
  this implies $\f_i(v)<\f_i(w)$. 
  For $r\gg1$ this translates into $v(|\fM|)>w(|\fM|)$
  as desired.

  Finally, if $v$ and $w$ are general semivaluations in $\cVe{\infty}$
  with $v\not\le w$, then we can pick an admissible compactification 
  $X$ of $\A^2$ such that $r_X(v)\not\le r_X(w)$.
  By the previous construction there exists a free linear system $|\fM|$
  on $\A^2$ such that $r_X(v)(|\fM|)>r_X(w)(|\fM|)$. 
  But since the linear system $|\fM|$ was free also on $X$,
  it follows that $v(|\fM|)=r_X(v)(|\fM|)$ and $w(|\fM|)=r_X(w)(|\fM|)$.
  This concludes the proof.
\end{proof}
The following result is a partial analogue of Corollary~\ref{C112}
and characterizes integral affine functions on $\hcVe{\infty}$.
\begin{Prop}\label{P114}
  For any integral affine function $\f$ on $\hcVe{\infty}$
  there exist free linear systems $|\fM_1|$ and $|\fM_2|$ on $\A^2$
  and an integer $n\ge 1$ such that 
  $\f=\frac1n(\log|\fM_1|-\log|\fM_2|)$.
\end{Prop}
\begin{proof}
  Pick an admissible compactification $X$ of $\A^2$ such that 
  $\f$ is associated to divisor $Z\in\Div_\infty(X)$.
  We may write $Z=Z_1-Z_2$, where $Z_i\in\Div_\infty(X)$ is 
  ample. For a suitable $n\ge 1$, $nZ_1$ and $nZ_2$ are
  very ample, and in particular base point free. 
  We can then take $|\fM_i|=|\cO_X(nZ_i)|$, $i=1,2$.
\end{proof}
It seems harder to describe the parametrization $\a$. 
While~\eqref{e151} implies
\begin{equation*}
  \a(v)\ge\sup_{\phi\in R\setminus 0}\frac{v(\phi)}{\ord_\infty(\phi)}
\end{equation*}
for any $v$, it is doubtful that equality holds 
in general.\footnote{In fact, P.~Mondal has given examples in~\texttt{arXiv:1301.3172} showing that 
  equality does not always hold.}

One can show that equality does hold when $v$ is a quasimonomial
valuation in the tight tree $\cV'_\infty$, to be defined shortly.
%
%
%
%
\subsection{The tight tree at infinity}\label{S276}
For the study of polynomial dynamics in~\S\ref{S104}, the 
full valuative tree at infinity is too large.
Here we will introduce a very interesting and useful subtree.
%
%
\begin{Def}
  The \emph{tight tree at infinity} is the subset $\cV'_\infty\subseteq\cVe{\infty}$
  consisting of semivaluations $v$ for which $A(v)\le0\le\a(v)$. 
\end{Def}
Since $\a$ is decreasing and $A$ is increasing in the partial ordering on 
$\cVe{\infty}$, it is clear that $\cV'_\infty$ is a subtree of $\cVe{\infty}$.
Similarly, $\a$ (resp.\ $A$) is lower semicontinuous (resp.\ upper semicontinuous)
on $\cVe{\infty}$, 
which implies that $\cV'_\infty$ is a closed subset of $\cVe{\infty}$.
It is then easy to see that $\cV'_\infty$ is a metric tree in the sense 
of~\S\ref{S131}.

Similarly, we define $\hcV'_\infty$ as the set of semivaluations $v\in\hcVe{\infty}$
satisfying $A(v)\le 0\le\a(v)$. Thus $\hcV'_\infty=\R_+^*\cVe{\infty}$
The subset $\hcV'_\infty\subset\BerkAtwo$ does not depend on 
the choice of embedding $\A^2\hookrightarrow\P^2$. In particular,
it is invariant under polynomial automorphisms of $\A^2$.
Further, $\hcV'_\infty$ is nowhere dense as it contains no curve semivaluations. 
Its closure is the union of itself and the trivial valuation $\triv_{\A^2}$. 
%
%
\subsubsection{Monomialization}\label{S152}
The next, very important result characterizes some of the ends of 
the tree $\cV'_\infty$.
\begin{Thm}\label{T105}
  Let $\ord_E$ be a divisorial valuation centered at infinity such that 
  $A(\ord_E)\le0=(\vE\cdot\vE)$.
  Then $A(\ord_E)=-1$ and there exist coordinates $(z_1,z_2)$
  on $\A^2$ in which $\ord_E$ is monomial with
  $\ord_E(z_1)=-1$ and $\ord_E(z_2)=0$.
\end{Thm}
This is proved in~\cite[Theorem~A.7]{eigenval}. Here we provide 
an alternative, more geometric proof. This proof uses the Line
Embedding Theorem and is the reason why we work in characteristic
zero throughout~\S\ref{S108}. 
(It is quite possible, however, that Theorem~\ref{T105} is 
true also over an algebraically closed field of positive characteristic).
\begin{proof}
  Let $X$ be an admissible compactification of $\A^2$ on  
  which $E$ is a prime at infinity. 
  The divisor $\vE\in\Div_\infty(X)$ is nef at infinity. 
  It is also effective, and hence nef, since $(\vE\cdot\vE)\ge 0$;
  see Proposition~\ref{P201}.

  Let $K_X$ be the canonical class of $X$. 
  We have $(\vE\cdot K_X)=A(\ord_E)-1<0$.
  By the Hirzebruch-Riemann-Roch Theorem we have 
  \begin{equation*}
    \chi(\cO_X(\vE))=\chi(\cO_X)+\frac12((\vE\cdot\vE)-(\vE\cdot K_X))
    >\chi(\cO_X)=1.
  \end{equation*}
  Serre duality yields $h^2(\cO_X)=h^0(\cO_X(K_X-\vE))=0$, so 
  since $h^1(\cO_X(\vE))\ge0$ we conclude that $h^0(\cO_X(\vE))\ge 2$.
  Thus there exists a nonconstant polynomial $\phi\in R$ that defines a global 
  section of $\cO_X(\vE)$. Since $\vE$ is effective, $\phi+t$
  is also a global section for any $t\in K$.
  
  Let $C_t$ be the closure in $X$ of the affine curve 
  $(\phi+t=0)\subset\A^2$. 
  For any $t$ we have $C_t=\vE$ in $\Pic(X)$, so 
  $(C_t\cdot E)=1$ and $(C_t\cdot F)=0$ for all primes $F$ at infinity
  different from $E$. 
  This implies that $C_t$ intersects $X\setminus\A^2$ at a unique
  point $\xi_t\in E$; this point is furthermore free on $E$, $C_t$ is 
  smooth at $\xi_t$, and the intersection is transverse. 
  Since $\ord_E(\phi)=(\vE\cdot\vE)=0$, the image of the 
  map $t\mapsto\xi_t$ is Zariski dense in $E$.
  
  For generic $t$, the affine curve $C_t\cap\A^2=(\phi+t=0)$ is smooth, hence 
  $C_t$ is smooth for these $t$. By adjunction,  $C_t$ is
  rational. 
  In particular, $C_t\cap\A^2$ is a smooth curve with one place at infinity.

  The Line Embedding Theorem by Abhyankar-Moh
  and Suzuki~\cite{AM,Suzuki} now shows that 
  there exist coordinates $(z_1,z_2)$ on $\A^2$ such that $\phi+t=z_2$. 
  We use these coordinates to define a compactification
  $Y\simeq\P^1\times\P^1$ of $\A^2$. Let $F$ be the irreducible
  compactification of $Y\setminus\A^2$ that intersects the strict
  transform of each curve $z_2=\mathrm{const}$. Then the birational
  map $Y\dashrightarrow X$ induced by the identity on $\A^2$
  must map $F$ onto $E$. It follows that $\ord_E=\ord_F$.
  Now $\ord_F$ is monomial in $(z_1,z_2)$ with 
  $\ord_F(z_1)=-1$ and $\ord_F(z_2)=0$.
  Furthermore, the 2-form $dz_1\wedge dz_2$ has a pole of order 2 along 
  $F$ on $Y$ so $A(\ord_F)=-1$.
  This completes the proof.
\end{proof}
%
%
\subsubsection{Tight compactifications}\label{S150}
We say that an admissible compactification $X$ of $\A^2$ is \emph{tight} if 
$|\Delta(X)|\subseteq\cV'_\infty$. 
Let $E_i$, $i\in I$ be the primes of $X$ at infinity. 
Since the parametrization $\a$ and the
log discrepancy $A$ are both affine on the simplices of $|\Delta(X)|$,
$X$ is tight iff $A(v_i)\le 0\le\a(v_i)$ for all $i\in I$.
In particular, this implies $(\vE_i\cdot\vE_i)\ge0$, so the 
the divisor $\vE_i\in\Div_\infty(X)$ is nef for all $i\in I$. 
Since every divisor in $\Div_\infty(X)$ that is nef at infinity is 
a positive linear combination of the $\vE_i$, we conclude
\begin{Prop}\label{P115}
  If $X$ is a tight compactification of $\A^2$,
  then the nef cone of $X$ is simplicial.
\end{Prop}
See~\cite{CPR1,CPR2,GM1,GM2,Monserrat} for other cases 
when the nef cone is known to be simplicial. 
For a general admissible compactification of $\A^2$ one would,
however, expect the nef cone to be rather complicated.
\begin{Lemma}\label{L127}
  Let $X$ be a tight compactification of $\A^2$ and 
  $\xi$ a closed point of $X\setminus\A^2$.
  Let $X'$ be the admissible compactification of $\A^2$ 
  obtained by blowing up $\xi$. Then $X'$ is tight unless $\xi$
  is a free point on a prime $E$ for which $\a_E=0$ or $A_E=0$.
\end{Lemma}
\begin{proof}
  If $\xi$ is a satellite point, then $X'$ is tight since $|\Delta(X')|=|\Delta(X)|$.
  
  Now suppose $\xi$ is a free point, belonging to a 
  unique prime on $E$. 
  Let $E'$ be the prime of $X'$ resulting from blowing up $\xi$. 
  Then $X'$ is tight iff $\a_{E'}:=(\vE'\cdot\vE')\ge0\ge A_{E'}$.
  But it follows from~\eqref{e152} that
  $A_{E'}=A_E+1$ and $\a_{E'}=\a_E-1$.
  Hence $\a_{E'}\ge0\ge A_{E'}$ unless $\a_E=0$ or $A_E=0$.
  The proof is complete.
\end{proof}
\begin{Cor}\label{C113}
  If $X$ is a tight compactification of $\A^2$ and 
  $v\in\hcV'_\infty$ is a divisorial valuation, then there exists
  a tight compactification $X'$ dominating $X$ such that 
  $v\in|\hDelta^*(X')|$.
\end{Cor}
\begin{proof}
  In the proof we shall repeatedly use the analogues at infinity
  of the results in~\S\ref{S163}, in particular Lemma~\ref{L123}.
  
  We may assume $v=\ord_E$ for some prime $E$ at infinity.
  By Lemma~\ref{L123}, the valuation $w:=r_X(v)$ is divisorial and 
  $b(w)$ divides $b(v)$. We argue by induction on the integer $b(v)/b(w)$.

  By the same lemma we can 
  find an admissible compactification $X_0$ dominating $X$
  such that $|\hDelta^*(X_0)|=|\hDelta^*(X)|$,
  and $w$ is contained in a one-dimensional
  cone in $|\hDelta^*(X_0)|$. 
  Then the center of $v$ on $X_0$ is a free point $\xi_0$. 
  Let $X_1$ be the blowup of $X_0$ in $\xi_0$.
  Note that since $v\ne w$ we have
  $\a(w)>\a(v)\ge0\ge A(v)>A(w)$, so by Lemma~\ref{L127}
  the compactification $X_1$ is tight. 

  If $v\in|\hDelta^*(X_1)|$ then we are done.
  Otherwise, set $v_1=r_{X_1}(v)$. If the center $\xi_1$ of $v$ on $X_1$
  is a satellite point, then it follows from Lemma~\ref{L123}
  that $b(v_1)>b(v_0)$. 
  If $b(w)=b(v)$, this is impossible and if $b(w)<b(v)$,
  we are done by the inductive hypothesis.

  The remaining case is when $\xi_1$ is a free point on $E_1$, the 
  preimage of $\xi_0$ under the blowup map. 
  We continue this procedure: assuming that
  the center of $v$ on $X_j$ is a free point $\xi_j$, we let
  $X_{j+1}$ be the blowup of $X_j$ in $\xi_j$.
  By~\eqref{e152} we have $A_{E_n}=A_{E_0}+n$.
  But $A_{E_n}\le0$ so the procedure must stop 
  after finitely many steps. When it stops, we either have 
  $v\in|\hDelta^*(X_n)|$ or the center of $v$ on $X_n$ is a
  satellite point. In both cases the proof is complete in
  view of what precedes.
\end{proof}
\begin{Cor}
  If $X$ is a tight compactification of $\A^2$ and $f:\A^2\to\A^2$ is
  a polynomial automorphism, then there exists a tight compactification
  $X'$ such that the birational map $X'\dashrightarrow X$ induced by 
  $f$ is \emph{regular}.
\end{Cor}
\begin{proof}
  Let $E_i$, $i\in I$ be the primes of $X$ at infinity. Now $f^{-1}$ maps
  the divisorial valuations $v_i:=\ord_{E_i}$ to divisorial valuations 
  $v'_i=\ord_{E'_i}$. We have $v'_i\in\hcV'_\infty$, so after a repeated application
  of Corollary~\ref{C113} we find an admissible compactification $X'$
  of $\A^2$ such that $v'_i\in|\hDelta^*(X')|$ for all $i\in I$. 
  But then it is easy to check that $f:X'\to X$ is regular.
\end{proof}
\begin{Cor}
  Any two tight compactifications can be dominated by a third,
  so the set of tight compactifications is a directed set.
  Furthermore, the retraction maps  $r_X:\hcVe{\infty}\to|\hDelta^*(X)|$
  give rise to homeomorphisms
  \begin{equation*}
    \hcV'_\infty\simto\varprojlim_X|\hDelta^*(X)|
    \quad\text{and}\quad
    \cV'_\infty\simto\varprojlim_X|\Delta(X)|,
  \end{equation*}
  where $X$ ranges over all tight compactifications of $\A^2$.  
\end{Cor}
%
%
%
%
\subsection{Other ground fields}\label{S246}
Throughout the section we assumed that the ground field was
algebraically closed and of characteristic zero. Let us briefly
discuss what happens when one or more of these assumptions 
are not satisfied.

First suppose $K$ is algebraically closed but of characteristic $p>0$.
Everything in~\S\ref{S108} goes through, except for the proof of
the monomialization theorem, Theorem~\ref{T105}, which relies
on the Line Embedding Theorem. On the other hand, it is quite possible
that the proof of Theorem~\ref{T105} can be modified to work also
in characteristic $p>0$.

Now suppose $K$ is not algebraically closed. There are two
possibilities for studying the set of semivaluations in $\BerkAtwo$
centered at infinity.
One way is to pass to the algebraic closure $K^a$. 
Let $G=\Gal(K^a/K)$ be the Galois group. 
Using general theory we have an identification 
$\BerkAtwo(K)\simeq\BerkAtwo(K^a)/G$ and $G$ preserves the
open subset $\hcVe{\infty}(K^a)$ of semivaluations centered at infinity.
Any embedding $\A^2(K)\hookrightarrow\P^2(K)$ induces an embedding
$\A^2(K^a)\hookrightarrow\P^2(K^a)$ and allows us to define subsets
$\cVe{\infty}(K)\subseteq\hcVe{\infty}(K)$ and 
$\cVe{\infty}(K^a)\subseteq\hcVe{\infty}(K^a)$.
Each $g\in G$ maps $\cVe{\infty}(K^a)$ into itself and preserves
the partial ordering parametrizations
as well as the parametrizations $\a$ and $A$ and the 
multiplicity $m$. 
Therefore, the quotient $\cVe{\infty}(K)\simeq\cVe{\infty}(K^a)/G$
also is naturally a tree that we equip with a metric that takes into account 
the degree of the map $\cVe{\infty}(K^a)\to\cVe{\infty}(K)$. 

Alternatively, we can obtain the metric tree structure directly
from the dual graphs of the admissible compactifications by
keeping track of the residue fields of the closed points being blown up.
%
%
%
%
\subsection{Notes and further references}
The valuative tree at infinity was introduced in~\cite{eigenval}
for the purposes of studying the dynamics at infinity of 
polynomial mappings of $\C^2$ (see the next section).
It was not explicitly identified as a subset of the Berkovich
affine plane over a trivially valued field.

In~\cite{eigenval}, the tree structure of $\cVe{\infty}$ was deduced
by looking at the center on $\P^2$ of a semivaluation 
in $\cVe{\infty}$. Given a closed point $\xi\in\P^2$, the 
semivaluations having center at $\xi$ form a tree 
(essentially the valuative tree at $\xi$ but normalized by
$v(L_\infty)=1$). 
By gluing these 
trees together along $\ord_\infty$ we see that $\cVe{\infty}$
itself is a tree. The geometric approach here, using admissible
compactifications, seems more canonical and amenable to
generalization to higher dimensions.

Just as with the valuative tree at a point, I have allowed myself
to change the notation from~\cite{eigenval}. Specifically,
the valuative tree at infinity is (regrettably) denoted $\cV_0$
and the tight tree at infinity is denoted $\cV_1$. The notation
$\cVe{\infty}$ and $\cV'_\infty$ seems more natural.
Further, the valuation $\ord_\infty$ is denoted $-\deg$
in~\cite{eigenval}.

The tight tree at infinity $\cV'_\infty$ 
was introduced in~\cite{eigenval}
and tight compactifications in~\cite{dyncomp}.
They are both very interesting notions. The tight tree
was studied in~\cite{eigenval} 
using key polynomials, more or less in the 
spirit of Abhyankar and Moh~\cite{AM}. While key polynomials
are interesting, they are notationally cumbersome as they
contain a lot of combinatorial information and they depend on
a choice of coordinates, something that I have striven to avoid here.

As indicated in the proof of Theorem~\ref{T105}, it is possible to study the tight 
tree at infinity using the basic theory for compact surfaces.
In particular, while the proof of the structure result for $\cV'_\infty$
in~\cite{eigenval} used the Line Embedding Theorem in a 
crucial way (just as in Theorem~\ref{T105}) one can use the 
framework of tight compactifications together with surface theory to
give a proof of the Line Embedding Theorem.
(It should be mentioned, however, that by now there are quite a few
proofs of the line embedding theorem.)

One can also prove Jung's theorem, on the structure $\Aut(\C^2)$
using the tight tree at infinity. It would be interesting to see
if there is a higher-dimensional version of the tight tree at
infinity, and if this space could be used to shine some 
light on the wild automorphisms of $\C^3$, the existence of 
which was proved by Shestakov and Umirbaev in~\cite{UmSh}.

The log discrepancy used here is a slight variation of the 
standard notion in algebraic geometry (see~\cite{graded})
but has the advantage of not depending on the 
choice of compactification. 
If we fix an embedding $\A^2\hookrightarrow\P^2$ and 
$A_{\P^2}$ denotes the usual log discrepancy on $\P^2$, 
then we have $A(v)=A_{\P^2}(v)-3v(|\fL|)$. 
%
%
%
%
%
%
\newpage
\section{Plane polynomial dynamics at infinity}\label{S104}
We now come to the third type of dynamics on Berkovich spaces:
the dynamics at infinity of polynomial mappings of $\A^2$.
The study will be modeled on the dynamics near a (closed) fixed 
point as described in~\S\ref{S107}. We will refer to the latter
situation as the local case.
%
%
%
%
\subsection{Setup}
Let $K$ is an algebraically closed field 
of characteristic zero, equip\-ped with the trivial valuation.
(See~\S\ref{S272} for the case of other ground fields.)
Further, $R$ and $F$ are the coordinate ring and
function field of the affine plane $\A^2$ over $K$.
Recall that the Berkovich affine plane $\BerkAtwo$ 
is the set of semivaluations on $R$ that restrict to the trivial valuation
on $K$.
%
%
%
%
\subsection{Definitions and results}
We keep the notation from~\S\ref{S108} and
consider a polynomial mapping 
$f:\A^2\to\A^2$, which we assume to be \emph{dominant} 
to avoid degenerate cases.
Given an embedding $\A^2\hookrightarrow\P^2$,
the degree $\deg f$ is defined as the degree of 
the curve $\deg f^*\ell$ for a general line 
$\ell\in|\fL|$.

The degree growth sequence 
$(\deg f^n)_{n\ge0}$ is submultiplicative, 
\begin{equation*}
  \deg f^{n+m}\le\deg f^n\cdot\deg f^m,
\end{equation*}  
and so the limit 
\begin{equation*}
  d_\infty=\lim_{n\to\infty}(\deg f^n)^{1/n}
\end{equation*}
is well defined.
Since $f$ is assumed dominant, $\deg f^n\ge 1$ for all $n$, 
hence $d_\infty\ge 1$. 
\begin{Exer}
  Verify these statements!
\end{Exer}
\begin{Example}
  If $f(z_1,z_2)=(z_2,z_1z_2)$, then $\deg f^n$ is the $(n+1)$th Fibonacci number
  and $d_\infty=\frac12(\sqrt{5}+1)$ is the golden mean.
\end{Example}
\begin{Example}
  For $f(z_1,z_2)=(z_1^2,z_1z_2^2)$, $\deg f^n=(n+2)2^{n-1}$ and $d_\infty=2$.
\end{Example}
\begin{Exer}
  Compute $d_\infty$ for a skew product 
  $f(z_1,z_2)=(\phi(z_1),\psi(z_1,z_2))$.
\end{Exer}
Here is the result that we are aiming for.
\begin{ThmC}
  The number $d_\infty=d_\infty(f)$ is a quadratic integer: there exist
  $a,b\in\Z$ such that $d_\infty^2=ad_\infty+b$. 
  Moreover, we are in exactly one of the following two cases:
  \begin{itemize}
  \item[(a)]
    there exists $C>0$ such that 
    $d_\infty^n\le\deg f^n\le Cd_\infty^n$ for all $n$;
  \item[(b)]
    $\deg f^n\sim nd_\infty^n$ as $n\to\infty$.
  \end{itemize}
  Moreover, case~(b) occurs iff $f$, after conjugation 
  by a suitable polynomial automorphism of $\A^2$,
  is a skew product of the form
  \begin{equation*}
    f(z_1,z_2)=(\phi(z_1),\psi(z_1)z_2^{d_\infty}+O_{z_1}(z_2^{d_\infty-1})),
  \end{equation*}
  where $\deg\phi=d_\infty$ and $\deg\psi>0$.
\end{ThmC}
The behavior of the degree growth sequence does not
depend in an essential way on our choice of 
embedding $\A^2\hookrightarrow\P^2$.
To see this, fix such an embedding, 
let $g:\A^2\to\A^2$ be  a polynomial automorphism and 
set $\tf:=g^{-1}fg$. Then 
$\tf^n=g^{-1}f^ng$, $f^n=g\tf^ng^{-1}$ and so
\begin{equation*}
  \frac1{\deg g\deg g^{-1}}
  \le \frac{\deg{\tf^n}}{\deg f^n}
  \le \deg g\deg g^{-1}
\end{equation*}
for all $n\ge1$.
As a consequence, when proving Theorem~C, we may
conjugate by polynomial automorphisms of $\A^2$, if necessary.
%
%
%
%
\subsection{Induced action}
The strategy for proving Theorems~C is superficially very
similar to the local case explored in~\S\ref{S107}.
Recall that $f$ extends to a map
\begin{equation*}
  f:\BerkAtwo\to\BerkAtwo,
\end{equation*} 
given by $f(v)(\phi):= v(f^*\phi)$.

We would like to study the dynamics of $f$ at infinity. 
For any admissible compactification $X$
of $\A^2$, $f$ extends to a rational map 
$f:X\dashrightarrow\P^2$.
Using resolution of singularities we can find $X$ such that
$f:X\to\P^2$ is a morphism. There are then two cases:
either $f(E)\subseteq L_\infty$ for every prime $E$ of $X$ at infinity,
or there exists a prime $E$ such that $f(E)\cap\A^2\ne\emptyset$.
The first case happens iff $f$ is \emph{proper}.

Recall that $\hcVe{\infty}$ denotes the set of semivaluations 
in $\BerkAtwo$ having center at infinity. It easily follows that $f$
is proper iff $f(\hcVe{\infty})\subseteq\hcVe{\infty}$. Properness
is the analogue of finiteness in the local case.
%
%
\subsubsection{The proper case}
When $f$ is proper, if induces a selfmap
\begin{equation*}
  f:\hcVe{\infty}\to\hcVe{\infty}.
\end{equation*} 
Now $\hcVe{\infty}$ is the pointed cone over the valuative 
tree at infinity $\cVe{\infty}$, whose elements are normalized by 
the condition $v(|\fL|)=-1$.
As in the local case, we can break the action of 
$f$ on $\hcVe{\infty}$ into two parts: the induced dynamics
\begin{equation*}
  f_\bullet:\cVe{\infty}\to\cVe{\infty},
\end{equation*}
and a multiplier $d(f,\cdot):\cVe{\infty}\to\R_+$.
Here 
\begin{equation*}
  d(f, v)=-v(f^*|\fL|),
\end{equation*}  
Further, $f_\bullet$ is defined by
\begin{equation*}
  f_\bullet v=\frac{f(v)}{d(f, v)}.
\end{equation*}
The break-up of the action is compatible with the dynamics
in the sense that $(f^n)_\bullet=(f_\bullet)^n$ and
\begin{equation*}
 d(f^n, v)=\prod_{i=0}^{n-1}d(f, v_i),
 \quad\text{where $v_i=f^i_\bullet v$}.
\end{equation*} 
Recall that $\ord_\infty\in\cVe{\infty}$ is the valuation given by
$\ord_\infty(\phi)=-\deg(\phi)$ for any polynomial $\phi\in R$.
We then have 
\begin{equation*}
  \deg f^n=d(f^n,\ord_\infty)=\prod_{i=0}^{n-1}d(f, v_i),
  \quad\text{where $v_i=f^i_\bullet\ord_\infty$}.
\end{equation*}
Now $v_i\ge\ord_\infty$ on $R$, so it follows that 
$\deg f^n \le(\deg f)^n$ as we already knew. 
The multiplicative cocycle $d(f,\cdot)$ is the main
tool for studying the submultiplicative sequence $(\deg f^n)_{n\ge 0}$.
%
%
\subsubsection{The non-proper case}
When $f:\A^2\to\A^2$ is dominant but not necessarily proper,
there exists at least one divisorial valuation $v\in\hcVe{\infty}\subseteq\BerkAtwo$
for which $f(v)\in\BerkAtwo\setminus\hcVe{\infty}$. 
We can view $f:\hcVe{\infty}\dashrightarrow\hcVe{\infty}$
as a partially defined map. Its domain of definition is the 
open set $\hD_f\subseteq\hcVe{\infty}$  consisting of semivaluations
for which there exists an affine function $L$ with
$v(f^*L)<0$. Equivalently, if we as before define 
$d(f, v)=-v(f^*|\fL|)$, then $\hD_f=\{d(f,\cdot)>0\}$.
On $D_f:=\hD_f\cap\cVe{\infty}$ we define $f_\bullet$ as before,
namely $f_\bullet v=f(v)/d(f, v)$.

Notice that $D_{f^n}=\bigcap_{i=0}^{n-1}f_\bullet^{-i}D_f$, so the
domain of definition of $f_\bullet^n$ decreases as $n\to\infty$.
One may even wonder whether the intersection $\bigcap_nD_{f^n}$
is empty. However, a moment's reflection reveals that $\ord_\infty$ 
belongs to this intersection. More generally, it is not hard to see
that the set of valuations $v\in\cVe{\infty}$ for which 
$v(\phi)<0$ for all nonconstant polynomials $\phi$, is a subtree of
$\cVe{\infty}$ contained in $D_f$ and invariant under $f$,
for any dominant polynomial mapping $f$. 

For reasons that will become apparent later, we will in fact
study the dynamics on the even smaller subtree, namely the 
tight subtree $\cVe{\infty}'\subseteq\cVe{\infty}$
defined in~\S\ref{S276}. We shall see shortly that 
$f_\bullet\cVe{\infty}'\subseteq\cVe{\infty}'$, so we have a natural
induced dynamical system on $\cVe{\infty}'$ for 
any dominant polynomial mapping $f$.
%
%
%
%
\subsection{Invariance of the tight tree $\cVe{\infty}'$}\label{S147}
Theorem~B, the local counterpart to Theorem~C,
follows easily under the
additional assumption (not always satisfied) that there exists 
a quasimonomial valuation $v\in\cV_0$ such that 
$f_\bullet v= v$. Indeed, such a valuation satisfies
\begin{equation*}
  \ord_0\le v\le\a v,
\end{equation*}
where $\a=\a(v)<\infty$. If $f(v)=c v$,
then this gives $c=c_\infty$ and 
$\a^{-1}c_\infty\le c(f^n)\le c_\infty^n$. Moreover,
the inclusion $c_\infty\Gamma_v=\Gamma_{f(v)}\subseteq\Gamma_v$
implies that $c_\infty$ is a quadratic integer. See~\S\ref{S148}.

In the affine case, the situation is more complicated. We cannot
just take \emph{any} quasimonomial fixed point $v$ for $f_\bullet$.
For a concrete example, consider the product map
$f(z_1,z_2)=(z_1^3,z_2^2)$ and let $v$ be the monomial 
valuation with weights $v(z_1)=0$, $v(z_2)=-1$. 
Then $f(v)=2 v$, whereas $d_\infty=3$. The problem here is that while 
$v\ge\ord_\infty$, the reverse inequality 
$v\le C\ord_\infty$ does not hold for any constant $C>0$.

The way around this problem is to use the tight tree $\cVe{\infty}'$
introduced in~\S\ref{S276}. Indeed, if $v\in\cVe{\infty}'$ is
quasimonomial, then either there exists $\a=\a(v)>0$
such that $\a^{-1}\ord_\infty\le v\le\ord_\infty$ on $R$, or
$v$ is monomial in suitable coordinates on $\A^2$, see 
Theorem~\ref{T105}.
As the example above shows, the latter case still has to be treated
with some care.

We start by showing that the tight tree is invariant.
\begin{Prop}\label{P117}
  For any dominant polynomial mapping 
  $f:\A^2\to\A^2$
  we have $f(\hcV'_\infty)\subseteq\hcV'_\infty$.
  In particular, $\cV'_\infty\subseteq D_f$ and 
  $f_\bullet\cV'_\infty\subseteq\cV'_\infty$.
\end{Prop}
\begin{proof}[Sketch of proof]
  It suffices to prove that if $v\in\hcV'_\infty$
  is divisorial, then $f(v)\in\hcV'_\infty$.
  After rescaling, we may assume $v=\ord_E$.
  Arguing  using numerical invariants as in~\S\ref{S118},
  we show that $f(v)$ is divisorial, of the form 
  $f(v)=r\ord_{E'}$ for some prime divisor $E'$ on $\A^2$
  (a priori not necessarily at infinity).

  We claim that the formula
  \begin{equation}\label{e156} 
    A(f(v))=A(v)+v(Jf)
  \end{equation}
  holds, where $Jf$ denotes the Jacobian determinant of $f$. 
  Note that the assumption $\a(v)\ge 0$ implies $v(Jf)\le 0$
  by~\eqref{e151}. Together with the assumption $A(v)\le0$,
  we thus see that $A(f(v))\le 0$. In particular, the 2-form $\omega$
  on $\A^2$ has a pole along $E'$, which implies that $E'$ must 
  be a prime at infinity. 

  Hence $f(v)\in\hcVe{\infty}$ and $A(f(v))\le 0$. It remains to
  prove that $\a(f(v))\ge0$. Let $X'$ be an admissible
  compactification of $\A^2$ in which $E'$ is a prime at infinity and 
  pick another compactification $X$ of $\A^2$ such that the induced
  map $f:X\to X'$ is regular. 
  The divisors $\vE\in\Div_\infty(X)$ and $\vE'\in\Div_\infty(X')$
  are both nef at infinity and satisfies $f_*\vE=r\vE'$.
  Since $(\vE\cdot\vE)=\a(v)\ge0$, $\vE$ is effective (and hence
  nef). As a consequence, $\vE'=r^{-1}f_*\vE$ is effective and
  hence nef. In particular, $\a(f(v))=r^2(\vE'\cdot\vE')\ge0$,
  which completes the proof.

  Finally we prove~\eqref{e156}. Write $A_E=A(\ord_E)$ and 
  $A_{E'}=A(\ord_{E'})$. Recall that $\omega$ is a 
  nonvanishing 2-form on $\A^2$. Near $E'$ it has a zero of order
  $A_{E'}-1$. From the chain rule, and the fact that 
  $f(\ord_E)=r\ord_{E'}$, it follows that $f^*\omega$ has a 
  zero of order $r-1+r(A_{E'}-1)=rA_{E'}-1$ along $E$.
  On the other hand we have 
  $f^*\omega=Jf\cdot\omega$ in $\A^2$
  and the right hand side vanishes to order 
  $\ord_E(Jf)+A_E-1$ along $E$. This concludes the proof.
\end{proof}
%
%
%
%
\subsection{Some lemmas}\label{S149}
Before embarking on the proof of Theorem~C, let us 
record some useful auxiliary results.
\begin{Lemma}\label{L105}
  Let $\phi\in R$ be a polynomial, $X$ an admissible compactification
  of $\A^2$ and $E$ a prime of $X$ at infinity. Let $C_X$ be the closure
  in $X$ of the curve $\{\phi=0\}$ in $\A^2$ and assume that $C_X$
  intersects $E$. Then $\deg p\ge b_E$, where $b_E:=-\ord_E(|\fL|)$.
\end{Lemma}
\begin{proof}
  This follows from elementary intersection theory. 
  Let $\pi:X\to\P^2$ be the birational morphism induced by the 
  identity on $\A^2$ and let $C_{\P^2}$ be the closure in $\P^2$
  of the curve $\{\phi=0\}\subseteq\A^2$. Then $\ord_E(\pi^*L_\infty)=b_E$.
  Assuming that $C_X$ intersects $E$, we get
  \begin{equation*}
    b_E
    \le b_E(C_X\cdot E)
    \le(C_X\cdot\pi^*L_\infty)
    =(C_{\P^2}\cdot L_\infty)
    =\deg p,
  \end{equation*}
  where the first equality follows from the projection
  formula and the second from B\'ezout's Theorem.
\end{proof}
Applying Lemma~\ref{L105} and Lemma~\ref{L126} to 
$\phi=f^*L$, for $L$ a general affine function, we obtain
\begin{Cor}\label{C101}
  Let $f:\A^2\to\A^2$ be a dominant polynomial mapping, 
  $X$ an admissible compactification
  of $\A^2$ and $E$ a prime of $X$ at infinity.
  Assume that $\deg(f)<b_E$. Then 
  $d(f, v)=d(f, v_E)$ for all $v\in\cVe{\infty}$
  such that $r_X(v)=v_E$.
\end{Cor}
%
%
%
%
\subsection{Proof of Theorem~C}
If we were to follow the proof in the local case, we would
pick a log resolution at infinity of the linear system 
$f^*|\fL|$ on $\P^2$. 
By this we mean an admissible compactification 
$X$ of $\A^2$ such that the strict transform of $f^*|\fL|$
to $X$ has no base points on $X\setminus\A^2$. 
Such an admissible compactification exists by resolution
of singularities. At least when $f$ is proper, we get a well defined
selfmap $r_Xf_\bullet:|\Delta(X)|\to|\Delta(X)|$. 
However, a fixed point $v$ of this map does not
have an immediate bearing on Theorem~C. 
Indeed, we have seen in~\S\ref{S147} that even when $v$
is actually fixed by $f_\bullet$, so that $f(v)=d v$ for some 
$d>0$, it may happen that $d<d_\infty$. 

One way around this problem would be to ensure that the 
compactification $X$ is tight, in the sense of~\S\ref{S150}. 
Unfortunately, it is not always
possible, even for $f$ proper, to find a tight $X$ that defines
a log resolution of infinity of $f^*|\fL|$.

Instead we use a recursive procedure. 
The proof below in fact works also when $f$ is merely
dominant, and not necessarily proper.
Before starting the procedure, let us write down a 
few cases where we actually obtain a proof of Theorem~C.
\begin{Lemma}\label{L106}
  Let $X$ be a tight compactification of $\A^2$ with associated
  retraction $r_X:\cVe{\infty}\to|\Delta(X)|$. Consider a fixed point 
  $v\in|\Delta(X)|$ of the induced selfmap 
  $r_Xf_\bullet:|\Delta(X)|\to|\Delta(X)|$. 
  Assume that we are in one of the following three situations:
  \begin{itemize}
  \item[(a)]
    $f_\bullet v= v$ and $\alpha(v)>0$;
  \item[(b)]
    $f_\bullet v\ne v$, 
    $\alpha(v)>0$, $v$ is divisorial 
    and $b(v)>\deg(f)$;
 \item[(c)]
    $\alpha(v)=0$ and $(r_Xf_\bullet)^n w\to v$ as $n\to\infty$
    for $w\in|\Delta(X)|$ close to $v$.
  \end{itemize}
  Then Theorem~C holds.
\end{Lemma}
\begin{proof}
  Case~(a) is treated as in the local situation. 
  Since $\a:=\a(v)>0$ we have 
  $\a^{-1} v\le\ord_\infty\le v$ on $R$.
  Write $f(v)=d v$, where $d=d(f, v)>0$. Then
  \begin{equation*}
    \deg f^n
    =-\ord_\infty(f^{n*}|\fL|)
    \le-\a^{-1} v(f^{n*}|\fL|)
    =-\a^{-1}d^n v(|\fL|)
    =\a^{-1}d^n.
  \end{equation*}
  Similarly, $\deg f^n\ge d^n$. This proves statement~(a) of Theorem~C
  (and that $d_\infty=d$). The fact that $d=d_\infty$ is a quadratic
  integers
  is proved exactly as in the local case, using value groups.
  Indeed, one obtains $d\Gamma_v\subseteq\Gamma_v$. 
  Since $\Gamma_v\simeq\Z$ or 
  $\Gamma_v\simeq\Z\oplus\Z$, $d$ must be a quadratic integer.

  Next we turn to case~(b). 
  By the analogue of Lemma~\ref{L123} we may assume that 
  the center of $f_\bullet v$ on $X$
  is a free point $\xi$ of $E$. By Corollary~\ref{C101} we have
  $d(f,\cdot)\equiv d:=d(f, v)$ on $U(\xi)$. 
  As in the local case, this implies that 
  $f_\bullet\overline{U(\xi)}\subseteq U(\xi)$, $d(f^n, v)=d^n$,
  $d^n\le\deg(f^n)\le\alpha^{-1} d^n$, so that we are in case~(a)
  of Theorem~C, with $d_\infty=d$. 
  The fact that $d=d_\infty$ is a quadratic integer follows from
  $d\Gamma_v\subseteq\Gamma_v\simeq\Z$.
  In fact, $d\in\N$.

  Finally we consider case~(c). Recall that the statements of 
  Theorem~C are invariant under conjugation
  by polynomial automorphisms.
  Since $X$ is tight and $\a(v)=0$, we may 
  by Theorem~\ref{T105}  choose coordinates 
  $(z_1,z_2)$ on $\A^2$ in which $v$ is monomial
  with $v(z_1)=0$, $v(z_2)=-1$.
  Since $v$ is an end in the $f_\bullet$-invariant tree 
  $\cV'_\infty$ and $r_Xf_\bullet v= v$,
  we must have $f_\bullet v= v$.
  In particular, $f_\bullet v(z_1)=0$, which implies that
  $f$ is a skew product of the form
  \begin{equation*}
    f(z_1,z_2)=(\phi(z_1),\psi(z_1)z_2^d+O_{z_1}(z_2^{d-1})),
  \end{equation*}
  where $d\ge1$ and  $\phi,\psi$ are nonzero polynomials.
  The valuations in $|\Delta(X)|$ close to $v$ must 
  also be monomial valuations, of the form $w_t$,
  with $w_t(z_1)=-t$ and $w_t(z_2)=-1$, where $0\le t\ll1$.
  We see that $f(w_t)(z_1)=-t\deg\phi$ and 
  $f(w_t)(z_2)=-(d+t\deg q)$. When $t$ is irrational,
  $f_\bullet w_t$ must be monomial, of the form $ w_{t'}$,
  where $t'=t\frac{\deg p}{d+t\deg q}$. By continuity, 
  this relationship must hold for all real $t$, $0\le t\ll1$.
  By our assumptions, $t'<t$ for $0<t\ll1$. This implies
  that either $\deg p<d$ or that $\deg p=d$, $\deg q>0$.
  It is then clear that $d_\infty=d$ is an integer, proving 
  the first statement in Theorem~C.
  Finally, from a direct computation, that we leave as an 
  exercise to the reader, it follows that $\deg f^n\sim nd^n$.
\end{proof}
The main case not handled by Lemma~\ref{L106} is the
case~(b) but without the assumption that $b_E>\deg f$.
In this case we need to blow up further.
\begin{Lemma}\label{L107}
  Let $X$ be a tight compactification of $\A^2$ with associated
  retraction $r_X:\cVe{\infty}\to|\Delta(X)|$. Assume that 
  $v= v_E=b_E^{-1}\ord_E\in|\Delta(X)|$ 
  is a divisorial valuation such that 
  $r_Xf_\bullet v_E= v_E$ but $f_\bullet v_E\ne v_E$.
  Then there exists a tight compactification $X'$ of $\A^2$ 
  dominating $X$ and a valuation
  $v'\in|\Delta(X')|\setminus|\Delta(X)|$
  such that $r_{X'}f_\bullet v'= v'$ and such that we are in 
  one of the following cases:
  \begin{itemize}
  \item[(a)]
    $f_\bullet v'= v'$ and $\alpha(v')>0$;
  \item[(b)]
    $f_\bullet v'\ne v'$, 
    $v'$ is divisorial, $\alpha(v')>0$ and $b(v')>b(v)$;
  \item[(c)]
    $\alpha(v')=0$ and $(r_{X'}f_\bullet)^n w\to v'$ as $n\to\infty$
    for $w\in|\Delta(X)|$ close to $v'$.
  \end{itemize}
\end{Lemma}
It is clear that repeated application of Lemma~\ref{L106} and
Lemma~\ref{L107} leads to a proof of Theorem~C. 
The only thing remaining is to prove Lemma~\ref{L107}.
\begin{proof}
  Write $v_0= v$. 
  By (the analogue at infinity of) Lemma~\ref{L123} 
  we may find an admissible compactification
  $X_0$ dominating $X$, 
  such that $|\Delta_0|:=|\Delta(X_0)|=|\Delta(X)|$, 
  $r_0:=r_{X_0}=r_X$ and such that the center of 
  $v_0= v$ on $X_0$ is a prime $E_0$ of $X_0$ at infinity. 
  Since $f_\bullet v_0\ne v_0$, the center of $f_\bullet v_0$
  must be a free point $\xi_0\in E_0$.
  Let $X_1$ be the blowup of $X_0$ at $\xi_0$, 
  $E_1$ the exceptional divisor and $v_1=b_1^{-1}\ord_{E_1}$
  the associated divisorial valuation. Note that 
  $b_1=b_0$ and $\a(v_1)=\a(v_0)-b_0^{-1}$ by~\eqref{e155}.
  In particular, $X_1$ is still tight.
  Write $|\Delta_1|=|\Delta(X_1)|$ and $r_1:=r_{X_1}$.
  We have 
  $r_1f_\bullet v_0\in|\Delta_1|\setminus|\Delta_0|=\,] v_0, v_1]$.
  Thus there are two cases:
 \begin{itemize}
  \item[(1)]
    there exists a fixed point 
    $v'\in\,] v_0, v_1[$ for $r_1f_\bullet$;
  \item[(2)]
    $(r_1f_\bullet)^n\to v_1=r_1f_\bullet v_1$ as $n\to\infty$;
  \end{itemize}
  Let us first look at case~(1). Note that $\a(v')>\a(v_1)\ge0$.
  If $f_\bullet v'= v'$, then we are in situation~(a) and the
  proof is complete. Hence we may assume that 
  $f_\bullet v'\ne v'$.
  Then $v'$ is necessarily divisorial. 
  By Lemma~\ref{L123} we have $b(v')>b_0=b(v)$.
  We are therefore in situation~(b), so the 
  proof is complete in this case.

  It remains to consider case~(2). 
  If $\a(v_1)=0$, then we set $X'=X_1$, 
  $v'= v_1$ and we are in situation~(c).
  We can therefore assume that $\a(v_1)>0$.
  If $f_\bullet v_1= v_1$, then we set
  $X'=X_1$, $v'= v_1$ and we are in 
  situation~(a).  
  If $f_\bullet v_1\ne v_1$, so that the center of 
  $f_\bullet v_1$ is a free point $\xi_1\in E_1$,
  then we can repeat the procedure above.
  Let $X_2$ be the blowup of $X_1$ at $\xi_1$, let $E_2$
  be the exceptional divisor and $v_2=b_2^{-1}\ord_{E_2}$
  the associated divisorial valuation. 
  We have $b_2=b_1=b$ and $\a(v_2)=\a(v_1)-b^{-1}=\a(v)-2b^{-1}$ 
  by~\eqref{e155}.

  Continuing the procedure above must eventually lead us to
  the situation in~(a) or~(c). Indeed, all of our compactifications
  are tight, so in particular all valuations $v_n$ satisfy
  $\a(v_n)\ge0$. But $\a(v_n)=\a(v)-nb^{-2}$. 
  This completes the proof.
\end{proof}
%
%
%
%
\subsection{Further properties}\label{S170}
The presentation above was essentially optimized 
to give a reasonably short proof of Theorem~C. 
While it is beyond the scope of these notes to present the details, let 
us briefly summarize some further results from~\cite{eigenval,dyncomp}.
Let $f:\A^2\to\A^2$ be a polynomial mapping and write $f$
also for its extension $f:\BerkAtwo\to\BerkAtwo$.

To begin, $f$ interacts well with the classification of points: 
if $v\in\hcVe{\infty}$ and $f(v)\in\hcVe{\infty}$ then $f(v)$ is
of the same type as $v$ (curve, divisorial, irrational or
infinitely singular). This is proved using numerical invariants
in the same way as in~\S\ref{S118}.

At least when $f$ is proper the induced map 
$f_\bullet:\cVe{\infty}\to\cVe{\infty}$ is continuous, finite and 
open. This follows from general results on Berkovich spaces,
just as in Proposition~\ref{P107}. As a consequence, the 
general results on tree maps in~\S\ref{S130} apply.

In~\cite{eigenval,dyncomp}, the existence of an \emph{eigenvaluation} 
was emphasized. This is a valuation $v\in\cVe{\infty}$ such that 
$f(v)=d_\infty v$. One can show from general tree arguments that 
there must exist such a valuation in the tight tree $\cV'_\infty$.
The proof of Theorem~C gives an alternative construction of 
an eigenvaluation in $\cV'_\infty$. 

Using a lot more work, the global dynamics on $\cV'_\infty$ is 
described in~\cite{dyncomp}. Namely, the set $\cT_f$ 
of eigenvaluations in $\cV'_\infty$ is either a singleton 
or a closed interval. (The ``typical'' case is that of 
a singleton.) In both cases we have $f^n_\bullet v\to\cT_f$ 
as $n\to\infty$, for all but at most one $v\in\cV'_\infty$.
This means that the dynamics on the tight tree 
$\cVe{\infty}$ is globally contracting, as opposed to 
a rational map on the Berkovich projective line, which
is globally expanding.

Using the dynamics on $\cV'_\infty$, the cocycle $d(f^n,v)$ can be very well
described: for any $v\in\cV'_\infty$ the sequence $(d(f^n,v))_{n\ge 0}$ 
satisfies an integral recursion relation. Applying this to $v=\ord_\infty$
we see that the degree growth sequence $(\deg(f^n))$ satisfies such
a recursion relation.

As explained in the introduction, one motivation for the results in this
section comes from polynomial mappings of the complex plane 
$\C^2$, and more precisely
understanding the rate at which orbits are attracted to infinity.
Let us give one instance of what can be proved.
Suppose $f:\C^2\to\C^2$ is a dominant polynomial mapping
and assume that $f$ has ``low topological degree'' in the sense
that the asymptotic degree $d_\infty(f)$ is strictly larger than the topological
degree of $f$, \ie the number of preimages of a typical point.
In this case, we showed in~\cite{dyncomp} that the functions
\begin{equation*}
  \frac1{d_\infty^n}\log^+\|f^n\|
\end{equation*}
converge uniformly on compact subsets of $\C^2$ to a plurisubharmonic
function $G^+$ called the \emph{Green function} of $f$. 
Here $\|\cdot\|$ is any norm on $\C^2$ 
and we write $\log^+\|\cdot\|:=\max\{\log|\cdot|,0\}$.
This Green function is important for understanding the ergodic
properties of $f$, as explored by 
Diller, Dujardin and Guedj~\cite{DDG1,DDG2,DDG3}.
%
%
%
%
\subsection{Other ground fields}\label{S272}
Throughout this section we assumed that the ground field was
algebraically closed and of characteristic zero. Let us briefly
discuss what happens when one or more of these assumptions 
are not satisfied.

\subsection{Other ground fields}\label{S272}
Throughout the section we assumed that the ground field was
algebraically closed and of characteristic zero.

The assumption on the characteristic was used in the
proof of formula~\eqref{e156} and hence of Proposition~\ref{P117}.
The proof of the monomialization result (Theorem~\ref{T105})
also used characteristic zero.
It would be interesting to  have an argument for Theorem~C that works in
arbitrary characteristic.

On the other hand, assuming that $\charac K=0$,
the assumption that $K$ be algebraically closed is
unimportant for Theorem~C, at least for statements~(a) and~(b).
Indeed, if $K^a$ is the algebraic closure of $K$, then any polynomial
mapping $f:\A^2(K)\to\A^2(K)$ induces a polynomial
mapping $f:\A^2(K^a)\to\A^2(K^a)$.
Further, an embedding $\A^2(K)\hookrightarrow\P^2(K)$
induces an embedding $\A^2(K^a)\hookrightarrow\P^2(K^a)$
and the degree of $f^n$ is then independent of whether we work over
$K$ or $K^a$. Thus statements~(a) and~(b) of Theorem~C trivially
follow from their counterparts over an algebraically closed field
of characteristic zero.

%
%
%
%
\subsection{Notes and further references}
  The material in this section is adapted from the papers~\cite{eigenval,dyncomp}
  joint with Charles Favre,
  but with a few changes in the presentation. 
  In order to keep these lecture notes reasonably coherent, I have also 
  changed some of the notation from the original papers.
  I have also emphasized a geometric 
  approach that has some hope of being applicable in higher 
  dimensions and the presentation is streamlined to give a reasonably
  quick proof of Theorem~C.
  
  Instead of working on valuation space, it is possible to consider
  the induced dynamics on divisors on the Riemann-Zariski space.
  By this we mean the data of one divisor at infinity for each
  admissible compactification of $\A^2$ (with suitable compatibility 
  conditions when one compactification dominates another.
  See~\cite{dyncomp} for more details and~\cite{deggrowth} for 
  applications of this point of view in a  slightly different context.
%
%
%
%
%
%


\begin{thebibliography}{KKMS73}

\bibitem[Aba10]{Abate}
{\sc M.~Abate} --
``Discrete holomorphic local dynamical systems",
in \emph{Holomorphic dynamical systems},
Lecture Notes in Mathematics, vol 1998, 1--55,
Springer, 2010.

\bibitem[AM73]{AM}
{\sc S.~S.~Abhyankar and  T.~T.~Moh} --
   \emph{Newton-Puiseux expansion and generalized 
  Tschirnhausen transformation.~I,~II}.
    J. Reine Angew. Math. 
\textbf{260} (1973), 47--83, 
\textbf{261} (1973), 29--53.

\bibitem[Art66]{Artin}
{\sc M.~Artin} --
   \emph{On isolated rational singularities of surfaces}. 
   Amer. J. Math. \textbf{88} (1966), 129–136.

\bibitem[BdM09]{BdM}
{\sc M.~Baker and L.~DeMarco} --
   \emph{Preperiodic points and unlikely intersections}.
   Duke Math. J. \textbf{159} (2011), 1--29.

\bibitem[Bak64]{Bak64}
{\sc I.~N.~Baker} --
  \emph{Fixpoints of polynomials and rational functions}.
   J. London Math. Soc. \textbf{39} (1964),  615--622.

\bibitem[Bak06]{Bak06}
{\sc M.~Baker} --
 ``A lower bound for average values of dynamical Green's functions",
   Math. Res. Lett. \textbf{13} (2006),  245--257.

\bibitem[Bak08]{Bak08}
 \bysame,
``An introduction to Berkovich analytic spaces and non-Archimedean potential theory on curves",
   In \emph{$p$-adic geometry}, 123--174,
   Univ. Lecture Ser. 45.
   Amer. Math. Soc., Providence, RI, 2008.

\bibitem[Bak09]{Bak09}
 \bysame,
``A finiteness theorem for canonical heights attached to 
  rational maps over function fields",
   J. Reine Angew. Math. \textbf{626} (2009), 205--233.

\bibitem[BH05]{BH05}
{\sc M.~Baker and L.-C. Hsia} --
``Canonical heights, transfinite diameters, and polynomial dynamics",
   J. Reine Angew. Math. \textbf{585} (2005), 61--92.

\bibitem[BR06]{BR06}
{\sc M.~Baker and R.~Rumely} --
``Equidistribution of small points, rational dynamics, and potential theory",
    Ann. Inst. Fourier \textbf{56} (2006), 625--688. 

\bibitem[BR10]{BRBook}
 \bysame,    \emph{Potential theory on the Berkovich projective line}.
    Mathematical surveys and monographs, vol 159.
   American Math. Soc., 2010.

\bibitem[Bea91]{Beardon}
{\sc A.~F.~Beardon} --
   \emph{Iteration of rational functions}.
   Graduate Texts in Mathematics, 132. 
   Springer-Verlag, New York, 1991.

\bibitem[BT82]{BT1}
{\sc E.~Bedford and B.~A.~Taylor} --
 ``A new capacity for plurisubharmonic functions",
   Acta Math.  149  (1982), 1--40. 

\bibitem[BT87]{BT2}
 \bysame,
``Fine topology, Shilov boundary and $(dd^c)^n$",
    J. Funct. Anal.  72  (1987),  225--251. 

\bibitem[Ben98]{BenedettoThesis} 
{\sc R.~L.~Benedetto} --
  \emph{Fatou components in $p$-adic dynamics}. 
   Ph.D. Thesis. Brown University, 1998.
   Available at \texttt{www.cs.amherst.edu/$\sim$rlb/papers/}.

\bibitem[Ben00]{Benedetto1} 
 \bysame,
   ``$p$-adic dynamics and Sullivan's no wandering domains theorem",
   Compositio Math. \textbf{122} (2000), 281-–298.

\bibitem[Ben01a]{Benedetto2}
 \bysame,
``Reduction, dynamics, and Julia sets of rational functions", 
   J. Number Theory \textbf{86} (2001), 175--195.

\bibitem[Ben01b]{Benedetto5}
 \bysame,
``Hyperbolic maps in $p$-adic dynamics",
   Ergodic Theory Dynam. Systems \textbf{21} (2001), 1--11.

\bibitem[Ben02a]{Benedetto6}
 \bysame,
``Components and periodic points in non-Archimedean dynamics",
   Proc. London Math. Soc. \textbf{84} (2002) 231--256. 

\bibitem[Ben02b]{Benedetto3}
\bysame,  
``Examples of wandering domains in $p$-adic polynomial dynamics",
   C. R. Math. Acad. Sci. Paris \textbf{335} (2002), 615--620.

\bibitem[Ben05a]{Benedetto7}
 \bysame,
``Wandering domains and nontrivial reduction in non-Archimedean dynamics",
   Illinois J. Math. \textbf{49} (2005), 167--193.

\bibitem[Ben05b]{Benedetto4}
 \bysame,
``Heights and preperiodic points of polynomials over function fields",
   Int. Math. Res. Not. \textbf{62} (2005), 3855--3866.

\bibitem[Ben06]{Benedetto8}
 \bysame,
``Wandering domains in non-Archimedean polynomial dynamics",
   Bull. London Math. Soc. \textbf{38} (2006), 937–-950.

\bibitem[Ben10]{BenedettoNotes}
 \bysame,
``non-Archimedean dynamics in dimension one"
   Lecture notes from the 2010 Arizona Winter School,
   \texttt{http://math.arizona.edu/$\sim$swc/aws/2010/}.

\bibitem[Ber90]{BerkBook}
{\sc V.~G.~Berkovich} --
   \emph{Spectral theory and analytic
geometry over non-Archi\-medean fields}.
    Mathematical Surveys and Monographs, 33. 
   American Mathematical Society, Providence, RI, 1990.

\bibitem[Ber93]{Berkihes}
 \bysame,
   ``\'Etale cohomology for non-Archimedean analytic spaces",
   Publ. Math. Inst. Hautes \'Etudes Sci. \textbf{78} (1993), 5--161. 

\bibitem[Ber94]{BerkVan1}
 \bysame,
   ``Vanishing cycles for formal schemes",
   Invent. Math. \textbf{115} (1994), 539--571.


\bibitem[Ber99]{BerkLC1}
 \bysame,  ``Smooth $p$-adic analytic spaces are locally contractible. I",
   Invent. Math. \textbf{137} (1999), 1--84.

\bibitem[Ber04]{BerkLC2} 
 \bysame,
   ``Smooth $p$-adic analytic spaces are locally contractible. II",
   In \emph{Geometric aspects of Dwork theory}, 293-–370. 
   Walter de Gruyter and Co. KG, Berlin, 2004.

\bibitem[Ber09]{BerkHodge}
 \bysame,   ``A non-Archimedean interpretation of the weight zero
  subspaces of limit mixed Hodge structures",
   In \emph{Algebra, arithmetic, and geometry: in honor of Yu.~I.~Manin}.
   Progr. Math., vol 269, 49--67.
   Birkh\"auser, Boston, MA, 2009.

\bibitem[BdFF10]{BdFF} 
{\sc S.~Boucksom, T.~de Fernex and C.~Favre} -- 
``The volume of an isolated singularity",
 Duke Math. J. 161, Number 8 (2012), 1455--1520.

\bibitem[BFJ08a]{deggrowth}
{\sc S.~Boucksom, C.~Favre and M.~Jonsson} --
``Degree growth of meromorphic surface maps",
   Duke Math. J. \textbf{141} (2008), 519--538

\bibitem[BFJ08b]{hiro} 
\bysame,
``Valuations and plurisubharmonic singularities",
   Publ. Res. Inst. Math. Sci. \textbf{44} (2008), 449--494. 

\bibitem[BFJ12]{siminag} 
\bysame,
``Singular semipositive metrics in non-Archimedean geometry",
\texttt{arXiv:1201.0187}. To appear in J. Algebraic Geom. 

\bibitem[BFJ14]{nama}
\bysame,
``Solution to a non-Archimedean Monge-Amp{\`e}re equation",
J. Amer. Math. Soc., electronically published on May 22, 2014.

\bibitem[BGR84]{BGR}
{\sc S.~Bosch, U.~G\"untzer and R.~Remmert} --
  \emph{Non-Archimedean Analysis}.
   Springer-Verlag, Berlin, Heidelberg, 1994.

\bibitem[BD01]{BriendDuval}
{\sc J.-Y. Briend and J. Duval} --
``Deux caract{\'e}risations de la mesure d'{\'e}quilibre 
d'un endomorphisme de ${\rm P}\sp k(\bf C)$",
    Publ. Math. Inst. Hautes {\'E}tudes Sci. \textbf{93} (2001), 145--159. 

\bibitem[Bro65]{Brolin}
{\sc H.~Brolin} --
``Invariant sets under iteration of rational functions",
   Ark. Mat. \textbf{6}\,(1965), 103--144.

\bibitem[CPR02]{CPR1}
{\sc A.~Campillo, O.~Piltant and A.~Reguera} --
 ``Cones of curves and of line bundles on surfaces 
  associated with curves having one place at infinity",
   Proc. London Math. Soc. \textbf{84} (2002), 559--580.

\bibitem[CPR05]{CPR2}
\bysame,
``Cones of curves and of line bundles at infinity",
   J. Algebra \textbf{293} (2005), 513--542.

\bibitem[CG93]{CG}
{\sc L.~Carleson and T.~Gamelin} --
   \emph{Complex dynamics}.
    Springer-Verlag, New York, 1993.

\bibitem[CL06]{CL}
{\sc A.~Chambert-Loir} --
 ``Mesures et {\'e}quidistribution sur les espaces de Berkovich",
   J. Reine Angew. Math. \textbf{595} (2006), 215-–235.

\bibitem[Con08]{ConradNotes}
{\sc B.~Conrad} --
 ``Several approaches to non-archimedean geometry",
   In \emph{$p$-adic Geometry (Lectures from the 2007 Arizona Winter School)}.
   AMS University Lecture Series, volume 45.
   Amer. Math. Soc., Providence, RI, 2008.

\bibitem[CLM07]{CLM}
{\sc T.~Coulbois, A.~Hilion, and M.~Lustig} --
  ``Non-unique ergodicity, observers{'} topology 
  and the dual algebraic lamination for $\mathbb{R}$-trees",
   Illinois J. Math. \textbf{51} (2007), 897-–911.

\bibitem[DDG1]{DDG1}
{\sc J.~Diller, R.~Dujardin and V.~Guedj} --
``Dynamics of meromorphic maps with small 
topological degree I: from cohomology to currents",
   Indiana Univ. Math. J. \textbf{59} (2010), 521--562.

\bibitem[DDG2]{DDG2}
\bysame,
``Dynamics of meromorphic maps with small topological degree II: Energy and invariant measure",
Comment. Math. Helv. 86 (2011), pp. 277--316.

\bibitem[DDG3]{DDG3}
\bysame,
``Dynamics of meromorphic maps with small topological degree III: geometric currents and ergodic theory",
   Ann. Sci. {\'E}cole Norm. Sup. \textbf{43} (2010), 235--278.

\bibitem[DS08]{DSAENS}
{\sc T.-C.~Dinh and N.~Sibony} --
``Equidistribution towards the Green current for
  holomorphic maps",
   Ann. Sci. {\'E}cole Norm. Sup. \textbf{41} (2008), 307--336.

\bibitem[Dub11]{Dubouloz}
{\sc A. Dubouloz} --
   Personal communication.

\bibitem[ELS03]{ELS}
{\sc L.~Ein, R.~Lazarsfeld, and K.~E.~Smith} --
``Uniform approximation 
of Abhyankar valuations in smooth function fields",
   Amer. J. Math. \textbf{125} (2003), 409--440.

\bibitem[Fab09]{Faber09}
{\sc X.~Faber} --
 ``Equidistribution of dynamically small subvarieties 
  over the function field of a curve",
   Acta Arith. \textbf{137} (2009), 345--389.

\bibitem[Fab13a]{Faber1}
\bysame,
``Topology and geometry of the Berkovich 
  ramification locus for rational functions",
  Manuscripta Math. \textbf{142} (2013), 439-–474. 

\bibitem[Fab13b]{Faber2}
\bysame,
``Topology and geometry of the Berkovich 
  ramification locus for rational functions, II",
  Math. Ann. \textbf{356} (2013), 819--844.

\bibitem[Fab14]{Faber3}
\bysame,
``Rational Functions with a Unique Critical Point",
Int. Math. Res. Not. IMRN 2014, no. 3, 681--699. 

\bibitem[Fav05]{FavreHab}
{\sc C.~Favre} --
   \emph{Arbres r\'eels et espaces de valuations}.
   Th\`ese d'habilitation, 2005.

\bibitem[FJ03]{brolin}
{\sc C.~Favre and M.~Jonsson} --
``Brolin's theorem for curves in two complex dimensions",
   Ann. Inst. Fourier \textbf{53} (2003), 1461--1501.

\bibitem[FJ04]{valtree}
\bysame, 
 \emph{The valuative tree}.
   Lecture Notes in Mathematics, vol 1853.
   Springer, 2004.

\bibitem[FJ05a]{pshsing}
\bysame,
``Valuative analysis of planar pluri\-sub\-harmonic functions",
    Invent. Math.  162  (2005),  no. 2, 271--311. 

\bibitem[FJ05b]{valmul}
\bysame,
``Valuations and multiplier ideals",
   J. Amer. Math. Soc, \textbf{18} (2005), 655--684.

\bibitem[FJ07]{eigenval}
\bysame,
``Eigenvaluations",
   Ann. Sci. {\'E}cole Norm. Sup. \textbf{40} (2007), 309--349. 

\bibitem[FJ11]{dyncomp}
\bysame,
``Dynamical compactifications of $\C^2$",
   Ann.\ of Math. \textbf{173} (2011), 211--249.

\bibitem[FKT11]{FKT}
{\sc C.~Favre, J.~Kiwi and E.~Trucco} --
 ``A non-archimedean Montel's theorem",
 Compositio 148 (2012), 966–-990.

\bibitem[FR04]{FR3}
{\sc C.~Favre and J.~Rivera-Letelier} --
 ``Th\'eor\`eme d'\'equidistribution de Brolin en dynamique $p$-adique",
   C. R. Math. Acad. Sci. Paris \textbf{339} (2004), 271--276.

\bibitem[FR06]{FR1}
C.~Favre and J.~Rivera-Letelier.
``{\'E}quidistribution quantitative des points de petite 
  hauteur sur la droite projective",
   Math. Ann. \textbf{335} (2006), 311-–361.

\bibitem[FR10]{FR2}
\bysame,
``Th\'eorie ergodique des fractions rationnelles
sur un corps ultram\'etrique",
   Proc. London Math. Soc. \textbf{100} (2010), 116--154.

\bibitem[FLM83]{FLM}
{\sc A.~Freire, A.~Lopez, and R.~Ma{\~n}{\'e}} --
``An invariant measure for rational maps",
   Bol. Soc. Bras. Mat. \textbf{14}\,(1983), 45--62.

\bibitem[Fol99]{Folland}
{\sc G.~B.~Folland} --
   \emph{Real analysis: modern techniques and their
  applications, second edition}.
   Pure and applied mathematics (New York). 
A Wiley-Interscience Publication.
   John Wiley \& Sons, Inc., New York, 1999.

\bibitem[Fre93]{Fremlin}
{\sc D.~H.~Fremlin} --
``Real-valued measurable cardinals",
   In \emph{Set theory of the reals (Ramat Gan, 1991)}.
   Israel Math. Conf. Proc., vol 6, 151--304.
   Bar-Ilan Univ, Ramat Gan, 1993.
   See also \texttt{www.essex.ac.uk/maths/people/fremlin/papers.htm}.

\bibitem[Ful93]{Fulton}
{\sc  W.~Fulton } --
   \emph{Introduction to toric varieties}.
   Annals of Mathematics Studies, 131. 
Princeton University Press, Princeton, NJ, 1993.

\bibitem[GM04]{GM1}
{\sc C.~Galindo and F.~Monserrat} --
``On the cone of curves and of line bundles of a rational surface",
   Internat. J. Math. \textbf{15} (2004), 393–-407.

\bibitem[GM05]{GM2}
\bysame,
``The cone of curves associated to a plane configuration",
   Comment. Math. Helv. \textbf{80} (2005), 75-–93.

\bibitem[GTZ08]{GTZ}
{\sc D.~Ghioca, T.~J.~Tucker and M.~E.~Zieve} --
  {Linear relations between polynomial orbits}.
   Duke Math. J. 161 (2012), 1379–1410.

\bibitem[GH90]{GH}
{\sc {\'E}.~Ghys and P.~de la Harpe} --
  {Sur les groupes hyperboliques d'apr{\`e}s Mikhael Gromov}.
   Progress in Mathematics, vol 83.
   Birkh\"auser, Boston, 1990.

\bibitem[Gra07]{Granja}
{\sc A.~Granja} --
``The valuative tree of a two-dimensional regular local ring",
   Math. Res. Lett. \textbf{14} (2007), 19–34. 

\bibitem[Gub08]{Gubler}
{\sc W.~Gubler} --
``Equidistribution over function fields",
   Manuscripta Math. \textbf{127} (2008), 485--510. 

\bibitem[Har77]{Hartshorne}
{\sc R.~Hartshorne} --
   \emph{Algebraic geometry}.
   Graduate Texts in Mathematics, No. 52. 
   Springer-Verlag, New York-Heidelberg, 1977.

\bibitem[Hsi00]{Hsia}
{\sc L.-C.~Hsia} --
``Closure of periodic points over a non-Archimedean field",
    J. London Math. Soc. \textbf{62} (2000), 685-–700.

\bibitem[HS01]{HuSw}
{\sc R.~H\"ubl and I.~Swanson} --
``Discrete valuations centered on local domains",
   J. Pure Appl. Algebra 161 (2001), 145--166.

\bibitem[Izu85]{Izumi}
{\sc S.~Izumi} --
``A measure of integrity for local analytic algebras",
   Publ. RIMS Kyoto Univ. \textbf{21} (1985), 719--735.

\bibitem[Jec03]{Jec03}
{\sc T.~Jech} --
  \emph{Set theory}.
   The third millenium edition, revised and expanded.
   Springer Monographs in Mathematics.
   Springer-Verlag, Berlin, 2003.

\bibitem[JM12]{graded}
{\sc M.~Jonsson and M.~Musta\c{t}\u{a}} --
``Valuations and asymptotic invariants for sequences of ideals",
Ann. Inst. Fourier \textbf{62} (2012), 2145--2209.

\bibitem[Ked10]{Ked2}
{\sc K.~Kedlaya} --
``Good formal structures for flat meromorphic connections, I: surfaces",
   Duke Math. J. \textbf{154} (2010), 343--418.

\bibitem[Ked11a]{Ked3}
\bysame,
``Good formal structures for flat meromorphic connections, II: excellent schemes",
   J. Amer. Math. Soc. \textbf{24} (2011), 183--229. 

\bibitem[Ked11b]{Ked1}
\bysame,
``Semistable reduction for overconvergent F-isocrystals, IV: Local semistable
reduction at nonmonomial valuations",
   Compos. Math. \textbf{147} (2011), 467--523.

\bibitem[KKMS73]{KKMS}
{\sc G.~Kempf, F.~F.~Knudsen, D.~Mumford and B.~Saint-Donat} --
   \emph{Toroidal embeddings. I}. 
   Lecture Notes in Mathematics. Vol. 339. 
   Springer-Verlag, Berlin, 1973.

\bibitem[Kis02]{Kishimoto}
{\sc T.~Kishimoto} --
``A new proof of a theorem of Ramanujan-Morrow",
   J. Math. Kyoto \textbf{42} (2002), 117--139.

\bibitem[Kiw06]{Kiwi1}
{\sc J.~Kiwi} --
``Puiseux series polynomial dynamics and 
  iteration of complex cubic polynomials",
   Ann. Inst. Fourier \textbf{56} (2006), 1337--1404.

\bibitem[Kiw14]{Kiwi2}
\bysame,
``Puiseux series dynamics of quadratic rational maps",
Israel J. Math. \textbf{201} (2014), 631--700.

\bibitem[Kol97]{Kollar}
{\sc J.~Koll\'ar} --
   \emph{Singularities of pairs}.
   Proc. Symp. Pure Math., 62, Part 1, 
   AMS, Providence, RI,  1997.

\bibitem[KM98]{KollarMori}
{\sc J.~Koll\'ar and S.~Mori} --
  \emph{Birational geometry of algebraic varieties}.
   Cambridge Tracts in Mathematics, 134. 
   Cambridge University Press, Cambridge, 1998

\bibitem[Lan02]{LangAlgebra}
{\sc S.~Lang} --
  \emph{Algebra}.
   Revised third edition. 
   Graduate Texts in Mathematics, 211. 
   Springer-Verlag, New York, 2002.

\bibitem[Lip69]{Lipman}
{\sc J.~Lipman} --
``Rational singularities with applications to
  algebraic surfaces and unique factorization",
   Publ. Math. Inst. Hautes \'Etudes Sci. \textbf{36} (1969), 195--279. 

\bibitem[Lyu83]{Lyubich}
{\sc M.~Lyubich} --
``Entropy properties of rational endomorphisms of the 
    Riemann sphere",
   Ergodic Theory Dynam. Systems \textbf{3} (1983), 351--385.

\bibitem[Mac36]{MacLane}
{\sc S.~MacLane} --
``A construction for prime ideals as absolute values of an algebraic field",
   Duke M. J. \textbf{2} (1936), 363--395.

\bibitem[Mat89]{Matsumura}
{\sc H.~Matsumura} --
  \emph{Commutative Ring Theory}.
   Cambridge Studies in Advanced Mathematics, 8.
   Cambridge University Press, Cambridge, 1989.

\bibitem[Mil06]{Milnor}
{\sc J.~Milnor} --
   \emph{Dynamics in one complex variable}.
   Annals of Mathematics Studies, 160.
   Princeton University Press, Princeton, NJ, 2006.

\bibitem[Mon07]{Monserrat}
{\sc F.~Monserrat} --
``Curves having one place at infinity and linear systems on rational surfaces",
   J. Pure Appl. Algebra \textbf{211} (2007), 685–-701.

\bibitem[Mor73]{Morrow}
{\sc J.~A.~Morrow} --
``Minimal normal compactifications of $\mathbf{C}^2$",
   In \emph{Complex analysis, 1972} (Proc. Conf., Rice Univ. Houston, Tex., 1972.
   Rice Univ. Studies \textbf{59} (1973) 97--112.

\bibitem[MS95]{MortonSilverman}
{\sc P.~Morton and J.~H.~Silverman} --
 ``Periodic points, multiplicities, and dynamical units",
   J. Reine Angew. Math. \textbf{461} (1995), 81--122.

\bibitem[Oda88]{Oda}
{\sc T.~Oda} --
  \emph{Convex bodies and algebraic geometry. 
  An introduction to the theory of toric varieties}.
   Ergebnisse der Mathematik und ihrer Grenzgebiete (3), 15.
   Springer-Verlag, Berlin, 1988.

\bibitem[Oku11a]{Oku11a}
{\sc Y.~Okuyama}--
``Repelling periodic points and logarithmic equidistribution 
  in non-archimedean dynamics",
Acta Arith.152, No. 3 (2012), 267--277. 

\bibitem[Oku11b]{Oku11b}
\bysame,
``Feketeness, equidistribution and critical orbits 
  in non-archimedean dynamics",
  Math. Z (2012), DOI:10.1007/s00209-012-1032-x.

\bibitem[Par11]{Parra}
{\sc M.~R.~Parra} --
``The Jacobian cocycle and equidistribution towards the Green current",
   \texttt{arXiv:1103.4633}.

\bibitem[Pop11]{Popescu-Pampu}
{\sc P.~Popescu-Pampu} --
``Le cerf-volant d'une constellation",
Enseign. Math. 57 (2011), 303--347.

\bibitem[PST09]{PST}
{\sc C.~Petsche, L~Szpiro and M.~Tepper} --
``Isotriviality is equivalent to potential good reduction for 
    endomorphisms of $\mathbb{P}^N$ over function fields",
   J. Algebra \textbf{322} (2009), 3345--3365.

\bibitem[Ree89]{Rees}
{\sc D.~Rees} --
 ``Izumi's theorem",
   In Commutative algebra (Berkeley, CA, 1987), 407--416.
   Math. Sci. Res. Inst. Publ., 15, Springer, New York 1989. 

\bibitem[Riv03a]{Rivera1}
{\sc J. Rivera-Letelier} --
  ``Dynamique des fonctions rationnelles sur des corps locaux",
   Ast{\'e}risque \textbf{287} (2003), 147--230.

\bibitem[Riv03b]{Rivera2}
\bysame,
``Espace hyperbolique $p$-adique et dynamique des 
  fonctions rationnelles", 
   Compositio Math. \textbf{138} (2003), 199--231.

\bibitem[Riv04]{Rivera3}
\bysame,
   ``Sur la structure des ensembles de Fatou $p$-adiques",
   Available at \texttt{arXiv:math/0412180}.

\bibitem[Riv05]{Rivera4}
\bysame,
``Points p{\'e}riodiques des fonctions rationnelles dans 
l'espace hyperbolique $p$-adique",
   Comment. Math. Helv. \textbf{80} (2005), 593--629.

\bibitem[Rob00]{Robert}
{\sc A.~Robert} --
  \emph{A course in $p$-adic analysis}.
   Graduate Texts in Mathematics, 198.
   Springer-Verlag, New York, 2000.

\bibitem[Rug12]{Ruggiero}
{\sc M.~Ruggiero} --
 ``Rigidification of holomorphic germs with non-invertible differential",
 Michigan Math. J. \textbf{61} (2012), 161--185.

\bibitem[SU04]{UmSh}
{\sc I.~P.~Shestakov and U.~U.~Umirbaev} --
``The tame and the wild automorphisms of 
  polynomial rings in three variables",
   J. Amer. Math. Soc, \textbf{17} (2004), 197--227.

\bibitem[Sib99]{Sibony}
{\sc N.~Sibony} --
``Dynamique des applications rationnelles de $\mathbf{P}\sp k$",
   In \emph{Dynamique et g{\'e}om{\'e}trie complexes (Lyon, 1997)}, 
   Panor. Synth{\`e}ses, 8, 97--185.
   Soc. Math. France, Paris, 1999.

\bibitem[Sil07]{SilvBook}
{\sc J.~H.~Silverman} --
   \emph{The arithmetic of dynamical systems}.
   Graduate Texts in Mathematics, volume 241.
   Springer, New York, 2007.

\bibitem[Sil10]{SilvNotes}
\bysame,
``Lecture notes on arithmetic dynamics",
   Lecture notes from the 2010 Arizona Winter School.
   \texttt{math.arizona.edu/$\sim$swc/aws/10/}

\bibitem[Spi90]{spiv}
{\sc M.~Spivakovsky} --
``Valuations in function fields of surfaces",
   Amer. J. Math. \textbf{112} (1990), 107--156.

\bibitem[Suz74]{Suzuki}
{\sc M.~Suzuki} --
  ``Propri\'et\'es topologiques des polyn\^omes de 
    deux variables complexes, et automorphismes alg\'ebriques de
    l'espace $\mathbf{C}^2$",
   J. Math. Soc. Japan, 26 (1974), 241--257.

\bibitem[Tem10a]{temkinstable}
{\sc M.~Temkin} --
``Stable modification of relative curves",
    J.~Algebraic Geometry \textbf{19} (2010), 603--677. 

\bibitem[Tem10b]{temkinnotes}
\bysame,
  ``Introduction to Berkovich analytic spaces",
     \texttt{arXiv:math/1010.2235v1}.
     To appear in ``Berkovich Spaces and Applications",
     Springer Lecture Notes in Mathematics.

\bibitem[Thu05]{ThuillierThesis} 
{\sc A.~Thuillier} --
   \emph{Th\'eorie du potentiel sur les courbes
 en g\'eom\'etrie analytique non archim\'edienne.
 Applications \`a la th\'eorie d'Arakelov}.
   Ph.D. thesis, University of Rennes, 2005. 
   \texttt{tel.archives-ouvertes.fr/tel-00010990}.

\bibitem[Thu07]{ThuillierStepanov}
\bysame,
 ``G\'eom\'etrie toro\"{\i}dale et g\'eom\'etrie analytique non archim\'edienne. Application au type d'homotopie de certains sch\'emas formels",
   Manuscripta Math.  \textbf{123} (2007),  no. 4, 381--451.

\bibitem[Tou72]{Tougeron}
{\sc J.-C.~Tougeron} --
   \emph{Id\'eaux de fonctions diff\'erentiables}.
   Ergebnisse der Mathematik und ihrer Grenzgebiete, Band 71. 
   Springer-Verlag, Berlin-New York, 1972.

\bibitem[Tru09]{Trucco}
{\sc E.~Trucco} --
  ``Wandering Fatou components and algebraic Julia sets",
 To appear in Bull.\ de la SMF,  \texttt{arXiv:0909.4528v2}.

\bibitem[Vaq00]{Vaquie1}
{\sc M.~Vaqui\'e} --
``Valuations",
   In \emph{Resolution of singularities (Obergurgl, 1997)}.
   Progr. Math., 181,  539--590.
   Birkha\"user, Basel, 2000. 

\bibitem[Vaq07]{Vaquie2}
\bysame,
``Extension d'une valuation",
 Trans. Amer. Math. Soc. \textbf{359} (2007), 3439--3481.

\bibitem[Yua08]{Yuan}
{\sc X.~Yuan} --
``Big line bundles over arithmetic varieties",
   Invent. Math. \textbf{173} (2008), 603--649.

\bibitem[YZ09a]{YZa}
{\sc X.~Yuan and S.-W.~Zhang} --
``Calabi-Yau theorem and algebraic dynamics",
    Preprint 
\texttt{www.math.columbia.edu/$\sim$szhang/papers/Preprints.htm}. 

\bibitem[YZ09b]{YZb}
\bysame,
 ``Small points and Berkovich metrics",
    Preprint, 2009, available at 
\texttt{www.math.columbia.edu/$\sim$szhang/papers/Preprints.htm}. 

\bibitem[ZS75]{ZS}
{\sc O.~Zariski and P.~Samuel} --
   \emph{Commutative algebra}.
   Vol 2, Graduate Texts in Mathematics 29.
   Springer, 1975.

\end{thebibliography}
\end{document}